\documentclass[letterpaper]{book}

\usepackage[utf8]{inputenc}
\usepackage[backend=biber, style=alphabetic, sorting=nyt, doi=false, eprint=false, isbn=false]{biblatex}
  \usepackage{amsmath,amssymb, amsthm}
  \usepackage{mathrsfs}
  \usepackage[mathscr]{euscript}
  \usepackage{stmaryrd}
  \SetSymbolFont{stmry}{bold}{U}{stmry}{m}{n}
  \usepackage[top=30truemm, bottom=40truemm, left=35truemm, right=35truemm]{geometry}
  \usepackage{xcolor}
  \usepackage[all]{xy}
  \usepackage{tikz-cd}
  \usepackage{color}
  \usepackage{mathtools}
  \usepackage{graphicx}
  \usepackage{scalerel}
  \usepackage{enumitem}  
  \usepackage[colorlinks,final,hyperindex]{hyperref}
  \usepackage[noabbrev,capitalize]{cleveref}
  \usepackage{etoolbox}
  \usepackage{stackengine,scalerel}
  \usepackage{calc}
  \usepackage{adjustbox}
  \usepackage{todonotes}
  \usepackage{aliascnt}

\newtheorem{theorem}{Theorem}[section]
\newtheorem*{theorem*}{Theorem}

\newtheorem{manualtheoreminner}{Theorem}
\newenvironment{manualtheorem}[1]{%
  \renewcommand\themanualtheoreminner{#1}%
  \manualtheoreminner
}{\endmanualtheoreminner}

\newaliascnt{proposition}{theorem}
\newtheorem{proposition}[proposition]{Proposition}
\aliascntresetthe{proposition}

\newaliascnt{corollary}{theorem}
\newtheorem{corollary}[corollary]{Corollary}
\aliascntresetthe{corollary}

\newaliascnt{lemma}{theorem}
\newtheorem{lemma}[lemma]{Lemma}
\aliascntresetthe{lemma}

\newaliascnt{hypothesis}{theorem}
\newtheorem{hypothesis}[hypothesis]{Hypothesis}
\aliascntresetthe{hypothesis}

\newtheorem*{proposition*}{Proposition}

\theoremstyle{definition}

\newaliascnt{definition}{theorem}
\newtheorem{definition}[definition]{Definition}
\aliascntresetthe{definition}

\newaliascnt{variant}{theorem}

\aliascntresetthe{variant}

\newaliascnt{predefinition}{theorem}
\newtheorem{predefinition}[predefinition]{Pre-definition}
\aliascntresetthe{predefinition}

\newaliascnt{example}{theorem}
\newtheorem{example}[example]{Example}
\aliascntresetthe{example}

\newaliascnt{notation}{theorem}
\newtheorem{notation}[notation]{Notation}
\aliascntresetthe{notation}

\newaliascnt{question}{theorem}

\aliascntresetthe{question}

\newaliascnt{note}{theorem}

\aliascntresetthe{note}

\newaliascnt{remark}{theorem}
\newtheorem{remark}[remark]{Remark}
\aliascntresetthe{remark}
\newtheorem*{remark*}{Remark}

\newaliascnt{warning}{theorem}
\newtheorem{warning}[warning]{Warning}
\aliascntresetthe{warning}

\newaliascnt{construction}{theorem}

\aliascntresetthe{construction}

\crefname{theorem}{theorem}{theorems}
\Crefname{theorem}{Theorem}{Theorems}
\crefname{proposition}{proposition}{propositions}
\Crefname{proposition}{Proposition}{Propositions}
\crefname{corollary}{corollary}{corollaries}
\Crefname{corollary}{Corollary}{Corollaries}
\crefname{lemma}{lemma}{lemmas}
\Crefname{lemma}{Lemma}{Lemmas}
\crefname{hypothesis}{hypothesis}{hypotheses}
\Crefname{hypothesis}{Hypothesis}{Hypotheses}
\crefname{definition}{definition}{definitions}
\Crefname{definition}{Definition}{Definitions}
\crefname{variant}{variant}{variants}
\Crefname{variant}{Variant}{Variants}
\crefname{predefinition}{pre-definition}{pre-definitions}
\Crefname{predefinition}{Pre-definition}{Pre-definitions}
\crefname{example}{example}{examples}
\Crefname{example}{Example}{Examples}
\crefname{notation}{notation}{notations}
\Crefname{notation}{Notation}{Notations}
\crefname{conjecture}{conjecture}{conjectures}
\Crefname{conjecture}{Conjecture}{Conjectures}
\crefname{Assumption}{assumption}{assumptions}
\Crefname{Assumption}{Assumption}{Assumptions}
\crefname{question}{question}{questions}
\Crefname{question}{Question}{Questions}
\crefname{note}{note}{notes}
\Crefname{note}{Note}{Notes}
\crefname{remark}{remark}{remarks}
\Crefname{remark}{Remark}{Remarks}
\crefname{warning}{warning}{warnings}
\Crefname{warning}{Warning}{Warnings}
\crefname{construction}{construction}{constructions}
\Crefname{construction}{Construction}{Constructions}

\mathchardef\mhy="2D
\newcommand{\blank}{-} %
\newcommand{\eps}{\varepsilon}
\DeclareFontFamily{U}{min}{}
\DeclareFontShape{U}{min}{m}{n}{<-> udmj30}{}
\newcommand{\yo}{{\!\text{\usefont{U}{min}{m}{n}\symbol{'210}}\!}}
\newcommand{\vS}{\vec{S}}
\def\slashedrightarrow{\relbar\joinrel\mapstochar\joinrel\rightarrow}

\renewcommand*{\do}[1]{%
\expandafter\newcommand \csname c#1\endcsname{\mathcal{#1}}}
\docsvlist{A,B,C,D,E,F,G,H,I,J,K,L,M,N,O,P,Q,R,S,T,U,V,W,X,Y,Z}
\renewcommand*{\do}[1]{%
\expandafter\newcommand \csname e#1\endcsname{\mathscr{#1}}}
\docsvlist{A,B,C,D,E,F,G,H,I,J,K,L,M,N,O,P,Q,R,S,T,U,V,W,X,Y,Z}
\renewcommand*{\do}[1]{%
\expandafter\newcommand \csname s#1\endcsname{\mathsf{#1}}}
\docsvlist{A,B,C,D,E,F,G,H,I,J,K,L,M,N,O,P,Q,R,S,T,U,V,W,X,Y,Z}
\renewcommand*{\do}[1]{%
\expandafter\newcommand \csname r#1\endcsname{\mathrm{#1}}}
\docsvlist{A,B,C,D,E,F,G,H,I,J,K,L,M,N,O,P,Q,R,S,T,U,V,W,X,Y,Z,f,g,h,s,u,v}
\renewcommand*{\do}[1]{%
\expandafter\newcommand \csname #1#1\endcsname{\mathbb{#1}}}
\docsvlist{B,C,D,E,F,G,H,I,J,K,L,M,N,O,P,Q,R,T,U,V,W,X,Y,Z}
\renewcommand{\AA}{\mathbb{A}}
\renewcommand{\SS}{\mathbb{S}}
\renewcommand*{\do}[1]{%
\expandafter\newcommand \csname b#1\endcsname{\mathbf{#1}}}
\docsvlist{A,B,C,D,E,F,G,H,I,J,K,L,M,N,O,P,Q,R,S,T,U,V,W,X,Y,Z}
\renewcommand*{\do}[1]{%
\expandafter\newcommand \csname f#1\endcsname{\mathfrak{#1}}}
\docsvlist{A,B,C,D,E,F,G,H,I,J,K,L,M,N,O,P,Q,R,S,T,U,V,W,X,Y,Z, p, q}

\newcommand{\abs}[1]{{\lvert#1 \rvert}}

\renewcommand*{\do}[1]{%
\expandafter\newcommand \csname #1\endcsname{\mathsf{#1}}}
\docsvlist{
Set, FinSet, Poset, sSet, msSet, PreComp, cSet, mcSet, Comic, Cptd,
Fin, Cat, CatSp, CatPSp, PrCatSp, CAT, LinCat, Gph,
PSh, Shv, Exc, Prof,
Mon, CMon, Ab, Grp, Sp, Alg, Algbrd, Algd, CAlg, Rig, CRig, Ring, CRing, Mod, LMod, RMod, grMod, Ch, adCh, dCh, Tw, coCart, Cart, ADC, Span, Gaunt, Adj, Bord, FinVect, Vect, Ani, Hor, Ver, Lat, GrayCat, Mfld, Exit, Strat}
\newcommand{\BMod}[2][]{{}_{#1}\mathsf{BMod}_{#2}}
\renewcommand{\Pr}{\mathsf{Pr}}
\newcommand{\PrL}{{\mathsf{Pr}^\mathsf{L}}}
\newcommand{\PrR}{{\mathsf{Pr}^\mathsf{R}}}
\newcommand{\Spx}{{\boldsymbol{\Delta}}}

\DeclareSymbolFont{bbold}{U}{bbold}{m}{n}
\DeclareMathSymbol{\Ori}{\mathord}{bbold}{"01}

\newcommand{\strCat}{\Cat^\mathrm{str}}

\newcommand{\SMC}{\mathsf{SMCat}}

\renewcommand*{\do}[1]{%
\expandafter\newcommand \csname #1\endcsname{\mathrm{#1}}}
\docsvlist{op, co, coop, lax, oplax, cn, nc, gp, st, fin, lfin, fpqc, fppf, uns, An, IAn, LAn, RAn, Cof, Fib, rev, strict, inj, cg, Assoc, Triv, can, Ste, adj, ladj, radj, dual, fr, sfr, un, Fam, comp, conj, const, Top, PL, sm}

\renewcommand*{\do}[1]{%
\expandafter\DeclareMathOperator \csname #1\endcsname{#1}}
\docsvlist{id, pr, ev, coev, ob, Hom, Nat, Map, el, core, Sym, Sing, Aut, Spec, Lan, Tor, Ext, THH, Tot, St, Un, cell, ho, fib, cof, cyl, cofib, Free, Nm, Pic, Ind, Env, Sq, res, dom, cod, Th, End, LEnd, REnd, Fun, LFun, RFun}

\DeclareMathOperator*{\colim}{colim}
\DeclareMathOperator{\im}{Im}

\newcommand{\cube}{\Box}

\newcommand{\Cube}{{\scalerel*{\mbox{$\square$\hspace{-0.76em}\scalebox{0.85}[1]{$\square$}}\hspace{0.1em}}{\Spx}}}

\newlength\shlength
\newcommand\vvec[2][0]{\ThisStyle{\setlength\shlength{#1\LMpt}
  \stackengine{-5.6\LMpt}{$\SavedStyle#2$}{\smash{$\kern\shlength
    \stackengine{\dimexpr 1.3pt+6.25\LMpt}{$\SavedStyle\mathchar"017E$}
      {\rule{\widthof{$\SavedStyle#2$}}{\dimexpr.1pt+.5\LMpt}\kern.4\LMpt}{O}{r}{F}{F}{L}\kern-\shlength$}}
      {O}{c}{F}{T}{S}}}

\DeclareMathOperator{\rcof}{\overrightarrow{\mathrm{cof}}}
\DeclareMathOperator{\rfib}{\overrightarrow{\mathrm{fib}}}
\DeclareMathOperator{\rcyl}{\overrightarrow{\mathrm{cyl}}}
\DeclareMathOperator{\rcone}{\overrightarrow{\mathrm{cone}}}
\DeclareMathOperator{\rpath}{\overrightarrow{\mathrm{path}}}
\DeclareMathOperator{\lcof}{\overleftarrow{\mathrm{cof}}}
\DeclareMathOperator{\lfib}{\overleftarrow{\mathrm{fib}}}
\DeclareMathOperator{\lcyl}{\overleftarrow{\mathrm{cyl}}}
\DeclareMathOperator{\lcone}{\overleftarrow{\mathrm{cone}}}\DeclareMathOperator{\lpath}{\overleftarrow{\mathrm{path}}}

\newcommand{\ramalg}{\mathbin{\overrightarrow{\amalg}}}
\newcommand{\lamalg}{\mathbin{\overleftarrow{\amalg}}}

\makeatletter
\newcommand{\xRrightarrow}[2][]{\ext@arrow 0359\Rrightarrowfill@{#1}{#2}}
\newcommand{\Rrightarrowfill@}{\arrowfill@\equiv\equiv\Rrightarrow}
\newcommand{\xLleftarrow}[2][]{\ext@arrow 3095\Lleftarrowfill@{#1}{#2}}
\newcommand{\Lleftarrowfill@}{\arrowfill@\Lleftarrow\equiv\equiv}
\makeatother

\newlength\squareheight
  \tikzset{Rightarrow/.style={double equal sign distance,>={Implies},->},
  Rightarrow wide/.style={double, double distance=2.5pt, >={Implies},->},
  triple/.style={-, preaction={draw, Rightarrow wide}},
  quadruple/.style={
    preaction={draw, double, double distance=4pt, -, >={Implies}, ->, shorten >=22.5pt},                 %
    -, double, double distance=1pt, shorten >=25pt, shorten <=25pt
  }
  }

\title{The Algebra of Categorical Spectra}
\author{Naruki Masuda}
\date{July 26, 2024 (revised in April 2026)}
\begin{document}
\maketitle
\tableofcontents

\chapter{Introduction}
This thesis focuses on developing the fundamental theory of \emph{categorical spectra}, a higher-categorical generalization of spectra. 
Let $\sS$ denote the $(\infty, 1)$-category of $\infty$-groupoids (a.k.a.\ anima, weak homotopy types, or spaces) and $\sS_{\ast}$ be the category of pointed $\infty$-groupoids. 
Recall that the category of \emph{spectra} is defined as the limit along the loop functor on $\sS_\ast$: 
\[\Sp\coloneqq \lim(\cdots \xrightarrow{\Omega} \sS_\ast\xrightarrow{\Omega} \sS_\ast).\]
In other words, a spectrum $X$ is a sequence of pointed $\infty$-groupoids $(X_n)_{n\in \NN}$ equipped with equivalences $X_n\xrightarrow{\sim}\Omega X_{n+1}$. 
For $0\leq n \leq \infty$, let $n\Cat$ denote the $(\infty, 1)$-category of (small) $(\infty, n)$-categories. That is, we set $0\Cat = \sS$ and inductively define $(n+1)\Cat = (n\Cat)\mhy\Cat$ as the $(\infty, 1)$-category of $n\Cat$-enriched categories. We also set $\infty\Cat=\lim_n n\Cat$, i.e., an $(\infty, \infty)$-category $X$ is a compatible collection of the underlying $(\infty, n)$-categories $X^{\leq n}$ obtained by discarding higher noninvertible cells. An essential feature of $\infty\Cat$ is that it is a fixed point of enrichment: $(\infty\Cat)\mhy\Cat \simeq \infty\Cat$. 
The $(\infty, 1)$-category $\CatSp$ of categorical spectra is defined by replacing $\sS$ in the definition of $\Sp$ with $\infty\Cat$:
\[\CatSp\coloneqq \lim(\cdots \xrightarrow{\Omega}\infty\Cat_\ast\xrightarrow{\Omega} \infty\Cat_\ast),\]
so a categorical spectrum $X$ is a sequence $(X_n)_{n\in\NN}$ of pointed $(\infty, \infty)$-categories equipped with equivalences $X_n\xrightarrow{\sim} \Omega X_{n+1}$. 
Here, the loop of a pointed $(\infty, \infty)$-category means the $(\infty, \infty)$-category of \emph{endomorphisms} of the base object: $\Omega(X, x) = (\Hom_X(x, x), \id_x)$. 
A spectrum is thus a special case of a categorical spectrum where every component is an $\infty$-groupoid: $\Sp\subset \CatSp$. 

Notice that $X_0\xrightarrow{\sim} \Omega X_1\xrightarrow{\sim} \Omega^2 X_2\xrightarrow{\sim} \cdots$; just as for $\infty$-groupoids, the $n$-fold loop object is canonically an $\EE_n$-monoid object.
In the limit, each component $X_n$ of a categorical spectrum acquires the structure of a symmetric monoidal $(\infty, \infty)$-category.
Conversely, given a symmetric monoidal $(\infty, \infty)$-category $X_0$, there is a minimal choice of $X_n$ with $X_0\xrightarrow{\sim} \Omega^n X_n$, namely the \emph{(connective) delooping} $X_n = \rB^n X_0$ of $X_0$. A categorical spectrum is \emph{connective} if it is determined by $X_0$ in this way. These form another class of examples, equivalent to symmetric monoidal $(\infty, \infty)$-categories: $\rB^\infty: \CMon(\infty\Cat)\xrightarrow{\simeq}\CatSp^\cn\subset \CatSp$. Note that a non-grouplike $\EE_\infty$-space is a special case of this example. 

Of course, this is a naive categorification that one might try (in fact, it was considered independently by several groups: \cref{catsp_literature}). Categorification is more an art than a recipe and a naive approach is not always the correct one. It is therefore reasonable to ask how fruitful this notion is. 
As is typical for an abstract mathematical structure, there are two layers of affirmative answers to this question. On the surface level, one seeks sufficient supply of examples, old and new, to demonstrate that the notion provides a useful organizational language. This is largely achieved in the thesis of Stefanich \cite{stefanichHigherQuasicoherentSheaves2021}. 
The notion of categorical spectra provides a convenient framework to bundle iterated categorifications. The $(\infty, 1)$-category $\CatSp$ contains $\Sp$ and $\CMon(\infty\Cat)$ as full subcategories, as well as more intricate examples such as the categorical spectra of iterated spans and iterated modules. The functoriality of the $n$-quasi-coherent sheaves and the compatibility across different $n$ can be expressed as a morphism from the categorical spectrum of iterated spans of the (pre)stacks to the categorical spectum iterated modules.

In the deeper layer, however, one asks whether these are legitimate mathematical \emph{objects} rather than merely a language.
The utility of spectra partially comes from the fact that they admit various interrelated interpretations (the following list is roughly taken from \cite[\S 0.2.3]{lurie2018spectral}):
\begin{itemize}
    \item Spectra are infinite loop spaces; this is the definition we have chosen. 
    \item Spectra are generalized abelian groups; 
    there is an equivalence between infinite loop spaces $\Omega^\infty X$ for a spectrum $X$ and grouplike $\EE_\infty$-spaces. In particular, we have $\pi_0: \Sp\to \Ab$ and an inclusion $\Ab\to \Sp$ as the heart of the Postnikov $t$-structure. Moreover, there is a symmetric monoidal structure $\otimes = \otimes_{\SS}$ deriving the tensor product of abelian groups. This viewpoint allows us to develop the whole ``brave new'' versions of classical algebra and algebraic geometry, known as \emph{higher algebra} and \emph{spectral algebraic geometry}. 
    \item Spectra are stable homotopy types; for instance, this is apparent in the definition of finite spectra using the Spanier--Whitehead category. Namely, for two finite pointed CW complexes $X, Y$, the space of maps between their suspension spectra is computed as 
    \[\Map_{\Sp}(\Sigma^\infty X, \Sigma^\infty Y)\simeq \colim(\Map_{\sS_\ast}(X, Y)\to \Map_{\sS_\ast}(\Sigma X, \Sigma Y)\to\cdots ).\]
    A finite spectrum is a shift $\Sigma^{\infty-n} X$ of a suspension spectrum of a finite CW complex, and general spectra are filtered colimits of finite spectra: $\Sp = \Ind(\Sp^\fin)$. Moreover, the Freudenthal suspension theorem implies that the above colimit sequence stabilizes after finitely many suspensions. 
    \item Spectra are generalized (co)homology theories; 
    more precisely, given a spectrum $E$, one defines the associated cohomology theory $E^n(X) = \pi_0\Map_{\sS}(X, \Omega^{\infty-n} E)$ and homology theory $E_n(X)= \pi_n(\Sigma^{\infty} X\otimes E)$. The Brown representability theorem states this gives a bijection between equivalence classes of cohomology theories and those of spectra. As for homology, it leads to another useful description of spectra as reduced excisive copresheaves on finite pointed $\infty$-groupoids (or CW complexes): $\Sp\simeq \Exc(\sS^\fin_\ast, \sS)$. This perspective is essential in Goodwillie's calculus of functors; spectra play the role of the linearization of $\infty$-groupoids. 
    \item Spectra form the universal stable $(\infty, 1)$-category: more precisely, $\Sp$ is the free stable presentable $(\infty, 1)$-category generated by a single object (the sphere spectrum) $\SS\in \Sp$, i.e., $\ev_{\SS}: \LFun(\Sp, \eC)\to \eC$ is an equivalence for any stable presentable $(\infty, 1)$-category $\eC$. It implies that $\Sp$ is the tensor unit of the symmetric monoidal $(\infty, 1)$-category $\Pr^\sL_\st$ of presentable stable $(\infty, 1)$-categories for the Lurie tensor product of presentable categories, providing an elegant way to define the tensor product of spectra. This machinery is also essential for categorified sheaf theory, including recent developments in six functor formalism \cite{liuEnhancedSixOperations2017}\cite{mannAdic6FunctorFormalism2022}.
\end{itemize}
Heuristically, we would like categorical spectra to admit similar interpretations, but with everything directed or lax:
\begin{itemize}
    \item Categorical spectra are infinite loop higher categories; again, this is essentially our definition. 
    \item Categorical spectra are generalized commutative monoids and symmetric monoidal categories; we have already seen that connective spectra are equivalent to symmetric monoidal $(\infty, \infty)$-categories, and a general categorical spectrum is a compatible sequence $(X_n)$ of possibly nonconnective deloopings of $X_0$. In particular, a robust theory of categorical spectra leads to a robust derived algebra of commutative monoids. 
    We wish categorical spectra to play a similarly fundamental role in categorified algebra as spectra do in higher algebra. In particular, we need a tensor product of categorical spectra.
    In many examples, $X_0$ is a classical object (of categorical level $0$ or $1$), and $X_n$ is obtained by iterated categorifications. In other words, the new interesting information grows in the \emph{cohomological} direction. For this reason, Johnson-Freyd \cite{pirsa_PIRSA:23090104} proposes to call the resulting categorified algebra \emph{deeper algebra}, in contrast to \emph{higher algebra}, whose derived information grows in the homological direction. 
    \item Categorical spectra are stable directed homotopy types; this is based on the interpretation that $(\infty, \infty)$-categories are directed homotopy types. That is, $(\infty, \infty)$-categories are like cell complexes, but with directions. Just as spectra are like CW complexes with negative dimensional cells, categorical spectra are like $(\infty, \infty)$-categories with negative dimensional cells. 
    Additionally, a Spanier--Whitehead style definition and a form of the Freudenthal suspension theorem are desirable.
    \item Categorical spectra are ``(co)homology theory'' of higher categories; in other words, we would like an axiomatization of the functors on $(\infty, \infty)$-categories that categorical spectra represent. 
    In particular, we must understand the correct analogue of excision properties. This is more or less equivalent to understanding which higher categorical colimits in $\CatSp$ are also limits, i.e., we would like to classify the \emph{absolute colimits} in $\CatSp$. 
    \item Categorical spectra are universal stable presentable Gray-bimodules; as we will see below, the natural replacement for the Cartesian product in the higher categorical world is the (lax) Gray tensor product. Accordingly, the ambient categorical setting will in general only have compositions of higher cells up to some noninvertible higher cells. 
    Once we overcome the relevant difficulties around the Gray tensor product, it is more or less clear that $\CatSp$ is the universal presentable Gray-bimodule with invertible loop-suspension. The tensor product of categorical spectra then follows from this. 
    However, stability for $(\infty, 1)$-categories can be defined in many other ways; understanding the appropriate notion of stability boils down to understanding the excision properties, or equivalently, the absolute colimits. 
\end{itemize}
In essence, the theme of this project is to explore this list. In particular, we will pave the way toward the tensor product of categorical spectra based on the last viewpoint, and begin the study of absolute colimits in $\CatSp$ ---since the notion of weighted colimits uses tensoring by the enriching category, we already need the tensor product to study the notion of stability. We now expand on the contents of the paper, following the table of analogies:
\begin{center}
    \begin{tabular}{|c|c|c|}
        \hline
        Classical Mathematics & Homotopy Theory & Higher Category Theory \\
        \hline\hline
        equality & homotopy & morphism\\
        \hline
        sets $\Set = (0, 0)\Cat$ & spaces/$\infty$-groupoids $\sS = 0\Cat$ & $(\infty, \infty)$-categories $\infty\Cat$\\
        \hline
        --- & homotopy $n$-type & $(\infty, n)$-category \\\hline
        Cartesian product $\times$ & Cartesian product $\times$ & lax Gray tensor product $\otimes$ \\\hline
        $(1, 1)$-categories $(1, 1)\Cat$ & $(\infty, 1)$-categories $1\Cat$ & Gray-categories $(\infty\Cat^\otimes)\mhy\Cat$ \\\hline
        abelian groups $\Ab$ & spectra $\Sp$ & categorical spectra $\CatSp$\\\hline
         & grouplike $\EE_\infty$-spaces $\simeq \Sp^\cn$ & $\CMon(\infty\Cat)\simeq \CatSp^\cn$ \\\hline
        --- & loop $\Omega (X, x) = \Aut_X(x)$ & $\Omega (X, x) = \End_X(x)$ \\\hline
        --- & suspension $\Sigma = \rB\ZZ\wedge (\blank)$ &  $\Sigma = \rB\Free_{\EE_1} = \rB\NN\owedge (\blank)$\\\hline
        free functor $\Set\to \Ab$ & suspension spectra ${\Sigma}_+^\infty: \sS\to \Sp$ & $\Sigma_+^\infty: \infty\Cat\to \CatSp$ \\\hline
        integers $\ZZ$ & sphere spectrum $\SS$ & finite set categorical spectrum $\FF$\\\hline
        tensor product $\otimes_\ZZ$ & tensor (smash) product $\otimes_\SS$ & tensor product $\otimes_\FF$ \\\hline
        abelian categories & (pre)stable categories & stable Gray-bimodules \\\hline
    \end{tabular}
    \end{center}
\subsection*{Higher category theory}
The first two chapters of the body of the thesis, as well as the appendix, are devoted to preparation. 
In \cref{chapter_prelim_category_theory}, we give a model-independent exposition of the theory of $(\infty, \infty)$-categories.
Our main focus will be on comparing enriched-categorically defined notions and those that are native and specific to $\infty$-category theory. 
In doing so, we will pass between weak and strict notions of higher categories; we perform computations on strict objects and then left Kan extend to the weak setting. 
The Steiner theory of the \cref{appendix_Steiner_theory} provides an algebraic language for computations in strict $\infty$-categories.

The central object on the $(\infty, \infty)$-categorical side of the story is the \emph{(lax) Gray tensor product}. It is a biclosed monoidal structure on $\infty\Cat$ that acts \emph{additively} on categorical levels (compare with the fact that $n\Cat\subset\infty\Cat$ is closed under the Cartesian product (in fact, under all limits and colimits)). More precisely, it is characterized by $\cube^m\otimes \cube^n\simeq \cube^{m+n}$, where $\cube^n$ denotes the $n$-category called the (fully lax) $n$-cube (we refer the reader to \cref{pictures_of_the_cubes} for illustrations).
The main technical complication arises from the \emph{non-commutativity}, essentially coming from the choice of direction of the $2$-cell in the cube category $\cube^2 = \cube^1\otimes \cube^1$. 
For our purposes, the most important result is \cref{pushout_formula_unreduced_suspension} that compares the enriched-categorical suspension with the suspension as a quotient of the Gray cylinder, i.e., the Gray tensor product with the interval. 

Most of the material in this section is well-known to experts and appear in the literature, such as \cite{campionGrayTensorProduct2023}\cite{loubatonTheoryModels$infty2023} (see also \cite{araJointTranchesPour2020} for the strict case). 
However, we pay special attention to avoiding mistakes with duality involutions (this corresponds to sign issues in the presence of negatives, e.g. in $\Sp$), as this will be crucial in later chapters.

\subsection*{Categorical spectra}
\cref{chapter_categorical_spectra} is another preparatory chapter. We begin with the study of loop-suspension adjunctions and the delooping hypothesis. In particular, in \cref{suspension_is_tensor_S1}, we observe that the loop-suspension adjunction is the tensor-hom adjunction for the \emph{lax smash product} with the directed circle $\vS^1 = \rB\NN$, so $\Sigma X\simeq \vS^1\owedge X$ (where $\owedge$ is to $\otimes$ as $\wedge$ is to $\times$). 
We then define categorical spectra precisely and include examples that arise universally by imposing levelwise properties of categorical spectra. We conclude the section with a formal study of the finiteness properties of categorical spectra. More than half of this chapter is a summary of \cite[Chapter 13]{stefanichHigherQuasicoherentSheaves2021}, simplified by specializing to the situation of interest, namely from enriched $(\infty,1)$-categories to $(\infty,\infty)$-categories. 

\subsection*{Tensor product of categorical spectra}
\cref{chapter_tensor_product_of_catsp} is devoted to proving our first main theorem: the construction and universal properties of the tensor product of categorical spectra. The strategy was outlined above, but we now describe the method and the main obstacle in more detail.

Recall that the tensor product (always assumed to be biclosed) of abelian groups is characterized by the fact that the free abelian group functor $\Free: \Set\to \Ab$ promotes to a symmetric monoidal functor. 
The tensor product of spectra is similarly characterized by the fact that $\Sigma_+^\infty: \sS\to \Sp$ promotes to a symmetric monoidal functor. However, the homotopical nature of the object makes it impossible to formulate a correct universal property without a solid foundation in $(\infty,1)$-category theory. 
Lurie first constructed a symmetric monoidal structure $\otimes$ on $\PrL$ that promotes the presheaf functor $\PSh: \Cat\to \PrL$ to a symmetric monoidal functor\footnote{This is an example of the microcosm principle: to talk about an object with a certain structure (e.g.\ a commutative monoid), we must first equip the ambient category with the corresponding structure (e.g.\ a symmetric monoidal structure).}. A (commutative) algebra object in $\PrL$ is precisely a presentably (symmetric) monoidal category.
\begin{theorem*}[\cite{LurieHA}]
    $\Sigma^\infty_+: \sS\to \Sp$ is an idempotent $\EE_0$-algebra in $\PrL$, i.e., $\Sigma^\infty_+\otimes \id: \Sp\simeq \sS\otimes \Sp\to \Sp\otimes\Sp$ is an equivalence. Since the forgetful functor $\CAlg^{\mathrm{idem}}(\PrL)\to \Alg_{\EE_0}^{\mathrm{idem}}(\PrL)$ is an equivalence, $\Sp$ uniquely promotes to an object of $\CAlg(\PrL)$ whose unit is the sphere spectrum $\SS$.
\end{theorem*}
This method is robust and elegant, so we aim to apply the same strategy to categorical spectra. 
However, it turns out to be more subtle than it may first appear. First, one cannot expect $\Sigma_+^\infty: \sS\hookrightarrow\infty\Cat\to\CatSp$ to be an idempotent $\EE_0$-algebra in $\PrL$, since the category $\infty\Cat$ is already not idempotent\footnote{In fact, the equivalences $\CAlg^\mathrm{idem}(\PrL)\xrightarrow{\sim}\Alg^\mathrm{idem}_{\EE_1}(\PrL)\xrightarrow{\sim}\Alg^\mathrm{idem}_{\EE_0}(\PrL)$ and the fact that $(\infty\Cat, \times)\in \CAlg(\PrL)$, $(\infty\Cat, \otimes^{\mathrm{(op)lax}})\in \Alg(\PrL)$ have the same underlying $\EE_0$-algebra (i.e., the unit is terminal) disprove idempotence.} over $\sS$.
One may instead ask whether $\Sigma_+^\infty: \infty\Cat\to \CatSp$ (or equivalently, $\Sigma^\infty: \infty\Cat_\ast\to \CatSp$) is idempotent, but to make sense of this, we must first choose a monoidal structure on $\infty\Cat$. 
The obvious choice would be the Cartesian product, but the suspension is far from being a module map over it:
if $X, Y$ are $m$- and $n$-categories respectively, then $X\wedge \Sigma Y$ is a $\max\{m, n+1\}$-category, while $\Sigma(X\wedge Y)$ is a $\max\{m, n\}+1$-category. Thus $X\wedge \Sigma Y$ and $\Sigma(X\wedge Y)$ do not even have the same categorical level in general. This is the same problem as $\Sigma X\not\simeq \vS^1\wedge X$.

This observation suggests that we should use a monoidal structure that adds categorical levels; fortunately, we have already proven that the Gray tensor product satisfies $\Sigma X\simeq \vS^1\owedge X$. In particular, $\Sigma^\infty: \infty\Cat_\ast\to \CatSp$ is a morphism of right $\infty\Cat_\ast$-modules. 
However, the noncommutativity now becomes a serious obstacle: it makes no sense to ask for idempotence of a right module over a noncommutative algebra, since there is no relative tensor product. 
It turns out that one can resolve this issue by lifting $\Sigma$ to a \emph{bimodule} morphism $\infty\Cat_\ast\to \infty\Cat_\ast$. In fact, we will prove that there is a unique such lift if we twist the bimodule structure on one copy of $\infty\Cat_\ast$ by the \emph{total dual} involution. 
In terms of the coefficient object $\vS^1$, this means that we may commute objects with $\vS^1$ in a completely canonical manner, provided we twist the object by the total dual involution: $\vS^1\owedge X\simeq X^\circ\owedge \vS^1$.
We call this the \emph{half-central structure} of $\vS^1$. This is of fundamental importance beyond this application: it is the higher-categorical incarnation of the \emph{Koszul sign rule}. Since we have no negatives to express the sign, we must instead twist one side of the equivalence by an appropriate duality.
Once the half-central structure of $\vS^1$ is understood, the remaining work to prove the following theorem is largely formal:
\begin{manualtheorem}{A}[=\cref{thm_tensor_product_of_catsp}]
        The functor $\Sigma_+^\infty: \infty\Cat\to \CatSp$ is an idempotent $\EE_0$-algebra in $\BMod{\infty\Cat^\otimes}(\PrL)$. 
        In particular, $\CatSp$ admits a unique biclosed $\EE_1$-monoidal structure $\otimes$ 
        underlying an $\infty\Cat^\otimes$-algebra structure (in $\PrL$) that promotes $\Sigma_+^\infty$ to an $\infty\Cat^\otimes$-algebra morphism. 
        Moreover, the monoidal category $\CatSp^\otimes$ can be obtained by universally inverting $\vS^1\in \infty\Cat_\ast^{\owedge}$. 
\end{manualtheorem}
The reader who finds the noncommutativity of the tensor product unsettling is invited to consult \cref{remark_remedy_of_noncommutativity} for several ideas that may help address this issue.
However, the Gray tensor product appears to be unavoidable in higher category theory, at least from the viewpoint that $(\infty,\infty)$-categories are directed homotopy types. 
Furthermore, we expect that any useful commutative variant of the tensor product receives a lax monoidal functor from our tensor product. In other words, the results established using our tensor product are universal and can be transferred to a preferred setting if necessary. 
We hope that the following chapters justify the richness and naturality of the theory developed from this tensor product of categorical spectra.

\subsection*{Duality, absolute colimits and stability}
Recall that one of the standard motivations for introducing spectra is \emph{Spanier--Whitehead duality}. That is, unlike the symmetric monoidal categories $\sS_\ast$ or $\Sp^\cn$ (where the only dualizable objects are the free ones on finite sets), $\Sp$ has many interesting dualizable objects: any finite (i.e., compact) spectrum is dualizable, which provides an upper bound on dualizable objects (as in any closed monoidal category with a compact unit object; cf.\cref{dualizable_implies_compact}).
This reflects the fact that the category of spectra has better exactness properties: since $\Sp\simeq \LFun(\Sp, \Sp)\simeq \RFun(\Sp, \Sp)^\op$, an object $X\in \Sp$ is dualizable if and only if $X\otimes (\blank)\in \RFun(\Sp, \Sp)$, i.e., if it preserves limits. By stability, limits commute with finite colimits, so if $X$ is built from finite colimits (and shifts) of $\FF$, then $X$ is dualizable.
More generally, weighted colimits in $\Sp$ with finite weights (i.e., diagrams $J\to \Sp^\fin$ for finite categories $J$) are limit-preserving, or \emph{absolute}: they are preserved by all $\Sp$-enriched functors. This captures the essence of stability in a way that does not depend on an ad hoc choice of diagram shapes.

In the same spirit, we expect a good supply of dualizable objects and absolute colimits in $\CatSp$. Although we are not yet able to classify all of them, we show that many fundamental finite categorical spectra and ``finite weights'' are dualizable (or absolute).
For the following, let $X\leftarrow Y \rightarrow Z$ be a span of categorical spectra, and let $X\ramalg_Y Z$ denote the directed pushout $X\amalg_{\{0\}\otimes Y} (\cube^1\otimes Y)\amalg_{\{1\}\otimes Y} Z$. 
\begin{manualtheorem}{B}[=\cref{theorem_directed_pushouts_are_absolute}, \cref{retracts_of_cubes_are_dualizable}]
    The functor $\ramalg : \Fun((\bullet\leftarrow \bullet\rightarrow\bullet), \CatSp)\to \CatSp$ admits a right adjoint given by 
    \[X\mapsto (\Sigma^{\infty-1}I^\op\otimes X\leftarrow \Sigma^{-1}X\rightarrow \Sigma^{\infty-1}I \otimes X),\] where $I$ is the interval category $0\to 1$ with basepoint at $0$. 
    It follows that suspension spectra of the cubes, orientals, and objects of Joyal's $\Theta$ category are all dualizable. 
\end{manualtheorem}
We expect that proving the absoluteness of these basic examples, together with an analysis of closure properties of absolute weights, will lead to a complete classification of absolute weights. 
Note, however, that ordinary pushouts are not absolute; in fact, it force invertibility of $(\infty,1)$-categorical loop-suspension, i.e., $(\infty,1)$-categorical stability. Therefore, to classify absolute colimits, we must define ``finite weights'' in a way that excludes such overly invertible examples. 
A particularly important special case is the following:
\begin{manualtheorem}{C}[=\cref{extension_of_catsp}]\label{theorem_C}
    Let $f: X\to Y$ be a morphism of categorical spectra. There is a canonical equivalence between the lax cofiber $\rcof(f) = 0\ramalg_X Y$ and the lax fiber of the suspension $\rfib(\Sigma f)= 0\overrightarrow{\times}_{\Sigma Y}\Sigma X$. Moreover, the canonical triangle $Y\to (\rcof(f)\simeq \rfib(\Sigma f)) \to \Sigma X$ is a (non-lax) bifiber sequence. 
\end{manualtheorem}
In this situation, we say that $Z = \rcof(f)$ is an \emph{extension} of $\Sigma X$ by $Y$ (and similarly for a \emph{coextension} $\lcof(f)$). 
This theorem lifts the Barratt--Puppe sequence in the category of spectra. This is obvious in stable $(\infty,1)$-categories, but the result here is surprising, since naive extensions of standard facts in stable $(\infty,1)$-categories tend to fail; for example, the pasting law for directed pushouts fails, and lax fiber sequences do not coincide with lax cofiber sequences. 
This suggests that the classical notion of (co)fiber sequences in a stable $(\infty,1)$-category splits into three different classes of sequences—namely, lax cofiber sequences, bifiber sequences, and lax fiber sequences—and that they appear in a three-periodic pattern.

\subsection*{Categorical spectra with adjoints and applications to TQFT}
The study of absolute colimits in categorical spectra is not only theoretically important but also a useful computational tool. One can often draw strong conclusions from coincidences between limits and colimits\footnote{For instance, ambidexterity is a technique of this flavor, which has recently opened new directions in chromatic homotopy theory.}. In particular, even if one is only interested in symmetric monoidal $(\infty,n)$-categories, studying categorical spectra leads to a better understanding of their (pre)stability. We apply this principle to the study of TQFTs. 
We say that a categorical spectrum $(X_n)$ is \emph{$d$-adjointful} if, for $k<d+n$, every $k$-cell of $X_n$ admits both left and right adjoints. For instance, an $(\infty,d)$-category $\eC$ \emph{has duals} in the sense of \cite{lurieClassificationTopologicalField2009} (i.e., its objects are fully dualizable) if the corresponding connective categorical spectrum $\rB^\infty\eC$ is $d$-adjointful (note that objects of $X_n$ are dualizable if $1$-morphisms of $X_{n+1}$ admit adjoints). 
We denote by $\CatSp^{d\mhy\adj}\subset \CatSp$ the full sub $(\infty,1)$-category of $d$-adjointful categorical spectra. We also let $d\CatSp\subset \CatSp$ denote the full subcategory of $d$-categorical (i.e., $d$-truncated) spectra, and set $d\CatSp^{\adj}=d\CatSp\cap \CatSp^{d\mhy\adj}$.  
\begin{manualtheorem}{D}[=\cref{theorem_tensor_product_localizes_to_adj}, \cref{n_adjointable_catsp_are_closed_under_extensions}]
    The subcategories $\CatSp^{d\mhy\adj}\subset \CatSp$ are closed under extensions. Moreover, the tensor product of categorical spectra localizes to categorical spectra with adjoints. More precisely, there is a unique tensor product functor making the following diagram commute (where the vertical arrows are the localizations):
    \[\begin{tikzcd}
        k\CatSp \otimes l\CatSp\ar[r, "\otimes"]\ar[d, "L^{k\mhy\adj}\otimes L^{l\mhy\adj}"] & (k+l)\CatSp \ar[d, "L^{(k+l)\mhy\adj}"]\\
        k\CatSp^{\adj} \otimes l\CatSp^{\adj}\ar[r, dashed, "\otimes^\adj"] & (k+l)\CatSp^{\adj}
    \end{tikzcd}\]
    In particular, there are unique tensor products on $0\CatSp^{\adj}$ and $\infty\CatSp^{\adj}$ that promote the localizations $\CatSp\twoheadrightarrow 0\CatSp^{\adj}, \infty\CatSp^{\adj}$ to monoidal functors. 
\end{manualtheorem}
The proof reduces to a certain pushout formula for $\cube^1\otimes \Adj$, where $\Adj$ denotes the walking adjunction $2$-category. This proof occupies the bulk of \cref{section_categorical_spectra_with_adjoints}. 
Using this result, we can formally assemble the framed cobordism hypotheses in different dimensions into a multiplicatively structured object\footnote{It is convenient to normalize by shifting $\rB^\infty\Bord_n\in\CatSp^{n\mhy\adj}$ to $\rB^{\infty-n}\Bord_n\in \CatSp^{0\mhy\adj}$ so that we may avoid using a graded monoidal structure. This also aligns with the philosophy of deeper algebra, using codimensional indexing and placing the \emph{partition function} in degree zero.}:
\begin{manualtheorem}{E}[=\cref{multiplicative_cobordism_hypothesis},\cref{stable_cobordism_hypothesis}]
    Assume the framed cobordism hypothesis. Then:
    \begin{itemize}
        \item $\bigoplus_{n\geq 0}\rB^{\infty-n}\Bord^\fr_n$ is a tensor algebra in $\CatSp^{0\mhy\adj}$ freely generated by a $(-1)$-cell. 
        \item The stably framed bordism $(\infty,\infty)$-category $\rB^\infty\Bord^\sfr$ is the tensor unit of $\CatSp^{\infty\mhy\adj}$.
    \end{itemize}
\end{manualtheorem}
In both cases (at least heuristically), multiplication in the algebra structure is given by the Cartesian product of manifolds.
Note that it is not possible to formulate the above structure without the lax tensor product, since the Cartesian product adds the dimension of manifolds\footnote{This theorem was roughly mentioned in \cite{MOpostRigCategoryStructureBord}; however, the discussion is not made rigorous due to the lack of correct background setup.}.
Note, however, that everything up to this point could have been carried out within the setting of connective categorical spectra, i.e., symmetric monoidal $(\infty,\infty)$-categories, after suitable reindexing. 
The true strength of the theory lies in stability results such as Theorem~\ref{theorem_C}, which we now apply. 
A \emph{cobordism category with singularities} is a cell-complex-like object built from the ordinary cobordism categories (with tangential structures), i.e., objects $B^k$ fitting into the following iterated (co)extensions of categorical spectra (in $\CatSp^{0\mhy\adj}$): 
\[\begin{tikzcd}[column sep=small]
    0\ar[r, tail] & B^d\ar[r, tail]\ar[d, two heads] & B^{d-1}\ar[r, tail]\ar[d, two heads] & \cdots \ar[r, tail] & B^1\ar[r, tail]\ar[d, two heads] & B^0\ar[d, two heads] \\
    & \rB^{\infty-d} \Bord_d^{\tilde{X}^d} & \rB^{\infty-(d-1)}\Bord_{d-1}^{\tilde{X}^{d-1}} & &\rB^{\infty-1} \Bord_1^{\tilde{X}^1} & \rB^{\infty}\Bord_0^{\tilde{X}^0}
\end{tikzcd}\]
where $\tilde{X}^k$ are $\rO(k)$-spaces classifying tangential structures. Each extension $B^k$ is classified by an $\rO(k)$-equivariant morphism $E^k: \tilde{X}^k\to (\Omega^{\infty-k-1} B^{k+1})^{\leq 0}$.

Recall that an extension of categorical spectra can be described either as a lax fiber or as a lax cofiber. 
The former describes maps \emph{into} $B^k$, allowing us to interpret $B^k$ geometrically, while the latter describes maps \emph{out of} $B^k$, i.e., TQFTs. This dual description is precisely the content of the cobordism hypothesis with singularities. 
\begin{manualtheorem}{F}[=\cref{cobordism_hypothesis_with_singularities} {\cite[Theorem 4.3.11]{lurieClassificationTopologicalField2009}}]
    The categorical spectrum $B^k$ admits a description as the cobordism category of $\vec{X}$-manifolds as in \cite[Definition Sketch 4.3.2]{lurieClassificationTopologicalField2009}, with singularity datum $\vec{X} = (\tilde{X}^d, \tilde{X}^{d-1}, E^{d-1}, \cdots, \tilde{X}^k, E^k)$. 
    Moreover, for any $0$-adjointful categorical spectrum $A$, there is a Cartesian square 
    \[\begin{tikzcd}
        \Map(B^k, A)\ar[r]\ar[d] & \Map(B^{k+1}, A)\ar[d]\ar[r, phantom, "\ni"] & Z_0\ar[d, mapsto] \\
        \Map_{\rO(k)}(\tilde{X}^k, (\Alg_{\EE_0}(A_{k+1}))^{\leq 0}) \ar[r] & \Map_{\rO(k)}(\tilde{X}^k, (A_{k+1})^{\leq 0})\ar[r, phantom, "\ni"] & \Omega^{\infty-k-1}Z_0\circ E^k. 
    \end{tikzcd}\]
\end{manualtheorem}
Note that this formulation is in a sense dimension-independent: $B^d$ can be $0$ if $\tilde{X}^d =\emptyset$, so the value of $d$ does not matter, as long as the sequence is bounded below. This is somewhat curious since in the study of TQFTs dimension is usually a fixed parameter. 

\subsection*{Notations and Terminologies}
\begin{itemize}
    \item We mainly follow the standard notations of \cite{LurieHTT}\cite{LurieHA}, with a few exceptions below. In particular, for a type of categorical object, we use $\widehat{(\blank)}$ to denote the (very large) category of the large variant of that object.
    \item From now on (except in \cref{appendix_Steiner_theory}), we will assume the objects are homotopical by default; in particular, we leave ``$\infty$-'' in $(\infty, n)$-categories or $\infty$-category/groupoid implicit and call them $n$-categories or category/groupoid. The ($(\infty, 1)$-)category of $n$-categories will be denoted by $n\Cat$. Similarly, $\infty$-operads in the sense of \cite{LurieHA} will be called operads. 
    We still denote the category of ($\infty$-)groupoids by $\sS\coloneqq 0\Cat$, even though we will avoid calling them \emph{spaces} unless they are supposed to have additional structures of topological spaces or CW complexes. 
    Presheaves takes values in groupoids by default: $\PSh_{\eC}(\eD)\coloneqq \Fun(\eD^\op, \eC)$ and $\PSh(\eD) = \PSh_{\sS}(\eD) =\Fun(\eD^\op, \sS)$.
    \item We let $\PrL\subset \widehat{\Cat}\supset \PrR$ denote the (non-full) subcategories whose objects are presentable categories and the morphisms are left and right adjoint functors, respectively.
    We let $\LFun(\eC, \eD)\subset \Fun(\eC, \eD)\supset \RFun(\eC, \eD)$ denote the full subcategories spanned by the left and right adjoints. 
    We see $\PrL$ as a symmetric monoidal category by Lurie tensor product $\otimes$ (the internal hom is $\LFun$).
    Also, $\Pr^\sL_{\omega}\subset \PrL$ will denote the (non-full) symmetric monoidal subcategory of compactly generated categories and compact object preserving left adjoints. 
    The morphisms of its category, $\Pr^\sR_{\omega}\subset \PrR$ are filtered colimit preserving right adjoints. 
    \item The $(1, 1)$-category of strict $n$-categories will be denoted by $n\strCat$ (the notation $(n, n)\Cat$ will mean the $(n+1, 1)$-category of weak $(n, n)$-categories).
    \item We write $\Map$ for the hom groupoid of a (possibly underlying) $(\infty, 1)$-category and $\Hom$ for a generic hom object of an algebroid or an enriched category, in particular, the Cartesian internal hom for $\infty\Algbrd$.
    The notation $[\blank, \blank]$ (resp. $\llbracket \blank, \blank\rrbracket$) will mean the \emph{left} (resp. \emph{right}) internal hom for the Gray-type tensor product, i.e., $\Map(X\otimes Y, Z)\simeq \Map(Y, [X, Z])\simeq \Map(X, \llbracket Y, Z\rrbracket)$.
    As above, $\Fun$ will mean the functor ($(\infty, 1)$-)category between category-type objects. 
    \item We let $\ast$ denote generically a terminal object (most often the contractible category), while $1$ denote the unit of the monoidal structure under consideration. We use $\eC_\ast$ and $\eC_{\ast\ast}$ to denote the category of pointed and bipointed objects, i.e., the category of objects under $\ast$ and $\ast\amalg \ast$. 
    \item We tend to use $\sqcup$ to denote a disjoint union (i.e., a coproduct that is disjoint) whereas $\amalg$ will mean a generic coproduct. The wedge sum $\vee$ either means the coproduct in $\eC_\ast$ or the \emph{bipointed wedge sum} (typically \emph{sink-source}), i.e., $(X, x_0, x_1)\vee(Y, y_0, y_1) = (X\amalg Y/(x_1=y_0), x_0, y_1)$. We will not be very strict about the distinction of the notation. 
    \item $\times$, $\wedge$ will mean the Cartesian product and the corresponding smash product. $\otimes$ denotes the lax Gray tensor product of unpointed $\infty$-categories or the tensor product of categorical spectra, whereas $\owedge$ denotes the Gray smash product of pointed $\infty$-categories. 
    \item We write $\sigma$ for the unreduced suspensions (of $\infty$-algebroids/categories and augmented directed complexes), reserving $\Sigma$ for the (categorical) reduced suspension. Note that this does not agree with the usual suspension for $\infty$-groupoids. 
    We use $\rB$ to denote (the univalent completion of) the delooping of a monoidal category. 
    For categorical spectra, we may both use $\Sigma$ and $\rB$ for the shift functor (also denoted by $[1]$ as usual); we will be flexible with the notation here. 
    \item $\GG$, $\Ori$, $\Cube$, $\Theta$ denote the $(1, 1)$-category of globes, orientals, (fully lax) cubes, and Joyal's theta. 
    $C_n$, $\Ori^n$, $\cube^n$ denote the $n$-cell, the $n$-oriental (i.e., the fully lax simplex) and the $n$-cube. 
    Note the somewhat nonstandard notation for the orientals. We use $\Theta_1$ to mean the usual simplex category, whose objects ($1$-categorical simplexes) are denoted by $[n]$. 
    \item For $\eC\in \widehat{\Cat}$, we denote the full subcategory of (homotopically) $n$-truncated objects by $\eC_{\leq n}\subset \eC$. In contrast, for $X\in \infty\Algbrd$, we denote the $n$-categorical truncation by $X^{\leq n}\subset X$, i.e., the right adjoint to the inclusion $n\Algbrd\hookrightarrow\infty\Algbrd$. 
    There is also a left adjoint to the inclusion, which we denote by $X\mapsto {}^{\leq n}X$, but this notation is ambiguous and it can mean the further localization to $n\Cat$, $n\strCat$ or $n\Gaunt$. See \cref{section_n_cats_and_algebroids} for details. 
    \item There seems to be no consensus about whether our Gray tensor product should be called the lax or oplax Gray tensor product. We will follow \cite{araJointTranchesPour2020}. That is, the one making the linearization functor $\infty\strCat\to \adCh$ strong monoidal with respect to the usual Koszul sign rule is the \emph{oplax} Gray tensor product, and we take its opposite, the \emph{lax} tensor product as default. This choice is more convenient for us, ultimately because of the choice we made about the way a monoidal functor is identified with a bimodule in \cref{chapter_tensor_product_of_catsp}. 
\end{itemize}

\subsection*{Acknowledgement}
I would like to thank my advisor, David Gepner, for his continuous encouragement and for fostering an exceptionally friendly learning environment. His optimism regarding my thesis project was a crucial counterforce to my pessimism.

Special thanks go to Tim Campion, who taught me through discussions how to work with $(\infty, n)$-categories. His insights and numerous MathOverflow questions deeply influenced this thesis.

I am thankful to Mayuko Yamashita for drawing my attention to physicists' use of categorical spectra and for organizing valuable learning seminars around QFTs.

I am grateful to Bruno Vallette and Tasuki Kinjo for collaborations during my first and second research experiences. They taught me how to produce a paper.

I thank all the faculty members and staff of Johns Hopkins University, especially Emily Riehl, Nitu Kitchloo, Katia Consani, Yiannis Sakellaridis, David Savitt, and Steve Wilson, for their mathematical and non-mathematical support during my graduate study. I also thank Lars Hesselholt and Jack Morava for their encouragement and discussions.

For conversations regarding the project, I thank Anish Chedalavada, David Ayala, German Stefanich, John Francis, Kantaro Ohmori, Ko Aoki, Niranjan Ramachandran, Shai Keidar, Takumi Maegawa, and Thomas Blom.

I thank the organizers and participants of various learning seminars from which I learned a lot, especially Kenta Kobayashi and Takumi Maegawa, who were frequent and excellent speakers. Milton Lin and Rok Gregoric were active organizers of the seminars where I learned some topics efficiently.

Regarding my life as a graduate student, I especially thank Milton Lin for introducing me to addictive exercise habits and mountain trips. Without him, I would have suffered from back pain during the writing of this thesis. I thank Rahul Dalal for frequent company in bouldering gym sessions and Jonathan Lin for advice on my projects. I thank Akira Tominaga for sharing drinks and connecting me to a community of Japanese graduate students. I also thank every graduate student and postdoc who shared offices, beer, parties, and climbing sessions with me.

Lastly, I thank my family, especially my parents, for their love and support. Without them, it would have been impossible to find something I am passionate about.

\chapter{Preliminaries on category theory}\label{chapter_prelim_category_theory}
The goal of this chapter is to equip the reader with background knowledge on $n$-category theory (i.e., $(\infty, n)$-category theory), including the case $n=\infty$. 
It is largely meant to be expository, but we also provide some proofs of folklore results that the author could not find in the literature.
Steiner's theory of strict $\infty$-categories is at the core of the techniques, but it is separated into the appendix so that the reader can avoid getting too distracted by the combinatorics of strict $\infty$-categories.

We start with a general introduction in \cref{section_n_cats_and_algebroids}. 
Our focus is to provide (without proofs) various natively $(\infty, 1)$-categorical treatments of $n$-categories. 
They are roughly divided into two flavors: one is as \emph{enriched categories}: $n$-categories are categories enriched in $(n-1)$-categories. This is inductive by nature. 
Another is by \emph{presentation}: the category $n\Cat$ of $n$-categories is a localization of presheaves $\PSh(\eC)$, where $\eC\subset n\Cat$ is some full sub $(1, 1)$-category of combinatorial shapes, and the localization is prescribed by gluings that exist in $\eC$. Heuristically, $\eC$ is a lax version of a test category, i.e., we probe the structure of $n$-categories by mapping combinatorial shapes into it.
This approach separates combinatorial complexity from homotopical complexity and makes combinatorial calculations possible. Another advantage is its flexibility, offering different choices of $\eC$ depending on our purpose. 

\cref{section_suspension} and \cref{section_duality} introduce two fundamental operations: the \emph{(unreduced) suspension} and \emph{duality involutions}. Roughly speaking, the suspension takes $X\in n\Cat$ to $\sigma X\in (n+1)\Cat$, which is generated by two objects $\bot, \top$ with $\Hom_{\sigma X}(\bot, \top) = X$. This is treated most naturally using enriched category theory. 
Duality involutions generalize the operation $(\blank)^\op: \Cat \to \Cat$ of taking the opposite category.
There are $(\ZZ/2)^n$-worth of duality involutions on $n\Cat$, obtained by flipping morphisms of specified dimensions. In fact, these exhaust all automorphisms of $n\Cat$. 

\cref{section_cubes_and_Gray_tensor} introduces another fundamental operation: the \emph{Gray tensor product} of $\infty$-categories. It is a noncommutative monoidal structure that will replace the role of Cartesian products in many situations.
Heuristically, the Gray tensor product of two $\infty$-categories looks like the Cartesian product, except that every product cell is filled with a non-invertible arrow of a coherently chosen direction. 
The tensor product is naturally characterized on the dense full subcategory $\Cube\subset \infty\Cat$ of the \emph{cubes}, and as such it is rooted in the presentation over the cube category.
Using the Gray tensor product, we can define two suspension-like operations by tensoring with the interval either on the left or on the right and collapsing the top and bottom faces of the cylinder.
We will show in \cref{section_Gray_suspension} that these agree with the suspension (up to a duality involution in one case). This will be a crucial ingredient in the later chapters. 

\section{\texorpdfstring{$n$}{n}-categories and \texorpdfstring{$n$}{n}-algebroids}\label{section_n_cats_and_algebroids}
There are a few different approaches when working with $(\infty, n)$-categories. In this thesis, we take a \emph{model-independent} approach; we assume and work natively in $(\infty, 1)$-category theory (as developed in \cite{LurieHTT}) and let $n\Cat$ be \emph{the} (large) category (recall the ``implicit $\infty$-'' convention) of (small) $(\infty, n)$-categories without choosing a point-set presentation. 
Thus, our treatment is similar in spirit to \cite{barwickUnicityTheoryHigher2021} (see also \cite{campion$inftyn$categoricalPastingTheorem2023}\cite{campionGrayTensorProduct2023}): we use certain kinds of \emph{strict} $\infty$-categories as the combinatorial blueprint of weak $\infty$-categories.
We will also use the larger category $n\Algbrd = n\Cat^\rf \supset n\Cat$ of $n$-algebroids (a.k.a.\ flagged $(\infty, n)$-categories) to encompass both strict and weak categories.
We begin with some definitions on the strict side. 

\begin{definition}\label{def_strict_omega_categories}
    A \emph{strict $0$-category} is a set: $0\strCat\coloneqq \Set$.
    Let $n\geq 0$ and suppose inductively that the $(1, 1)$-category $n\strCat$ of strict $n$-categories is already defined. Then a \emph{strict $(n+1)$-category} is a strictly $n\strCat$-enriched category: $(n+1)\strCat\coloneqq (n\strCat)\mhy\strCat$. These categories are presentable, and the inclusion $n\strCat\hookrightarrow (n+1)\strCat$ admits both left and right adjoints, denoted by ${}^{\leq n}(\blank)$ (or ${}^{\leq n, \mathrm{str}}(\blank)$ if there is a risk of confusion) and $(\blank)^{\leq n}$, respectively. Let $\infty\strCat$ be the colimit 
    \[\colim(0\strCat\hookrightarrow 1\strCat\hookrightarrow \cdots \hookrightarrow n\strCat\hookrightarrow\cdots) \in \PrL,\]
    or equivalently, the limit in $\PrR$ or $\widehat{\Cat}$ along the \emph{truncations} $(\blank)^{\leq n}$. 
\end{definition}

\begin{definition}\label{definition_suspension_strict}
    Let $X$ be a strict $n$-category. The \emph{suspension} $\sigma X$ is a strict $(n+1)$-category with two objects $\{\bot, \top\}$ and the hom categories 
    \[\Hom_{\sigma X}(\bot, \top) = X,\quad \Hom_{\sigma X}(\bot, \bot)=\ast =\Hom_{\sigma X}(\top, \top),\quad \Hom_{\sigma X}(\top, \bot) = \emptyset\] 
    equipped with uniquely determined compositions. 
    A suspension has a canonical \emph{source-sink} bipointing $\ast\sqcup\ast = \sigma\emptyset \to \sigma X$. The functor $\sigma: n\strCat\to (n+1)\strCat_{\ast\ast}$ is colimit-preserving, and the right adjoint is $(X, x_0, x_1)\mapsto \Hom_X(x_0, x_1)$. 
\end{definition}

\begin{definition}
    Let $(X, x_0, x_1)$, $(Y, y_0, y_1)$ be bipointed strict $\infty$-categories. The \emph{wedge sum} $X\vee Y$ is the quotient $(X\sqcup Y)/(x_1=y_0)$ equipped with the induced bipointing $(x_0, y_1)$. 
\end{definition}
\begin{definition}\label{def_cells_and_Theta}
    \begin{itemize}
        \item The \emph{$n$-cell} $C_n$ is the strict $n$-category $\sigma^n(\ast)$. We define the \emph{(reflexive) globe category} $\GG_n$ as the full subcategory $\{C_k\mid 0\leq k\leq n\}\subset n\strCat$ and $\GG\coloneqq \GG_{\infty}$.
        \item The \emph{canonically bipointed theta category} $\Theta^\can_{\ast\ast}\subset\infty\strCat_{\ast\ast}$ is the smallest full subcategory containing the terminal object $C_0$ and closed under suspension and wedge sum operations.
        The (Joyal's) \emph{theta category} $\Theta\subset \infty\strCat$ is the image of $\Theta^\can_{\ast\ast}$ under the forgetful functor. Let $\Theta_n = \Theta\cap n\strCat$. The bipointing of $\theta\in \Theta$ will always be the canonical one, i.e., given by the source and the sink vertices. 
    \end{itemize}    
\end{definition} 

\begin{example}
    Let us give a generic example of an object of $\Theta$. The following diagram depicts the generating cells of $(\sigma^2 C_0)\vee \sigma(\sigma C_0\vee\sigma^2 C_0)\vee \sigma C_0 = C_2\sqcup_{C_0} (C_2\sqcup_{C_1}C_3)\sqcup_{C_0} C_1$:
\[\begin{tikzcd}
    \bullet \ar[r, bend left=60, "{}" name=P]\ar[r, bend right=60, "{}" name=Q]\ar[Rightarrow, from=P, to=Q, shorten <=1ex] & \bullet \ar[r, bend left=80, "{}" name=A]\ar[r, "{}" name=B]\ar[r, bend right=80, "{}" name=C]\ar[Rightarrow, from=A, to=B, shorten <=.7ex]\ar[Rightarrow, from=B, to=C, shift left=2ex, "{}" name=D, shorten <=.5ex]\ar[Rightarrow, from=B, to=C, shift right=2ex, "{}" name=E, shorten <=.5ex]\ar[triple, from=E, to=D, shorten >=.5ex] & \bullet \ar[r] & \bullet
\end{tikzcd}.\]
\end{example}

\begin{remark}\label{remark_strcat_as_set_valued_theta_presheaves}
    The full subcategory $\Theta_n\subset n\strCat$ is (already $(1, 1)$-categorically) dense, i.e., the restricted Yoneda embedding $n\strCat\to \PSh_{\Set}(\Theta_n)$ is fully faithful. The image is characterized by the so-called Segal condition, a certain locality that can be stated as follows: any object $\theta\in \Theta$ admits a canonical colimit representation as the maximal cells (under inclusion) glued along their shared boundary: $\colim_{i} C_{k_i}\xrightarrow{\sim}\theta$. Now a presheaf $P: \Theta^\op_n\to \eC$ valued in an arbitrary category $\eC$ satisfies the \emph{Segal condition} if for any $\theta\in \Theta_n$, the induced map $P(\theta)\to \lim_i P(C_{k_i})$ is an equivalence.
    In contrast, the full subcategory $\GG_n\subset n\strCat$ is a set of colimit generators but not dense, i.e., the further restricted Yoneda embedding $n\strCat\to \PSh_{\Set}(\GG_n)$ is conservative (in fact, monadic) but not fully faithful; it fails to remember compositions (cf.\ \cref{globular definition of strict omega categories}). 
\end{remark}

Now we describe the category $n\Algbrd = n\Cat^\rf$ of $n$-algebroids, a.k.a.\ flagged $n$-categories in a few different ways. They play the role of a minimal common generalization of strict $n$-categories and $(\infty, n)$-categories: 
\begin{itemize}
    \item (iterated enrichment, \cite{lurie$infty2$CategoriesGoodwillieCalculus2009}\cite{gepnerEnriched$infty$categoriesNonsymmetric2015}\cite{hinichYonedaLemmaEnriched2020}\cite{stefanichHigherQuasicoherentSheaves2021}) For a given groupoid $X\in \sS$ of objects, one can functorially assign the ``$X$-worth-of-objects'' version of the associative (nonsymmetric) operad $\Assoc_X$, which agrees with $\Assoc$ when $X=\ast$. For example, when $X$ is a set, $\Assoc_X$ is equivalent to the multicategory with objects $(x, y)$ for $x, y \in X$ and multimorphisms 
    \[\Map((x_0, y_1), (x_1, y_2), \ldots, (x_{k-1}, y_k); (y_0, x_k))\coloneqq \prod_{i=0}^k\delta_{x_i, y_i}, \quad k\geq 0,\] where $\delta_{x, y} = \ast$ if $x=y$ and $\emptyset$ if $x\neq y$. 
    For a monoidal category $\sV$, 
    we define the category $\Algbrd_X(\sV)\coloneqq \Alg_{\Assoc_X}(\sV)$ of \emph{$\sV$-algebroid\footnote{These are called \emph{categorical algebras} in \cite{gepnerEnriched$infty$categoriesNonsymmetric2015}. In \cite{stefanichHigherQuasicoherentSheaves2021}, where the author took the terminology, a more general setting with \emph{categories} of objects is considered. It is also worth mentioning that since the relevant symmetric monoidal structure on $\sV$ is Cartesian and $\sV$ contains $\sS$, $\sV$-algebroids can also be defined as simplicial objects $X: \Spx^\op\to \sV$ satisfying $X_0\in \sS$ and the Segal condition $X_n\xrightarrow{\sim} X_1\times_{X_0}\cdots\times_{X_0} X_1$.} with groupoid of objects $X$}.
    Roughly speaking, $\eA\in \Algbrd_X(\sV)$ assigns $\eA(x_0, x_1)\in \sV$ for a pair of points $(x_0, x_1)\in X^2$ and a composition morphism $\eA(x_0, x_1)\otimes\cdots \otimes \eA(x_{k-1}, x_k)\to \eA(x_0, x_k)$ for $(x_0, \ldots, x_k)\in X^k$, $k\geq 0$, in a coherently unital and associative way.  
    We define the category of $\sV$-algebroids as the domain of the Cartesian fibration $\ob: \Algbrd(\sV)\to \sS$ classifying the contravariant functor $X\mapsto \Algbrd_X(\sV)\in \widehat{\Cat}$.
    When $X$ is a set and $\sV$ is a $(1, 1)$-category, we recover the notion of strictly $\sV$-enriched categories with the set of objects $X$, so we have $\Algbrd(\sV)|_{\Set} = \sV\mhy\strCat$. 
    
    Let $n\Algbrd\coloneqq \Algbrd(\ldots(\Algbrd(\sS))\ldots)$ be the $n$-fold iteration (with the Cartesian monoidal structure). 
    The inclusion $\ast\xrightarrow{\emptyset}\sS$ in $\PrL$ induces the inclusion $\operatorname{disc}: \sS\simeq \Algbrd(\ast)\hookrightarrow \Algbrd(\sS)$ (given by the initial section of the object fibration) and inductively $(n-1)\Algbrd(\sS)\hookrightarrow n\Algbrd(\sS)$ in $\PrL$. We let $\infty\Algbrd$ be the colimit as $n\to \infty$. 
    \item ($\Theta$-presheaves, \cite{rezkCartesianPresentationWeak2010})  
    Iteratively applying the inclusion $\Algbrd(\sV)|_{\Set}\hookrightarrow \Algbrd(\sV)$ starting with $\sV= \ast$, we have the inclusion $n\strCat\hookrightarrow n\Algbrd$. Its restriction $\Theta_n\hookrightarrow n\Algbrd$ is dense, i.e., the restricted Yoneda embedding $n\Algbrd\to \PSh(\Theta_n)$ is fully faithful. The essential image is characterized by the same Segal condition as in \cref{remark_strcat_as_set_valued_theta_presheaves}.
    \item (presheaves on a suitable site) More generally, if $\cS\subset n\strCat\hookrightarrow n\Algbrd$ is dense, then one can study $n\Algbrd$ by describing the localization $\PSh(\cS)\to n\Algbrd$ and the combinatorics of $\cS$. The localization admits an explicit description by cell attachments of torsion-free complexes when $\cS$ is \emph{suitable} in the sense of \cite[Theorem B]{campionGrayTensorProduct2023}. 
    \item If we take the notion of $n$-categories (see below) as the primary one, the notion of $n$-algebroids is equivalent to the notion of \emph{flagged} $n$-categories of \cite{ayalaFlaggedHigherCategories2018}: $n\Algbrd \simeq n\Cat^\rf$. Flagging is an extra structure of an $n$-category. Roughly speaking, it keeps track of the choices of the groupoids of objects ($X$ above) at each stage of enrichment, which are not invariant under categorical equivalence.
\end{itemize}
Our official definition is the first one; various presheaf presentations will replace the traditional use of models. 
Notice that $n$-algebroids are ``evil'' as a notion of categories, taken up to isomorphisms instead of equivalences. This is why it contains the $(1, 1)$-category of strict $n$-categories, constructed from strict enrichment. 
We define a localization $n\Cat\subset n\Algbrd$ by fixing this:
\begin{itemize}
    \item The category of $\sV$-categories $\sV\mhy\Cat\subset\Algbrd(\sV)$ is the localization by \emph{categorical equivalences}, i.e., fully faithful and essentially surjective maps of algebroids. It suffices to invert $E\to \ast$, where  $E$ is (the base change to $\sV$ of) the contractible groupoid with two objects. Let $\sS \coloneqq 0\Cat$, $(n+1)\Cat \coloneqq (n\Cat)\mhy\Cat$ and let $\infty\Cat\coloneqq \colim (\cdots\hookrightarrow n\Cat\hookrightarrow (n+1)\Cat\hookrightarrow\cdots)$ in $\PrL$. 
    The localization $n\Cat\subset n\Algbrd$ are generated by the \emph{Rezk maps} $\sigma^k (E\to \ast)$, $0\leq k < n$. 
    The local objects are also called \emph{univalent} or \emph{Rezk-complete}, meaning that the prescribed groupoid of objects of the algebroid is in fact the maximal subgroupoid, so it can be recovered from the notion of equivalences \emph{internal} to the algebroid.
    \item The Rezk maps, considered as maps in $\PSh(\Theta)$, yields the localization $\PSh(\Theta_n)\to n\Algbrd\to n\Cat$. This is the original context of Rezk.
    \item In terms of flagged $n$-categories, the univalent complete objects are those with the \emph{maximal} flags.
\end{itemize}
Note that a strict $n$-category is usually not univalent. For instance, the delooping $\rB'\ZZ$ of $\ZZ$ as an algebroid with a single object $\ast$ is a strict $1$-category, but its univalent completion $\rB\ZZ$ is the circle $S^1$; the automorphisms of $\ast$ must be already in the groupoid of objects to be univalent. In fact, a strict $n$-category is univalent if and only if it is \emph{gaunt}, i.e., if no cell has a nonidentity automorphism \cite[\S 3]{barwickUnicityTheoryHigher2021}. 
Summarizing the discussion, we obtain the following diagram: 
\begin{equation}\label{diagram_sorts_of_nCat}
    \begin{tikzcd}
        n\Gaunt\ar[r, hook]\ar[d, hook] & n\strCat \ar[d, hook] \ar[r, hook] & \PSh_{\Set}(\Theta_n)\ar[d, hook]\\
        n\Cat\ar[r, hook] & n\Algbrd\ar[r, hook] & \PSh(\Theta_n)
    \end{tikzcd}
\end{equation}
Each inclusion in the diagram is right adjoint to an $\omega$-accessible localization. The top row consists of $(1, 1)$-categories and is the $0$-truncated part of the bottom row. The middle (resp. left) column is the part of the right (resp. middle) satisfying the Segal conditions (resp. univalence). In particular, these are compactly generated with a set of compact generators $\GG_n$.
The distinction between $n\Cat$ and $n\Cat^\rf = n\Algbrd$ is often irrelevant, but one is more appropriate in some cases. It is usually clear whether the discussion works in both settings or just one, but we will clarify when necessary. 

We write ${}^{\leq n}(\blank)$, $(\blank)^{\leq n}$ for the left and right adjoints of the inclusion $n\Algbrd\hookrightarrow \infty\Algbrd$.
The right adjoint $(\blank)^{\leq n}$ preserves univalence and homotopically $k$-truncated objects (in particular strictness or gauntness). We call $X^{\leq n}$ the \emph{underlying $n$-algebroid (or category)} or \emph{$n$-(categorical) truncation} (there is potential confusion with homotopical truncation, but the relevant notion is usually clear from the context).  On the other hand, the left adjoint ${}^{\leq n}(\blank)$ (the \emph{$n$-categorical localization}) does not preserve univalence or strictness in general. For example, $\rB\NN$ is a gaunt $1$-category freely generated by an object and an endomorphism, but its $0$-categorical localization depends on the ambient setting; in $0\Algbrd= 0\Cat= \sS$ we have ${}^{\leq 0}(\rB\NN)\simeq \rB\ZZ \simeq S^1$, whose image in $0\strCat$ and $0\Gaunt$ is $\pi_0 S^1 = \ast$. Moreover, $\rB^2\NN$ is a gaunt $2$-category but ${}^{\leq 1, \mathrm{algbrd}}(\rB^2\NN) = \rB' S^1$ is an algebroid generated by a point and a $S^1$-worth of automorphisms, whereas ${}^{\leq 1, \mathrm{cat}}(\rB^2\NN) = \rB S^1 \simeq \CC P^\infty$. 
The localization $\infty\Cat\twoheadrightarrow n\Cat$ preserves finite products \cite[Proposition 3.6.13]{stefanichHigherQuasicoherentSheaves2021}. 

Note that $\Algbrd_{X}(\blank)$ preserves limits of operads and thus of symmetric monoidal categories. Integrating the equivalences $\Algbrd_X(\infty\Algbrd)\simeq \lim_n\Algbrd_X(n\Algbrd)$ and $\Algbrd_X(\infty\Cat)\simeq \lim_n\Algbrd_X(n\Cat)$ over $X$ and with some univalence considerations, one sees that $\infty\Algbrd$ and $\infty\Cat$ are fixed points of the corresponding constructions \cite[Remark 3.6.12]{stefanichHigherQuasicoherentSheaves2021}:
\[\infty\Cat\simeq (\infty\Cat)\mhy\Cat,\quad \infty\Algbrd \simeq \Algbrd(\infty\Algbrd). \]
Moreover, \cite{goldthorpeHomotopyTheoriesInfty2023} shows that these are universal among such fixed points of enrichment in $\PrL$. 
\begin{remark}
    There are also a few combinatorial presentations of the categories $n\Cat$ and $n\Algbrd$ using marked simplicial and cubical sets \cite{verityWeakComplicialSets2008}\cite{campionCubicalModelInfty2020}. 
    These are convenient for many purposes; not only are they \emph{set-valued} presheaves on combinatorial shapes, but they also handle relative categories by design. 
    However, the representable presheaves are not fibrant. To compute the correct mapping groupoid one must perform a fibrant replacement, so the calculation is not as straightforward as in the $n=1$ case of (naturally marked) quasicategories.
    We will not rely on these combinatorial approaches because passing between them and ours is tricky. See \cite{loubatonComplicialSetsModel2022} for proof that the model category for $n$-complicial sets indeed model $n\Cat\in \widehat{\Cat}$.
\end{remark}

\section{Suspension}\label{section_suspension}
The suspension functor on $n\strCat$ extends to $n\Algbrd$. 
There are many equivalent definitions; we first give one based on enriched category theory. 

Let $\Assoc_{\{\bot, \top\}}$ be the nonsymmetric operad for two-object algebroids. 
There is a morphism of nonsymmetric operads $\Triv = \Spx_{\mathrm{inert}} \to \Assoc_{\{\bot, \top\}}$ characterized by $[1]\mapsto (\bot, \top)$, which by definition corepresents $\Hom_{(\blank)}(\bot, \top): \Algbrd_{\{\bot, \top\}}(\sV)\to \sV$. 
When $\sV$ has an initial object that is compatible with the monoidal structure, there is a left adjoint given by the operadic left Kan extension (for nonsymmetric operads, see \cite[\S A.4]{gepnerEnriched$infty$categoriesNonsymmetric2015}). 
Also notice that $\Algbrd_{\{\bot, \top\}}\subset \Algbrd(\sV)_{\ast\ast}$ is the fiber of the Cartesian fibration $\ob: \Algbrd(\sV)_{\ast\ast}\to \sS_{\ast\ast}$ over the initial object, so the inclusion admits a right adjoint that sends $X$ to its full subalgebroid spanned by the base objects. 

\begin{definition}(\cite[Example 3.3.6]{stefanichHigherQuasicoherentSheaves2021})
        The (unreduced) suspension functor is the composition of the left adjoints
        \[\sigma: \sV\to \Algbrd_{\{\bot, \top\}}(\sV)\hookrightarrow \Algbrd(\sV)_{\ast\ast},\]
        whose right adjoint is $(X, x_0, x_1)\mapsto \Hom_X(x_0, x_1)$. In particular, we have $\sigma: n\Algbrd\to (n+1)\Algbrd_{\ast\ast}$ and, in the limit, $\sigma: \infty\Algbrd\to \Algbrd(\infty\Algbrd)_{\ast\ast}\simeq \infty\Algbrd_{\ast\ast}$. 
\end{definition}
\begin{remark}
    \begin{enumerate}
        \item If $\sV$ is a $1$-category, $\sigma$ factors through $\sV\to \Algbrd(\sV)|_{\Set, \ast\ast} = \sV\mhy\strCat_{\ast\ast}$. This agrees with the suspension of strictly $\sV$-enriched categories. More generally, the suspension functor is fully faithful and has the same description of hom objects as in \cref{definition_suspension_strict} (use the description of the operadic left Kan extension). In particular, $\sigma$ on $\infty\Algbrd$ restricts to the $\sigma$ on $\infty\strCat$ previously defined. 
        \item The suspension $\sigma$ lands in $\sV\mhy\Cat_{\ast\ast}$ unless $\sV=\ast$. To see this, assume there is a nontrivial map $E\to \sigma X$, so $f: 1_{\sV}\to X \simeq \Hom_{\sigma X}(\bot, \top)$ and $g: 1_{\sV}\to \emptyset\simeq \Hom_{\sigma X}(\top, \bot)$ are inverses of each other. Then $\Hom_{\sigma X}(\top, \top) \xrightarrow{\id \otimes g\otimes f} \Hom_{\sigma X}(\top, \top)\otimes \Hom_{\sigma X}(\top, \bot)\otimes \Hom_{\sigma X}(\bot, \top)\to \Hom_{\sigma X}(\top, \top)$ is an equivalence, so $1_{\sV}\simeq \Hom_{\sigma X}(\top, \top)$ is a retract of $\emptyset$. It follows that $1_{\sV}\simeq \emptyset$, inducing a natural equivalence $\id_{\sV}\simeq \mathrm{const}_{\emptyset}$, i.e., $\sV$ must be trivial. 
        \item In particular, when $X$ is an $\infty$-category (i.e., univalent), $\sigma X$ is also univalent, so we will use the same notation for the restricted functors. 
        \item There is also a definition based on $\Theta$-presheaves \cite[Theorem 2.25]{campion$inftyn$categoricalPastingTheorem2023}. Notice that by our definition of $\Theta$, the suspension of \cref{definition_suspension_strict} restricts to a functor $\sigma: \Theta_n\to (\Theta^\can_{\ast\ast})\cap (n+1){\Cat}_{\ast\ast}$. 
        Let $\tilde{\sigma}: \PSh(\Theta_n)\to \PSh(\Theta_{n+1})_{\ast\ast}$ be the unique colimit-preserving extension. This restricts to a colimit-preserving functor $\sigma: n\Algbrd\to (n+1)\Algbrd_{\ast\ast}$. 
        The two definitions of $\sigma$ given above are equivalent, as both are colimit-preserving and agree on $\Theta_n$.  
    \end{enumerate}
\end{remark}
Later, we will give another description using the Gray tensor product in \cref{pushout_formula_unreduced_suspension}.

\section{Duality}\label{section_duality}
All categories in the diagram \ref{diagram_sorts_of_nCat} have the same group of automorphisms: 
\begin{proposition}
    Let $\sC$ denote one of $\Gaunt$, $\strCat$, $\Cat$, $\Algbrd$ and $0\leq n\leq \infty$. 
    Then any automorphism of $n\sC$ preserves the subcategories $\GG_n$, $\Theta_n$, and the restriction $\Aut(n\sC)\to \Aut(\GG_n)\simeq (\ZZ/2)^n$ is an equivalence. 
\end{proposition}
When $n$ is finite, the proposition is \cite[Theorem 4.13, Lemma 10.2]{barwickUnicityTheoryHigher2021} for $n\Gaunt$, $n\Cat$, and the same argument works for $n\strCat$, $n\Algbrd$. 
The key idea is to characterize the $n$-cell $C_n$ as an object of the abstract category $n\Gaunt$, $n\Cat$, etc., as the smallest generator (i.e., an object corepresenting a conservative functor) with respect to the retract relation. This shows that any automorphism restricts to an automorphism of $\GG_n$ (which is the identity on objects).
Each copy of $\ZZ/2$ in $\Aut(\GG_n)\simeq (\ZZ/2)^n$ corresponds to flipping the cosource and cotarget maps $s, t: C_{k-1}\to C_{k}$ for $0<k\leq n$. These automorphisms uniquely extend to $\PSh(\Theta_n)$ and fix all relevant subcategories. 

The same idea does not apply directly when $n=\infty$ because the infinite cell $C_\infty$ is a proper retract of itself.
However, the following lemma inductively reconstructs the subcategories $n\sC\subset \infty\sC$ from the abstract category $\infty\sC$\footnote{This idea is based on the post \cite{MOAuthoCat} characterizing the posets inside $\ho\Cat_{(1, 1)}$.}. Consequently, any automorphism of $\infty\sC$ preserves the subcategory $n\sC$ and therefore restricts to $\GG_n$ for any $n\geq 0$, and the proposition follows.
\begin{lemma}
    \begin{enumerate}
        \item $0\sC\subset \infty\sC$ is the colimit-closure of the terminal object. 
        \item Assume $n\geq 1$ and suppose we have already identified the full subcategory $(n-1)\sC\subset \infty\sC$ and, in particular, the right adjoint $(\blank)^{\leq n-1}$ to the inclusion.
        Let $n\sC'\subset \infty\sC$ be the full subcategory spanned by the objects $X$ satisfying the following condition:
        \begin{quote}
            for any (homotopically) $0$-truncated object $Y\in (\infty\sC)_{\leq 0}$, the counit $Y^{\leq n-1}\to Y$ induces a monomorphism $\Map(Y, X)\to \Map(Y^{\leq n-1}, X)$ in $\sS$. 
        \end{quote} 
        Then $n\sC$ is the colimit-closure of $n\sC'\subset \infty\sC$. 
    \end{enumerate}
\end{lemma}

\begin{proof}
    The first point is clear. For the second point, it suffices to prove $\GG_n\subset n\sC'\subset n\sC$.
    First, we show $\GG_n\subset n\sC'$. 

    Consider the $(-1)$-truncation functor $\Set\to \{\emptyset, \ast\}$ in $\PrL$. As the localization is product-preserving, it induces the \emph{underlying $n$-preordered set} functor $\Set(\mhy\strCat)^n\to \{\emptyset, \ast\}(\mhy\strCat)^n$. Note that the latter contains $\GG_n$. 
    For any $X\in \{\emptyset, \ast\}(\mhy\strCat)^n$, the map $(s_{n-1}, t_{n-1}): X_{n}\rightarrow X_{n-1}\times X_{n-1}$ is mono and $X_{\geq n}$ is constant. It follows that for any strict $\infty$-category $Y$, 
    the map $\Map_{\PSh_{\Set}(\GG)}(Y_\bullet, X_\bullet) \to \Map_{\PSh_{\Set}(\GG)}(Y^{\leq n-1}_\bullet, X_\bullet)$ is a monomorphism. Since $\infty\strCat\to \PSh_{\Set}(\GG)$ is faithful, $\Map_{\infty\strCat}(Y, X)\to \Map_{\infty\Cat}(Y^{\leq n-1}, X)$ is also a monomorphism. 
    Next assume $X\in n\sC'$. We must show that any $k$-cell of $X$ is degenerate for $k>n$, i.e., the map $\Map(C_k, X)\to \Map(C_{k+1}, X)$ induced by the projection $C_{k+1}\twoheadrightarrow C_k$ is an isomorphism.
    We claim that the following conditions are equivalent when $k\geq n$:
    \begin{enumerate}
        \item The map $\Map(C_k, X)\to \Map(\partial C_k, X)$ induced by $\partial C_k\hookrightarrow C_k$ is mono. 
        \item The map $\Map(C_k, X)\to \Map(\partial C_{k+1}, X)$ induced by $\partial C_{k+1}\twoheadrightarrow C_k$ is an isomorphism. 
        \item The maps $\Map(C_k, X)\to \Map(C_{k+1}, X)\to \Map(\partial C_{k+1}, X)$ induced by $\partial C_{k+1}\hookrightarrow C_{k+1}\twoheadrightarrow C_k$ are isomorphisms. 
    \end{enumerate}
    (1) and (2) are equivalent because the second map is the diagonal of the first. (2) and (3) are equivalent because the maps in (3) are mono by assumption: they are maps between subobjects of $\Map(\partial C_n, X)$ because $(\partial C_{k+1})^{\leq n-1} = (C_{k+1})^{\leq n-1} = (C_k)^{\leq n-1} = \partial C_n$ when $k\geq n$. 
    Now, again by assumption, (1) is true for $k=n$. By induction using the equivalence, we see that (3) is true for all $k\geq n$, so we have $X\in n\sC$. 
\end{proof}
\begin{remark}
    The category $n\Gaunt'$ is the category of so-called $n$-posets; by convention, a $(-1)$-poset is the terminal object, and $n$-posets are those enriched in $(n-1)$-posets. 
\end{remark}

\begin{definition}
    For a function $\tau: \ZZ_{\geq 1}\to \ZZ/2$, we let $D_\tau$ denote the corresponding involution of the categories of $n$-categories.
    That is, if we let $s_k, t_k: C_k\to C_{k-1}$ be the $k$-th co-source and co-target maps, $D_\tau$ is characterized by 
        \[D_{\tau}(s_k) = \begin{cases}
            s_k & \text{if $\tau(k)=0$},\\
            t_k & \text{if $\tau(k)=1$}.
        \end{cases}\]
    One can think of $\tau$ as the indicator function of the dimensions of the cells that get flipped. 
\end{definition}

The following copy of $\ZZ/2\times\ZZ/2\subset \Aut(\infty\Algbrd)$ is of dimension-independent importance (one reason is \cref{total_dual}; the proof does not work for a general $\tau$):

\begin{definition}
    The \emph{odd dual} (resp.\ \emph{even dual}) flips $s, t: C_{k-1}\to C_k$ for $k$ odd (resp.\ even), i.e., they are the duality involution $D_\tau$ when $\tau$ is the indicator function of the odd (resp.\ even) numbers. 
    We denote the odd and even duals by $(\blank)^\op$ and $(\blank)^\co$, respectively. 
    The \emph{total dual} flips cells of all dimensions, i.e., $(\blank)^\coop$, which we denote by $(\blank)^\circ$ or $D$. 
\end{definition}

\section{The cubes and the Gray tensor product}\label{section_cubes_and_Gray_tensor}
The (lax) Gray tensor product is a monoidal structure on $\infty\Algbrd$ and its various localizations. It differs from the Cartesian product in a few important ways: 
\begin{itemize}
    \item When $X, Y$ are $m$-, $n$-categories respectively, the Gray tensor product $X\otimes Y$ is an $(m+n)$-category while the Cartesian product $X\times Y$ is a $\max\{m, n\}$-category.
    \item The Cartesian product is symmetric, but the Gray tensor product is not. It behaves like a $*$-algebra with respect to the odd (or even) dual. 
    \item Both the Cartesian product and the Gray tensor product are closed with the terminal unit. However, the Gray internal hom classifies the functor category with \emph{(op)lax} natural transformations, while the Cartesian internal hom classifies the functor category with strong natural transformations. 
\end{itemize}
\begin{remark}
    In the classical $2$-categorical literature, sometimes the Gray tensor product refers to the \emph{pseudo}-Gray tensor product (as opposed to the \emph{lax} or oplax Gray tensor product). We will never use this language because the pseudo-Gray tensor product is just the Cartesian product in our natively homotopical setting. 
\end{remark}

We will freely use Steiner's theory (see \cref{appendix_Steiner_theory} for a review). Here, let us only recall the existence of an adjunction
\[\begin{tikzcd}
    \infty\strCat \arrow[r, shift left = 1ex, "\lambda"]
    & \adCh \arrow[l, shift left = .5ex, "\nu"]
    \arrow[l, phantom, shift right = .2ex, "\scriptscriptstyle\boldsymbol{\bot}"]
\end{tikzcd}\]
that restricts to an equivalence $\infty\Gaunt^\Ste\simeq \adCh^\Ste$ of the full subcategories of \emph{strong Steiner objects} on both sides (for the $\infty\strCat$ side they are automatically gaunt). $\adCh$ is the category of \emph{augmented directed complexes} (homologically graded augmented chain complexes with additional data of a \emph{positivity submonoid}). Strong-Steinerness is a reasonably checkable ``free and loop-free'' condition and is satisfied by many simple and combinatorially important gaunt $\infty$-categories (an important non-example is the walking adjunction category $\Adj$, however). Steiner's theory gives a neat way to define the (lax) Gray tensor product of strict $\infty$-categories; it corresponds to the tensor product of chain complexes.
We endow the category $\adCh$ with a monoidal structure by the (reversed) Koszul sign rule (i.e., $\partial (x\otimes y) = (-1)^{\deg(y)}\partial x\otimes y + x\otimes \partial y$). One can check that $\adCh^\Ste\subset \adCh$ is a monoidal subcategory. 
\begin{definition}
    \begin{enumerate}
        \item There exists a unique biclosed monoidal structure, called the \emph{(strict) Gray tensor product}, on $\infty\strCat$ such that $\adCh^\Ste\simeq \infty\Gaunt^{\Ste}\hookrightarrow \infty\strCat$ promotes to a monoidal functor. 
        \item The \emph{cube category} $\Cube\subset \infty\Gaunt^{\Ste}\subset\infty\strCat$ is the monoidal full subcategory generated by the interval $\cube^1 \coloneqq C_1$. The objects of $\Cube$ are the \emph{$n$-cubes} $\cube^n\coloneqq (\cube^1)^{\otimes n}\in n\Gaunt$. 
    \end{enumerate}
\end{definition}

\begin{remark}
    Campion \cite{campionCubesAreDense2022} shows that
    $\Theta$ is contained in the idempotent completion of $\Cube$ (see also \cref{cubes_and_orientals_are_dense} for another proof with the same idea). 
    In particular, $\Cube\subset\infty\strCat$ is dense, i.e., the left Kan extension $\PSh_{\Set}(\Cube)\to \infty\strCat$ is a localization. 
    This shows the uniqueness part of (1); moreover it is characterized as the unique biclosed monoidal structure promoting $\Cube\hookrightarrow \infty\strCat$ to a monoidal functor. 
\end{remark}

\begin{example}\label{pictures_of_the_cubes}
    The $n$-cube $\cube^n$, or a strong Steiner category in general, is an example of \emph{computads} (a.k.a.\ \emph{polygraphs}). It is the relevant notion of freeness for strict $n$-categories, defined similarly to the notion of CW-complexes \cref{definition_polygraph_strong_steiner}.
    By definition, we have $\cube^n = \lambda\bigl((\cdots\to 0 \to \underline{?}\ZZ\xrightarrow{(1, -1)}\underline{0}\ZZ\oplus\underline{1}\ZZ)^{\otimes n}\bigr)$.
    A polygraphic basis element (i.e., an atomic cell) of $\cube^n$ is a string of the letters $0, 1, ?$. The number of $?$ is the dimension of the cell. The differential $\partial(\underline{?}) = \underline{1}-\underline{0}$, together with the Koszul sign rule, describes the domain and the codomain of these cells.     
    The following depicts the atomic cells of the first three cubes $\cube^1, \cube^2, \cube^3$: 
    \[\begin{tikzcd}
        0\ar[r] & 1,
    \end{tikzcd}\quad
    \begin{tikzcd}[row sep=2em]
        00\ar[r]\ar[d] & 01 \ar[d] \\
        10\ar[r]\ar[ru, Rightarrow, shorten <=2ex, shorten >=2ex] & 11,
    \end{tikzcd}\quad
    \begin{tikzcd}[column sep=small, row sep=1em]
    && 001\arrow[rrd] &&&&& 001\arrow[rrd]\arrow[dd] &&\\
    000\arrow[dd]\arrow[rrd]\arrow[rru] &&&& 011\arrow[dd] & 000\arrow[dd]\arrow[rru] &&&& 011\arrow[dd] \\
    && 010\arrow[dd]\arrow[rru]\arrow[uu, Rightarrow, shorten=2.5ex] && {}\arrow[r, "\Rrightarrow", phantom] & {} && 101\arrow[rrd]\arrow[rru, Rightarrow, shorten=1ex] && \\
    100\arrow[rrd]\arrow[rru, Rightarrow, shorten=1ex] &&&& 111 & 100\arrow[rrd]\arrow[rru]\arrow[rruuu, Rightarrow, shorten=6.5ex] &&&& 111. \\
    && 110\arrow[rru]\arrow[rruuu, Rightarrow, shorten=6.5ex] &&&&& 110\arrow[rru]\arrow[uu, Rightarrow, shorten=2.5ex] &&
    \end{tikzcd}\]
\end{example}
\begin{remark}
    A monoidal structure $\otimes'$ on $\Cube$ such that $\cube^n\otimes' \cube^m = \cube^{n+m}$ is completely characterized by the bifunctor $\otimes': \Cube\times \Cube\to \Cube$; since $\Cube$ is a $0$-truncated object of $\Cat$ (see \cref{cube_category_is_gaunt}), any $\AA_\infty$-structure is strictly associative, and associativity is a property of the underlying $\AA_2$-structure. 
\end{remark}
Next, we explain the Gray tensor product for weak categories. 
\begin{theorem}[{\cite{campionGrayTensorProduct2023}}]
    There exists a unique closed $\EE_1$-monoidal structure, called the \emph{(lax) Gray tensor product}, on $\infty\Algbrd$ promoting the inclusion $\infty\Gaunt^{\Ste}\hookrightarrow \infty\Algbrd$ to a strong monoidal functor. 
    Moreover, the reflective subcategories $\infty\Cat$, $\infty\strCat$, and $n\Algbrd$ are all exponential ideals (see the remark below). In particular, the tensor product localizes to these categories and their intersections. 
\end{theorem}
\begin{definition}
    We denote the Gray tensor product of $\infty\Algbrd$ and $\infty\Cat$ by $\otimes$, and the left and right internal homs by $\Fun^\oplax(\blank, \blank)$ and $\Fun^\lax(\blank, \blank)$\footnote{in the original version we used $[-, -]$ and $\llbracket-, -\rrbracket$.}, so we have a natural equivalence
    \[\Map_{\infty\Cat}(X, \Fun^\oplax(Y, Z))\simeq \Map_{\infty\Cat}(Y\otimes X, Z)\simeq \Map_{\infty\Cat}(Y, \Fun^\lax(X, Z))\] and similarly for $\infty\Algbrd$. 
    Beware that $X \otimes Y$ can be ambiguous up to univalent completion when $X, Y$ are $\infty$-categories; it is usually clear from the context if the tensor product is localized or not (there is no such ambiguity for the internal hom).
\end{definition}
\begin{remark}
    By definition, a full subcategory $\eC\subset \infty\Algbrd$ is an \emph{exponential ideal} if for any $Y\in \eC$ and $X\in \infty\Algbrd$, we have $\Fun^\lax(X, Y), \Fun^\oplax(X, Y)\in \eC$. When $L: \infty\Algbrd\to \eC$ is a localization, this means that $L$-equivalences are preserved by tensoring objects from both sides, or that there is a (necessarily unique) monoidal structure on $\eC$ promoting $L$ to a monoidal functor \cite[Proposition 2.2.1.9]{LurieHA}. If we denote the localized tensor product by $X \otimes^L Y\simeq L(X\otimes Y)$, the comparison map $X\otimes Y\to X\otimes^L Y$ is not an equivalence in general.
\end{remark}
\begin{remark}\label{Gray_tensor_cubes_compact}
    Because $\Cube\subset \infty\Algbrd$ is dense, promoting this inclusion to a strong monoidal functor is enough to characterize the biclosed monoidal structure on $\infty\Algbrd$; if it exists, it must be induced from the localization $\PSh(\Cube)\to \infty\Algbrd$, where $\PSh(\Cube)$ is endowed with the Day convolution monoidal structure. The same is true for further localizations. Since $\Cube$ is a set of compact generators, the Gray tensor product is compactly generated, i.e., $\infty\Cat^\otimes \in\Alg(\Pr^\sL_\omega)$.
\end{remark}
\begin{remark}\label{remark_Gray_cylinder_preserves_gauntness}
    The theorem in particular states that the weak and strict tensor products agree on the strong Steiner categories, i.e., the comparison map $X\otimes Y\to \tau_{\leq 0} (X\otimes Y)$ for $X, Y\in \infty\Algbrd$ (or $\infty\Cat$) is an equivalence when $X, Y$ are strong Steiner. 
    The monoidal category $\infty\Gaunt^{\Ste}$ in the theorem can be replaced by a larger category of \emph{torsion-free complexes}.
    The author does not know if gauntness is preserved in general by the Gray tensor product. However, Loubaton \cite[Theorem 4.3.3.26]{loubatonTheoryModels$infty2023} shows that the tensor product with the cubes preserves gauntness\footnote{Technically speaking, the theorem is proven for the monoidal structure transferred from Verity's monoidal structure on complicial sets, which was previously not known to be equivalent to Campion's monoidal structure. However, the theorem in particular proves that the transferred monoidal structure agrees with the strict one on $\Cube$, so by the characterization, the equivalence of the two monoidal structures follows.}.
\end{remark}
\begin{remark}
    The Gray tensor product is \emph{additive} on category levels, i.e., if $X$, $Y$ are $k$- and $l$-categories respectively, then $X\otimes Y$ is an $(k+l)$-category. To see this, consider the subcategory $\eC$ of pairs $(X, Y)\in \infty\Cat\otimes \infty\Cat$ (beware the use of $\otimes$ for presentable categories) such that $X\otimes Y$ is an $n$-category. $\eC$ is closed under colimits and contains $(\cube^k, \cube^l)$ for $k+l\leq n$, so it contains $\bigcup_{k+l\leq n}k\Cat\otimes l\Cat$. 
\end{remark}
\begin{remark}
    \cite{loubatonComplicialSetsModel2022}, \cite{loubatonTheoryModels$infty2023} also show the existence of the Gray tensor product by showing the equivalence between $\infty$-categories and complicial sets and allowing one to transfer the Gray tensor product of complicial sets \cite{verityGrayTensorProduct2023}. 
\end{remark}
The Gray tensor product is \emph{semiCartesian}, that is, its unit object is terminal. In this case, there is a natural transformation $X\otimes Y\to X\times Y$. The following lemma shows that this extends to a lax monoidal functor. This fact seems standard, but we include the proof because the author does not know of a proof in the literature:
\begin{lemma}
    Let $\eC^\otimes\to \Delta^\op$ be a semiCartesian monoidal category. Then there is a lax monoidal functor $\eC^\times\to \eC^\otimes$ from the Cartesian monoidal structure whose underlying functor is $\id_{\eC}$. 
\end{lemma}
\begin{proof}
    One can construct the opposite monoidal category $(\eC^\op)^\otimes$ by postcomposing the involution $(\blank)^\op:\Cat \to \Cat$ with the Segal object $\Delta^\op\to \Cat$ it classifies. 
    Then $\eC^\otimes$ is semiCartesian if and only if $(\eC^\op)^\otimes$ is \emph{unital} in the sense of \cite[Definition 2.3.1.1]{LurieHA}. Let $(\blank)^\mathrm{sym}: \Mon(\Cat)\to \CMon(\Cat)$ be the left adjoint to the forgetful functor. Then constructing a lax monoidal functor $\eC^\times\to \eC^\otimes$ which is the identity on the underlying categories is equivalent to constructing a lax monoidal functor $((\eC^\op)^\otimes)^\mathrm{sym}\to (\eC^\op)^\amalg$ which is the identity on underlying categories. Now \cite[Proposition 2.4.3.9]{LurieHA} says that a lax monoidal functor from a unital operad to a coCartesian operad is determined by the underlying functor.
\end{proof}
\begin{remark}
    Because $\otimes$ is closed and semiCartesian, the comparison morphism $X\otimes Y\to X\times Y$ is an isomorphism when $X$ or $Y$ is a $0$-category. It follows by adjunction that $\Fun^\lax(X, Y)^{\leq 0}\simeq \Map_{\infty\Cat}(X, Y)\simeq \Fun^\oplax(X, Y)^{\leq 0}$. 
\end{remark}
\begin{remark}\label{remark_lax_monoidal_structure_on_id_from_cart_to_semicart}
    In particular, the identity functor promotes to a lax monoidal functor $(\infty\Algbrd, \times)\to (\infty\Algbrd, \otimes)$ and $(\infty\Cat, \times)\to (\infty\Cat, \otimes)$. This induces a fully faithful change-of-enrichment functor $\Algbrd(\infty\Algbrd)\to \Algbrd(\infty\Algbrd^\otimes)$ and $\iota: \infty\Cat\simeq (\infty\Cat)\mhy\Cat\to (\infty\Cat^\otimes)\mhy\Cat$, which moreover has a left adjoint.
\end{remark}
\begin{definition}
    A \emph{Gray category} is a category enriched in the Gray tensor product of $\infty\Cat$, i.e., an object of $\infty\Cat^\otimes\mhy\Cat$.
\end{definition}
We close the section by showing that some of the duality functors interact well with the Gray tensor product. 
\begin{proposition}\label{total_dual}
    The total dual functor promotes to a monoidal endofunctor $\infty\Algbrd^\otimes \to \infty\Algbrd^\otimes$. 
    The odd and even dual functors promote to antimonoidal endofunctors $\infty\Algbrd^{\otimes} \to \infty\Algbrd^{\otimes^\oplax}$.
    Similar statements hold for $\infty\Cat$, $\infty\strCat$, $\infty\Gaunt$. 
\end{proposition}
\begin{proof}
    By the universal property of the Gray tensor product, it suffices to show the following:
    \begin{enumerate}
        \item The total dual $(\blank)^\circ$ restricts to $\infty\Gaunt^{\Ste}$ and promotes to a monoidal functor.
        \item The odd and even duals $(\blank)^\op$, $(\blank)^\co$ restrict to $\Cube$ and promote to antimonoidal functors\footnote{The previous version claimed that these duality functors restrict to strong Steiner complexes. However, Marcus Nicolas pointed out that this is false and provided the following counterexample: $\cube^2\cup_{\{10, 01\}}\cube^1$ is strong Steiner but its $\op$-dual $\cube^2\cup_{\{01, 10\}}\cube^1$ is not.}.
    \end{enumerate}
    These follow from the corresponding claims in the category $\adCh$ (see \cref{operations_on_ADCs}).
    To check the monoidal property in (1), note that $(A\otimes B)^\circ$ and $A^\circ\otimes B^\circ$ have the same underlying graded abelian group. The identification clearly respects the augmentation and the positive parts. For differentials, observe
    \begin{align*}
        \partial^{(A\otimes B)^\circ}(a\otimes b) = -\partial^{A\otimes B}(a\otimes b) &=-\bigl((-1)^{\deg b}(\partial^A a)\otimes b + a\otimes (\partial^B b)\bigr)\\
         &= (-1)^{\deg b}(\partial^{A^\circ} a)\otimes b + a \otimes (\partial^{B^\circ} b) = \partial^{A^\circ\otimes B^\circ}(a\otimes b).
    \end{align*}
    The proof of (2) is similar. 
\end{proof}

\section{The Gray suspension is the suspension}\label{section_Gray_suspension}
In this section, we prove the key lemma that connects the Gray tensor product and the self-enrichment $\infty\Algbrd \simeq \Algbrd(\infty\Algbrd)$. This will feed into \cref{suspension_is_tensor_S1} and ultimately become one of the main ingredients in the construction of the tensor product of categorical spectra. 
Another important byproduct is the pushout formula for $(\sigma X)\otimes \cube^1$; see \cref{pushout_formula_for_gray_cylinder_of_suspensions}.

\begin{lemma}\label{pushout_formula_unreduced_suspension}
    There are functors
    \[P, P^\circ: \infty\Algbrd\to \Fun(\Delta^1\times \Delta^1, \infty\Algbrd)\]
    whose values $P(X), P^\circ(X)$ are pushout squares with the following specifications (only the bottom arrows and the $2$-cells are not predetermined):
    \begin{equation}\label{equation_Gray_suspension}\begin{tikzcd}[column sep=huge]
        {\partial \cube^1\otimes X} \arrow[r, "\partial\cube^1\otimes (X\to \ast)"] \arrow[d, hook] & {\partial \cube^1} \arrow[d, hook] \\
        {\cube^1\otimes X} \arrow[r] & \sigma X , 
    \end{tikzcd} \quad
    \begin{tikzcd}
        {X\otimes \partial\cube^1} \arrow[r] \arrow[d, hook] & {\partial \cube^1} \arrow[d, hook] \\
        {X\otimes \cube^1} \arrow[r] & \sigma(X^\circ) .
    \end{tikzcd}\end{equation} 
\end{lemma}
\begin{remark}
    The diagrams restrict to $X\in \infty\Cat, \infty\strCat, \infty\Gaunt$. For strict categories, we could run the same proof with the strict tensor product, but by \cref{remark_Gray_cylinder_preserves_gauntness}, this coincides with the restriction of the lemma. 
    The left pushout formula is \cite[Lemma 3.8]{campionGrayTensorProduct2023} (based on the strict $\infty$-category case \cite[Cor. B.6.6]{araJointTranchesPour2020}), but we include a detailed proof for completeness. 
\end{remark}
\begin{proof}
    Notice the canonical equivalence $\sigma(X^\coop) = \sigma(X^\co)^\co$. From this, it follows that $P(X) = P(X^\co)^\co$, so it suffices to only construct the pushout square $P(X)$.

    We first analyze $P$, assuming such functors exist. First, $P(\emptyset)$ must be the identity square
    \[\begin{tikzcd}
        \emptyset\ar[r]\ar[d] & \partial\cube^1\ar[d]\\
        \emptyset\ar[r] & \partial\cube^1,
    \end{tikzcd}\]
    so $P$ must lift to functors $\infty\Algbrd\to \Fun(\Delta^1\times \Delta^1, \infty\Algbrd)_{P(\emptyset)/}$. 
    These are necessarily colimit-preserving because colimits in the codomain are computed componentwise (in each coslice category)\footnote{Recall that the colimit of $p: X\to Y_{y/}$ is (almost by definition) the colimit of the corresponding cone $\bar{p}: X^{\triangleleft}\to Y$.}. Let $\Fun^{\mathrm{po}}(\Delta^1\times \Delta^1, \infty\Algbrd)_{P(\emptyset)/}$ denote the (colimit-closed) full subcategory of the codomain consisting of pushout squares. 
    Notice that, once we have a functor $P|_{\Cube}: \Cube\to \Fun^{\mathrm{po}}(\Delta^1\times\Delta^1, \infty\Algbrd)_{P(\emptyset)/}$ with the components as specified, its unique colimit-preserving extension to $\PSh(\Cube)$ automatically descend to $\infty\Algbrd$ and meet the specifications:
    \begin{align*}
        \Fun(\Cube, \Fun(\Delta^1\times\Delta^1,\infty\Algbrd)_{P(\emptyset)/}) & \simeq \LFun(\PSh(\Cube), \Fun(\Delta^1\times\Delta^1,\infty\Algbrd)_{P(\emptyset)/}) \\
        &\hookleftarrow \LFun(\infty\Algbrd, \Fun(\Delta^1\times \Delta^1, \infty\Algbrd)_{P(\emptyset)/}),
    \end{align*}
    They descend to $\infty\Algbrd$ because each component of the square does (which is a consequence of the existence of the Gray tensor product), and the extended functor lands in pushout squares if the original functor does. 
    We construct $P|_\Cube$ in two steps: (1) construct the commutative squares valued in $\infty\Gaunt^\Ste\simeq \adCh^{\Ste}$ and (2) check that they are pushouts in $\infty\Algbrd$. Note that the (weak) tensor product of objects in $\infty\Gaunt^\Ste$, as well as the suspension and the total dual, agrees with the strict notion.
    \begin{enumerate}
        \item Since our definition of the Gray tensor product involves Steiner theory, so does the construction of the squares $P(\cube^n)$. 
        We define the functor $\tilde{P}: \adCh^\Ste\to \Fun(\Delta^1\times \Delta^1, \adCh^\Ste)$ by
        \[\tilde{P}(A) = \begin{tikzcd}
            \lambda(\partial\cube^1)\otimes A\ar[r]\ar[d] & \lambda(\partial\cube^1)\ar[d]\\
            \lambda\cube^1\otimes A \ar[r] & \sigma(A), 
        \end{tikzcd}\]
        with the bottom map (under the notation $\lambda\cube^1 = (e\ZZ\xrightarrow{\binom{-1}{1}} \bot\ZZ\oplus \top\ZZ)$)
        \begin{align*}
            A_q\otimes\bot\ZZ\oplus A_q\otimes \top\ZZ\oplus A_{q-1}\otimes e\ZZ\xrightarrow{(0, 0, 1)} A_{q-1} & \quad (q>0), \\
            A_0\otimes \bot \ZZ\oplus A_0\otimes \top\ZZ\xrightarrow{(\eps, \eps)} \bot\ZZ\oplus\bot\ZZ & \quad (q=0).
        \end{align*}
        It is straightforward to check that this defines a map of augmented directed complexes functorial in $A$, so it indeed defines a (strictly) commutative square 
        $\tilde{P}: \adCh^\Ste\to \Fun(\Delta^1\times \Delta^1, \adCh^\Ste)$. 
        Through the equivalence $\infty\Gaunt^\Ste\simeq \adCh^\Ste$ and restricting to $\Cube\cup\{\emptyset\}$, we obtain
        $P: \Cube\to \Fun(\Delta^1\times \Delta^1, \infty\Gaunt^\Ste)_{P(\emptyset)/}\subset \Fun(\Delta^1\times \Delta^1, \infty\Algbrd)_{P(\emptyset)/}$.
        \item It remains to show that $P(\cube^n)$ are pushout squares. We Kan-extend $P$ to $\infty\Algbrd$ and say that $X\in \infty\Algbrd$ is good if $P(X)$ is a pushout. The terminal category $\cube^0$ is clearly good. We show that if $X\in \infty\Gaunt^\Ste$ is good, then so is $X\otimes \cube^1$.
        By assumption, $P(X)\otimes \cube^1$ is a pushout, so we can factor $P(X\otimes \cube^1)$ as follows:
        \[\begin{tikzcd}
            {\partial \cube^1\otimes X\otimes \cube^1} \ar[r] \ar[d, hook]\ar[rd, phantom, "\ulcorner", very near end] & {\partial \cube^1\otimes \cube^1}\ar[r]\ar[d, hook] & {\partial \cube^1} \ar[d, hook] \\
            {\cube^1\otimes X\otimes \cube^1} \ar[r] & (\sigma X)\otimes \cube^1 \ar[r] & \sigma(X\otimes \cube^1) , 
        \end{tikzcd}\]
        The outer rectangle is a pushout if and only if the right square is. 
        Factor the right squares into $(a)+(b)$ in the next diagrams by factoring the left inclusions as follows ($0, 1$ (resp. $\bot, \top$) are the source and the sink of $\cube^1$ (resp. $\sigma X$); note that wedge sums are both strict and weak \cite[Theorem 2.31]{campion$inftyn$categoricalPastingTheorem2023}):
        \begin{align*}
            \partial\cube^1\otimes \cube^1 &\hookrightarrow \bigl((\sigma X\otimes \{0\})\vee (\{\top\}\otimes \cube^1)\bigr) \sqcup \bigl((\{\bot\}\otimes \cube^1)\vee (\sigma X\otimes \{1\})\bigr) \to (\sigma X)\otimes \cube^1;%
        \end{align*}
        \[\begin{tikzcd}
            &\partial\cube^1\otimes \cube^1 \ar[r]\ar[d]\ar[rd, phantom, "(a)"] & \ast\sqcup \ast \ar[d]\\
            \sigma X\sqcup \sigma X\ar[r]\ar[d]\ar[rd, phantom, "(c)"]&\sigma X\vee \cube^1 \sqcup \cube^1\vee\sigma X \ar[r]\ar[d]\ar[rd, phantom, "(b)"] & \sigma X\sqcup \sigma X \ar[d]\\
            \sigma(X\otimes \cube^1)\ar[r] & (\sigma X)\otimes  \cube^1 \ar[r] & \sigma(X\otimes \cube^1)
        \end{tikzcd}\]
        Since $(a)$ is a pushout, it suffices to construct a section $(c)$ of $(b)$ (i.e., the horizontal composite $(b)+(c)$ is the identity) and to show that $(c)$ is a pushout. Again, $(c)$ can be explicitly constructed in $\adCh^{\Ste}$; horizontal maps pick up the ``long'' hom of the wedge sums and the ``diagonal'' hom of the tensor product. 

        It suffices to show that $(c)$ is a pushout in $\PSh(\Theta)$, i.e., for $\theta\in \Theta$, the ($\Set$-valued) squares $\Map(\theta, (c))$ are pushouts in $\sS$.
        Separating the cases based on the map on vertices, it reduces to the case $\theta = \sigma\theta'$ and, moreover, to checking that $\Hom(\theta', X\otimes \cube^1)\cong \Hom_{\ast\ast}(\sigma \theta', \sigma(X\otimes \cube^1))\to \Hom_{\ast\ast}(\sigma\theta', (\sigma X)\otimes \cube^1)$ is a bijection. Writing $\theta'$ as a colimit of cells, we may assume $\theta'=C_n$. In this case, one explicitly checks (on the Steiner complexes side) that any bipointed map $C_{n+1}\to (\sigma X) \otimes \cube^1$ must factor through $\sigma(C_{n}\to X\otimes \cube^1)$.
    \end{enumerate}
\end{proof}
    By colimit-extending the pushout square $(c)$ in the same fashion as in the first part of the proof (i.e., in the category under the value for $X=\emptyset$), we get the following pushout formula: 
\begin{corollary}\label{pushout_formula_for_gray_cylinder_of_suspensions}
    There exist the following pushout squares that are functorial in $X\in \infty\Algbrd$:
     \[\begin{tikzcd}
        \sigma X\sqcup \sigma X \ar[r]\ar[d]& \sigma X\vee \cube^1\sqcup \cube^1\vee \sigma X \ar[d] \\
        \sigma(X\otimes \cube^1) \ar[r] & (\sigma X)\otimes \cube^1, 
     \end{tikzcd}\quad
     \begin{tikzcd}
        \sigma X\sqcup \sigma X \ar[r]\ar[d]& \sigma X\vee \cube^1\sqcup \cube^1\vee \sigma X \ar[d] \\
        \sigma((\cube^1)^\op\otimes X) \ar[r] & \cube^1\otimes (\sigma X).
     \end{tikzcd}\]
    The top arrows pick up the ``long'' hom between the source and the sink objects of the wedge sum. The bottom arrows pick up the ``diagonal'' hom category.
\end{corollary}
Here $(\cube^1)^\op\cong \cube^1$ indicates that, with the vertex labels compatible with the other entries of the diagram, the $1$-morphism goes from $1$ to $0$. The right pushout can be obtained from the left by replacing $X$ by $X^\op$ and using $A^\op\otimes B^\op = (B\otimes A)^\op$ and $(\sigma A)^\co = \sigma(A^\op)$. 
For later use, we record a generalization of the lemma for higher suspensions and a consequence on iterated hom categories.
\begin{corollary}
    Let $k\geq 0$ be an integer. There is the following pushout square that is functorial in $X\in\infty\Algbrd$.
    \[\begin{tikzcd}
        \partial C_k\otimes X \ar[r, two heads]\ar[d]\ar[rd, phantom, "\ulcorner" very near end] & \partial C_k \ar[d]\\
        C_k\otimes X\ar[r, two heads] & \sigma^k X
    \end{tikzcd}\]
\end{corollary}
\begin{proof}
    The case $k=0$ is obvious, and $k=1$ is the first diagram of the lemma. We proceed by induction, so assume the conclusion for some $k\geq 1$. Plugging 
    $X=\sigma Y = (\cube^1\otimes Y)\cup_{\partial \cube^1\otimes Y}\partial \cube^1$ into the induction hypothesis, we see that $\sigma^{k+1} Y$ is the colimit of the following punctured cube diagram:
    \[\begin{tikzcd}[column sep=small, row sep=small]
        & \partial C_k\otimes \partial\cube^1\ar[rr]\ar[dd] && \partial C_k\ar[dd, dashed]\\
        \partial C_k\otimes \partial \cube^1\otimes\ar[rr, crossing over]\ar[ru]\ar[dd] Y && \partial C_k\otimes \cube^1\otimes Y \ar[ru] &\\
        & C_k\otimes \partial\cube^1\ar[rr, dashed] && \sigma^{k+1} Y\\
        C_k\otimes \partial\cube^1\otimes Y\ar[rr]\ar[ru] && C_k\otimes \cube^1\otimes Y\ar[ru, dashed]\ar[from=uu, crossing over]&
    \end{tikzcd}\]
    The pushout of the span in the back face is $(C_k\sqcup C_k)\cup_{(\partial C_k\sqcup \partial C_k)} \partial C_k\simeq \partial C_{k+1}$, so we have the following pushout diagram:
    \[\begin{tikzcd}
        \left((C_k\otimes \partial \cube^1)\cup_{(\partial C_k\otimes \partial\cube^1)} (\partial C_k\otimes\cube^1)\right)\otimes Y  \ar[r]\ar[d] & \partial C_{k+1} \ar[d]\\
        C_k\otimes \cube^1\otimes Y \ar[r] & \sigma^{k+1} Y
    \end{tikzcd}\]
    The top arrow factors through $\partial C_{k+1}\otimes Y$ by pushing out along $(\partial C_k\otimes \cube^1\twoheadrightarrow\partial C_k)\otimes Y$ (note that the spans in the front and the back faces are, up to the factor of $Y$, different only by the $\cube^1$ in the entry $\partial C_k\otimes \cube^1\otimes Y$). Since $(C_k\otimes \cube^1)\cup_{\partial C_k\otimes\cube^1}\partial C_k \simeq C_{k+1}$, the pushout along the same map factors the bottom map as $C_k\otimes \cube^1\otimes Y\to C_{k+1}\otimes Y\to \sigma^{k+1} Y$.
\end{proof}
As $\sigma$ preserves contractible colimits, there is an adjunction $\begin{tikzcd}
    \infty\Algbrd_{A/} \ar[r, shift left=1ex, "\sigma"] & \infty\Algbrd_{\sigma(A)/} \ar[l, shift left=.5ex, "\omega"] \ar[l, phantom, shift right = .2ex, "\scriptscriptstyle\boldsymbol{\bot}"]\end{tikzcd}$
    for any $A\in\infty\Algbrd$. By iteration, we have 
    \[\sigma^k: \begin{tikzcd}
    \infty\Algbrd \ar[r, shift left=1ex, "\sigma"] & \infty\Algbrd_{\partial C_1/} \ar[l, shift left=.5ex, "\Hom"] \ar[l, phantom, shift right = .2ex, "\scriptscriptstyle\boldsymbol{\bot}"] \ar[r, shift left=1ex, "\sigma"] & \cdots \ar[l, shift left=.5ex, "\Hom"] \ar[l, phantom, shift right = .2ex, "\scriptscriptstyle\boldsymbol{\bot}"]\ar[r, shift left=1ex, "\sigma"] & \infty\Algbrd_{\partial C_k /} \ar[l, shift left=.5ex, "\Hom"] \ar[l, phantom, shift right = .2ex, "\scriptscriptstyle\boldsymbol{\bot}"]
    \end{tikzcd}: \omega^k.\]
    For a parallel pair of $(k-1)$-morphisms $(s_{k-1}, t_{k-1}): \partial C_k\to X$, let $X(s_{k-1}, t_{k-1})$ denote the mapping category $\omega^k(X, s_{k-1}, t_{k-1})$. By adjunction, we have the following pullback square:
    \[\begin{tikzcd}
        \Hom(Y, X(s_{k-1}, t_{k-1}))\ar[r]\ar[d] & \Hom(\sigma^k Y, X)\ar[d] \\
        \ast\ar[r] & \Hom(\partial C_k, X)
    \end{tikzcd}.\] 
    Composing with the pullback square of the corollary (after applying $\Hom(\blank, X)$), we obtain the following:
\begin{corollary}\label{pullback_formula_for_hom_category}
    Let $X\in\infty\Algbrd$ and $(s_{k-1}, t_{k-1}): \partial C_{k}\to X$ be parallel $(k-1)$-morphisms. Then we have the pullback square 
    \[\begin{tikzcd}
        X(s_{k-1}, t_{k-1}) \ar[r]\ar[d]\ar[rd, phantom, "\lrcorner" very near start] & {\Fun^\oplax(C_k, X)}\ar[d]\\
        \ast\ar[r, "{(s_{k-1}, t_{k-1})}"] & {\Fun^\oplax(\partial C_k, X)}.
    \end{tikzcd}\]
\end{corollary}

\chapter{Categorical spectra}
\label{chapter_categorical_spectra}
In this chapter, we introduce our main object of study: \emph{categorical spectra}. 
We define the category $\CatSp$ of categorical spectra in a completely analogous way to that of spectra, i.e., as the limit along the sequence of the \emph{loop} endofunctors on $\infty\Cat_\ast$, instead of $\sS_\ast$:
\[\CatSp\coloneqq \lim(\cdots \xrightarrow{\Omega}\infty\Cat_\ast\xrightarrow{\Omega}\infty\Cat_\ast).\]
Roughly speaking, $\Omega$ takes $(X, x)$ to the endomorphism category $\End_X(x)$ of the basepoint. 
In \cref{section_loop_suspension}, we start by studying the loop-suspension adjunction $\Sigma\dashv \Omega$ for pointed $\infty$-categories.
This is the \emph{reduced} version of $\sigma\dashv \Hom$ from the previous chapter, i.e., $\Sigma X = (\sigma X)/(\sigma \ast)$. 
Unlike in classical algebraic topology, $\Sigma$ is not equivalent to $\sigma$; the latter is simpler because of the absence of morphisms in one direction. 
Collapsing the basepoints breaks this feature and, in particular, does not preserve gauntness in general. In other words, it is essentially an operation for weak $\infty$-categories. 
As before, we first define the suspension using enriched category theory and connect it to the Gray tensor product, observing that the loop-hom adjunction is equivalent to the tensor-hom adjunction for $\vS^1$ with respect to the \emph{Gray smash product}, where $\vS^1$ is the \emph{directed circle}, the category freely generated by a loop on a basepoint.
Analogously to May's recognition principle, the delooping hypothesis states that the $n$-fold loop construction provides an equivalence between $\EE_n$-monoidal $\infty$-categories and \emph{$n$-connective} $\infty$-categories, i.e., those $\infty$-categories that are trivial up through the $(n-1)$-th categorical level. This will be made precise in \cref{section_connectivity_delooping_hypothesis}. 

We then define categorical spectra in \cref{section_categorical_spectra}. By an argument similar to that for infinite loop objects, we show that it naturally fits into the following pullback square: 
\[\begin{tikzcd}
    \CMon^\gp(\sS)\ar[r, hook, "\rB^\infty"]\ar[d, hook] & \Sp \ar[d, hook] \\
    \CMon(\infty\Cat)\ar[r, hook, "\rB^\infty"] & \CatSp.
\end{tikzcd}\]
The left column can be thought of as the \emph{connective part} of the right. We will see that the horizontal arrows have both left and right adjoints: the right adjoint takes the maximal Picard subgroupoid, and the left adjoint inverts all stable cells. These offer ways to extract information from categorical spectra in a classical form. 
Iteration of categorification often provides examples of categorical spectra. We try to give a unified description in \cref{section_categorical_spectra_stable_properties}. 
We end the chapter by studying finiteness properties of categorical spectra in \cref{section_categorical_spectra_finiteness_properties}. We can make definitions analogous to various finiteness properties of spectra. They are all equivalent for spectra, but for categorical spectra the situation is subtler. At this point, we only treat the formal implications between them.

This chapter largely overlaps with \cite[Chapter 13]{stefanichHigherQuasicoherentSheaves2021}. We refer the reader there for further details and different perspectives. 

\section{Loop and suspension}\label{section_loop_suspension}
Let $\sV$ be a presentable Cartesian closed category (whose default monoidal structure is the Cartesian one). We denote the category of pointed objects $\ast\xrightarrow{x} X$ (which we write $(X, x)$, $x\in X$) by $\sV_\ast$. Notice $\sV_\ast\coloneqq \sV_{\ast/} \simeq \Alg_{\EE_0}(\sV)\simeq \sS_\ast\otimes \sV$. We will omit the basepoint from the notation when there is no risk of confusion. 
\begin{definition}
    Notice that the underlying groupoid functor $\ob: \Algbrd(\sV)_\ast\simeq \Algbrd(\sV)\times_{\sS}\sS_\ast\to \sS_\ast$ is a Cartesian fibration. 
    The \emph{algebroid delooping} functor is the inclusion of the fiber over the initial object $\ast$: 
    \[\rB'_{\sV}: \Mon(\sV)\simeq \Alg(\sV)= \Algbrd_{\ast}(\sV)\hookrightarrow \Algbrd(\sV)_\ast,\]
    and the \emph{delooping} functor is its univalent completion $\rB_{\sV}: \Mon(\sV)\xrightarrow{\rB'_{\sV}}\Algbrd(\sV)_\ast\xrightarrow{L^\mathrm{uni}} \sV\mhy\Cat_\ast$.
    The \emph{loop} functor is its right adjoint (i.e., the Cartesian transport of $\ob$ along the basepoint map $\ast\to X$), $\Omega_{\sV}: \Algbrd(\sV)_\ast\to \Alg(\sV)$, or its restriction to $\sV\mhy\Cat_\ast$. We will omit $\sV$ in the subscript when it is not confusing. 
\end{definition}
\begin{definition}
    By abuse of notation, we continue to denote by $\Omega_{\sV}$ the functor that returns the underlying pointed (resp.\ unpointed) object of the loop:
    \[(\sV\mhy\Cat_\ast\hookrightarrow)\Algbrd(\sV)_\ast\xrightarrow{\Omega_{\sV}}\Alg(\sV)\to \Alg_{\EE_0}(\sV)\simeq \sV_{\ast} (\to \sV).\]
    These functors have left adjoints, called the \emph{suspension}, given by the (partial) composite
    \[\bigl(\sV\xrightarrow{(\blank)_+}\bigr) \sV_{\ast}\xrightarrow{\Free_{\EE_1/\EE_0}}\Alg(\sV)\xrightarrow{\rB'}\Algbrd(\sV)_{\ast}\bigl(\xrightarrow{L^\mathrm{uni}} \sV\mhy\Cat_\ast\bigr);\]
    we will use the notation $\Sigma'=\rB'\circ \Free$, $\Sigma = \rB\circ\Free$ and $\Sigma'_+$, $\Sigma_+$ for their unpointed versions.
\end{definition}

\begin{remark}\label{B_is_fully_faithful}
    $\Omega_{\sV}: \sV\mhy\Cat_\ast\to \Alg(\sV)_\ast$ depends functorially on the monoidal category $\sV$ via change-of-enrichment \cite[Remark 13.1.7]{stefanichHigherQuasicoherentSheaves2021}.
    The loop of a pointed $\sV$-algebroid $(X, x)$ is the object of endomorphisms $\End_X(x)$ endowed with a monoid structure by composition (so the basepoint is $\id_x$).
    In particular, the loop functor $\Omega_{\sV}: \Algbrd(\sV)_{\ast}\to \Alg(\sV)_\ast$ inverts fully faithful morphisms, so it factors through the univalent completion $\Algbrd(\sV)_\ast\to \sV\mhy\Cat_\ast$. It follows that $\rB$ is fully faithful: the unit $\id\to \Omega\rB\simeq \Omega L^{\mathrm{uni}} \rB' \simeq \Omega \rB'$ is an equivalence because $\rB'$ is fully faithful. 
\end{remark}
The following lemma relates the reduced and unreduced suspensions:

\begin{lemma}
    The composite $\sV_\ast\xrightarrow{\sigma}\Algbrd(\sV)_{\sigma(\ast)/} \xrightarrow{\cofib} \Algbrd(\sV)_\ast\bigl(\xrightarrow{L^{\mathrm{uni}}} \sV\mhy\Cat_\ast\bigr)$, 
    where the second arrow is the cobase change along $\sigma(\ast)\to \ast$,
    is equivalent to the reduced suspension $\Sigma'$ (resp.\ $\Sigma$).
\end{lemma}

\begin{proof}
    Observe that the right adjoint is given by $\Algbrd(\sV)_\ast\to \Algbrd(\sV)_{\sigma(\ast)/}\xrightarrow{\Hom_{(\blank)}(\ast, \ast)} \sV_\ast$, which is equivalent to $\Omega_{\sV}$.
\end{proof}
In other words, we have a cofiber sequence $\sigma(\ast)\xrightarrow{\sigma(x)} \sigma(X)\to \Sigma(X)$. 
Now we restrict attention to the most interesting case, when $\sV=\infty\Cat$ or $\infty\Algbrd$. We let $\Sigma$ (resp.\ $\Sigma'$) denote the 
\emph{endofunctor} $\infty\Cat_\ast\xrightarrow{\Sigma_{\infty\Cat}}(\infty\Cat)\mhy\Cat_\ast\simeq \infty\Cat_\ast$ (resp.\ its algebroid version), and similarly for $\Omega: \infty\Algbrd_\ast\to \Mon(\infty\Algbrd)$.

\begin{remark}
    The loop functor on $\infty\Algbrd_{\ast}$ restricts to $\infty\Cat_\ast\to \infty\Cat_\ast$ and $\sS_\ast = 0\Cat_\ast\to \sS_\ast$. 
    The latter agrees with the classical loop, which takes the groupoid of \emph{automorphisms} of the basepoint. 
    In contrast, the suspension functor $\Sigma: \infty\Cat_\ast\to \infty\Cat_\ast$ does not restrict to the classical suspension functor $\sS_\ast\to \sS_\ast$. Instead, the classical suspension is the delooping of the \emph{group completion} of the free monoid. 
\end{remark}

Now we combine the above lemma and \cref{pushout_formula_unreduced_suspension} to deduce the formula relating the Gray smash product (which we define now) and the suspension. This will later be the first level of the half-central structure of $\vS^1$.

\begin{definition}
    Since $\sS\to \sS_\ast$ is an idempotent algebra (with the smash product) in $\PrL$, the base-change promotes to a (symmetric) monoidal localization $\sS_\ast\otimes(\blank): \PrL\to \PrL$ and, in particular, to a functor $(\blank)_\ast: \Alg(\PrL)\to \Alg(\PrL)$, sending a presentably monoidal category to the category of its pointed objects with the ``smashed'' monoidal product. 
    Applying this construction to $\infty\Algbrd$ and $\infty\Cat$ with the Gray tensor product, we obtain the \emph{Gray smash product} on the pointed objects, which we denote by $\owedge$. 
    Explicitly, we have $X\owedge Y \simeq  \cofib(X\otimes \ast \sqcup_{\ast\otimes \ast} \ast\otimes Y\to X\otimes Y)$.
\end{definition}

\begin{definition}
    Let $\vS^1 = \rB\NN = \Sigma_{+}(\ast)$ (notice $\NN = \Free_{\EE_1}(\ast)$) be the \emph{directed circle}. Also let $\vS^n \coloneqq (\vS^1)^{\owedge n}$ be the \emph{directed $n$-sphere}.
\end{definition}

\begin{remark}
    The directed circle $\vS^1$ is a gaunt $1$-category, and it is the free category on the graph $\Delta^1/\partial\Delta^1$. In contrast, $\vS^n$ is not strict for $n> 1$; the next proposition shows that
    \[
    \vS^n = \Sigma^n S^0 = \rB^n\Free_{\EE_n}(\ast),
    \]
    and $\Free_{\EE_n}(\ast) \simeq \bigsqcup_{k\geq 0} \bigl(\mathrm{Conf}_k(\RR^n)_{h\Sigma_k}\bigr)$ is not a $0$-truncated homotopy type. 
\end{remark}
\begin{proposition}\label{suspension_is_tensor_S1}
    There exist canonical identifications, functorial in $X$:
    \[
    \Sigma X \simeq \vec{S}^1\owedge X,\quad X\owedge \vec{S}^1\simeq \Sigma X^\circ.
    \]
    In particular, this provides a natural isomorphism
    \[
    \tau_X: \vec{S}^1\owedge X\simeq \Sigma X \simeq X^\circ \owedge \vec{S}^1.
    \]
\end{proposition}

\begin{proof}    
    Consider the following diagram: 
    \[\begin{tikzcd}[row sep={35,between origins}, column sep={40,between origins}]
        & {\ast \otimes \ast}\ar{dd}\ar{dl} && {\ast\otimes \partial \cube^1} \ar{ll}\ar{rr}\ar{dd}\ar{dl} & & {\partial \cube^1} \ar{dd}\ar{dl} \\
        {X\otimes \ast}\arrow[dr, phantom, "\scriptstyle{(1)}"rotate=40]\ar{dd} && {X\otimes \partial \cube^1} \ar[crossing over]{ll} \ar[crossing over]{rr} \arrow[dr, phantom, "\scriptstyle{(2)}"rotate=40] && \partial\cube^1 \arrow[dr, phantom, "\scriptstyle{(3)}"rotate=40]\\
        & {\ast\otimes \vec{S}^1}\ar{dl}  && {\ast\otimes \cube^1} \ar{ll} \ar{rr} \ar{dl} & &  {\sigma\ast} \ar{dl} \\
        {X\otimes \vec{S}^1} && {X\otimes \cube^1} \ar{ll} \ar{rr} \ar[from=uu,crossing over] && {\sigma X^\circ} \ar[from=uu,crossing over]
    \end{tikzcd}\]
    All the arrows from back to front are induced by the basepoint $\ast\to X$. 
    The front-right and the back-right faces are the pushout diagrams of \cref{pushout_formula_unreduced_suspension}, whereas the front-left and back-left faces are pushouts by $\vec{S}^1 = \cube^1/\partial\cube^1$. 
    Since all the faces in front and back are pushouts, we obtain induced equivalences on the total cofibers:
    \[
    X\owedge\vec{S}^1\simeq \cofib(1)\leftarrow\cofib(2) \rightarrow \cofib(3) \simeq \Sigma X^\circ. 
    \]
    The other equivalence $\Sigma X\simeq \vec{S}^1\owedge X$ follows from the other pushout diagram of \cref{pushout_formula_unreduced_suspension}.
\end{proof}
\begin{remark}
    The first equivalence in the proposition is equivalent to the statement that the suspension $\infty\Cat_\ast\to \infty\Cat_\ast$ is a morphism in $\RMod_{\infty\Cat_\ast^\owedge}(\PrL)$. 
    The analogous statement for the usual smash product (for Cartesian product), instead of the Gray smash product, fails fundamentally: if $X$ is an $n$-category for $n\geq 1$, then $\vS^1\wedge X$ is still an $n$-category, while $\Sigma X$ is an $(n+1)$-category.
\end{remark}

\begin{remark}\label{loop_preserves_filtered_colimits}
    The loop $\Omega: \infty\Cat_\ast\to \infty\Cat_\ast$ preserves filtered colimits, or equivalently, the suspension $\Sigma$ preserves compact objects. 
    It is a general fact that filtered colimits of enriched categories commute with taking hom objects,
    but it also follows from the pushout formula and \cref{Gray_tensor_cubes_compact}.
\end{remark}

\section{Connectivity of \texorpdfstring{$\infty$-categories}{infinity-categories} and the delooping hypothesis}\label{section_connectivity_delooping_hypothesis}

In the literature, the delooping hypothesis is often phrased as ``an $(n+k)$-category with a single $0, 1, \cdots, (k-1)$-cells is equivalent to a $\EE_k$-monoidal $n$-category''; this is literally true in the flagged/algebroid setting, but such ``$k$-connective $(n+k)$-algebroids'' are usually not univalent. The goal of this section is to give a precise treatment of the notion of connectivity in the univalent setting. 

Let $\Map(X, Y)$ denote $\Hom(X, Y)^{\leq 0}\simeq [X, Y]^{\leq 0}$ (so the Cartesian and Gray enrichments are the same after $0$-truncation).
Recall that a functor $f: X\to Y$ in $\infty\Cat$ is essentially surjective if the map $f^{\leq 0}: X^{\leq 0}\to Y^{\leq 0}$ is an effective epimorphism, i.e., it induces a surjection on $\pi_0$. 

\begin{definition}\cite{lurieClassificationTopologicalField2009}
    Let $n\geq -1$ be an integer. We define the \emph{$n$-connectivity} of a map $f: X\to Y$ inductively as follows:
    \begin{itemize}
        \item By convention, any map $f$ is $(-1)$-connective. 
        \item If $n\geq 0$, a map $f$ is $n$-connective if $f$ is essentially surjective and, for any pair of objects $x, x' \in X$, the induced map $X(x, x')\to Y(fx, fx')$ is $(n-1)$-connective. 
    \end{itemize} 
    We say $f$ is \emph{$\infty$-connective} if it is $n$-connective for all $n$. 
\end{definition}
The notion of connectivity is closely related to that of surjectivity:
\begin{definition}[cf. \cite{baezLecturesNCategoriesCohomology2010}]
    Let $n\geq 0$. A morphism $f: X\to Y$ is \emph{$n$-surjective} if the natural map $\Map(C_n, X)\to \Map(\partial C_n,  X)\times_{\Map(\partial C_n, Y)} \Map(C_n, Y)$ is an effective epimorphism of groupoids (note the convention $\partial C_0=\emptyset$). 
\end{definition}
\begin{remark}
    A map is $0$-surjective iff $0$-connective iff essentially surjective. 
    When $n\geq 1$, a map $f: X\to Y$ is $n$-surjective if and only if, for any parallel pair of $(n-1)$-morphisms $(s_{n-1}, t_{n-1}): \partial C_n\to X$, the induced map $X(s_{n-1}, t_{n-1})\to Y(fs_{n-1}, ft_{n-1})$ is essentially surjective. 
    In fact, taking the $0$-truncation of \cref{pullback_formula_for_hom_category}, the latter statement is equivalent to the statement that the induced map 
    \[
    \fib(\Map(C_n, X)\to \Map(\partial C_n, X))\to \fib(\Map(C_n, Y)\to \Map(\partial C_n, Y))
    \]
    is an effective epimorphism for any choice of basepoint of $\Map(\partial C_n, X)$. This is equivalent to $\Map(C_n, X)\to \Map(\partial C_n, X)\times_{\Map(\partial C_n, Y)}\Map(C_n, Y)$ being an effective epimorphism because this property can be checked fiberwise over $\Map(\partial C_n, X)$.
\end{remark}

\begin{example}
    A morphism $f$ is $n$-surjective if and only if the $n$-truncation $f^{\leq n}: X^{\leq n}\to Y^{\leq n}$ is $n$-surjective. In particular, the inclusion $X^{\leq n}\hookrightarrow X$ is $k$-surjective for $k\leq n$. 
\end{example}

\begin{proposition}
    The following are equivalent: 
    \begin{enumerate}
        \item $f: X\to Y$ is $n$-connective. 
        \item $f: X\to Y$ is $k$-surjective for $0\leq k\leq n$. 
        \item $f^{\leq n}: X^{\leq n}\to Y^{\leq n}$ is $n$-connective. 
    \end{enumerate}
\end{proposition}

\begin{proof}
    Unpacking the induction, the $n$-connectivity of $f$ is equivalent to essential surjectivity together with the condition that for any $0\leq k\leq n-1$ and any parallel pair of $k$-morphisms $(s_k, t_k): \partial C_{k+1}\to X$, the map $X(s_k, t_k)\to Y(fs_k, ft_k)$ is $(n-k-1)$-connective. By downward induction on $k$, the latter condition can be weakened to the requirement that $X(s_k, t_k)\to Y(fs_k, ft_k)$ is essentially surjective, so the equivalence of (1) and (2) follows from the remark. Since the $k$-surjectivity of a map only depends on the $k$-truncation, (1) is equivalent to (3). 
\end{proof}
\begin{remark}
    The class of $n$-connected maps are closed under pullbacks, filtered colimits, and disjoint unions. %
    Classically, in an $(\infty, 1)$-topos, there is an inductive characterization of $n$-connected maps \cite[Proposition 6.5.1.18]{LurieHTT} using diagonals. In the higher-categorical setting, the diagonal $X\to X\times X$ of $\infty$-groupoids typically gets replaced either by $\Fun^\lax([1], X)\to X\times X$ or by $\mathrm{Tw}(X)\to X\times X$. We do not pursue such characterizations here. 
\end{remark}

\begin{definition}
    Let $n\in \ZZ_{\geq 0}\cup \{\infty\}$. 
    A pointed $\infty$-category $X$ is said to be \emph{$n$-connective} if the structure map $\ast\to X$ is $(n-1)$-connective.     
\end{definition}
\begin{remark}
    The definition extends to unpointed $\infty$-categories as follows: an $\infty$-category $X$ is always $(-1)$-connective. It is $0$-connective (resp.\ $1$-connective) if it is nonempty (resp.\ if $X^{\leq 0}$ is connected). For $n\geq 2$, $X$ is $n$-connective iff it is $1$-connective and $n$-connective for one (and therefore any) choice of a basepoint. 
\end{remark}
\begin{proposition}
    Let $(X, x)$ be a pointed $\infty$-category and $n\geq 0$. The following are equivalent: 
    \begin{enumerate}
        \item $X$ is $n$-connective, i.e., the structure map $\ast\xrightarrow{x} X$ is $(n-1)$-connective. 
        \item $\Omega^k X$ is connected (i.e., $(\Omega^k X)^{\leq 0}$ is connected) for $0\leq k\leq n-1$.
        \item The counit map of $\rB^n \Omega^n X\to X$ is an equivalence. 
        \item $X$ belongs to the essential image of the functor $\rB^n: \Mon_{\EE_n}(\infty\Cat)\to \infty\Cat_{\ast}$. 
        \item There exists an $\infty$-algebroid $\tilde{X}$ whose underlying $(n-1)$-algebroid is trivial and whose univalent completion is $\tilde{X}\xrightarrow{\sim} X$. 
    \end{enumerate}
\end{proposition}
\begin{proof}
    When $n=0$, all the conditions are empty, so we may assume $n\geq 1$ and proceed by induction. 
    When $n=1$, the equivalence of the first two conditions is clear from the definition. Since $\rB\Omega X\to X$ is always fully faithful, it is an equivalence if and only if $X$ is connected. This is also equivalent to (4) because $\rB$ is fully faithful by \cref{B_is_fully_faithful}. For (5), we may take $\tilde{X} = \rB'\Omega X$.
    Now assume $n\geq 1$. From the case $n=1$, we know that all the conditions imply $X = \rB Y$ for some $Y\in \Mon(\infty\Cat)$. 
    By induction, it suffices to observe that each condition is equivalent to the corresponding condition for $Y$ with $n$ replaced by $n-1$. This is clear for (1) and (2).
    For (3), observe that the counit map $\rB^n\Omega^n \rB Y\to \rB Y$ is equivalent to $\rB(\rB^{n-1}\Omega^{n-1} Y\to Y)$ and that $\rB: \infty\Cat_\ast\xrightarrow{\rB} \Mon(\infty\Cat)\xrightarrow{\text{forget}}\infty\Cat_\ast$ is conservative.
    The equivalence of (3) and (4) is easier to see directly because $\rB^n$ is fully faithful with right adjoint $\Omega^n$ (the existence of the functor $\rB^n$ follows from that $\rB$ is product preserving and Dunn additivity; see also the argument for \cref{CatSp_is_limit_of_SMC}). 
    For (5), if either of $\tilde{Y}\xrightarrow{\sim} Y$ or $\tilde{X}\xrightarrow{\sim} X$ is given, we may take $\tilde{X} = \rB'\tilde{Y}$ and $\tilde{Y} = \Omega \tilde{X}$ for the other. 
\end{proof}

\begin{proposition}
    If $(X, x)$ is a $n$-connective pointed $(n-1)$-category, then $X$ is in fact a groupoid.
\end{proposition}
\begin{proof}
    $\Omega^{n-1} X$ is an $\EE_{n-1}$-groupoid if $X$ is an $(n-1)$-category, so $\Omega^n X$ is a grouplike $\EE_n$-groupoid and $\rB^n\Omega^n X$ is a groupoid. If $X$ is connective, the counit map $\rB^n\Omega^n X\to X$ is an equivalence, so $X$ is also a groupoid. 
\end{proof}

\begin{corollary}
    A pointed $\infty$-category $(X, x)\in \Cat_\ast$ is $\infty$-connected if and only if $X$ is terminal. 
\end{corollary}
\begin{proof}
    The structure map $\ast\to X$ is $(n-1)$-connective if and only if $\ast\to X^{\leq n-1}$ is, so $X^{\leq n-1}$ is terminal by the proposition. Applying \cref{objects_of_colim_in_PrL_omega_is_colim}, we have $X=\colim_{n\in\NN} X^{\leq n}$, so $X\simeq \ast$. 
\end{proof}

\section{Categorical spectra}\label{section_categorical_spectra}
Our goal here is to give the definitions and some examples of categorical spectra, roughly summarizing \cite[Chapter 13]{stefanichHigherQuasicoherentSheaves2021}.
\begin{definition}
    The category of \emph{categorical spectra} is the limit of right adjoints 
    \[\CatSp\coloneqq \lim(\cdots\xrightarrow{\Omega}\infty\Cat_\ast\xrightarrow{\Omega}\infty\Cat_\ast)\]
    in $\Pr^\sR_{\omega}$ (or $\widehat{\Cat}$).
    Its objects, categorical spectra $X\in \CatSp$, are sequences $(X_n, x_n)_{n\in \NN}$ of pointed $\infty$-categories equipped with identifications $f_n: (X_n, x_n)\xrightarrow{\sim} (\End_{X_{n+1}}(x_{n+1}), \id_{x_{n+1}})$. We will often suppress $x_n$ and $f_n$ in the notation. 
    We write $\Omega^{\infty-n}$ for the projection to the $n$-th component, so $\Omega^{\infty-n} X = X_n$ for $X=(X_n)$. We let $\Sigma^{\infty-n}$ denote the left adjoint of $\Omega^{\infty-n}$.
    One can also define the obvious variant of flagged categorical spectra: $\CatSp^\rf\coloneqq \lim(\cdots\xrightarrow{\Omega} \infty\Algbrd_\ast\xrightarrow{\Omega} \infty\Algbrd_\ast)$. 
\end{definition}
\begin{example}
    Recall that the following diagram commutes:
    \[\begin{tikzcd}
        \sS_\ast\ar[r, "\Omega"]\ar[d, hook] & \sS_\ast \ar[d, hook] \\
        \infty\Cat_\ast \ar[r, "\Omega"] & \infty\Cat_\ast
    \end{tikzcd}\]
    Therefore the vertical maps induce a fully faithful functor $\Sp\hookrightarrow \CatSp$. We will see in \cref{categorical_truncations_of_catsp} that this inclusion admits both left and right adjoints, and in the next chapter we will see that the sphere spectrum $\SS$ is an idempotent algebra and $\Sp$ is precisely the category of $\SS$-modules.
\end{example}
\begin{remark}
    The endomorphism $\Omega$ on $\infty\Cat_\ast$ commutes with the limit diagram, so it induces an endomorphism $\Omega: \CatSp\to \CatSp$. This is actually an automorphism by the coinitiality of $(\NN_{\geq 1})^\op\hookrightarrow \NN^\op$, and its adjoint inverse $\Sigma$ is induced by taking the colimit of $\Sigma: \infty\Cat_\ast\to \infty\Cat_\ast$ in $\PrL$. 
    This is an example of the general fact that one can universally invert an endomorphism on a category by passing to the sequential limit along the endomorphism. 
    More precisely, the left adjoint to the inclusion $\Fun(\rB\ZZ, \widehat{\Cat})\hookrightarrow\Fun(\rB\NN, \widehat{\Cat})$ (resp.\ $\Fun(\rB\ZZ, \PrL)\hookrightarrow \Fun(\rB\NN, \PrL)$) sends the pair $(\infty\Cat_{\ast}, \Omega)$ (resp.\ $(\infty\Cat_\ast, \Sigma)$) of a category and an endomorphism to $(\CatSp, \Omega)$ (resp. $(\CatSp, \Sigma)$). 
    Moreover, it sends the morphism $(\sS_\ast, \Omega)\to (\infty\Cat_\ast, \Omega)$ of pairs to $(\Sp, \Omega)\to (\CatSp, \Omega)$. 
\end{remark}
\begin{warning}
    We will occasionally use the standard notation $X[n]\coloneqq \Sigma^n X$ for $n\in \ZZ$; we warn that this is the \emph{left} action by $\vS^n\in \infty\Cat_\ast^\otimes$. So, for instance, in the later chapters, $X\otimes (Y[n])$ and $(X\otimes Y)[n]$ are not necessarily equivalent. 
\end{warning}
\begin{remark}
    The reader may wonder what happens if we choose to work with a fixed finite category level and make the same definition. 
    It turns out that we gain nothing new: we end up with the category $\Sp$. We may regard $\Omega: n\Cat_\ast\to (n-1)\Cat_\ast$ as an endomorphism of $n\Cat_\ast$ by composing with the inclusion. Now note that $\Omega^n$ factors through $\sS_\ast$:
    \[\begin{tikzcd}
        \cdots \ar[r, "\Omega^n"] & \sS_\ast \ar[r, "\Omega^n"]\ar[d, hook] & \sS_\ast\ar[d, hook] \\
        \cdots \ar[ru, dashed, "\Omega^n"]\ar[r, "\Omega^n"] & n\Cat_\ast \ar[r, "\Omega^n"]\ar[ru, dashed, "\Omega^n"] & n\Cat_\ast
    \end{tikzcd}\]
    Taking the limit horizontally, we observe that the dashed arrows induce the inverse to the inclusion by the coinitiality of $(\NN_{\geq n})^\op\hookrightarrow \NN^\op$, so if we universally invert $\Omega: n\Cat_\ast\to n\Cat_\ast$, we recover the category $\Sp$.  
    If we allow the category levels to vary over the limiting diagram, there is a meaningful notion of categorical levels of categorical spectra (see \cref{definition_k_categorical_spectra}). 
\end{remark}
\begin{remark}\label{catsp_literature}
    The definition is not new and was made or indicated independently by many authors, including Horiuchi \cite{horiuchiObservationsSphereSpectrum2018}, the author \cite{MasudaOralExam} (informally, inspired by the works of Connes and Consani around $\FF_1$ \cite{connesOverlineOperatornameSpec2020}), Stefanich \cite{stefanichHigherQuasicoherentSheaves2021} (who attributes the notion to Constantin Teleman), and Johnson-Freyd \cite{pirsa_PIRSA:23090104} (who attributes the notion to Claudia Scheimbauer).
\end{remark}
Recall that the category of spectra is also the stabilization of the category of connective spectra: 
\[\Sp\coloneqq \lim(\cdots \xrightarrow{\Omega} \Sp^\cn\xrightarrow{\Omega} \Sp^\cn).\]
Here $\Sp^\cn\simeq \CMon^\gp(\sS)$ is the category of infinite loop groupoids. We now give an analogous description of $\CatSp$ as the ``stabilization'' of $\infty\SMC$, following \cite[\S 13.4]{stefanichHigherQuasicoherentSheaves2021}.
Note that $\Omega: \infty\Cat_\ast\to \Mon(\infty\Cat)$ preserves the Cartesian product. In particular, it induces
$\Omega: \Mon_{\EE_n}(\infty\Cat)\to \Mon_{\EE_n}(\Mon(\infty\Cat))\simeq \Mon_{\EE_{n+1}}(\infty\Cat)$ for $0\leq n\leq \infty$, and these commute with the forgetful functors between different $n$. Now consider the following diagram in $\Pr^\sR_{\omega}$: 
\[\begin{tikzcd}
    \vdots\ar[d, equal] &{\cdots} & \vdots\ar[d] & \vdots\ar[d] & \vdots\ar[d] \\    
    \CatSp\ar[d, equal] \ar[r]&{\cdots}\ar[r] & \Mon(\infty\Cat)\ar[r, "\Omega"]\ar[d] & \Mon_{\EE_2}(\infty\Cat) \ar[r, "\Omega"]\ar[d] & \Mon_{\EE_3}(\infty\Cat)\ar[d] \\
    \CatSp\ar[r]\ar[d, equal]&{\cdots}\ar[r] & \infty\Cat_\ast\ar[r, "\Omega"]\ar[d, equal] & \Mon(\infty\Cat) \ar[r, "\Omega"]\ar[d] & \Mon_{\EE_2}(\infty\Cat)\ar[d] \\
    \CatSp\ar[r]\ar[d, equal]&{\cdots}\ar[r] & \infty\Cat_\ast\ar[d, equal]\ar[r, "\Omega"] & \infty\Cat_\ast\ar[r, "\Omega"]\ar[d, equal] & \Mon(\infty\Cat)\ar[d] \\
    \CatSp\ar[r] &{\cdots}\ar[r] & \infty\Cat_\ast\ar[r, "\Omega"] & \infty\Cat_\ast \ar[r, "\Omega"]& \infty\Cat_\ast.
\end{tikzcd}\]
Taking the vertical limit (along the forgetful functors), one obtains the following:
\begin{proposition}\label{CatSp_is_limit_of_SMC}
    The diagram defining $\CatSp$ lifts to the following diagram in $\Pr_\omega^\sR$: 
    \[\CatSp\xrightarrow{\sim} \lim(\cdots\xrightarrow{\Omega} \CMon(\infty\Cat)\xrightarrow{\Omega} \CMon(\infty\Cat)).\]
    In terms of left adjoints, we have
    \[\begin{tikzcd}
    \infty\Cat_\ast\ar[r, "\Sigma"]\ar[d, "{\Free_{\EE_\infty/\EE_0}}"] & \infty\Cat_\ast\ar[r, "\Sigma"]\ar[d, "{\Free_{\EE_\infty/\EE_0}}"] & \cdots \ar[r] & \CatSp\ar[d, equal] \\
    \CMon(\infty\Cat)\ar[r, hook, "\rB"] & \CMon(\infty\Cat) \ar[r, hook, "\rB"] & \cdots\ar[r, hook] & \CatSp.
    \end{tikzcd}\]
    The bottom row lies in $\Mod_{\CMon}(\PrL)$, so $\CatSp$ is semiadditive, i.e., it has a zero object $0$ and biproducts $\oplus$.
    Also, we have an equivalence $\Sigma^\infty\simeq \rB^\infty\circ \Free_{\EE_\infty/\EE_0}$. 
\end{proposition} 
    We continue to denote the projection to the $n$-th component by $\Omega^{\infty-n}$. The left adjoint of $\Omega^{\infty-n}$ will be denoted by $\rB^{\infty-n}: \CMon(\infty\Cat)\to \CatSp$. By cofinality, we may $1$-periodically extend the above $\NN^{(\op)}$-indexed diagrams to $\ZZ$-indexed diagrams without changing the (co)limits; thus we use the notations $\rB^{\infty-n}$, etc., for $n\in \ZZ$.
\begin{definition}
    We let $\CatSp^\cn\subset \CatSp$ denote the essential image of $\rB^\infty$ (so $\CatSp^\cn\simeq \CMon(\infty\Cat)$) and call its objects \emph{connective} categorical spectra.
\end{definition}
\begin{warning}
    Even though we use the term ``connective,'' we do not yet know an appropriate definition of an analog of a $t$-structure on a stable $(\infty, 1)$-category.
\end{warning}
\begin{example}
    For $\eC\in \CMon(\infty\Cat)$, the categorical spectra $\rB^{\infty-n} \eC$ is the sequence of $\infty$-categories
    \[(\rB^{\infty-n}\eC)_k = \begin{cases}
        \Omega^{n-k}\eC & (k\leq n),\\
        \eC & (k=n), \\
        \rB^{k-n} \eC & (k\geq n).
    \end{cases}\]
\end{example}
    For the next definition, recall that the free $\EE_\infty$-algebra on $X\in \infty\Cat$ (resp.\ $\infty\Cat_\ast$) is given by the symmetric algebra $\bigsqcup_{n\geq 0} X^{\times n}/\Sigma_n$ (resp.\ $\bigvee_{n\geq 0} X^{\wedge n}/\Sigma_n$). 
\begin{definition}
    A categorical spectrum of the form $\Sigma^\infty X$ for $X\in \omega\Cat_\ast$ (in particular, $\Sigma_+^\infty X = \rB^\infty\Free_{\EE_\infty} X$ for $X\in \infty\Cat$) is called a \emph{suspension spectrum}. We let 
    \[\FF \coloneqq \Sigma^\infty_+ (\ast) = \Sigma^\infty S^0 = \rB^\infty\Fin^{\simeq}\] be the suspension spectrum on a point. We will call it the \emph{finite set spectrum}, the \emph{unit}, or the \emph{directed sphere spectrum}. 
\end{definition}
\begin{remark}
    Since $\Omega$ preserves filtered colimits (\cref{loop_preserves_filtered_colimits}), $\Omega^\infty: \CatSp\to \infty\Cat_\ast$ is in $\Pr^\sR_\omega$. It follows that $\FF$ is a compact object of $\CatSp$. 
\end{remark}

We close this section with some remarks on the duality involutions.
\begin{remark}(\cite[Definition 13.2.12]{stefanichHigherQuasicoherentSheaves2021})\label{remark_cells_of_categorical_spectra}
    Just as a spectrum can be seen as a ``CW complex with possibly negative-dimensional cells,'' one can regard a categorical spectrum as an ``$\infty$-category with negative-dimensional cells.'' To make this precise, let $(X_n)$ be a categorical spectrum and consider the map
    \[\Map(C_m, X_n)\xrightarrow{\sim} \Map(C_m, \Omega X_{n+1})\xrightarrow{\sim} \Map(\Sigma C_m, X_{n+1})\to \Map(C_{m+1}, X_{n+1})\] 
    induced by $C_{m+1} = \sigma C_m\twoheadrightarrow \Sigma C_m$.
    We define the pro-representable globular presheaf $\cell_\bullet(\blank): \CatSp\to \PSh(\GG)$ as the colimit of the following diagram of monomorphisms:
    \[\Map(C_\bullet, X_0)\rightarrowtail \Map(C_{\bullet+1}, X_1)\rightarrowtail \Map(C_{\bullet+2}, X_2)\rightarrowtail \cdots. \]
    Moreover, this can be extended to a presheaf over $\GG_{(-\infty, \infty)} = \colim(\GG\xrightarrow{\sigma}\GG\xrightarrow{\sigma}\cdots)$ and given a structure of compositions. 
\end{remark}
\begin{remark}\label{remark_involutions_of_catsp}
    Let $\tau: \ZZ_{\geq 1}\to \ZZ/2$ be a function and let $[n]$ (for $n\leq 0$) be the shift operator, i.e., $\tau[n]$ is the function $\tau[n](k) = \tau(k-n)$. Then the following diagram commutes (to construct the commuting homotopy, observe that the two compositions are equal on $\Theta$):
    \[\begin{tikzcd}
        \infty\Cat\ar[r, "\sigma"]\ar[d, "D_{\tau[-1]}"] & \infty\Cat\ar[d, "D_{\tau}"] \\
        \infty\Cat \ar[r, "\sigma"] & \infty\Cat.
    \end{tikzcd}\]
    Because $\Sigma X = \cof(\sigma(\ast)\to \sigma X)$, it follows that the following diagrams also commute (the right one is obtained by taking right adjoints; the truncated information of $\tau(1) = \tau[-1](0)$ acts by reversal of the monoidal structure, which is invisible on the underlying category):
    \[\begin{tikzcd}
        \infty\Cat_\ast\ar[r, "\Sigma"]\ar[d, "D_{\tau[-1]}"] & \infty\Cat_\ast\ar[d, "D_{\tau}"] \\
        \infty\Cat_\ast \ar[r, "\Sigma"] & \infty\Cat_\ast, 
    \end{tikzcd}
    \quad\begin{tikzcd}
        \infty\Cat_\ast \ar[r, "\Omega"]\ar[d, "D_\tau"] & \infty\Cat_\ast\ar[d, "D_{\tau[-1]}"] \\
        \infty\Cat_\ast \ar[r, "\Omega"] & \infty\Cat_\ast
    \end{tikzcd}.\]
    Consequently, the category $\CatSp$ admits an action of $\prod_\ZZ \ZZ/2\simeq \lim(\cdots \to \prod_{\ZZ_{\geq 1}} \ZZ/2\xrightarrow{[-1]}\prod_{\ZZ_{\geq 1}} \ZZ/2).$
    We continue to denote by $D_{\tau}$ the involution corresponding to $\tau: \ZZ\to \ZZ/2$. 
    By definition, this is the involution satisfying $D_{\tau[n]}\circ \Omega^{\infty-n} = \Omega^{\infty-n}\circ D_{\tau}$ and $D_\tau\circ\Sigma^{\infty-n}= \Sigma^{\infty-n}\circ D_{\tau[n]}$; i.e., it sends $(X_n)$ to $(D_{\tau[n]}X_n)$ in components. One may think of $\tau$ as the indicator function of the dimensions of stable cells that get flipped. 
\end{remark}
\begin{definition}\label{definition_duality_involutions_of_catsp}
    The \emph{total dual} (resp.\ \emph{odd dual}, \emph{even dual}) is the duality involution $D_\tau$ corresponding to the case $\tau$ is the indicator function of $\ZZ$ (resp.\ the odd numbers, the even numbers). We continue to denote the total dual by $D = (\blank)^\circ$ and the odd and even duals by $(\blank)^\op$, $(\blank)^\co$, respectively. 
\end{definition}
Explicitly, the total dual of $X=(X_n)$ is given by $X^\circ=(X_n^\circ)$, and the odd dual $X^\op$ is given by the ``alternating'' sequence 
$(X_0^\op, X_1^\co, X_2^\op, \cdots)$, and similarly for the even dual. 

\section{Levelwise properties of categorical spectra}\label{section_categorical_spectra_stable_properties}
Many properties and structures of categorical spectra are defined levelwise. We will list some examples and formulate typical ways to universally impose such properties. 
\begin{definition}
    Let $P = \{P(n)\}_n$ be a sequence of properties of symmetric monoidal (flagged) $\infty$-categories such that if $X_{n+1}$ satisfies $P(n+1)$, then $\Omega X_{n+1}$ satisfies $P(n)$. 
    We say a categorical spectrum $X=(X_n)_n$ satisfies $P$ (or is a $P$-categorical spectrum) if $X_n$ satisfies $P(n)$. We denote the full subcategory of $P$-categorical spectra by
    \[\CatSp^{P}\coloneqq \lim_n (\cdots \to \infty\SMC^{P(n)}\xrightarrow{\Omega}\infty\SMC^{P(n+1)}\to\cdots )\subset \CatSp.\]
    If the property only depends on the underlying pointed $\infty$-categories, one can also define
    \[\CatSp^{P}\coloneqq \lim_n (\cdots \to \infty\Cat_\ast^{P(n)}\xrightarrow{\Omega}\infty\Cat_\ast^{P(n+1)}\to\cdots )\subset \CatSp.\]
    More generally, if we have a sequence of categories $\{i_n: \eC_n\to \infty\SMC\}$ equipped with a lift of the loop, i.e., functors $\tilde{\Omega}: \eC_{n+1}\to \eC_n$ with the following commutative square $(\ast_n)$, one can define $\CatSp^\eC$ as the limit $\lim_n\eC_n$ in $\widehat{\Cat}$:
    \[\begin{tikzcd}
        \CatSp^{\eC}\ar[r]\ar[d, "i"]& \cdots \ar[r] &\eC_{n+1}\ar[r, "\tilde{\Omega}"]\ar[d, "i_{n+1}"]\ar[rd, phantom, "(\ast_n)"] & \eC_{n} \ar[d, "i_n"]\\
        \CatSp \ar[r] & \cdots \ar[r] & \infty\SMC \ar[r, "\Omega"] & \infty\SMC,
    \end{tikzcd}\]
    and similarly using $\infty\Cat_\ast$ or $\infty\Cat^\rf_\ast$ instead of $\infty\SMC$. 
\end{definition}
\begin{remark}
    An index-shift of a levelwise property is again a levelwise property. The corresponding full subcategory is some shift $\CatSp^P[n]$ of $\CatSp^P$.
\end{remark}
\begin{remark}
    If $i_n$ and $\tilde{\Omega}$ admit left adjoints $L_n\dashv i_n$ and $\tilde{\rB}\dashv \tilde{\Omega}$, then the above diagram lives in $\PrR$, so $i$ also admits a left adjoint $L$.
    If $\eC_n$ is a localization, i.e., $i_n$ is fully faithful with left adjoint $L_n$, the existence of $\tilde{\rB}$ is automatic. In fact, we take $\tilde{\rB}= L_{n+1}\circ B\circ i_n$ with the unit and counit
        \[\tilde{\rB}\tilde{\Omega}= L_{n+1} \rB i_n\tilde{\Omega}\simeq L_{n+1} \rB \Omega i_{n+1}\to L_{n+1}i_{n+1} \xrightarrow{\sim}\id_{\eC_{n+1}},\]
        \[\id_{\eC_{n}}\xleftarrow{\sim} L_n i_n \to L_n \Omega \rB i_n\xrightarrow{\sim} L_n \Omega i_{n+1} L_{n+1} \rB i_n\simeq L_n i_n \tilde{\Omega} L_{n+1} \rB i_n \xrightarrow{\sim} \tilde{\Omega} L_{n+1}\rB i_n = \tilde{\Omega}\tilde{\rB}, \]
    and moreover $\tilde{\rB}$ is fully faithful.
    In other words, for a levelwise property $P=\{P(n)\}$, 
    we may levelwise perform $P$-delooping $\rB^P$ by localizing the connective delooping $\rB$, and the \emph{$P$-envelope} $\rB^{\infty, P}$ of $\rB^\infty X$ is (as a categorical spectrum) the colimit of these deloopings: $\rB^{\infty, P} X\simeq \colim_n\rB^{\infty-n}(\rB^P)^n X$.
    An analogous claim except for the full faithfulness of $\tilde{\rB}$ is true when $\infty\SMC$ and $\rB$ are replaced by $\infty\Cat_\ast$ and $\Sigma$. 
    Note that the left adjoint is in general not given levelwise by $L_n$.
\end{remark}
We describe a condition under which the left or right adjoint to $i$ is given levelwise. Let $\Adj$ be the free adjunction $2$-category and $l: C_1\to \Adj$ be the inclusion of the universal left adjoint $1$-morphism (see \cref{section_categorical_spectra_with_adjoints} for more detail); note that this is an epimorphism.
Recall the following consequence of \cite[Proposition 5.3.17]{stefanichHigherQuasicoherentSheaves2021}:
\begin{proposition}
    Let $I$ be a category, let $\eD$ be a $2$-category (e.g., $\eD = \widehat{\Cat}$) and let $X: I\to \Fun(\Adj, \eD)$ be a diagram whose restriction $Y: I\to\Fun(\Adj, \eD)\xrightarrow{l^\ast} \Fun(C_1, \eD)$ admits a limit $Y^\triangleleft: I^\triangleleft \to \Fun(C_1, \eD)$ (which is computed pointwise). Then there is a limit diagram $X^\triangleleft: I^\triangleleft\to \Fun(\Adj, \eD)$ making the following diagram commute:
    \[\begin{tikzcd}
        I\ar[r, "X"]\ar[d, hook] & \Fun(\Adj, \eD) \ar[d, hook, "l^\ast"]\\
        I^\triangleleft \ar[r, "Y^{\triangleleft}"]\ar[ru, "X^\triangleleft"]& \Fun(C_1, \eD)
    \end{tikzcd}\]
\end{proposition}
\begin{remark}
    By replacing $\eD$ by $\eD^\op$ or $\eD^\co$, similar consequences with ``limit'' replaced by ``colimit'' or ``left adjoint'' replaced by ``right adjoint'' hold. 
\end{remark}
\begin{corollary}
    Suppose the square $(\ast_n)$ is vertically right (resp.\ left) adjointable for every $n$ with $i_n\dashv R_n$ (resp.\ $L_n\dashv i_n$), so the diagram defines $\NN^\op\to \Fun(\Adj, \widehat{\Cat})$. Then the morphism $i$ admits a right adjoint $R$ levelwise given by $R_n$ (resp.\ a left adjoint $L$ levelwise given by $L_n$), i.e., the natural map $\tilde{\Omega}^{\infty-n} \circ R \to R_n\circ \Omega^{\infty-n}$ is an equivalence. 
\end{corollary}
\begin{remark}
    For the assumption of the corollary (for $R_n$), it suffices to check the horizontal left adjointability of $(\ast)_n$, i.e., that $\tilde{\rB}\dashv \tilde{\Omega}$ exists and commutes with $i_n$, and the existence of the right adjoints $R_n$. 
\end{remark}

\begin{example}[Univalence]
    Being a categorical spectrum is a levelwise property of a flagged categorical spectrum; if $X\in \infty\Algbrd_\ast$ is an $\infty$-category, then $\Omega X$ is also an $\infty$-category. There is a univalent completion functor $L^\mathrm{uni}: \CatSp^{\rf}\to \CatSp$ which levelwise is $L^\mathrm{uni}: \infty\Cat^\rf\to \infty\Cat$ because $(\ast_n)$ is vertically left adjointable by \cref{B_is_fully_faithful}.
    As we have seen, the infinite delooping spectrum of a symmetric monoidal $\infty$-category $\eC$ is the univalent completion of a more naive algebroid delooping: $\rB^\infty\eC = L^\mathrm{uni}{\rB'}^\infty\eC$. 
\end{example}

\begin{example}[$k$-categorical spectra, {\cite[Definition 13.2.17, Proposition 13.2.20]{stefanichHigherQuasicoherentSheaves2021}}]\label{definition_k_categorical_spectra}
    Let $-\infty\leq k\leq \infty$. We say $X$ is a \emph{$k$-categorical spectrum} if $X_n$ is a $\max\{n+k, 0\}$-category. We let $k\CatSp\subset \CatSp$ be the full subcategory of $k$-categorical spectra.
    On one extreme, we have $\infty\CatSp = \CatSp$, while on the other we have $-\infty\CatSp=\Sp$. For finite $k$, the categories $k\CatSp$ are shifts of one another. Many interesting examples of categorical spectra live in $0\CatSp$ or $1\CatSp$. This defines the categorical hierarchy that interpolates between spectra and categorical spectra: 
    \[\Sp=-\infty\CatSp\subset \cdots\subset (-1)\CatSp\subset 0\CatSp\subset 1\CatSp\subset \cdots \subset \infty\CatSp = \CatSp. \]
    Recall that the inclusion $n\Cat\hookrightarrow \infty\Cat$ admits both left and right adjoints, denoted by ${}^{\leq n}(\blank)$ and $(\blank)^{\leq n}$. In particular, $k\CatSp\subset \CatSp$ is closed under limits and colimits, so the inclusion has both left and right adjoints, again denoted by ${}^{\leq k}(\blank)$ and $(\blank)^{\leq k}$. In the following diagram, the square $(\ast_n)$ is horizontally left adjointable, so in this range, the right adjoint is given levelwise, i.e., $\Omega^{\infty-n} (X^{\leq k}) \simeq (\Omega^{\infty-n} X)^{\leq (n+k)}$ if $n\geq -k$. 
    \[\begin{tikzcd}
        k\CatSp\ar[r]\ar[d, "i", hook]& \cdots \ar[r] &(n+k+1)\Cat_\ast\ar[r, "\Omega"]\ar[d, "i_{n+1}", hook]\ar[rd, phantom, "(\ast_n)"] & (n+k)\Cat_\ast \ar[d, "i_n", hook]\\
        \CatSp \ar[r] & \cdots \ar[r] & \infty\Cat_\ast \ar[r, "\Omega"] & \infty\Cat_\ast.
    \end{tikzcd}\]
    However, the left adjoint $\Sigma: \sS_\ast\to \sS_\ast$ to $\Omega$ is not the restriction of $\Sigma: \infty\Cat_\ast\to \infty\Cat_\ast$, so we must take the monoidal structure into account, as the next remark shows. 
\end{example}
\begin{remark}\label{categorical_truncations_of_catsp}
    The condition that $X_n$ is a groupoid for $n\leq -k$ in fact forces $X_n$ to be \emph{grouplike} for $n< -k$, so $k\CatSp = \lim_n\eC_n$ for $\eC_n\subset\infty\SMC$ defined by 
    \[\eC_n = \begin{cases}
        \CMon((n+k)\Cat) & (n\geq -k), \\
        \CMon^\gp(\sS) & (n< -k). 
    \end{cases}\]
    The left and right adjoints of $n\Cat \hookrightarrow \infty\Cat$ preserve products, so they induce left and right adjoints of $\CMon(n\Cat)\hookrightarrow\CMon(\infty\Cat)$. 
    Also note that $\CMon^\gp(\sS)\subset \CMon(\sS)$ admits a left adjoint $(\blank)^\gp$ given by group completion and a right adjoint $(\blank)^\times$ that takes the maximal Picard subgroupoid, i.e. the components of invertible objects. Composing these, $i_n: \eC_n\hookrightarrow\infty\Cat$ admits left and right adjoints. The square $(\ast_n)$ is vertically right adjointable, so in particular, the underlying spectrum functor $(\blank)^{\leq -\infty}: \CatSp\to \Sp$ is levelwise given by $(X_n)\mapsto ((X_n)^{\leq 0, \times})$. Notice, however, that $(\ast_n)$ is still not vertically left adjointable \cite[Remark 13.4.21]{stefanichHigherQuasicoherentSheaves2021}.
\end{remark}

\begin{example}[Connectivity]
    Let $-\infty\leq k\leq \infty$, and consider the property $P(n)$ of being $(n+k)$-connective. We say $X$ is \emph{$k$-connective} when it satisfies this levelwise property. Denote the corresponding full subcategory by $\CatSp^{k\mhy\cn}$. When $k$ is finite, it is the essential image of the fully faithful functor $\rB^{\infty+k}:\infty\SMC\to \CatSp$ whose right adjoint is $\Omega^{\infty+k}$. The $k$-connective cover of $X$ is $X^{k\mhy\cn} = \rB^{\infty+k}\Omega^{\infty+k}X$, i.e., the terminal $k$-connective categorical spectrum with a map to $X$. 
\end{example}
\begin{example}[Adjoints and Duals]
    In \cref{chapter_categorical_spectra_with_adjoints}, we will discuss the levelwise property of being $n$-adjointful. The cobordism hypothesis gives a geometric description of the adjointful envelope for some categorical spectra. 
\end{example}

\begin{remark}
    Many categorical spectra in nature arise from the $P$-envelope construction for some property $P$. We expect that the examples include the following, but we will not pursue the details here because they require extra work and some are not yet in the literature.
    \begin{enumerate}
    \item ($n$-semiadditivity)
        For finite $n$ and an $(\infty, n)$-category $\cC$, \cite[\S 3.2]{lurieClassificationTopologicalField2009} outlines the definition of the $n$-category $\Fam_n^k(\cC)$. 
        Roughly speaking, 
        it is the $n$-category of spans of $k$-truncated $\pi$-finite groupoids coherently decorated by cells of $\cC$. 
        There is a morphism $\cC\to\Fam^k_n(\cC)$ exhibiting $\Fam^k_n(\cC)$ as the universal $k$-semiadditive $n$-category under $\cC$, as proven by \cite{harpazAmbidexterityUniversalityFinite2020} in the $n=1$ case; the general case (including the definition of a $k$-semiadditive $n$-category) has been announced by Scheimbauer--Walde \cite{scheimbauerUniversalPropertyHigher2023}. If $\cC$ itself is $k$-semiadditive, it gives the \emph{finite path integral} functor\footnote{The importance of this functor is explained in \cite{freedTopologicalQuantumField2009}. $\Fam^k_n(\cC)$ classifies classical field theories, and the composition with $\int$ gives the \emph{quantization}. An important example is the Dijkgraaf--Witten theory.} $\int: \Fam^k_n(\cC)\to \cC$.
        If $X = (X_n)$ is a categorical spectrum, then almost by definition $\Fam^k(X)\coloneqq \{\Fam^k_n(X_n)\}_{n\geq 0}$ forms a categorical spectrum. One can define $k$-semiadditivity of categorical spectra so that $\Fam^k(X)$ is the $k$-semiadditive envelope of $X$. 
        \item ($n$-stability) For a ring spectrum $R$, the spectrally enriched symmetric monoidal category $\rB R$ admits a stable envelope $\mathsf{Perf}_R$ and a stable presentable envelope $\LMod_R$. 
        In \cite{stefanichPresentableInftyCategories2020}, Stefanich defined the categorical spectrum $\underline{R} = \{n\Mod_R\}$ and the notion of $n$-presentable (stable) $n$-categories. We expect that the construction $\rB^\infty R\mapsto \underline{R}$ can be realized as the $0$-presentable stable envelope, and similarly for a finitary version of it.
        \item (separable closure: \cite{pirsa_PIRSA:23090104})
        With an appropriate finitary version of $\underline{R}$ as above, for a ring of characteristic $0$, Johnson-Freyd and Reutter defined a notion of higher categorical separable closure (either characterized by having a trivial \'etale homotopy type or by a Nullstellensatz-like condition). For complex numbers, it constructs the categorical spectrum of super-vector spaces, super-algebras, and so on. This is likely another example of an envelope construction with an appropriate property of algebraic closedness. 
    \end{enumerate}
    Combining these envelopes with the right adjoint $\GG_m \coloneqq (\blank)^{\leq -\infty}: \CatSp\to \Sp$, we can extract a spectrum containing interesting new information. For instance, $\GG_m(\underline{R})$ gives an infinite sequence of nontrivial deloopings extending the classically well-known $R^\times, \mathrm{Pic}(R), \mathrm{Br}(R)$, i.e., the units, the Picard space, and the Brauer space of a spectrum. Part of the characterizing properties of the separable closure of $\underline{\CC}$ is that $\GG_m(\underline{\CC}^\mathrm{sep}) = I_{\CC^\times}$, the Brown--Comenetz dualizing spectrum of $\CC^\times$.
    This recovers the Freed--Hopkins proposal that $I_{\CC^\times}$ is the universal target of a ``physical'' invertible TQFT. 
\end{remark}

\section{Finiteness properties of categorical spectra}\label{section_categorical_spectra_finiteness_properties}
For a spectrum $X=(X_n)$, a fundamental observation is the formula $X\simeq \colim_n \Sigma^{\infty -n}X_n$. 
It remains valid for categorical spectra by the following standard argument:
\begin{proposition}\label{objects_of_colim_in_PrL_omega_is_colim}
    Let $\eC\to \cdots \to \eC_n \to \cdots\to \eC_1\to \eC_0$ be a limit diagram of categories whose structure morphisms $R_n: \eC\to \eC_n$ admit left adjoints $L_n$.
    Then the counit maps induce a colimit diagram $L_0R_0\to L_1R_1\to \cdots\to \id_{\eC}$ in $\Fun(\eC, \eC)$.
\end{proposition}
\begin{proof}
    For every $X, Y\in \eC$, we have an equivalence
    \[\Map_{\eC}(X, Y)\simeq \lim_n \Map_{\eC_n}(R_n X, R_n Y)\simeq \lim_n \Map_{\eC}(L_nR_n X, Y).\]
    By Yoneda, this induces the diagram in question and moreover exhibits $X$ as a sequential colimit of $L_nR_n X$. 
\end{proof}
\begin{corollary}\label{spectra_are_colimits_of_suspension_spectra}
    Let $X = (X_n)$ be a categorical spectrum. Then there are canonical equivalences $\colim_n \Sigma^{\infty-n} X_n\xrightarrow{\sim}\colim_n \rB^{\infty-n} X_n \xrightarrow{\sim} X$.
\end{corollary}
We will use this corollary to deduce that any categorical spectrum is a filtered colimit of finite categorical spectra. We must first define the notion of finite categorical spectra. Recall that for a spectrum $X\in\Sp$, the following conditions are equivalent:
\begin{enumerate}
    \item $X$ is \emph{finite}, i.e., $X\simeq \Sigma^{\infty-n} Y$ for some natural number $n$ and a finite pointed CW complex $Y$. 
    \item $X$ is \emph{perfect}, i.e., $X$ belongs to the smallest stable subcategory that contains $\SS$ and is closed under retracts.
    \item $X$ is \emph{compact}, i.e., $\Map_{\Sp}(X, \blank): \Sp\to \sS$ preserves filtered colimits. 
    \item $X$ is \emph{dualizable}, i.e., the functor $X\otimes (\blank): \Sp\to \Sp$ admits left or right (equivalently, both) adjoints. 
\end{enumerate}
Ideally, we wish to modify each definition for categorical spectra and prove that they are all equivalent. For now, we only work out the formal part of this.
To define perfectness of categorical spectra, we must first understand the notion of \emph{stability}. This is still a work in progress (see the introduction to \cref{chapter_ablosute_colimits_in_catsp}). What is clear is that perfect categorical spectra should not be closed under $1$-categorically finite colimits; instead, these must be replaced by some lax analogs. 
\begin{definition}
    A (pointed) $\infty$-category is \emph{finite} if it belongs to the smallest subcategory $\infty\Cat_{(\ast)}^\fin\subset\infty\Cat_{(\ast)}$ that contains the (pointed) cells $\GG = \{{C_n}_{(,+)}\}_{n\geq 0}$ and is closed under finite colimits.
    A categorical spectrum is \emph{finite} if it is of the form $\Sigma^{\infty-n} X$ for some integer $n$ and a finite pointed $\infty$-category $X$. We write $\CatSp^\fin\subset \CatSp$ for the full subcategory of finite categorical spectra. 
\end{definition}
\begin{remark}\label{remark_finite_category_Ind_generates}
    Finite $\infty$-categories are compact 
    and the inclusion $\infty\Cat^\fin\to \infty\Cat$ preserves finite colimits. By \cite[Proposition 5.3.5.11, Example 5.3.6.8]{LurieHTT}, the left Kan extension $\Ind(\infty\Cat^\fin_{(\ast)})\to \infty\Cat_{(\ast)}$ is fully faithful and colimit-preserving. Since the cells generate $\infty\Cat$ under colimits, we have $\Ind(\infty\Cat^\fin_{(\ast)})\xrightarrow{\sim}\infty\Cat_{(\ast)}$. In particular, every (pointed) $\infty$-category is canonically a filtered colimit of finite ones. 
\end{remark}
\begin{example}
    Any finite torsion-free complex is a finite $\infty$-category by \cite[Theorem B]{campion$inftyn$categoricalPastingTheorem2023}. In particular, any strong Steiner $\infty$-category corresponding to a finite-dimensional strong Steiner complex is finite; examples include the objects of $\Theta$, lax cubes, and orientals. The author does not know if any finite computad is a finite $\infty$-category. 
\end{example}

\begin{corollary}
    Any categorical spectrum is a filtered colimit of finite categorical spectra. More precisely, the inclusion $\CatSp^\fin\subset \CatSp$ induces an equivalence $\Ind(\CatSp^\fin)\xrightarrow{\sim} \CatSp$. 
\end{corollary}
\begin{proof}
    Since $\Omega^{\infty-n}: \CatSp\to \infty\Cat_\ast$ preserves filtered colimits, finite categorical spectra are compact.
    By \cite[Proposition 5.3.5.11]{LurieHTT}, the left Kan extension $\Ind(\CatSp^\fin)\to \CatSp$ is fully faithful.
    To show that it is an equivalence, we must check that the smallest full subcategory of $\CatSp$ containing $\CatSp^\fin$ and closed under filtered colimits is $\CatSp$ itself, which follows from \cref{spectra_are_colimits_of_suspension_spectra} and \cref{remark_finite_category_Ind_generates}.
\end{proof}
\begin{corollary}
    A categorical spectrum $X$ is compact if and only if it is a retract of a finite categorical spectrum.
\end{corollary}
\begin{remark}
    In the classical case of spectra, we can remove the retract from the statement. The standard argument uses homology theory and the Hurewicz theorem: homology theory precisely tells us the recipe for approximating a space by cells, and stably it is conservative enough that the recipe can completely recover the spectrum in question. 
    We do not know whether any compact categorical spectrum is a finite suspension categorical spectrum (also note that the analogous statement for equivariant spectra is known to be false). In any case, it is desirable to have an algebraic invariant that extracts the data of cells required to build an $\infty$-category or a categorical spectrum. 
\end{remark}
We briefly discuss the dualizability of a categorical spectrum. Recall the following definition:
\begin{definition}
    Let $\eC$ be a monoidal category. An object $X\in\eC$ is \emph{left- (resp.\ right-)dualizable} if the functor $X\otimes (\blank): \eC\to \eC$ admits a left (resp.\ right) adjoint. 
\end{definition}
Dualizability depends on the monoidal structure constructed in the next chapter. However, note the following general fact in a closed monoidal category: a dualizable object has as much compactness as the unit object.
\begin{proposition}\label{dualizable_implies_compact}
    Suppose we are given a closed monoidal structure $\otimes$ on $\CatSp$ whose unit is $\FF$. Then any left- or right-dualizable object is compact. 
\end{proposition}
\begin{proof}
    Suppose $X$ has a right dual $X^R$. By assumption, the tensor product admits an internal hom: $\Map(Z, [X, Y])\simeq \Map(X\otimes Z, Y)\simeq \Map(Z, X^R\otimes Y)$, so, in general, we have $[X, Y]\simeq X^R\otimes Y$ (similarly, $\llbracket X, Y\rrbracket\simeq Y\otimes X^L$ for the left dual $X^L$ and the right internal hom $\llbracket \blank, \blank\rrbracket$). Plugging in $Z=\FF$, we see that $\Map(X, Y) \simeq \Map(\FF, X^R\otimes Y)$, and since $\FF$ is compact, this functor is colimit-preserving in $Y$. 
\end{proof}
The assumption is clearly satisfied once the monoidal structure is constructed. In particular, any dualizable categorical spectrum is compact. 
However, not all finite categorical spectra are dualizable. For this, we must replace the finiteness discussed in this section by some lax analog (for example, the non-directed circle $\Sigma^{\infty} S^1$ is not dualizable).
We will give some examples of dualizable categorical spectra in \cref{chapter_ablosute_colimits_in_catsp}. 

\chapter{Tensor product of categorical spectra}
\label{chapter_tensor_product_of_catsp}
The main goal of this chapter is to define the tensor product of categorical spectra through its universal property. 
We have already explained the strategy in detail in the introduction. 
Here we motivate our approach slightly differently, following the history of the corresponding problem for spectra. 

The first \emph{homotopy} category of spectra $\mathsf{hSp}$ with the ``smash product'' symmetric monoidal structure was defined in \cite{boardmanStableHomotopyTheory1965}, and later a more handcrafted approach in \cite{adamsStableHomotopyGeneralised1995} became popular. 
However, for higher algebra, it had serious deficiencies: the homotopy category has bad formal properties (e.g., it does not have most limits and colimits), fails to encode homotopically nuanced algebraic structures, and does not work well in families.

The first batch of successful attempts to define a symmetric monoidal structure remembering all homotopical data used model categories (e.g. \cite{EKMM}, \cite{mandellModelCategoriesDiagram2001}). The trick was a reversed microcosm principle: if we could define a symmetric tensor product, we would have the symmetry of the unit, so instead of defining spectra as merely $\NN$-indexed families of spaces, we build a model that by design takes symmetries of the spheres into account. 
The minimalistic choice is to consider the symmetric group action on the spheres encoding the Koszul sign rule, i.e., we define spectra as a $\Fin^\simeq \coloneqq \bigsqcup_{n\geq 0}\rB\Sigma_n$-indexed family of spaces. This leads to the definition of symmetric spectra. 
They were good enough for many purposes, but model categories were too rigid to behave well in families, and the choice of a model was arbitrary, contrary to the canonicity of the stable homotopy category. 

The truly universal object was in between---the $(\infty, 1)$-category $\Sp$ of spectra (Boardman's original definition was close to the modern one, except that the language was missing back then). 
After thoroughly developing $(\infty, 1)$-category theory, Lurie characterized the symmetric monoidal ($(\infty, 1)$-)category of spectra as the unit of the symmetric monoidal category $\Pr^\sL_\st$ of presentable stable categories \cite[\S 4.8]{LurieHA}.
Note that once we pass to $(\infty, 1)$-category land, sequential spectra work perfectly; being natively enriched over homotopy types and not sets, the suspension functors carry automorphisms equivalent to those of the spheres.

The lesson is that we should take the symmetries of the spheres and suspension functors into account. Coming back to our problem, we have already seen that 
the suspension $\Sigma: \infty\Cat_\ast\to \infty\Cat_\ast$ is naturally equivalent to $\vS^1\owedge(\blank)$ and $D(\blank)\owedge\vS^1$. One can think of this as the \emph{twisted} symmetry of the suspension functor, where the twist comes from the total dual $D$. Classically, the sphere commutes with other CW complexes by the \emph{Koszul sign rule}, but since we do not have negatives and instead have directions of cells, we must express the sign rule \emph{externally} by switching the directions of the cells. 
This allows us to expect the formula 
\begin{align*}
    \Sigma^{\infty-m}X\otimes \Sigma^{\infty-n} Y &= \vS^{-m}\otimes \Sigma^\infty X\otimes \vS^{-n}\otimes \Sigma^\infty Y \\
    & = \vS^{-(m+n)}\otimes \Sigma^\infty D^n X\otimes \Sigma^\infty Y \\
    & = \Sigma^{\infty-(m+n)}((D^n X)\otimes Y)
\end{align*}
for the tensor product of suspension spectra.
Based on $\infty\Cat_\ast^\owedge$, we can hope for at most an $\EE_1$-monoidal structure, or some $\ast$-algebra strucutre on $\CatSp$. However, even to prove associativity (and its higher coherence), we must be able to move suspensions around \emph{canonically}; this information is packaged conveniently in the \emph{half-central structure} on $\vS^1$, which we will establish in \cref{section_half_central_structure_of_S1}. In fact, we will show the existence of a \emph{unique} half-central structure on $\vS^1$, which deserves to be called the \emph{categorical Koszul sign rule}.

The remaining work needed to define the tensor product is in \cref{section_construction_of_tensor_product} and is rather formal, modifying previously available techniques to the $\EE_1$-setting.
This will moreover prove the monoidal universal property $\CatSp^\otimes = \infty\Cat_\ast^\owedge[\vS^{-1}]$. 
Note that although $\CatSp$ is formally obtained by inverting the endomorphism $\vS^1\otimes (\blank)$, it does not immediately imply that this procedure inverts $\vS^1$ monoidally; this problem already exists in the commutative setting (see the references in the introduction to the section).
We close the chapter with \cref{section_stable_properties_tensor_product}, establishing some basic results on the tensor product. 

\section{Half-central structure of \texorpdfstring{$\vec{S}^1$}{S1}}\label{section_half_central_structure_of_S1}
This section contains perhaps the most important technical ingredient of this thesis: the half-central structure on the directed circle $\vS^1= \rB\NN$. 
This is the higher-categorical incarnation of the Koszul sign rule in ordinary homotopy theory. We will define the half-center of $\infty\Cat_{\ast}$ and prove that the directed circle $\vec{S}^1= \rB \NN$ admits a unique half-central structure (Theorem~\ref{half-center_S1}). This will be the key technical input for the construction of the tensor product of categorical spectra in section \cref{section_construction_of_tensor_product}\footnote{When inverting a set of elements $S$ of a (noncommutative) monoid $M$, one only requires a condition on $S$ weaker than $S\subset Z(M)$, called the Ore condition, to have good control over $S^{-1}M$. 
For the definition of a monoidal structure, it might be possible to define a categorified Ore condition instead. However, it will become routine to commute the directed spheres with other objects, so it is independently useful to know that such a maneuver is completely canonical and harmless.}.

\subsection{Half-center}
We begin by recalling the notion of the center of a monoidal category. A good reference for this and the next subsection is \cite{ben-zviIntegralTransformsDrinfeld2010}.
In the following definition, the classical case is when $A$ is a monoidal $(1, 1)$-category, i.e., when $\sV = (1, 1)\Cat$ with Cartesian monoidal structure.
The example that we will specialize to later is $A=\infty\Cat_\ast^{\owedge}$ and $\sV=\Pr^{\sL, \otimes}_\omega$. 
\begin{definition}
    Let $A$ be an $\EE_1$-algebra object of $\sV\in \CAlg(\PrL)$. The \emph{center} $\fZ_\sV(A)$ of $A$ (in $\sV$) is the object $\End_{\BMod{A}(\sV)}(A)$. 
    A \emph{central structure} on $s\in A$ (i.e. $s: 1_\sV\to A$) is a lift of $s$ along the forgetful functor\footnote{
        Not to be confused with another forgetful functor $\End_{\BMod{A}}(A)\to \End_{\LMod_{A}}(A)\simeq A^\rev$.} $\fZ_\sV(A) = \End_{\BMod{A}}(A)\to \End_{\RMod_{A}}(A)\simeq A$. 
    We will omit $\sV$ from the notation when it is not confusing or relevant. 
\end{definition}
To understand the meaning of the definition, suppose $s\in A$ admits a central structure. The object $s$ is identified with a right $A$-module morphism $s\otimes(\blank): A\to A$. A bimodule homomorphism structure promoting this includes the isomorphism $s\otimes (t\otimes (\blank))\simeq t\otimes (s\otimes (\blank))$ naturally in $t$, i.e., $\tau: s\otimes (\blank)\simeq (\blank)\otimes s$.
This is all we need when $\sV$ is a $(1, 1)$-category (e.g., $\sV = \Set$), and being central is a \emph{property}, but in general this is the first piece of the \emph{structure} with infinitely many coherence data.
\begin{remark}
    As we will see below, the center is just another name for the Hochschild cohomology of $A$, seen as a bimodule over itself. 
    The notion of center defined above is sometimes called the $\EE_2$-center because it admits a canonical structure of an $\EE_2$-algebra object of $\sV$, with which the center is characterized by a universal property (see \cite[\S 5.3]{LurieHA}).
    This $\EE_2$-structure on the center can be naturally understood from Morita theory. One expects that the Morita $2$-$\sV$-category $\mathfrak{Alg}(\sV)$ of $\EE_1$-algebras and bimodules in $\sV$ can be defined in a similar manner as in (the easiest case of) \cite{johnson-freydOplaxNaturalTransformations2017}. 
    In this category, $A$ as an algebra is an object, and $A$ as a bimodule is the identity morphism of the object $A$, so the center admits a description as the double-loop object
    \[\fZ_\sV(A) \simeq \Omega(\Omega(\mathfrak{Alg}(\sV), A), \id_A),\]
    from which the $\EE_2$-structure is clear. 
\end{remark}
According to \cref{suspension_is_tensor_S1}, it seems natural to expect that $\vec{S}^2 = \vec{S}^1\owedge \vec{S}^1$ lifts to the center, which would imply that $\Sigma^\infty: \infty\Cat_\ast\to \CatSp$ lifts to an $\infty\Cat_\ast^\owedge$-bimodule homomorphism. 
However, defining a central structure can be difficult in general, a priori requiring infinitely many coherence data. 
It turns out to be easier to directly formulate the coherence data extending \cref{suspension_is_tensor_S1}, namely the ``half-central'' structure of $\vec{S}^1$; the gauntness of $\vec{S}^1$ makes it homotopy-theoretically more tractable than $\vec{S}^2$.

To define the notion of the half-center (with respect to an involution $D$), let $D: A\to A$ be a monoidal endofunctor equipped with an equivalence $D\circ D\simeq \id_A$. There is a locally fully faithful functor $\Alg(\sV)\to \mathfrak{Alg}(\sV)$ which is the identity on objects and regards algebra homomorphisms as bimodules (we do not need a precise construction of $\mathfrak{Alg}(\sV)$, however). Explicitly, an algebra homomorphism $f: A\to B$ can be seen as an $(A, B)$-bimodule ${}^f B$, whose underlying right $B$-module is $B$ itself and whose left action of $A$ is provided by $f$. By abuse of notation, we denote the $(A, A)$-bimodule ${}^DA$ also by $D$. 
\begin{definition}
    The \emph{half-center} of $A$ with respect to $D$ is $\fZ_{\sV}(A, D) \coloneqq \Hom_{\BMod{A}}(A, D)$\footnote{Optimally, this is an object of $\sV$, but for our purposes the underlying object in $\sS$ suffices.}. 
\end{definition}
\begin{remark}
    The following diagram commutes (note that $A=D=D\otimes_A D$ after forgetting to $\RMod_A$):
    \[\begin{tikzcd}
        \fZ(A, D)\simeq \Hom_{\BMod{A}}(A, D)\ar[r, "D\otimes_A(\blank)", "\sim"']\ar[d, "\text{forget}"'] & \Hom_{\BMod{A}}(D, D\otimes D)\ar[d, "\simeq"]\\
        A\simeq \End_{\RMod_A}(A, A)& \Hom_{\BMod{A}}(D, A)\ar[l, "\text{forget}"']
    \end{tikzcd}\]
    Thus lifting $x\in A$ to $\fZ(A, D)$ in fact gives a simultaneous lift to $\Hom(A, D)$ and $\Hom(D, A)$. In particular, a half-central structure on $x$ induces a central structure on $x\otimes x$ by composition $\Hom(A, D)\times \Hom(D, A)\to \Hom(A, A) = \fZ(A)$.
\end{remark}

\subsection{Cyclic bar construction and Hochschild cohomology}
\label{subsection_cyclic_bar}
Here we review the standard resolution of a bimodule into free ones, called the \emph{cyclic bar construction}, and the resulting description of the half-center $\fZ_{\sV}(A, D)$ as the Hochschild cohomology of the $(A, A)$-bimodule $D$.

Let $A, B$ be $\EE_1$-algebras in $\sV$ and let $M, N$ be $(A, B)$-bimodules. Our goal is to find a convenient description of $\Hom_{\BMod[A]{B}(\sV)}(M, N)$. 
Using the equivalence $\BMod[A]{B}(\sV)\simeq \LMod_A(\RMod_B(\sV))$, we have an adjunction
\[\begin{tikzcd}
    \LMod_A(\RMod_B(\sV)) \arrow[r, shift right = .5ex]
    & \RMod_B(\sV) \arrow[l, shift right = 1ex, "A\otimes(\blank)"']
    \arrow[l, phantom, shift right = .2ex, "\scriptscriptstyle\boldsymbol{\bot}"]
\end{tikzcd}\]
with the comonad $T = A\otimes (\blank)$. The associated comonad resolution $M\simeq \colim_{[n]\in \Spx^\op}(T^{n+1} M)$ gives  
\[\Hom_{\BMod[A]{B}}(M, N)\simeq \lim_{[n]\in\Spx}\Hom_{\BMod[A]{B}}(A^{\otimes (n+1)}\otimes M, N)\simeq \lim_{[n]\in\Spx}\Hom_{\RMod_B}(A^{\otimes n}\otimes M, N).\]
Now we apply this to the case $A=B$, $M=A$, $N=D = {}^D A$, where $D: A\to A$ is an involution as before. Since the right $A$-module structure on $D$ is identical to that on $A$, we see 
\[\fZ_{\sV}(A, D)\coloneqq \Hom_{\BMod{A}}(A, D)\simeq \lim_{[n]\in \Spx} \Hom_{\RMod_A}(A^{\otimes n}\otimes A, A)\simeq \lim_{[n]\in \Spx}\Hom_\sV(A^{\otimes n}, A).\]
\begin{remark}\label{remark_cosimplicial_structure_of_Hochschild_cohomology}
    On the right-hand side, the data of the involution $D$ is encoded in the cosimplicial structure. Explicitly, the coface map $d^i: \Hom_\sV(A^{\otimes n}, A)\to \Hom_\sV(A^{\otimes (n+1)}, A)$ sends $f: A^{\otimes n}\to A$ to 
    \[d^i f: x_0\otimes \cdots \otimes x_n\mapsto \begin{cases}
        D(x_0)f(x_1\otimes\cdots \otimes x_n) \quad (i=0),\\
        f(x_0\otimes \cdots \otimes x_{i-1}x_i\otimes \cdots \otimes x_n) \quad (1\leq i\leq n), \\
        f(x_0\otimes \cdots\otimes x_{n-1})x_n \quad (i=n+1). 
    \end{cases}\]
\end{remark}

\subsection{Digression: obstruction theory for totalizations of cosimplicial spaces}
We digress a bit and try to explicate the coherence of (half)-central structures on an object. As we saw in the last section, the half-center is the totalization of a cosimplicial object, so one can try to slice the cosimplicial diagram in skeletal layers to write down the obstructions to the existence of a half-central structure. 
In our case of interest, the obstructions turn out to live in contractible spaces, in which case this section is subsumed by a simpler argument in \cref{trivial_obstruction_theory}. Nevertheless, we include the exposition to give the intuitive description of the half-central structure in \cref{subsection_half_central_structure_on_S1}. 
The material has been standard since \cite{bousfieldHomotopySpectralSequences1989} but we provide a concise, model-independent account. See also \cite{mathewFibersPartialTotalizations2015} for a similar treatment of related material.

\begin{notation}
    For a functor $f: \eC\to \eD$, we denote the restriction $\Fun(\eD, \sS)\to \Fun(\eC, \sS)$ by $f^\ast$ and the right Kan extension by $f_\ast$, so we have an adjunction $f^\ast\dashv f_\ast$. Recall the equivalence $\lim f_\ast F\xrightarrow{\sim} \lim F$. 
\end{notation}
Let $X = X^\bullet$ be a cosimplicial object in $\sS$. Our goal is to understand when a point $x_0\in X^0$ lifts to the totalization $\Tot X\coloneqq \lim X^\bullet$. Let $\Spx_{\leq n} \subset \Spx$ 
denote the full subcategory spanned by $[k]$ for $k\leq n$, let $X_{\leq n}: \Spx_{\leq n}\to \sS$ be the restriction of $X$, and let $\Tot_n X\coloneqq \lim X_{\leq n}$ be the $n$-th partial totalization. 
Since $\Spx = \colim_n \Spx_{\leq n}$, applying \cref{objects_of_colim_in_PrL_omega_is_colim} to $\eC = \Fun(\Spx, \sS)^\op$, we have $X= \lim_n (\Spx_{\leq n}\to \Spx)_\ast X_{\leq n}$. 
The totalization yields the following limit tower of partial totalizations:  
\[\Tot X \to \cdots \to \Tot_2 X\to \Tot_1 X\to \Tot_0 X = X^0. \]
Now our task is to understand the fiber of $\Tot_{n+1} X\to \Tot_n X$ at a given point $x_n\in \Tot_n X$. 

Let $R^n X: \Spx_{\leq n+1} \to \sS$ be the right Kan extension of $X_{\leq n}$ and $M^n X\coloneqq (R^n X)^{n+1}$ (called the \emph{$n$-th matching object} of $X$). The limit of the unit natural transformation $\alpha: X_{\leq n+1}\to R^n X$ is the map in question: $\Tot_{n+1} X\to \lim R^n X \xrightarrow{\sim}\Tot_n X$. 

Let $\Spx^\inj_{\leq n+1}\subset \Spx_{\leq n+1}$ denote the subcategory with the same objects but only injective morphisms. The forgetful functor from the category of sub-simplices $u: (\Spx^\inj_{\leq n+1})_{/[n+1]}\to \Spx_{\leq n+1}$ is coinitial (see e.g., \cite[1.2.4.17]{LurieHA}), so $\lim X\simeq \lim u^\ast X$. 
Let $\eC = (\Spx^\inj_{\leq n})_{/[n+1]}$, so that $\eC^\triangleright = (\Spx^\inj_{\leq n+1})_{/[n+1]}$, and let $\eC\xrightarrow{i} \eC^\triangleright \xleftarrow{j} \{[n+1]\}$ be the inclusions.
For any functor $F: \eC^\triangleright\to \sS$, the functor $j_\ast j^\ast F$ is constant at $F([n+1])$, and the counit $F\to i_\ast i^\ast F$ replaces the value $F[n+1]$ by the point $\ast$. It follows that the following square in $\Fun(\eC^\triangleright, \sS)$ is Cartesian: 
\[\begin{tikzcd}
    F \ar[r] \ar[d, "\eta"] & j_\ast j^\ast F \ar[d, "{\eta'}"] \\
    i_\ast i^\ast F \ar[r] & i_\ast i^\ast j_\ast j^\ast F .
\end{tikzcd}\]
Recall that the limit of a constant diagram is given by cotensoring with the geometric realization (groupoidification) of the diagram shape. Also note that the geometric realizations of $\eC$ and $\eC^\triangleright$ are $S^n$ and $\ast$ because, as a simplicial set, $\eC$ is the barycentric subdivision of $\partial \Delta^{n+1}$. As a result, the limit of the above square over $\eC^\triangleright$ is the following Cartesian square in $\sS$: 
\[\begin{tikzcd}
    \lim F \ar[r] \ar[d, "\eta"] & F[n+1] \ar[d, "{\eta'}"] \\
    \lim i^\ast F \ar[r] & (F[n+1])^{S^n}. 
\end{tikzcd}\]
Plugging $\alpha: X_{\leq n+1}\to R^n X$ into $F$, we get a Cartesian cube (i.e. the cube is a limit diagram, cf.\ \cite[section 6.1.1]{LurieHA}) $\eta(\alpha)\Rightarrow \eta'(\alpha)$. Comparing the initial vertices of $\eta(\alpha)$ and $\eta'(\alpha)$ with the pullback of the rest of the squares (and since $i^\ast \alpha$ is an equivalence), we see that the following square in $\sS$ is Cartesian: 
\[\begin{tikzcd}
    \Tot_{n+1} X \ar[r]\ar[d] & X^{n+1} \ar[d] \\
    \Tot_n X \ar[r] & (X^{n+1})^{S^n}\times_{(M^n X)^{S^n}} M^n X
\end{tikzcd}\]
Tracing the construction, one sees that the map $S^n\to X^{n+1}$ corresponding to an element $x_n\in \Tot_n X$ sends the basepoint of $S^n$ (coming from $0\in \partial \Delta^{n+1}$) to $(d^0)^{n+1}(x_0)$, where $x_0$ is the element of $X^0$ underlying $x_n$.
Now we can show the following result: 
\begin{proposition}\label{obstruction theory for Tot}
    A given point $x_n\in \Tot_n X$ lifts to $x_{n+1}\in \Tot_{n+1} X$ if and only if the induced map $o(x_n): S^n\to X^{n+1}$ is trivial in $\pi_n(X^{n+1}, (d^0)^{n+1}(x_0))$. 
    When a lift exists, the space of lifts is equivalent to $\Omega^{n+1}(N^n(X), (d^0)^{n+1}(x_0))$, where $N^n(X)$ is the fiber of $X^{n+1}\to M^n(X)$ at the image of $x_n$. 
\end{proposition}
\begin{proof}
    The image of $x_n$ in $(X^{n+1})^{S^n}\times_{(M^n X)^{S^n}} M^n X$ is equivalent to the data of the following commutative square:
    \[\begin{tikzcd}
        S^n\ar[r, "o(x_n)"]\ar[d] & X^{n+1}\ar[d] \\
        \ast\ar[r]\ar[ru, "x_{n+1}" description, dashed] & M^n X
    \end{tikzcd}\]
    and the data of the lift $x_{n+1}$ is equivalent to the dashed arrow with two homotopies filling the triangles. The lower-right triangle can always be filled by composing some $\ast\to S^n$ with the square, so the triviality of $[o(x_n)]\in \pi_n(X^{n+1})$ suffices.
    Now assume the triviality of $[o(x_n)]$. The space of fillers is equivalent to the space of nullhomotopies of $S^n\to N^n X$. This is $\Omega^{n+1}(N^n X)$. 
\end{proof}

\begin{remark}
    The argument in this subsection is valid for cosimplicial objects in any category with finite limits and totalizations.
    By a similar argument, one can show that, for any $X^\bullet$ satisfying $X^m=\ast$ for $m\neq n$, the totalization is either empty or $\Omega^n X^n$ (cf.\ \cite[Corollary 1.2.4.18]{LurieHA} for the stable variant).
\end{remark}

\subsection{Half-central structure on \texorpdfstring{$\vec{S}^1$}{S1}}\label{subsection_half_central_structure_on_S1}
Now we specialize to our case of interest:
\begin{notation}
    Let $\sV=\Pr_\omega^\sL$, and let $A = \infty\Cat_\ast^\owedge$ and $D = (\blank)^\circ: A\to A$ be the total dual (monoidal) functor, which flips cells of all dimensions (\cref{total_dual}). We continue to denote the associated bimodule ${}^DA$ by $D$. 
\end{notation}
The goal of this section is to prove the following: 
\begin{theorem}\label{half-center_S1}
    $\vec{S}^1\in A$ and $\Sigma\simeq \vec{S}^1\owedge (\blank) \in \End_{\RMod_{A}}(A)$ uniquely lifts along the forgetful functor $\fZ_{\sV}(A, D)\to \End_{\RMod_A}(A) \simeq A$.
\end{theorem}
\begin{corollary}
    The category $\CatSp$ and the functor $\Sigma^\infty: \infty\Cat_\ast\to \CatSp$ lift to $\BMod{\infty\Cat_\ast}(\Pr_\omega^\sL)$.
\end{corollary}
Unpacking the obstruction theory of the totalization of cosimplicial objects, we see that the data of a half-central structure on $\vec{S}^1$ amounts to the following:
\begin{itemize} 
    \item An object $\vec{S}^1\in \omega\Cat_\ast$, 
    \item A natural isomorphism $\tau_X: \vec{S}^1\owedge X\xrightarrow{\sim} X^\circ \owedge \vec{S}^1$, 
    \item A natural homotopy $\theta_{X, Y}$ filling the triangle 
    \[\begin{tikzcd}
        &X^\circ \owedge \vec{S}^1 \owedge Y \ar[rd, "X^\circ\owedge (\tau_Y)"]&\\
        \vec{S}^1\owedge X\owedge Y \ar[ru, "(\tau_X)\owedge Y"]\ar[rr, "\tau_{X\owedge Y}"]&& X^\circ \owedge Y^\circ \owedge \vec{S}^1
    \end{tikzcd}\]
    \item A natural homotopy filling the $3$-simplex
    \[\begin{tikzcd}
        \vec{S}^1\owedge X\owedge Y\owedge Z \ar[r]\ar[d]\ar[rd] & X^\circ\owedge Y^\circ\owedge Z^\circ\owedge \vec{S}^1 \\
        X^\circ\owedge \vec{S}^1\owedge Y\owedge Z\ar[ru, crossing over]\ar[r] & X^\circ\owedge Y^\circ \owedge \vec{S}^1\owedge Z \ar[u]
    \end{tikzcd}\]
    whose boundary is filled by the homotopies $\tau$ and $\theta$.
    \item and so on. 
\end{itemize}
It turns out that all the vertices live in contractible components, so there is no room for any nontrivial choice of coherence data at each step after providing the natural equivalence $\tau$. With this in mind, we have the following more direct argument:
\begin{lemma}\label{trivial_obstruction_theory}
    Let $X= X^\bullet$ be a cosimplicial groupoid and $x\in X^0$ be a point. Suppose that 
    \begin{enumerate}
        \item the connected component of $(d^0)^n(x)\in X^n$ is contractible for all $n\geq 0$, and
        \item there is a path $d^0 x \simeq d^1 x\in X^1$.
    \end{enumerate} 
    Then the fiber of $\Tot(X^\bullet)\to X^0$ over $x$ is contractible. 
\end{lemma}
\begin{proof}
    By left Kan extension along $\{[0]\}\hookrightarrow \Spx$, the data of $\ast\xrightarrow{x} X^0$ is equivalent to a natural transformation $x: I\to X$, where $I$ is the cosimplicial set $[n]\mapsto \Hom([0], [n]) = [0]^{\sqcup (n+1)}$. In this way, the map $\Tot(X^\bullet)\to X^0$ is corepresented by the map $I\to \ast$ to the terminal cosimplicial groupoid, i.e., the groupoid of lifts of $x$ is equivalent to that of factorizations $I\to \ast \to X$ of $x$. Consider the image factorization $I\twoheadrightarrow Y^\bullet\hookrightarrow X^\bullet$, so $Y^n\subset X^n$ is the union of connected components of the images of $x$ under the structure maps $X^0\to X^n$. The conditions (1) and (2) ensure that $Y^\bullet\xrightarrow{\sim} \ast$, so the factorization $I\to \ast\simeq Y^\bullet\to X^\bullet$ gives the canonical point in the fiber of $\Tot(X^\bullet)\to X^0$ over $x$. Conversely, any factorization $I\twoheadrightarrow \ast\to X$ uniquely factors through $Y$ and is an image factorization, so the groupoid of such factorizations is contractible. 
\end{proof}

In the following, let $\mu_n$ generically denote the $n$-ary multiplication map of algebra objects (in particular, monoidal categories).
\begin{lemma}\label{trivial_automorphisms}
    For $n\geq 0$, the connected component of $\Sigma\circ\mu_n$ in the underlying groupoid $\Map_{\sV}(A^{\otimes n}, A)$ of $\Hom_{\sV}(A^{\otimes n}, A) = \LFun_\omega(\infty\Cat_\ast^{\otimes n}, \infty\Cat_\ast)$ is contractible. 
\end{lemma}
We will prove the lemma later; let us first assume it and finish proving the theorem.
\begin{lemma}\label{half_central_S1_Tot1}
    There exists a unique equivalence of endofunctors $\tau: (\blank)\owedge \vec{S}^1\xrightarrow{\sim} \vec{S}^1\owedge (\blank)^\circ$. 
\end{lemma}
\begin{proof}
    The uniqueness follows from the $n=1$ case of \cref{trivial_automorphisms}. The existence is \cref{suspension_is_tensor_S1}. 
\end{proof}

\begin{proof}[Proof of \cref{half-center_S1}]
    Apply \cref{trivial_obstruction_theory} to $X^\bullet = \Map_{\sV}(A^{\otimes \bullet}, A)$ with the cosimplicial structure described in \cref{remark_cosimplicial_structure_of_Hochschild_cohomology}. Conditions (1) and (2) are \cref{trivial_automorphisms} and \cref{half_central_S1_Tot1}, respectively.
\end{proof}

\begin{proof}[Proof of Lemma \ref{trivial_automorphisms}]
Consider the composition 
\[j: \Cube^{\times n}\to \PSh(\Cube)^{\otimes n}\to \infty\Cat^{\otimes n}\xrightarrow{(\blank)_+}\infty\Cat_\ast^{\otimes n}.\]
The first functor is the Yoneda embedding, the second is the tensor power of the localization $\PSh(\Cube)\to \omega\Cat$, and the last is the base change along $\sS\to \sS_\ast$ in $\Pr^\sL$. From the universal property of each functor, $j$ induces a fully faithful embedding
\begin{align*}
    \LFun_\omega(\infty\Cat_\ast^{\otimes n}, \infty\Cat_\ast) &\subset \LFun(\infty\Cat_\ast^{\otimes n}, \infty\Cat_\ast)\\ 
    &\simeq \LFun(\infty\Cat^{\otimes n}, \infty\Cat_{\ast}) \hookrightarrow \Fun(\Cube^{\times n}, \infty\Cat_{\ast}).
\end{align*}
In particular, the connected component of $F\coloneqq \Sigma\circ \mu_n\in \LFun(\infty\Cat_{\ast}^{\otimes n}, \infty\Cat_\ast)$ is equivalent to that of $\Sigma\circ \mu_n\circ j \in \Fun(\Cube^{\times n}, \infty\Cat_\ast)$. Now we have the following commutative diagram: 
\[\begin{tikzcd}
    \infty\Cat_\ast\otimes\cdots\otimes\infty\Cat_\ast\ar[r,"\mu_n"] & \infty\Cat_\ast\ar[r, "\Sigma"] & \infty\Cat_\ast \\
    \infty\Cat\otimes\cdots\otimes\infty\Cat \ar[r, "\mu_n"] \ar[u, "(\blank)_+"] & \infty\Cat \ar[r, "\Free_{\EE_1}"] \ar[u] & \Mon(\infty\Cat) \ar[u, hook, "\rB"] \\
    \Cube\times\cdots\times\Cube \ar[u, hook]\ar[r, "\mu_n"]\ar[uu, bend left=20mm, shift left=8mm, "j"] & \Cube \ar[u, hook]\ar[r, "\Free_{\EE_1}"] & \Mon(\Gaunt) \ar[u, hook]
\end{tikzcd}\]
The lower-left square commutes by the characterization of the Gray tensor product. The free $\EE_1$-algebra functor restricted to $\Cube$ lands in the gaunt monoidal categories by the explicit formula $\Free_{\EE_1}(\cube^n) \simeq \coprod_{k\geq 0}\cube^{kn}$. 
Therefore the connected component of $F$ in $\Fun(\Cube^{\times n}, \infty\Cat_\ast)$ is equivalent to that of $\Free_{\EE_1}\circ\mu_n\in \Fun(\Cube^{\times n}, \Mon(\Gaunt))$. The latter is a $(1, 1)$-category, so the connected component is equivalent to a delooping of a monoid in $\Set$. 
We must show that the (ordinary) group $\Aut(F)$ of invertible objects of that monoid is trivial.
We have the following equalizer diagram of sets: 
\[\begin{tikzcd}
    \Aut(F) \ar[r] 
    &\prod_{x\in \Cube^{\times n}} \Aut_{\Mon(\Gaunt)}(Fx) \ar[r, shift left=.5 mm]\ar[r, shift right = .5 mm]
    & \prod_{x\to y\in \Cube^{\times n}} \Hom_{\Mon(\Gaunt)}(Fx, Fy). 
\end{tikzcd}\]
We claim that for any $x = (\cube^{k_1}, \ldots, \cube^{k_n})$, the group $\Aut(Fx) = \Aut(\Free_{\EE_1}(\cube^{k_1+\cdots+k_n}))$ is trivial. Note that the natural map $\Aut_{\Gaunt}(\cube^m)\to \Aut_{\Mon(\Gaunt)}(\Free(\cube^m))$ is a bijection; the inverse is given by the restriction to the indecomposable part\footnote{For a monoidal $\infty$-category $\eM$, its indecomposable part can be defined as the pullback of $\operatorname{indec}(\pi_0({}^{\leq 0} \eM))\hookrightarrow \pi_0({}^{\leq 0} \eM)\leftarrow \eM$, where $\infty\Cat\xrightarrow{{}^{\leq 0}(\blank)} \sS\xrightarrow{\pi_0}\Set$ are (product-preserving) left adjoints to the inclusions. Indecomposables are only functorial in monoidal equivalences.}. 
Now the lemma is reduced to \cref{aut_of_cubes_are_trivial}.
\end{proof}

\section{The tensor product of categorical spectra}\label{section_construction_of_tensor_product}
We are now ready to prove the main theorem, namely the existence and universal property of the tensor product of categorical spectra. With the half-central structure of $\vS^1$ at hand, this is now a special case of a general result on the monoidal inversion of a central object.
The part of its proof relying on even more general facts about idempotent algebras is separated in \cref{subsection_idempotent_E1_alg}. 
The commutative case of the results of this section is well known \cite{voevodskyMathbfHomotopyTheory1998}\cite{robaloKtheoryBridgeMotives2015}\cite{nikolausGroupCompletionTheorem2017}\cite[\S 4.8.2]{LurieHA}\cite[Lecture V]{clausenCondensedMathematicsComplex}, but the proof requires modification (e.g., see the footnote of \cref{prop_inverting_central_object}), so we will spend some pages spelling out the details.

\subsection{The main theorem}
Recall that \cref{half-center_S1} allows us to lift the defining colimit diagram 
\[\infty\Cat_\ast\xrightarrow{\Sigma} \infty\Cat_\ast\xrightarrow{\Sigma}\cdots\to \CatSp \in \Pr^\sL_\omega \]
to a telescope in the category of $(\infty\Cat_\ast, \infty\Cat_\ast)$-bimodules in $\Pr_\omega^\sL$:
\[A\xrightarrow{\Sigma} D\xrightarrow{\Sigma}A\xrightarrow{\Sigma} D\xrightarrow{\Sigma}\cdots \to A_{\Sigma}= \CatSp.\]
In particular, $\Sigma_\infty: \infty\Cat_{\ast}\to \CatSp$ canonically lifts to a map of $\infty\Cat_\ast^{\owedge}$-bimodules. We denote $\Alg(\BMod{A}(\sV))$ by $\Alg_A(\sV)$. We can now state and prove the main theorem:
\begin{theorem}\label{thm_tensor_product_of_catsp}
    \begin{enumerate}
        \item $\vec{S}^1\in \infty\Cat_\ast$ acts invertibly (from the left and right) on the bimodule $\CatSp$. 
        \item The map $\Sigma^\infty: \infty\Cat_\ast\to \CatSp$ exhibits $\CatSp$ as an idempotent $\EE_0$-algebra of $\BMod{\infty\Cat_\ast}(\Pr_\omega^\sL)$. 
        \item $\CatSp$ admits a unique lift to $\Alg_A(\sV)$. The lax monoidal forgetful functor $\BMod{A}(\sV)\to \sV$ induces the underlying presentably monoidal structure on $\CatSp$.
        \item The monoidal category $\CatSp^\otimes$ is the monoidal inversion $\infty\Cat^\owedge_\ast[\vec{S}^{-1}]$. That is, 
        $\CatSp^\otimes$ is the initial $\infty\Cat^\owedge_\ast$-algebra on which $\vec{S}^{-1}$ acts invertibly. 
    \end{enumerate}
\end{theorem}

The theorem is a consequence of the following more general observations:

\begin{proposition}\label{prop_inverting_central_object}
    Let $\sV=\Pr^\sL_\omega$\footnote{We will only need that the monoidal structure on $A$ commutes with sequential colimits.} and let $A^\otimes \in \Alg(\Pr^\sL_\omega)$ be a monoidal category (not necessarily $\infty \Cat_\ast$). Let $s\in \fZ(A)$ be an object with a central structure (in particular $\tau: l_s=s\otimes (\blank)\xrightarrow{\sim} r_s=(\blank)\otimes s$), with which we regard $l_s: A\to A$ as a morphism in $\BMod{A}(\sV)$. Let $A_s\coloneqq \colim(A\xrightarrow{l_s}A\xrightarrow{l_s}\cdots)\in\BMod{A}(\sV)$ be its telescope. Assume moreover that $\tau_s: s\otimes s\to s\otimes s$ is equivalent to $\id_{s\otimes s}$\footnote{This is a noncommutative variant of an argument usually attributed to Voevodsky but the complication is quite different. The main point of Voevodsky's argument was that when we want to invert an object $s\in A$ by telescoping (in a symmetric monoidal category), it suffices to check that the cyclic permutation $(123)$ on $s^3$ is homotopic to the identity. In our argument, this condition is trivially satisfied (we use double suspension to begin with, and there is no nontrivial automorphism of $s^2$). We use the telescope of \emph{left} multiplications, which would intuitively invert the left action of $x$, except that we need a central structure to make sense of it. The right action is inverted essentially because it can be identified with the left action through the central structure, and the triviality of transposition is used here again.}. Then:
    \begin{enumerate}
        \item The left and right actions of $s\in A$ on $A_s$ are equivalent to each other and are $A$-bimodule isomorphisms. Explicitly, they are computed as the shift functor of components.
        \item The canonical functor $\eta: A\to A_s$ exhibits $A_s$ as an idempotent $\EE_0$-algebra object of $\BMod{A}(\sV)$. In particular, it uniquely lifts to an idempotent $\EE_1$-algebra object in $\BMod{A}(\sV)$. 
        \item For any $B\in \Alg(\sV)$, the forgetful functor $\BMod[A_s]{B}(\sV)\to \BMod[A]{B}$ (resp.\ $\BMod[B]{A_s}(\sV)\to \BMod[B]{A}(\sV)$) is fully faithful with essential image consisting of bimodules $M$ on which $s$ acts invertibly from the left (resp. right).
        \item There is an inclusion $\Alg_{A_s}(\sV)\hookrightarrow\Alg_{A}(\sV)$ whose image is the $A$-algebras on which $s\in A$ acts invertibly from both sides. $A_s$ is initial among such $A$-algebras. 
        \end{enumerate}
    \end{proposition}
    \begin{remark}
        The author hopes that the following variants of the last claim are also true. 
        \begin{enumerate}
            \item $A_s$ is initial among objects of $\Alg_{A/}$ on which the image of $s\in A$ acts invertibly. 
            \item For any algebra object $B\in \Alg(\Pr^\sL_\omega)$, the induced map $\Fun^\otimes(A_s, B)\to \Fun^\otimes(A, B)$ is a full subcategory inclusion. The image consists of those functors $f: A\to B$ with $f(s)\in B$ invertible. 
        \end{enumerate}
        For now, we do not attempt to prove these claims.
    \end{remark}
    \begin{proof}[{Proof of \ref{thm_tensor_product_of_catsp}}]
        Apply \cref{prop_inverting_central_object} for $A=\infty\Cat_\ast^\owedge$ and $s=\vS^2$. To check that $\tau_{\vS^2}: \vS^2\owedge \vS^2\xrightarrow{\sim} \vS^2\owedge\vS^2$ is homotopic to the identity, it suffices to observe that the monoid of endomorphisms of $\vS^4= \rB^4 \Free_{\EE_4}$ is $\Free_{\EE_4} = \bigsqcup_{n\in \NN} \EE_4(n)/\Sigma_n$, so $\Aut(\vS^4)\simeq \ast$.
    \end{proof}
    \begin{proof}
        \begin{enumerate}
            \item Unwinding the bimodule morphism structure on $l_s: A\to A$, the left and right actions of $a\in A$ on $A_s$ are given as the colimits of the telescopes (in the horizontal direction) of the following commutative squares:
            \[\begin{tikzcd}
                A \ar[d, "l_a"'] \ar[r, "l_s"] & A \ar[d, "l_a"]\ar[ld, "\alpha"', Rightarrow, shorten <=2ex, shorten >=2ex]\\
                A \ar[r, "l_s"] & A,
            \end{tikzcd} \quad \text{resp.}\quad \begin{tikzcd}
                A \ar[d, "r_a"'] \ar[r, "l_s"] & A \ar[d, "r_a"]\ar[ld, "\beta"', Rightarrow, shorten <=2ex, shorten >=2ex]\\
                A \ar[r, "l_s"] & A.
            \end{tikzcd}.\]
            Here the invertible $2$-cell $\beta$ is just the associator of the monoidal structure of $A$, whereas $\alpha$ uses the central structure on $s$:
            \[a\otimes (s\otimes (-))\simeq (a\otimes s)\otimes (-)\xrightarrow{\tau_a} (s\otimes a)\otimes (-)\simeq s\otimes (a\otimes (-)).\]
            When $a=s$, the assumption that $\tau_{s}\simeq \id : s\otimes s\to s\otimes s$ implies $\alpha\simeq \id$, so we have the following factorization of the $2$-cell:
            \[\begin{tikzcd}[column sep=tiny, row sep=tiny]
                A \arrow[rrrr, "l_{s}"] \arrow[dd, "l_{s}"'] &{}\ar[d, equal, shift right=.7ex] &&& A \arrow[dd, "l_{s}"]\\
                &{}&& {}\ar[d, equal, shift left=1.3ex, shorten >=1.5ex]& \\
                A \arrow[rrrruu, "f = \id" description]\arrow[rrrr, "l_{s}"'] &&&{}& A.
            \end{tikzcd}\]
            By the cofinality of $\NN\xrightarrow{+1} \NN$, the maps $f$ induce an inverse to $l_s: A_s\to A_s$. The colimit is taken in the category of bimodules, so $l_s$ is an invertible $A$-bimodule morphism. From this description, it is also clear that $l_s$ is a shift functor.
            
            It remains to identify $r_s: A_s\to A_s$ with $l_s$. It suffices to provide an invertible $3$-cell filling the following cylinder (the subscript $s$ of $l, r$ are omitted) because then its telescope exhibits that $\tau$ induces the equivalence between left and right action on $A_s$:
            \[\begin{tikzcd}
                A \ar[d, "l", bend left]\ar[d, "r"', bend right] \ar[r, "l"]\ar[d, phantom, "{\scriptstyle\Leftarrow}", "{\scriptstyle \tau}" near end] & A \ar[d, "l", bend left]\ar[ld, "\alpha"', Rightarrow, shorten <=2ex, shorten >=2ex, bend left, shift={(0, 1ex)}]\\
                A \ar[r, "l"'] & A,
            \end{tikzcd} \quad \simeq \quad
            \begin{tikzcd}
                A \ar[d, "r"', bend right] \ar[r, "l"] & A \ar[d, "r"', bend right]\ar[d, "l", bend left] \ar[ld, "\beta"', Rightarrow, shorten <=2ex, shorten >=2ex, bend right, shift={(0, -1.5ex)}]\ar[d, phantom, "{\scriptstyle\Leftarrow}", "{\scriptstyle \tau}" near end]\\
                A \ar[r, "l"'] & A
            \end{tikzcd}, \]
            or a $2$-cell that fills the following triangle (here $s_v$, $s_h$ denote $s$ being multiplied in the vertical and horizontal arrow respectively, and $\tau^v$ and $\tau^h$ are the central structures of $s_v$, $s_h$): 
            \[\begin{tikzcd} s_v \otimes s_h\otimes (-) \ar[r, "(\tau^h)^{-1}"]\ar[rd, "\tau^v"] & s_h\otimes s_v\otimes (-)\ar[d, "\tau^v"]\\ & s_h\otimes (-)\otimes s_v \end{tikzcd}\]
            This is filled by $\tau^v_{s_h} = (\tau^h_{s_v})^{-1}$ and the coherence of $\tau^v$.
            \item The morphisms $A\otimes_A A_s \xrightarrow{\eta\otimes A_s}A_s\otimes A_s$ and $A_s\otimes_A A\xrightarrow{A_s\otimes \eta}A_s\otimes A_s$ are the colimits of the telescopes along the endomorphisms $l_s\otimes A_s$ and $A_s\otimes l_s$, which are the left and right action of $s$ on $A_s$, so they are invertible. Any idempotent $\EE_0$-algebra lifts uniquely to an idempotent $\EE_1$-algebra by \cref{prop_idem_E0_is_E1}.
            \item The unit transformation $M\to A_s\otimes_A M$ is idempotent, so the free-forgetful adjunction is a localization. The unit is an equivalence if and only if $s$ acts invertibly on $M$ from the left. 
        \end{enumerate}
    \end{proof}
    \begin{remark}
        If we start from $\infty\Cat_\ast$ with the oplax smash product $\owedge^\rev$, the directed circle $\vS^1\in \infty\Cat_\ast$ is still half-central. 
        To see this, note that we have an equivalence $\op: \infty\Cat^{\otimes}\xrightarrow{\sim} \infty\Cat^{\otimes^\rev}$ and similarly for $\infty\Cat_\ast$ as objects of $\Alg(\PrL)$. 
        Along this equivalence, $\vS^1$ and $D$ are preserved because $(\vS^1)^\op= \vS^1$ and $D\circ \op = \co = \op\circ D$. 
        Consequently, they acquire the transferred structure of a monoidal involution and a half-central object, so we can also form a monoidal structure on $\CatSp$ by inverting $\vS^1$ in $\infty\Cat^{\owedge^\rev}_\ast$. Using the universal property, it is the image of $\infty\Cat_\ast^\owedge\to \CatSp^\otimes\in \Alg(\PrL)$ under the reversal involution on $\Alg(\PrL)$.
    \end{remark}
    \begin{remark}\label{remark_remedy_of_noncommutativity}
        The non-commutativity is somewhat unfortunate, given that many constructions in algebraic geometry rely on the commutativity of algebra. 
        The author does not know whether there is a reasonably nondestructive localization to a commutative (or at least $\EE_2$) tensor product. Conjecturally, the tensor product of categorical spectra with adjoints considered in \cref{chapter_categorical_spectra_with_adjoints} promotes to an $\EE_\infty$-structure\footnote{The author thanks Mayuko Yamashita and Thomas Blom for originally suggesting that it should be more than $\EE_1$, perhaps (framed) $\EE_2$. In joint work with David Reutter this is reduced to a conjecture related to the $\rO(n)$-action on $n$-categories with adjoints.}.
        The variant of the theorem for $\infty\Cat$ with the Cartesian product and $\CatSp^\cn$ instead of $\CatSp$ is true for a formal reason, namely that $\CMon$ is idempotent in $\PrL$ \cite{gepnerUniversalityMultiplicativeInfinite2016}). 
        
        Also, note that the noncommutativity of the Gray tensor product is not too uncontrolled; there is a duality involution that reverses the tensor: $X^\op\otimes Y^\op\simeq (Y\otimes X)^\op$. In other words, it is a \emph{$\ast$-algebra} in $\PrL$. One sees that $\CatSp$ has the induced structure of a $\ast$-algebra. Therefore, to fully exploit the multiplicative structure of categorical spectra, the theory of categorical $\ast$-algebras seems to be a good language to develop. 
    \end{remark}
    
    \subsection{Idempotent \texorpdfstring{$\EE_1$}{E1}-algebras}\label{subsection_idempotent_E1_alg}
    If $\eC$ is a symmetric monoidal category, \cite[Proposition 4.8.2.9]{LurieHA} shows that the forgetful functor $\CAlg^{\mathrm{idem}}(\eC)\to \Alg_{\EE_0}^{\mathrm{idem}}(\eC)$ is an equivalence. We spell out verifications of the $\EE_1$-variant of the related statements. Let $\eC^\otimes$ be a presentably $\EE_1$-monoidal category with unit $1$.
    \begin{definition}
        An $\EE_0$-algebra $\eta: 1\to E$ in $\eC$ is said to be \emph{idempotent} if the maps $E\simeq 1\otimes E\xrightarrow{\eta\otimes E} E\otimes E$ and $E\simeq E\otimes 1\xrightarrow{E\otimes \eta}E\otimes E$ are equivalences. An $\EE_1$-algebra $E$ is \emph{idempotent} if the multiplication map $E\otimes E\to E$ is an equivalence, or equivalently, if the underlying $\EE_0$-algebra is idempotent. 
    \end{definition}
    By definition, an $\EE_0$-algebra $\eta: 1\to E$ is idempotent if and only if the functor $L_E^l = E\otimes(\blank): \eC\to \eC$ is a localization
    (as in \cite[Prop. 5.2.7.4]{LurieHTT}). 
    Notice the monoidal equivalence $\eC\simeq \End_{\RMod_\eC(\PrL)}(\eC)$; this implies that $\Alg(\eC)$ is equivalent to the category of (right-)$\eC$-linear monads on $\eC$. 
    Since the category of idempotent monads is equivalent to the category of localizations, this is equivalent to the category of idempotent $\EE_0$-algebras in $\eC$ by $L_E^l\leftrightarrow E$. Consequently, we have:
    \begin{proposition}\label{prop_idem_E0_is_E1}
        The forgetful functor $\Alg_{\EE_1}^{\mathrm{idem}}(\eC)\to \Alg_{\EE_0}^{\mathrm{idem}}(\eC)$ is an equivalence.
    \end{proposition}
    Note that we automatically have $E\otimes \eta \simeq \eta\otimes E$ because both are inverse to the multiplication map $E\otimes E\to E$. 
    \begin{remark}
        If $\eC$ is given a symmetric monoidal structure, any idempotent $\EE_1$-algebra automatically upgrades to an $\EE_\infty$-algebra, so this subsection is only relevant if $\eC$ itself is noncommutative. This observation also implies that $\sS\hookrightarrow n\Cat$ is not idempotent in $\Pr^\sL$ for $n> 0$, because these categories have noncommutative monoidal structures (namely, the lax Gray tensor products) on these categories\footnote{A separate argument is needed for $n=1$. For instance, the existence of a semiCartesian tensor product distinct from the Cartesian product (namely, the funny tensor product) contradicts idempotence.}, and similarly for $\Sigma_+^\infty:\sS\to\CatSp$. 
    \end{remark}
    \begin{remark}
        As usual, one may also check more directly that $\Alg_{\EE_1}^\mathrm{idem}(\eC)\to \Alg_{\EE_0}(\eC)$ is an equivalence of posets, as follows. 
        Suppose there exists an $\EE_0$-algebra map $f: A\to B$ between idempotent $\EE_1$-algebras. We wish to show that $\Hom_{\EE_0}(A, B)$, $\Hom_{\EE_1}(A, B)$ are both contractible. 
        Since $\eta_B\otimes B \simeq (1\otimes B\xrightarrow{\eta_A\otimes B} A\otimes B\xrightarrow{f\otimes B} B\otimes B)$ is an equivalence, $B$ is a retract of $A\otimes B$ and therefore $L^l_A$-local, i.e., $B\simeq A\otimes B'$ for some $B'$. Moreover, $(B\xrightarrow{\eta_A\otimes B} A\otimes B) \simeq (A\otimes B'\xrightarrow{\eta_A\otimes A\otimes B'} A\otimes A\otimes B')$ is an equivalence. 
        Now consider the square $(\eta_A: 1\to A)\otimes (f: A\to B)$; the diagonal is $(\eta_A\otimes B)\circ f \simeq (A\otimes f)\circ (\eta_A\otimes A) \simeq (A\otimes f)\circ (A\otimes \eta_A) \simeq A\otimes \eta_B$, so we have a canonical contraction $f\simeq (\eta_A\otimes B)^{-1}\circ (A\otimes \eta_B)$, i.e., $\Hom_{\EE_0}(A, B)\simeq \ast$. 
        The symmetric argument implies $B\simeq B\otimes A$, so $B$ lies in $\Alg_{\EE_1}(A\eC A)\subset\Alg_{\EE_1}(\eC)$ (notice that the inclusion $(A\eC A)^\otimes \subset \eC^\otimes$ is a map of (nonsymmetric) operads and is a nonunital monoidal functor). It follows that $\Hom_{\Alg_{\EE_1}(\eC)}(A, B)$ is contractible because $A$ is initial (the tensor unit) in $\Alg_{\EE_1}(A\eC A)$.
    \end{remark}
    
    \begin{proposition}\label{prop_module_univ_property_of_smashing_loc}
        Let $E$ be an idempotent $\EE_1$-algebra. Then 
        \begin{enumerate}
            \item The forgetful functor $\LMod_E(\eC)\to \eC$ (resp.\ $\RMod_E(\eC)\to \eC$) is an equivalence onto the full subcategory $E\eC$ (resp.\ $\eC E$). 
            \item The forgetful functor $\BMod{E}(\eC)^\otimes \to \eC^\otimes$ is an equivalence to the full suboperad $(E\eC E)^\otimes \subset \eC^\otimes$.
        \end{enumerate}
    \end{proposition}
    \begin{proof}
        We spell out the left module case. The unit law $M\simeq 1\otimes M\xrightarrow{\eta_E\otimes M} E\otimes M\xrightarrow{a} M$ implies that $M$ is a retract of $E\otimes M$, so $M$ is $L^l_E$-local and therefore $\eta_E\otimes M = \eta_E\otimes E\otimes M'$ (for some $M'$) is an equivalence. This means that $a$, the counit map of the free-forgetful adjunction, is also an equivalence.   
        Also, on these local objects, the unit map $E\otimes X\xrightarrow{\eta\otimes E\otimes X}E\otimes E\otimes X$ of the free-forgetful adjunction is an equivalence, so the adjunction induces the stated equivalence. 
        A similar argument for bimodule, performed componentwise, proves the second claim. 
    \end{proof}

\section{Basic properties of the tensor product}\label{section_stable_properties_tensor_product}
    The goal of this short section is to discuss explicit descriptions of the tensor product and internal homs. We will also show that the tensor product behaves additively on category level and connectivity, and obtain comparison results for the tensor product with those of spectra and symmetric monoidal categories.

\subsection{Monoidal involutions and the half-central structure of \texorpdfstring{$\FF[1]$}{F[1]}}
    The proof of \cref{prop_inverting_central_object} (1) admits an obvious modification to the case where $s$ is half-central, which shows the following:
    \begin{proposition}\label{left_action_of_S1_on_catsp}
        The left action of $\vS^1\in\infty\Cat_\ast$ on $\CatSp$ (which is also the left action of $\FF[1]$) is canonically isomorphic to the shift functor. 
    \end{proposition}
    \begin{proposition}
        The total dual functor $D = (\blank)^\circ: \CatSp\to \CatSp$ uniquely promotes to an automorphism in $\Alg(\BMod{\infty\Cat}(\PrL))$. The half-central structure on $\vS^1\in \infty\Cat^\owedge_\ast$ induces a canonical half-central structure on $\FF[1]\in \CatSp^\otimes$ with respect to this total dual. In particular, we have $X\otimes \FF[1]\simeq \Sigma(X^\circ)$.
    \end{proposition}
    \begin{proof}
        By the universal property of $\CatSp = \infty\Cat_\ast^\owedge[\vS^{-1}]$, there is a unique involution $\CatSp\to \CatSp$ in $\Alg_{\infty\Cat^\otimes}(\PrL)$ making the following diagram commute:
        \[\begin{tikzcd}
            \infty\Cat_\ast^\owedge \ar[r, "D"]\ar[d, "\Sigma^\infty"] & \infty\Cat_\ast^\owedge\ar[d ,"\Sigma^\infty"] \\
            \CatSp^\otimes \ar[r, dashed, "\tilde{D}"]& \CatSp^\otimes.
        \end{tikzcd}\]
        To see that the underlying functor of $\tilde{D}$ is the total dual functor defined in \cref{definition_duality_involutions_of_catsp}, note that since $\tilde{D}$ is a $\infty\Cat^\otimes_\ast$-bimodule homomorphism, it must commute with the shift operators. Composing it with the squares of the right adjoints to the above, we have the following commutative diagram:
        \[\begin{tikzcd}
            \CatSp \ar[r, "\tilde{D}"]\ar[d, "\Sigma^n"]\ar[dd, "\Omega^{\infty-n}"', bend right=40]& \CatSp \ar[d, "\Sigma^n"'] \ar[dd, "\Omega^{\infty-n}", bend left=40]\\
            \CatSp \ar[r, "\tilde{D}"]\ar[d, "\Omega^\infty"] & \CatSp \ar[d, "\Omega^\infty"']\\
            \infty\Cat_\ast \ar[r, "D"] & \infty\Cat_\ast. 
        \end{tikzcd}\]
        This canonically factors through the same square for $n=0$, so $\tilde{D}$ is the map induced by $D$ in each degree compatibly along the loops, i.e., it agrees with the total dual functor. 
        The following diagram gives the (twisted) left action of $\FF[1] = \Sigma^{\infty}_+\vS^1$ on $B = \CatSp$ an $A = \infty\Cat_\ast^\owedge$-bimodule:
        \[\begin{tikzcd}
        A\ar[r]\ar[d, "l_{\vS^1}"] & D \ar[r]\ar[d, "l_{\vS^1}"] & A \ar[r]\ar[d, "l_{\vS^1}"] & \cdots \ar[r] & B\ar[d, "l_{\FF[1]}"]\\
        D\ar[r] & A \ar[r] & D \ar[r]& \cdots \ar[r] & D\otimes_A B.
        \end{tikzcd}\]
        $D\otimes_A B$ is a priori an $(A, B)$-bimodule. The left action of $A$ on the underlying $B$-module $D\otimes_A B$ (which is equal to $B$) defines a monoidal functor $l: A\to \End_{\RMod_B}(D\otimes_A B) = \End_{\RMod_B}(B)\simeq B$.
        This sends $a\in A$ to $\Sigma^{\infty}_+ a^\circ$. In particular, the image of $\vS^1$ is invertible, so $l$ extends to $B$, $D\otimes_A B$ promotes to a $(B, B)$-bimodule, and $l_{\FF[1]}$ is a $(B, B)$-bimodule map. 
        Moreover, this extended map is $D$ above by its universal property.
    \end{proof}
    \begin{remark}
        Similarly, one can show that the $\co$- and $\op$-duals induce $\CatSp^{\otimes^\rev}\to \CatSp^\otimes$.
    \end{remark}
    
    \subsection{Formulas for the tensor product and internal hom}    
    As an immediate consequence of \cref{spectra_are_colimits_of_suspension_spectra}, we have the following: 
    \begin{corollary}
        Let $X=(X_n)$ and $Y=(Y_n)$ be categorical spectra. Then we have 
        \[X\otimes Y\simeq \colim_{i, j}(\Sigma^{\infty-i-j}(D^j X_i)\owedge Y_j)\simeq \colim_{n}(\Sigma^{\infty-4n}(X_{2n}\owedge Y_{2n})).\] 
        Similarly, we have the analogous formula with $\rB^{\infty-i}$ in place of $\Sigma^{\infty-i}$, regarding $X_i$, $Y_j$ as symmetric monoidal categories. 
    \end{corollary}
    \begin{proof}
        Recall that $X\otimes Y \simeq (\colim_i \Sigma^{\infty-i} X_i)\otimes (\colim_j \Sigma^{\infty-j} Y_j)\simeq \colim_{i, j}\Sigma^{\infty-i} X_i\otimes \Sigma^{\infty-j} Y_j\simeq \colim_{i, j} \Sigma^{-j}D^j(\Sigma^{\infty-i}X_i)\otimes \Sigma^\infty Y_j\simeq \colim_{i, j} \Sigma^{\infty-i-j} (D^j X_i)\owedge (Y_j)$. The second part results from restricting to the cofinal diagram $2\NN\xrightarrow{\mathrm{diag}} \NN\times \NN$.
    \end{proof}

    \begin{corollary}\label{formula_internal_hom_catsp}
        We have the corresponding formula for the internal hom of categorical spectra:
        $\Omega^{\infty-n}[X, Y] \simeq \lim_{k\to \infty}[D^n X_k, Y_{n+k}]$, where the internal hom on the right-hand side is taken in either $\infty\Cat_\ast^\owedge$ or $\infty\SMC^\otimes$. 
    \end{corollary}
    \begin{proof}
        We have the following natural equivalences for $Z\in\infty\Cat_{\ast}$: 
        \begin{align*}
            \Map_{\infty\Cat_\ast}(Z, \Omega^{\infty-n}[X, Y])
            & \simeq \Map_{\CatSp}(\Sigma^{\infty-n}Z, [X, Y])\\
            &\simeq \Map_{\CatSp}(\colim_k \Sigma^{\infty-k} X_k\otimes \Sigma^{\infty-n} Z, Y)\\
            &\simeq \lim_k \Map_{\CatSp}(\Sigma^{\infty-k-n} (D^n X_k\owedge Z), Y) \\
            &\simeq \lim_k \Map_{\CatSp}(D^n X_k\owedge Z, Y_{n+k})\\
            &\simeq \Map_{\infty\Cat_\ast}(Z, \lim_k[D^n X_k, Y_{n+k}]),
        \end{align*}
        and similarly for $Z\in \infty\SMC$.
    \end{proof}

    \subsection{Additivity on categorical levels}
    Just as with the Gray tensor product of $\infty$-categories, the tensor product of categorical spectra behaves additively on category levels. 
    \begin{proposition}\label{tensor_product_adds_category_levels}
        For $m, n\in \ZZ\cup\{\pm\infty\}$, the essential image of $m\CatSp\otimes n\CatSp$ under the tensor product is $(m+n)\CatSp$. Here the convention is $\infty+(-\infty)= -\infty$. In particular, $0\CatSp\subset \CatSp$ is a monoidal subcategory and $\Sp\subset \CatSp$ is a $\otimes$-ideal.
    \end{proposition}
    \begin{proof}
        The category $n\CatSp$ is the colimit-closure of the set $\{\Sigma^{\infty-i}_+\cube^j \mid 0\leq j\leq n+i \text{ or } j=0\}$. Since $D\cube^j\simeq \cube^j$, we have  $\Sigma^{\infty-i}_+ \cube^j\otimes \Sigma^{\infty-k}_+ \cube^l\simeq \Sigma^{\infty-i-k}_+\cube^{j+l}\in (n+m)\CatSp$. 
        To see surjectivity, observe that $\FF[n]\in n\CatSp$ and $(\blank)\otimes \FF[n]: m\CatSp\xrightarrow{\sim} (m+n)\CatSp$.
    \end{proof}
    \begin{proposition}
        $\Sp\subset \CatSp$ is an exponential ideal. In particular, the group completion $L: \CatSp\to \Sp$ uniquely induces a monoidal structure on $\Sp$, which agrees with the usual tensor product of spectra. The local unit $\SS$ is idempotent over $\FF$, and $L$ is a smashing localization by $\SS$. 
    \end{proposition}
    \begin{proof}
        Let $X\in \Sp$ and $Y\in \CatSp$. To show that $[Y, X]$ is a spectrum, it suffices to check that it is local for $f_{i, j}: \Sigma^{\infty-i}_+ (C_{j+1}\to C_j)$ for any $i, j\geq 0$. As $\Sp\subset\CatSp$ is closed under limits, we may assume that $Y$ is of the form $\Sigma^{\infty-k}_+ Y'$, so $Y\otimes f_{i, j}$ is of the form $\Sigma^{\infty-i-k}_+ D^i(Y')\otimes (C_{j+1}\to C_j)$, which is an $L$-equivalence because $\sS\subset\infty\Cat$ is an exponential ideal \cite[\S 3.2]{campionGrayTensorProduct2023}. Similarly, $\llbracket Y, X\rrbracket$ is a spectrum, so $\Sp$ is an exponential ideal. It follows that there exists a unique monoidal structure on $\Sp$ promoting $L$ to a monoidal localization. Since $\Sigma^\infty_{\Sp}: \sS_\ast\hookrightarrow \infty\Cat_\ast\xrightarrow{\Sigma^\infty}\CatSp\xrightarrow{L}\Sp$ is monoidal, the induced monoidal structure on $\Sp$ is the usual tensor product. Now observe that $\Hom(L(Z\otimes Y), X)\simeq \Hom_{\CatSp}(Z, [Y, X])\simeq \Hom_{\Sp}(LZ\otimes Y, X)$ for $X\in \Sp$, so $L(Z\otimes Y)\simeq LZ\otimes Y$. Plugging in $Z = \FF$, we get $LY\simeq \SS\otimes Y$. 
    \end{proof}
    \begin{warning}
        One should not expect $n\CatSp\subset \CatSp$ to be an exponential ideal unless $n=\pm\infty$. In fact, since $[\FF[-n], X]\simeq \FF[n]\otimes X$, this internal hom increases the category level if $n$ is negative.
        However, it is still a Gray-bimodule because it is closed under cotensoring by $\infty\Cat$, whose image under $\Sigma_+^\infty$ is connective.
    \end{warning}

    \subsection{Additivity on connectivity}
    Recall that we denoted the essential image of the embedding $\rB^{\infty-n}: \CMon(\infty\Cat)\hookrightarrow \CatSp$ by $\CatSp^{n\mhy\cn}$ and said that its objects are \emph{$n$-connective}.
    The following proposition shows that the tensor product respects connectivity; in particular, it restricts to what should be called the lax tensor product of symmetric monoidal $\infty$-categories.
    \begin{proposition}
        The inclusion $\rB^\infty: \CMon(\infty\Cat)\xrightarrow{\sim}\CatSp^\cn\subset \CatSp$ exhibits $\CMon(\infty\Cat)$ as a monoidal subcategory. 
        This monoidal structure is characterized by the fact that the functor $\Free_{\EE_\infty}: \infty\Cat\to \CMon(\infty\Cat)$ promotes to a monoidal functor with respect to the Gray tensor product of the domain. 
        Moreover, the image of $\CatSp^{m\mhy\cn}\otimes \CatSp^{n\mhy\cn}\subset \CatSp\otimes \CatSp$ under the tensor product is $\CatSp^{(m+n)\mhy\cn}$. 
    \end{proposition}
    \begin{proof}
        The tensor unit $\FF$ is connective. 
        To prove the first statement, by \cite[Proposition 2.2.1.1]{LurieHA}, it suffices to check that for any $X, Y\in \CMon(\infty\Cat)$, the tensor product $(\rB^\infty X)\otimes (\rB^\infty Y)$ lies in the image of $\rB^\infty$. 
        Write $X\simeq \colim_i \Free_{\EE_\infty}(\eC_i)$ and $Y\simeq \colim_j \Free_{\EE_\infty}(\eD_j)$. Now we have 
        \begin{align*}
            (\rB^\infty X)\otimes (\rB^\infty Y) &\simeq (\colim_i \Sigma^\infty_+(\eC_i))\otimes (\colim_j \Sigma^\infty_+(\eD_j))\\
            &\simeq \colim_{i, j}(\Sigma_+^\infty(\eC_i)\otimes \Sigma_+^\infty(\eD_j))\simeq \colim_{i, j}\Sigma_+^\infty(\eC_i\otimes \eD_j).
        \end{align*}
        Since $\CatSp^\cn\subset \CatSp$ is closed under colimits, the last colimit stays inside $\CatSp^\cn$. The characterization is also clear from this computation. 
    \end{proof}
    Recall that by \cite{gepnerUniversalityMultiplicativeInfinite2016} $\CMon\in \PrL$ is an idempotent algebra and the associated symmetric monoidal localization $\CMon(\blank) = \CMon\otimes (\blank): \PrL\to \PrL$ universally turns a presentable category into a \emph{semiadditive} presentable category. In particular, it induces
    $\CMon: \CAlg(\PrL)\to \CAlg(\PrL)$ and $\CMon: \Alg(\PrL)\to \Alg(\PrL)$. We denote the symmetric monoidal structure on $\CMon(\infty\Cat, \times)$ by $\oast$ and the monoidal structure on $\CMon(\infty\Cat, \otimes)$ by $\otimes$. The lax monoidal functor $\id: (\infty\Cat, \times)\to (\infty\Cat, \otimes)$ of \cref{remark_lax_monoidal_structure_on_id_from_cart_to_semicart} induces a lax monoidal structure on $\CMon(\infty\Cat)^\oast\to \CMon(\infty\Cat)^\otimes$. 
    \begin{corollary}
        The functor $\rB^\infty$ induces a functor $\Rig_{\EE_1}(\infty\Cat)\to \Alg_{\EE_1}(\CatSp)$. 
    \end{corollary}

\chapter{Absolute colimits in categorical spectra}
\label{chapter_ablosute_colimits_in_catsp}
    In this chapter, we begin our study of \emph{stability} in categorical spectra. 
    Understanding stability is clearly important: the category $\CatSp$ should not exist in isolation, but should be put into the larger context of \emph{stable Gray-categories}, just as the category $\Sp$ of spectra is a universal example of stable ($1$-)categories. Recall that stability for categories has many equivalent characterizations:
    \begin{proposition}[{\cite[Chapter 1]{LurieHA}}]\label{characterizations_of_stable_1cat}
        The following conditions on $\eC\in \Cat$ are equivalent:
        \begin{enumerate}
            \item $\eC$ is pointed, has cofibers, and the suspension functor is an equivalence. 
            \item $\eC$ is pointed, has fibers and cofibers, and a triangle is a fiber sequence if and only if it is a cofiber sequence.
            \item $\eC$ has finite limits and colimits, and a commutative square is a pushout if and only if it is a pullback. 
            \item $\eC$ has finite limits and $\ev_{S^0}: \Exc_\ast(\sS^\fin, \eC)\to \eC$
            is an equivalence. Here $\Exc_\ast(\eA, \eB)\subset \Fun(\eA, \eB)$ is the full subcategory of reduced excisive functors, i.e., functors that preserve terminal objects and send pushout squares to pullback squares. 
        \end{enumerate}
        If these conditions are satisfied, we say $\eC$ is a stable $1$-category. Moreover, if $\eC$ is presentable, stability is equivalent to having a (necessarily unique) module structure over $\Sp^\otimes$. 
    \end{proposition}
    Knowing that all of these are equivalent is extremely useful; for instance, one checks the minimalistic condition (1) and knows immediately that it has finite limits and colimits, and that they preserve both limits and colimits. 
    We wish to put $\CatSp$ into a similar context. This can be separated into two interrelated questions:
    \begin{enumerate}
        \item What is a structure on a category $\eC$ that puts it on the same stage as the category $\CatSp$ of categorical spectra, even to discuss stability?
        \item What are the appropriate analogs of the conditions listed above? Which generalize well and which do not?
    \end{enumerate}
    In the discussion of (2), it seems reasonable that pointedness is kept as is and suspension is replaced by its directed analog. 
    This indicates that for (1), the underlying $1$-category is not enough; we do not know any way to recover the suspension functor out of the bare $1$-category structure on $\infty\Cat_\ast$ or $\CatSp$. 
    We used either the fixed-point property under enrichment or the Gray-module structure, i.e., $\Sigma = \vS^1\owedge(\blank)$.
    The former is not too common compared to the latter, so we choose the latter; at least it is more convenient in the presentable situation. 
    Also, note that we forego the monoidal structure because we would like to think of stability as a kind of linearity.
    So, in the presentable context, either left modules or bimodules over $\infty\Cat^\otimes$ seem to be a good option\footnote{In the bimodule case, we might not want $\vS^1$ to act randomly from the right, compared to the left action. This indicates that we should ask for a $\ast$-bimodule-like compatibility. In terms of Gray-categories, this should imply that left-hom and right-hom enrichments (one is oplax) pass to each other by an appropriate duality. We will not discuss this point further here.}.
    For general categories, a reasonable guess for the relevant structure is that of Gray-categories/algebroids (i.e., $\infty\Cat^\otimes$-enriched categories); the two are related via the functor $\theta'_{\eA}: \LMod_{\eA}(\PrL)\to \widehat{\eA\mhy\Cat}$ of \cite[\S 4.2]{stefanichHigherQuasicoherentSheaves2021} (when $\eA=\sS$, it is the forgetful functor $\PrL\hookrightarrow\widehat{\Cat}$).
    
    Let us take this as an answer to (1) and move on to (2); from \cref{thm_tensor_product_of_catsp}, we know that $\CatSp$ is the universal module where $\vS^1$ acts invertibly, and a presentable $\infty\Cat^\otimes$-module admits a (necessarily unique) $\CatSp$-module structure if and only if the $\vS^1$-action is invertible.
    Moreover, we have the notion of tensoring for a general enriched category (which is a particular instance of \emph{weighted colimits}), so the suspension is (if it exists) a well-defined operation for any pointed Gray-category.
    Summarizing the argument, we make the following preliminary definition:
    \begin{predefinition}
        A Gray-category $\eC$ is \emph{stable} if it has a zero object and ``finite weighted colimits,'' and the suspension functor $\Sigma: \eC\to \eC$ is an equivalence.
    \end{predefinition}
    This has two problems. The first is that we do not know what finiteness is appropriate for the weight; our prototypical example $\CatSp$ has all weighted colimits, so it will not help us much. 
    The second problem is that this is not a very useful form of stability. Even for stable $1$-categories, the invertibility of loop-suspension does not obviously imply more useful characterizations, and the most standard argument fails in our Gray case (mainly because of the failure of the pasting law for directed pullbacks and pushouts).
    A fundamental problem with \cref{characterizations_of_stable_1cat} is that it depends on choices of a class of diagrams that happen to be both cone and cocone shapes, which is not robust enough to generalize.
    In fact, the naive analogs of (3)(4) (replacing (co)fibers by lax (co)fibers and pushouts/pullbacks by directed pushouts/pullbacks) are false in $\CatSp$, so the characterization of stability using diagrams that happen to be both limits and colimits should be regarded as a happy $1$-categorical accident. 

    This suggests that, instead of searching for a minimalistic definition, we should find \emph{all} exactness properties enjoyed by $\CatSp$. For stable $1$-categories, this question is asked in \cite{MOpostAllExactnessStableCategories}. 
    We take the viewpoint that stability allows us to recognize a finite colimit as a finite limit over another ($\Sp$-weighted) diagram: if $J$ is a finite category, the colimit functor $\Fun(J, \Sp)\to \Sp$ preserves both limits and colimits, so it has both left and right adjoints. This left adjoint is necessarily of the form $X\mapsto W\otimes X$ for a diagram $W: J\to \Sp$, so the colimit functor is also the limit functor weighted by $W$. In this case, we say the colimit is \emph{absolute}; a $J$-indexed colimit is preserved by any $\Sp$-enriched functor between $\Sp$-enriched categories. 
    We may treat limits and colimits more symmetrically by considering \emph{weighted colimits}. Tensoring by $X\in \Sp$ is an example, in which case absoluteness is equivalent to the dualizability of $X$. The abundance of dualizable objects, i.e., \emph{Spanier-Whitehead duality} is a key motivation for introducing spectra. 
    
    This circle of ideas is roughly summarized as follows: we may approach the stability by studying absolute colimits in $\CatSp$.
    The appropriate finiteness condition for the weight should be figured out along the way, but we will not reach that point in this thesis.
    Instead, we start by collecting basic examples of absolute colimits.
    In \cref{section_absoluteness_of_directed_pushouts}, we will show the absoluteness of directed pushouts. 
    In this case, the weight is simple and strict enough to allow explicit diagrammatic calculation. We will spend the first section spelling out basic knowledge and examples of directed pushouts. 
    We regard directed pushouts as a bootstrap for constructing other absolute colimits. We will already observe some interesting consequences, but the full strength of this result is yet to be explored.

    \section{Some weighted (co)limits in Gray-categories}
    \subsection{Bimodules and weighted colimits}
    Let $\eA\in \Alg(\PrL)$ be a presentably monoidal category. Our main examples of interest are $\eA= \infty\Cat^\otimes$ and $\CatSp^\otimes$.
    Here we give a minimalistic review of weighted colimits of $\eA$-enriched categories, with weights indexed only over unenriched categories. See e.g., \cite[Chapter 5]{stefanichHigherQuasicoherentSheaves2021} and \cite{hinichYonedaLemmaEnriched2020} for a more general discussion (the latter does not mention weighted colimits, but treats the basic setup of Morita theory over a noncommutative base). See also \cite{streetAbsoluteColimitsEnriched1983} for the original treatment of absolute colimits.
    Let $I, J\in \Cat$ be small categories. 
    We regard $\PSh_{\eA}(I) = \Fun(I^\op, \eA) \simeq \PSh(I)\otimes_{\sS} \eA$ as a $\eA$-bimodule in $\PrL$. 
    Let $\RMod_{\eA}^\mathrm{free}(\PrL)\subset \RMod_{\eA}(\PrL)$ be the full sub-$2$-category spanned by the presheaf categories. The hom category is given by 
    \[\LFun_{\eA}(\PSh_{\eA}(I), \PSh_{\eA}(J))\simeq \LFun(\PSh(I), \PSh_{\eA}(J))\simeq \Fun(I, \PSh_{\eA}(J))\simeq \Fun(I\times J^\op, \eA).\] Notice that this functor category admits a left $\eA$-module structure induced by that of $\PSh_{\eA}(J)$, i.e., by pointwise tensoring on the value. 
    \begin{definition}
        An $(I, J)$-bimodule, or a profunctor $I\slashedrightarrow J$, is a functor $W: I\times J^\op\to \eA$. Let $\Prof_{\eA}\coloneqq \RMod_{\eA}^\mathrm{free}(\PrL)$ denote the $2$-category of (unenriched) categories and profunctors. 
    \end{definition}
    We also have an equivalence 
    \[\Fun(I\times J^\op, \eA)\simeq \Fun(J^\op, \PSh_{\eA}(I^\op))\simeq \LFun_{\eA}(\PSh_{\eA}(J^\op), \PSh_{\eA}(I^\op)).\] We warn that $\PSh_{\eA}(I)\leftrightarrow \PSh_{\eA}(I^\op)$ induces an equivalence $\Prof\simeq \Prof^\op$, so to avoid confusion, we discuss adjunctions in this category only after embedding $\Prof$ into $\RMod_{\eA}(\PrL)$.

    Let $\eC$ be a left $\eA$-module. Consider the functor
    \[I\times I^\op\times J\times \Fun(J, \eC)\xrightarrow{I(\blank, \blank)\times \ev} \sS\times \eC\to \eA\times \eC\to \eC.\]
    By currying, we have a functor $I^\op\times J\to \Fun(\Fun(J, \eC), \Fun(I, \eC))$, which factors through $\LFun(\Fun(J, \eC), \Fun(I, \eC))$. 
    When $\eC$ is an $\eA$-bimodule (resp.\ $\eA$-algebra), it moreover factors through the hom category $\LFun_{\eA}(\Fun(J, \eC), \Fun(I, \eC))$ in $\RMod_{\eA}(\PrL)$ (resp.\ $\LFun_{\eC}(\Fun(J, \eC), \Fun(I, \eC))$ in $\RMod_{\eC}(\PrL)$). 
    In any of these cases, the left $\eA$-module structure on $\eC$ induces one on $\LFun(\Fun(J, \eC), \Fun(I, \eC))$ by pointwise tensoring on the value, so we may uniquely extend the functor to a functor of left $\eA$-modules\footnote{
        Notice the close analogy with matrix calculus: a profunctor corresponds to a matrix $W_{ij}: I\times J\to A$, where $A$ is a base ring and $I$, $J$ are sets. For a left $A$-module $M$, we can define multiplication $M^{\oplus J}\to M^{\oplus I}$ by a matrix $W$, by extending the action of the elementary matrix $I\times J\to \Hom(M^{\oplus I}, M^{\oplus J}); (i, j)\mapsto \iota_j\circ \pr_i$ to a left $A$-module homomorphism. 
    }:
    \[\Fun(I\times J^\op, \eA)\to \LFun_{\eA}(\Fun(J, \eC), \Fun(I, \eC)). \]
    This assignment is the above-mentioned equivalence when $\eC=\eA$. 
    The image of a bimodule $W: I\slashedrightarrow J$ is denoted by $W\otimes (\blank)$ or $\colim^W(\blank)$ and called the \emph{colimit weighted by $W$}. 
    This functor adimits a right adjoint $\Fun(I, \eC)\to \Fun(J, \eC)$, which we denote by $[W, \blank]$ or $\lim^W$ and call the \emph{limit weighted by $W$}. 
    By adjunction and the Yoneda lemma (we only need it for plain $(\infty, 1)$-categories), this is a unique left $\eA$-module functor $\Fun(I\times J^\op, \eA)\to \RFun(\Fun(I, \eC), \Fun(J, \eC))^\op$ extending $(i, j)\mapsto [J(j, \blank), \ev_{i}(\blank)]: I\times J\to \RFun(\Fun(I, \eC), \Fun(J, \eC))$. Here we give the codomain the transferred left $\eA$-module structure (by adjunction, an object $a$ of $\eA$ acts by precomposing with the pointwise cotensor $[a, \blank]$ on $\Fun(I, \eC)$).

    For colimit, we often take $I=\ast$, so $W: J^\op\to \eA$ is a presheaf. In this case, $\colim^W: \Fun(J^\op, \eA)\to \LFun(\Fun(J, \eC),\eC)$ is the unique left $\eA$-module morphism with the natural equivalence $\yo(j)\otimes F \simeq F(j)$; this is the so-called \emph{Yoneda reduction}. \cite[Example 5.5.18]{stefanichHigherQuasicoherentSheaves2021} implies that this definition of weighted colimits agrees with that of Stefanich. In this case, the weighted limit $\eC\to \Fun(J, \eC)$ is simply the pointwise cotensor $X\mapsto (j\mapsto [W(j), X])$. 
    The other extreme case $J=\ast$ switches the roles of limits and colimits: the functor $\Fun(I, \eA)\to \LFun_{\eA}(\eC, \Fun(I, \eC))$ takes $W$ to $X\mapsto (i\mapsto W(i)\otimes X)$. The weighted limit $[W, \blank]: \Fun(I, \eC)\to \eC$ is characterized by the Yoneda reduction $[I(i, \blank), F] = F(i)$ and the fact that it turns colimits into limits and tensors into cotensors ($[a\otimes W, F]\simeq [W, [a, F]]$).
    In particular, when $I = J = \ast$, the weighted limit and colimit are cotensoring and tensoring by the left action of $\eA$ on $\eC$. 

    For the study of stability, we are interested in the case where a weighted colimit $W\otimes (\blank)$ is also a weighted limit $[W', \blank]$:
    \begin{definition}
        A weight $W: I \slashedrightarrow J$ is \emph{absolute} if the weighted colimit functor $W\otimes(\blank): \Fun(J, \eA)\to \Fun(I, \eA)$ admits a left adjoint in $\RMod_{\eA}(\PrL)$. In this case, the left adjoint is given by $W^L\otimes(\blank)$ for some $W^L: J\slashedrightarrow I$, which we call the \emph{left dual} of $W$. 
    \end{definition}
    In this case, we have an equivalence $W\otimes(\blank)\simeq [W^L, \blank]$. 
    \begin{remark}
        Enriched categories and profunctors are algebras and bimodules in the $2$-category spanned by $\Fun(X, \eA)$ for $X\in \sS$, or the category of enriched quivers in \cite{hinichYonedaLemmaEnriched2020}. In this context, the absoluteness of a weight is equivalent to left dualizability as a bimodule. 
    \end{remark}
    \begin{example}
        A profunctor $\ast\slashedrightarrow \ast$ is given by an object $a\in \eA$. This weight is absolute precisely when $a$ is left dualizable. 
    \end{example}    
    
    \subsection{Directed pushouts and pullbacks}
    The goal of this section is to discuss an example of weighted (co)limits of particular interest: directed pushouts and directed pullbacks\footnote{The traditional names are cocomma and comma constructions. Other popular names are lax pushouts and lax pullbacks. The latter is somewhat more descriptive, but it risks conflicting with the general notion of (fully) lax limits and colimits. Lurie uses the term oriented fiber product for our directed pullback in \cite[\href{https://kerodon.net/tag/01KE}{Subsection 01KE}]{kerodon}.}. We will also address some subtleties around Gray-categories as we encounter them. Let us start with a minimalistic working definition of directed pushouts and pullbacks:
    \begin{definition}
        Let $\eC$ be a presentable category with an action map $\otimes: \Cat\otimes \eC\to \eC$ in $\PrL$, and let $[X, \blank]: \eC\to \eC$ denote the right adjoint to $X\otimes(\blank)$.  
        The \emph{directed pushout} of a span $B\leftarrow A\to C$, denoted by $B\ramalg_A C$, is the colimit of the following diagram in $\eC$:
        \[\begin{tikzcd}[column sep=small]
        &A\ar[rd, "0"]\ar[ld]&&A\ar[ld, "1"']\ar[rd]& \\
        B&&\cube^1\otimes A&& C
        \end{tikzcd}\]
        Dually, the \emph{directed pullback} of a cospan $B\to D\leftarrow C$, denoted by $B\overrightarrow{\times}_D C$, is the limit of the following diagram in $\eC$:
        \[\begin{tikzcd}[column sep=small]
            B\ar[rd] && {[\cube^1, D]}\ar[ld, "\ev_0"']\ar[rd, "\ev_1"]&& C\ar[ld] \\
            &D&&D&
        \end{tikzcd}\]
    \end{definition}
    A slightly more sophisticated definition is given by weighted (co)limits. Note that the directed pushout and pullback naturally fit into a cospan $B\to B\ramalg_A C\leftarrow C$ and a span $B\leftarrow B\overrightarrow{\times}_D C \to C$.
    Let $\eC\in\LMod_{\infty\Cat^\otimes}(\PrL)$. When $J$ is a $1$-category, $\Fun(J, \eC)\in \PrL$ means the functor ($1$-)category to the underlying category $\eC\in\PrL$. Let $J = \Lambda^2_0 = \{1\leftarrow 0\rightarrow 2\}$ be the walking cospan category and let $W_{\bullet\bullet}: J^\op\times J^\op\to \Cat$ be the weight given by the commutative diagram
    \[\begin{tikzcd}
        \cube^0 \ar[r, "="]\ar[d, "="] & \cube^0 \ar[d, "0"] & \emptyset \ar[d]\ar[l] \\
        \cube^0 \ar[r, "0"] & \cube^1 & \cube^0\ar[l, "1"] \\
        \emptyset \ar[u] \ar[r] & \cube^0 \ar[u, "1"] & \cube^0 \ar[l, "="] \ar[u, "="]
    \end{tikzcd}\]
    Note that $\cube^0\xrightarrow{0}\cube^1\xleftarrow{1}\cube^0$ is canonically expressed as the colimit of the following diagram of presheaves over $\Lambda^2_0$:
    \[\cube^0\otimes \yo(1)\leftarrow \cube^0\otimes \yo(0)\xrightarrow{0} \cube^1\otimes \yo(0)\xleftarrow{1}\cube^0\otimes \yo(0)\to \cube^0\otimes \yo(2).\] 
    As $W\mapsto \colim^W(\blank)$ preserves colimits, by Yoneda reduction, one sees that this weight $W_{\bullet\bullet}$ induces the following adjunction between the categories of spans and cospans:
    \[\begin{tikzcd}
        \ramalg: \Fun(J, \eC) \arrow[r, shift left = 1ex, "\colim^W"]
        & \Fun(J^\op, \eC) \arrow[l, shift left = .5ex, "\lim^W"]
        \arrow[l, phantom, shift right = .2ex, "\scriptscriptstyle\boldsymbol{\bot}"] : \overrightarrow{\times}
    \end{tikzcd}\]
    \begin{example}
        Directed pushouts and directed pullbacks specialize to many important constructions.
        \begin{enumerate}
            \item The \emph{suspension} functor $\Sigma: \eC\to \eC_{\ast\ast}$ is defined by $\Sigma X\coloneqq \ast\ramalg_{X} \ast$. When $\eC=\infty\Cat^\otimes$, it agrees with the original definition of the unreduced suspension thanks to \cref{pushout_formula_unreduced_suspension}. When $\eC= \infty\Cat_\ast^\owedge$ (where $X\in\infty\Cat$ acts by $X_+\owedge(\blank)$), we have $\eC = \eC_{\ast\ast}$ and recover the reduced suspension by \cref{suspension_is_tensor_S1} and 
            \[\vS^1\owedge X\simeq \colim(0\leftarrow S^0\xrightarrow{0}\cube^1_+\xleftarrow{1}S^0\rightarrow 0)\owedge X\simeq \colim(0\leftarrow X\xrightarrow{0}\cube^1_+\owedge X\xleftarrow{1} X\to 0).\]
            \item The \emph{hom} functor $\eC_{\ast\ast}\to \eC$ is defined by $(X; x_0, x_1)\mapsto X(x_0, x_1)\coloneqq \{x_0\}\overrightarrow{\times}_X \{x_1\}$. Again, by \cref{pullback_formula_for_hom_category} it recovers the self-enrichment of $\infty\Cat$. 
            \item The \emph{loop} functor $\Omega: \eC_{\ast}\to \eC_{\ast}$ is $(X, x)\mapsto \Omega (X, x) \coloneqq X(x, x)$. It is equivalent to the pointed internal hom $[\vS^1, \blank]$ and recovers the loop functors on $\infty\Cat_\ast$ and $\CatSp$. 
            \item The \emph{lax cofiber} (resp.\ \emph{oplax cofiber}) of $f: X\to Y$ in $\eC$ is $\rcof f \coloneqq \ast\ramalg_X Y$ (resp.\ $\lcof f \coloneqq Y\ramalg_X \ast$).
            \item Dually, the \emph{lax fiber} (resp.\ \emph{oplax fiber}) of a morphism $f: X\to Y\in \eC$ over $\ast\xrightarrow{y} Y$ is $\rfib f\coloneqq \{y\}\overrightarrow{\times}_Y X$ (resp.\ $\lfib f\coloneqq X\overrightarrow{\times}_Y \{y\}$). For $X\in\eC_\ast$, we let $\rpath(X) = \rfib(X\xrightarrow{\id} X)$ and $\lpath(X) = \lfib(X\xrightarrow{\id} X)$.
            \item Additionally, we will use the following notation: $\rcone(X)\coloneqq \rcof(X\xrightarrow{\id} X)$, $\lcone(X)\coloneqq \lcof(X\xrightarrow{\id} X)$, $\cyl(X)\coloneqq X\ramalg_X X\simeq \cube^1\otimes X$, $\rcyl(X\to Y) = X\ramalg_X Y$, and $\lcyl(X\to Y) = Y\ramalg_X X$. 
            Let $I\coloneqq \rcone(S^0)\in \infty\Cat_\ast$; it is the interval $0\to 1$ with the basepoint $0$. Notice that $I^\op = I^\circ = \lcone(S^0)$. If $\eC$ is pointed, we have $\rcone(X)\simeq I\owedge X$ and $\lcone(X)\simeq I^\op\owedge X$. 
        \end{enumerate}
    \end{example}
    For more specific examples, we give a few sample calculations of lax (co)fibers.
    \begin{example}\label{example_sample_calculations_of_lax_fib_cofib}
    \begin{enumerate}
        \item
        $\Sigma^\infty I = \rB^\infty\Free_{\EE_\infty}(I)$ is the symmetric monoidal category freely generated by an object $c$ and a morphism $1\to c$, or, in other words, a single $\EE_0$-algebra. In particular, we have an explicit description $\Free_{\EE_\infty}(I) = \operatorname{Env}(\EE_0) = \Fin^\inj$ (where $\operatorname{Env}$ is the monoidal envelope of \cite[\S 2.2.4]{LurieHA}). The cofiber map $\vS^0\to I$ induces the inclusion $\Fin^\simeq = \operatorname{Env}(\mathrm{Triv})\to\operatorname{Env}(\EE_0) = \Fin^\inj$.
        \item \Cref{lemma_cylinder_suspension_wedge_of_orientals} shows $\lcone(\Ori^n) = \Ori^{n+1}$ strictly, and \cite{gepnerOrientedSimplicialSpacesInpreparation} proves that this pushout is weak. One could define orientals inductively by the cone construction starting from $\Ori^0$. 
        More generally, one can think of $X\ramalg_{X\otimes Y} Y$ as the lax join of $X$ and $Y$. 
        \item Let $\NN^\delta$ denote the discrete commutative monoid of natural numbers, and let $\Fin^\simeq\to \NN^\delta$ be the $0$-truncation map in $\CMon(\sS)$. We (heuristically) compute the lax cofiber $0 \vec{\amalg}_{\Fin^\simeq} \NN^\delta$ in $\infty\SMC$, or equivalently in $\CatSp$. It is a symmetric monoidal category that corepresents a morphism $1\to c$, where $c$ \emph{strictly commutes} with itself with respect to the symmetric monoidal structure.
        As we saw above, without strictness, this is $\Fin^\inj$ with objects $c^{\otimes n}$ for $n\geq 0$, but we force the $\Sigma_n$-action on $c^{\otimes n}$ to be trivial,
        so the hom groupoid $\Hom(c^{\otimes m}, c^{\otimes n})$ is the quotient $\Hom^\inj(\langle m\rangle, \langle n\rangle)/\Sigma_n\simeq \rB\Sigma_{n-m}$.
        \item Next, let us compute the lax fiber $0\overrightarrow{\times}_{\rB\NN^\delta} \rB\Fin^\simeq$. This is equivalent to the pullback of $(\rB\NN^\delta)_{\ast//}\to \rB\NN^{\delta}\leftarrow \rB\Fin^\simeq$, where $\rB\NN^\delta_{\ast//}$ is the fiber of $\ev_0: [\Delta^1, \rB\NN^\delta]\to \rB\NN^\delta$, or the \emph{lax undercategory} of $\vS^1 = \rB\NN^\delta$ under the basepoint. 
        However, $\vS^1$ is just a $1$-category, so the lax undercategory is the same as the undercategory, i.e.,
        $\vS^1_{\ast//} = \vS^1_{\ast/} = \NN^{\leq}$; we get the \emph{ordered set} of natural numbers because $\ast\xrightarrow{n}\ast\in \vS^1_{\ast/}$ factors through another object $\ast\xrightarrow{m}\ast$ if and only if there is $k$ such that $\ast\xrightarrow{m}\ast\xrightarrow{k}\ast = \ast\xrightarrow{n}\ast$, i.e., $m+k=n$, or $m\leq n$. 
        The upshot is that the lax fiber in question is the fiber of $\NN^\leq \to \vS^1 = \rB\NN^\delta \leftarrow \rB\Fin^\simeq$. 
        This can be explicitly computed as a $(2, 1)$-category: the objects are the natural numbers, and the morphisms are given by $\Hom(m, n) = \rB\Sigma_{n-m}$.
        Thus we obtain $\rcof(\rB^\infty\Fin^\simeq \to \rB^\infty\NN^\simeq) \simeq \rfib(\rB^{\infty+1}\Fin^\simeq \to \rB^{\infty+1}\NN^\simeq)$ by computing both sides independently. In the next section, we will show that $\rcof(f) \simeq \rfib(\rB f)$ is generally true in categorical spectra.
    \end{enumerate}
    \end{example}
    We intuitively know that directed pushouts and pullbacks are universal lax squares extending spans and cospans. The next goal is to make this precise; since our target category $\eC$ (e.g., $\CatSp$) is only a Gray-category and not an honest $2$-category, we must be careful about what we even mean by a lax square in $\eC$. 
    \begin{remark}
        Recall the forgetful functor $\iota: \infty\Cat\simeq \infty\Cat^\times \mhy\Cat\to \infty\Cat^\otimes\mhy\Cat$ induced by the lax monoidal functor $\id: \infty\Cat^\times \to \infty\Cat^\otimes$. 
        It does not preserve colimits in general; for instance, the ``long'' hom category of $\iota(C_2\vee C_2)$ is $C_1\times C_1$, whereas that of $(\iota C_2)\vee (\iota C_2)$ is $\cube^2$.
        However, this is essentially the only failure; \cite[Theorem 4.10.7]{gepnerOrientedSimplicialSpacesInpreparation} shows that $\iota$ is a fully faithful right adjoint with essential image characterized as the local objects with respect to the maps
        \[\iota\sigma(C_m)\vee\iota\sigma(C_n)\to \iota(\sigma C_m\vee \sigma C_n), \quad m, n\geq 1.\]
        The following colimit decomposition remains true in $\infty\Cat^\otimes\mhy\Cat$:
        \[\iota \cube^2 = \iota\Delta^{\{00, 10, 11\}}\cup_{\iota\Delta^{\{00, 11\}}=\sigma(\{0\})}\iota\sigma \cube^1\cup_{\iota\sigma(\{1\}) = \Delta^{\{00, 11\}}}\iota\Delta^{\{00, 01, 11\}}. \]
    \end{remark}
    \begin{definition}
        A \emph{lax square} in $\eC$ is a functor $\iota\cube^2\to \eC$ in $\infty\Cat^\otimes\mhy\Cat$.
    \end{definition}
    Using the colimit-decomposition of $\iota\cube^2$, we see that the data of a lax square $\iota\cube^2\to \eC$, which we depict as
    \[\begin{tikzcd}
        A\ar[r]\ar[d] & B \ar[d] \\
        C\ar[r]\ar[ru, Rightarrow, shorten <=2ex, shorten >=2ex, "\alpha"'] & D
    \end{tikzcd}\]
    is functorially equivalent to either of the following commutative diagrams:
    \[\begin{tikzcd}
        & A \ar[r]\ar[d, "i_0"] & B \ar[dd]\\
        A\ar[d]\ar[r, "i_1"] & \cube^1\otimes A\ar[rd, "\alpha"] & \\
        C\ar[rr] & & D, 
    \end{tikzcd}\quad
    \begin{tikzcd}
        A\ar[dd]\ar[rr]\ar[rd, "\alpha"] & & B \ar[d]\\
        &{[\cube^1, D]}\ar[r, "\ev_0"] \ar[d, "\ev_1"] & D\\
        C\ar[r] & D &
    \end{tikzcd}\]
    \begin{remark}
        When we write a diagram with more than one cell of dimension at least $2$, e.g., 
        \[\begin{tikzcd}
            A\ar[r]\ar[d] & B \ar[d]\ar[r] & E \ar[d] \\
            C\ar[r]\ar[ru, Rightarrow, shorten <=2ex, shorten >=2ex, "\alpha"'] & D\ar[r]\ar[ru, Rightarrow, shorten <=2ex, shorten >=2ex, "\beta"']  & F,
        \end{tikzcd}\]
        in a Gray-category $\eC$, it should a priori be interpreted as a functor $\iota \cube^2\sqcup_{\iota \cube^1}\iota \cube^2\to \eC$, not $\iota(\cube^2\sqcup_{\cube^1}\cube^2)\to \eC$ (the latter implicitly asserts the existence of a well-defined $2$-dimensional composite $\beta\ast\alpha$). However, this particular case does not involve horizontal composition beyond dimension $1$, so one sees that the comparison map $\iota \cube^2\sqcup_{\iota \cube^1}\iota \cube^2\to \iota(\cube^2\sqcup_{\cube^1}\cube^2)$ is an equivalence. 
    \end{remark}

    The pasting law for directed pullbacks and pushouts is rather restrictive:
    \begin{proposition}\label{pasting_law_lax_pullbacks}
        In the following diagram, suppose that $\alpha$ is a directed pushout square. 
        \[\begin{tikzcd}
            A\ar[r]\ar[d] & B \ar[d]\ar[r] & E \ar[d] \\
            C\ar[r]\ar[d]\ar[ru, Rightarrow, shorten <=2ex, shorten >=2ex, "\alpha"'] & D\ar[r]\ar[d]\ar[ru, Rightarrow, shorten <=2ex, shorten >=2ex, "\beta"', "\simeq"]  & F\\
            G \ar[ru, Rightarrow, shorten <=2ex, shorten >=2ex, "\gamma"', "\simeq"]\ar[r] & H &         
        \end{tikzcd}\]
        Suppose moreover that $\beta, \gamma$ are invertible. Then
        \begin{enumerate}
            \item $\beta\ast \alpha$ is a directed pushout square if and only if $\beta$ is a pushout square.
            \item $\gamma\ast \alpha$ is a directed pushout square if and only if $\gamma$ is a pushout square.
        \end{enumerate}
        Similar assertions hold for directed pullback and pullback squares. 
    \end{proposition}
    \begin{proof}
        This follows from the pasting law for pullback squares together with the fact that the colimit of $C\leftarrow A\to \cube^1\otimes A\leftarrow A\to B$ can be computed by forming two pushout squares.
    \end{proof}
    Notice that even when $\eC = \infty\Cat$, the formation of directed pullback and directed pushout squares is not $2$-functorial. However, the half-central structure of $\vS^1$ gives rise to an exceptional functoriality of suspensions. Note that when the underlying category of $\eC$ is pointed (so $\eC\simeq \sS_\ast\otimes \eC$), the action of $\infty\Cat$ canonically factors through $\infty\Cat_\ast$ by $X\otimes (\blank) = X_{+}\owedge (\blank)$. 
    \begin{proposition}\label{half-2-functoriality_of_suspension}
        Let $\eC$ be an object of $\LMod_{\infty\Cat_{\ast}}(\PrL)$.
        The suspension functor $\Sigma: \eC\to \eC$ induces $\Map(\iota\cube^2, \eC)\to \Map(\iota(\cube^2)^{2\mhy\op}, \eC)$, which is depicted by
        \[\left(\begin{tikzcd}
            A\ar[r]\ar[d] & B \ar[d] \\
            C\ar[r]\ar[ru, Rightarrow, shorten <=2ex, shorten >=2ex, "\alpha"'] & D            
        \end{tikzcd}\right) \mapsto \left(\begin{tikzcd}
            \Sigma A\ar[r]\ar[d] & \Sigma B\ar[ld, Rightarrow, shorten <=2ex, shorten >=2ex, "\Sigma\alpha"'] \ar[d] \\
            \Sigma C\ar[r] & \Sigma D
        \end{tikzcd}\right)\]
    \end{proposition}
    \begin{proof}
        Using the half-central structure of $\vS^1$, we have $\Sigma(\cube^1_+\owedge A)\simeq (\vS^1\owedge \cube^1_+)\owedge A\xrightarrow{\sim} ((\cube^1)^{\op}_+\owedge \vS^1)\owedge A$. This flips the weight $\ast_+\xrightarrow{0}\cube^1_+\xleftarrow{1}\ast_+$ to $\ast_+\xrightarrow{1}\cube^1_+\xleftarrow{0}\ast_+$. 
    \end{proof}
    The proposition is already reflected in the classical definition of triangulated categories. Namely, this is the negative sign introduced when we rotate the triangle $A\xrightarrow{f} B\xrightarrow{g} C\xrightarrow{h} \Sigma A$ to $B\xrightarrow{g}C\xrightarrow{h}\Sigma A\xrightarrow{-\Sigma f} \Sigma B$; the suspension converts a lax fiber sequence to an oplax fiber sequence, so one flips the direction of the homotopy back to a lax fiber sequence by negating the maps. Since we do not have negatives in our categorical setting, we must distinguish lax and oplax fiber sequences. 

\section{Absoluteness of directed pushouts}\label{section_absoluteness_of_directed_pushouts}
The goal of this section is to prove the absoluteness of directed pushouts and begin the study of absolute colimits in $\CatSp$. Our strategy is to guess the dual weight and exhibit the adjunction of weighted colimits directly.
We will first prove the case of the cone construction, i.e., $0\ramalg_X X = I\owedge X$. In this case, absoluteness is equivalent to the dualizability of $\Sigma^\infty I$. Even though this is a special case of the main theorem, verification of the dualizability of $I$ nicely packages part of its proof, so we give a separate treatment. 
After the main theorem, we will list some immediate consequences. A corollary of particular importance is the equivalence $\rcof{f}\simeq \rfib{\Sigma f}$ for a morphism of categorical spectra $f: X\to Y$ (as observed in a particular case in \cref{example_sample_calculations_of_lax_fib_cofib}). Curious as it is, indicating both similarities and differences from stable $1$-categories, it is also useful for computations; we will apply this to the study of TQFTs in the next chapter. 

\subsection{\texorpdfstring{$\Sigma^{\infty} I$}{The suspension spectrum on I} is dualizable}
Recall the notation $I\coloneqq \rcof(S^0\to S^0)$. In this section, we will show that $\Sigma^\infty I$ is dualizable in categorical spectra. 
    To guess the dual object $(\Sigma^\infty I)^L$, recall that we hoped to show that the cone and the path category are equivalent up to a shift: $\rcone(X)\simeq \rpath(X[1])$. This is equivalent to $I\otimes X\simeq [I, \vS^1\owedge X]\simeq [\Sigma^{\infty-1} I, X]$, which suggests the following:
    \begin{proposition}
        $\Sigma^{\infty-1} I$ is the left dual of $\Sigma^\infty I$ in the monoidal category $\CatSp^\otimes$. 
    \end{proposition}
    \begin{remark}
        The proposition should eventually be understood as an example of \emph{categorical Atiyah duality} and the degree shift $1$ accounts for the dimension of the interval. Observe that the proof even resembles the usual construction of Spanier-Whitehead duality by embedding a space into a sphere and contracting the neighborhood, although it is unclear how this idea generalizes. This is the subject of an ongoing extension of this project.
    \end{remark}
    \begin{remark}
        Since the localization $\CatSp\to \Sp$ is monoidal, any monoidal duality in $\CatSp$ gives rise to one in $\Sp$. The duality of the proposition localizes to the trivial duality $0\dashv 0$. Note that the opposite direction for $\Sp\hookrightarrow\CatSp$ is false, as it does not preserve the unit object.
    \end{remark}
    \begin{remark}\label{remark_duals_of_I}
        Suppose $L$ is the left dual of $R$ in $\CatSp$ (so that $L\otimes (\blank)\dashv [L, \blank]\simeq R\otimes (\blank)$). By composing the adjoint inverses $[1]$ and $[-1]$, one sees $\Sigma^{\pm 1} L^\circ\dashv \Sigma^{\mp 1} R$ and $\Sigma^{\pm 1}L\dashv \Sigma^{\mp 1}R^\circ$. In particular, we have the four-periodic cycle 
        \[\Sigma^{\infty} I\dashv \Sigma^{\infty-1} I^\op\dashv \Sigma^{\infty}I^\op\dashv \Sigma^{\infty-1} I\dashv \Sigma^{\infty} I, \]
        so $\Sigma^{\infty} I$ is also right-dualizable.
    \end{remark}
    \begin{remark}
        Another clue for the dual object comes from Steiner's theory: $\lambda I$ is the augmented chain complex $\ZZ\underline{e}\to \ZZ\underline{0}\oplus \ZZ\underline{1}$ with basepoint $\underline{0}$. The dual complex of the reduced complex is $\ZZ\underline{1}^\vee\to \ZZ\underline{e}^\vee$ in degrees $[-1, 0]$. The dexterity of the dual determines the natural positive part, giving the desuspension of the reduced augmented directed complex of $\Sigma^{\infty} I^{(\op)}$. It seems possible to have a nonconnective version of Steiner theory relating pointed augmented complexes to a reasonably strict part of categorical spectra. Part of the difficulty would be that our version of suspension is essentially not strict, let alone loop-free. In analogy with the case of spectra, it should be rectified by passing to strictly commutative monoid objects, or $\rH\NN$-modules. The strictification process also seems likely to have a Dold--Thom style interpretation. 
    \end{remark}
    \begin{proof}
        Let us first construct the unit and counit of the duality.
        \begin{enumerate}
            \item To define the counit $\eps: \Sigma^{\infty-1} I\otimes \Sigma^{\infty} I\to \FF$, it suffices to define $\vS^1\otimes \eps: \Sigma^\infty (I\owedge I)\to \FF[1]$. We define it as $\Sigma^{\infty}$ applied to the map $I\owedge I\to \vS^1\simeq \rB \NN$ depicted as follows:
                \[\begin{tikzcd}
                    00 \ar[r, equal]\ar[d, equal] & 10\ar[d]\ar[ld, Rightarrow, shorten <=2ex, shorten >=2ex] \\
                    01\ar[r] & 11
                \end{tikzcd}\mapsto 
                \begin{tikzcd}
                \ast \ar[r, equal, "0"]\ar[d, equal, "0"] & \ast\ar[d, "1"]\ar[ld, equal, shorten <=2ex, shorten >=2ex] \\
                \ast\ar[r, "1"] & \ast.
                \end{tikzcd}\]
            The diagram shows the image of the atomic cells of $\cube^2$ under the composition $\cube^2\twoheadrightarrow I\owedge I\to \rB\NN$, which descends to the quotient $I\owedge I$ because the restriction to $\{0\}\otimes \cube^1\cup \cube^1\otimes \{0\}$ is trivial. 
            \item To define the unit map 
            $\eta: \FF\to \Sigma^\infty I\otimes \Sigma^{\infty-1}I = \Sigma^{\infty}I\otimes \vS^{-1}\otimes \Sigma^{\infty}I$, we use the half-central structure on $\vS^1$ to pull out the desuspension to the left. Namely, we define $\vS^1\otimes \eta$ so that
            \[\vS^1\otimes \FF\xrightarrow{\vS^1\otimes \eta} \vS^1\otimes \Sigma^\infty I \otimes \vS^{-1}\otimes \Sigma^\infty I\xrightarrow[\sim]{\tau_I\otimes \id} \Sigma^\infty I^\op\otimes \vS^1\otimes\vS^{-1}\otimes \Sigma^\infty I\simeq \Sigma^\infty (I^\op\owedge I)\] is induced by the map $\vS^1\to I^\op\owedge I$, $l\mapsto rs$, where $l$ is the generating loop of $\vS^1$ and $r, s$ are the $1$-cells depicted as follows (we identify $I^\op$ with the interval $0\to 1$ with the vertex $1$ marked, so $I^\op\owedge I$ is the quotient of $\cube^2$ by $\cube^1\otimes \{0\}\cup \{1\}\otimes\cube^1$):
            \[I^{\op}\owedge I = \begin{tikzcd}
                00 \ar[r, equal]\ar[d, "s"] & 10\ar[d, equal]\ar[ld, Rightarrow, shorten <=2ex, shorten >=2ex] \\
                01\ar[r, "r"] & 11.
            \end{tikzcd}\]
            We could have defined $\eta$ by pulling out the desuspension to the \emph{right}. The coherence data of the half-central structure of $\vS^1$ verifies that the two possible definitions are the same. To see this, notice that the following diagram canonically commutes:
            \[\adjustbox{scale=0.8, center}{
            \begin{tikzcd}
            \vS^1\otimes \FF \ar[r, "\vS^1\otimes \eta"] \ar[d, equal, "\tau_{\FF}"]
                & \vS^1\otimes \Sigma^\infty I\otimes \vS^{-1}\otimes \Sigma^\infty I \ar[r, "\tau_I"] \ar[d, "\tau_{I\otimes \vS^{-1}\otimes I}"]
                & \Sigma^{\infty} I^\op\otimes \vS^1\otimes \vS^{-1}\otimes \Sigma^\infty I\ar[r]\ar[d, "\tau_{\vS^{-1}}"]
                & \Sigma^{\infty} (I^\op\owedge I)\ar[d, equal]\\
                \FF^\circ\otimes \vS^1 \ar[r, "\eta^\circ\otimes\vS^1"]
                & \Sigma^\infty I^\op\otimes \vS^{-1}\otimes \Sigma^{\infty}I^\op\otimes \vS^1 \ar[r, "\tau_I^{-1}"]
                & \Sigma^\infty I^\op\otimes \vS^{-1}\otimes \vS^1\otimes \Sigma^\infty I \ar[r]
                & \Sigma^\infty(I^\op\owedge I).
            \end{tikzcd}}\]
            Taking the total dual of the bottom composite, we see that $\eta$ is also characterized by the fact that 
            \[\FF\otimes \vS^1\xrightarrow{\eta\otimes\vS^1}\Sigma^\infty I\otimes \vS^{-1}\otimes \Sigma^\infty I\otimes \vS^1\xrightarrow{} \Sigma^{\infty}I\otimes \vS^{-1}\otimes \vS^1\otimes \Sigma^{\infty} I^\op\to \Sigma^{\infty}(I\owedge I^\op)\] 
            is $\Sigma^\infty$ applied to the loop $r's':\vS^1\to I\owedge I^\op$, where $r', s'$ are the $1$-cells depicted as
            \[I \owedge I^\op = \begin{tikzcd}
                00 \ar[r, "{s'}"]\ar[d, equal] & 10\ar[d, "{r'}"]\ar[ld, Rightarrow, shorten <=2ex, shorten >=2ex] \\
                01\ar[r, equal] & 11.
            \end{tikzcd}\]
        \end{enumerate}
        Now we check the triangle identities, i.e., that the following compositions are equivalent to the identities:
        \begin{enumerate}
            \item $\Sigma^\infty I \xrightarrow{\eta \otimes \Sigma^\infty I} \Sigma^\infty I\otimes \Sigma^{\infty-1} I\otimes \Sigma^\infty I\xrightarrow{\Sigma^\infty I\otimes \eps}\Sigma^{\infty} I$,
            \item $\Sigma^{\infty-1} I\xrightarrow{\Sigma^{\infty-1} I\otimes \eta} \Sigma^{\infty-1}I\otimes \Sigma^{\infty} I\otimes \Sigma^{\infty-1} I\xrightarrow{\eps\otimes \Sigma^{\infty-1} I}\Sigma^{\infty-1}I$.
        \end{enumerate}
        For (1), consider the following diagram (we omit some $\Sigma^\infty$'s to save space):
        \[\begin{tikzcd}
            \vS^1\otimes \Sigma^\infty I \ar[r, "\vS^1\otimes \eta\otimes I"]\ar[rd, equal] & \vS^1\otimes \Sigma^\infty I\otimes \Sigma^{\infty-1} I\otimes \Sigma^{\infty} I \ar[r, "\tau_I\otimes \id", "\sim"']\ar[d, "\vS^1\otimes I\otimes \eps"] & \Sigma^\infty I^\op\otimes \Sigma^\infty I\otimes \Sigma^\infty I \ar[d, "I^\op\otimes \vS^1\otimes \eps"]\\
            & \vS^1\otimes \Sigma^{\infty} I \ar[r, "\tau_I", "\sim"'] & \Sigma^\infty I^\op\otimes \vS^1
        \end{tikzcd}\]
        The right square commutes, so the commutativity of the left triangle reduces to that of the outer compositions. By definition, these are $\Sigma^\infty$ applied to unstable maps $\vS^1\owedge I\to I^\op\owedge I\owedge I \to I^\op\owedge \vS^1$, which in turn are induced from maps of cubes $\cube^2\to \cube^3\to \cube^2$ by passing to quotients. Therefore, we may compute the compositions by tracing the assignments of the relevant atomic cells of the cubes:
        \[\begin{tikzcd}
            00 \ar[r, equal]\ar[d, equal] & 01 \ar[d, "l"] \\
            10 \ar[r, equal]\ar[ru, Rightarrow, shorten <=2ex, shorten >=2ex, "\theta"] & 11
        \end{tikzcd}\quad\to\quad \begin{tikzcd}[row sep=1em]
            & 001 \arrow[dd, "s"] \arrow[rd, equal] & \\
           000 \arrow[dd, equal] \arrow[ru, equal] & & 101 \arrow[dd, equal] \arrow[ld, Rightarrow, shorten <=1ex, shorten >=1ex, "\gamma"] \\
            & 011 \arrow[rd, "r"] & \\
           010 \arrow[rd, equal] \arrow[ru, "t"] \arrow[ruuu, Rightarrow, shorten <=6ex, shorten >=6ex, "\beta"] & & 111 \\
            & 110 \arrow[ru, equal] \arrow[uu, Rightarrow, shorten <=2ex, shorten >=2ex, "\alpha"] & 
           \end{tikzcd}\quad\to\quad\begin{tikzcd}
            00 \ar[r, "l'"]\ar[d, equal] & 01 \ar[d, equal] \\
            10 \ar[r, equal]\ar[ru, Rightarrow, shorten <=2ex, shorten >=2ex, "\theta'"] & 11
        \end{tikzcd}\]
        The first map assigns $l\mapsto rs$ and $\theta\mapsto \alpha\ast\beta$ and the second map assigns $s, t\mapsto l'$, $r\mapsto \id_\ast$, $\beta\mapsto \id_{l'}$ and $\alpha, \gamma\mapsto \theta'$. The composition therefore assigns $l\mapsto l'$, $\theta\mapsto \theta'$, so it is equal to $\tau_I: \vS^1\owedge I\to I\owedge I^{\op}\owedge \vS^1$. 
        Verification of (2) is similar. Tensoring $\vS^1$ from both the left and the right, (2) is equivalent to the identity if and only if the outer compositions of the following diagram commute:
        \[\begin{tikzcd}[column sep=large]
            \Sigma^{\infty} I\otimes \vS^1 \ar[r, "{I\otimes \eta\otimes \vS^1}"]\ar[rd, equal]
            & \Sigma^{\infty}I\otimes \Sigma^\infty I\otimes \Sigma^{\infty-1}I\otimes \vS^1 \ar[r, "I\otimes I\otimes \tau^{-1}_I", "\sim"']\ar[d, "\vS^1\otimes \eps\otimes \Sigma^{\infty-1}I\otimes \vS^1"]
            & \Sigma^\infty I\otimes \Sigma^\infty I\otimes \Sigma^\infty I^\op\ar[d, "\vS^1\otimes \eps\otimes\id"] \\
            & \Sigma^{\infty}I\otimes \vS^1 \ar[r, "\tau_I^{-1}", "\sim"']
            & \vS^1\otimes \Sigma^\infty I^\op.
        \end{tikzcd}\]
        Using the second description of $\eta$, the top-right composition is  $\Sigma^\infty$ applied to the unstable map $I\owedge \vS^1\to I\owedge I\owedge I^\op\to \vS^1\owedge I^\op$ induced from $\cube^2\to \cube^3\to \cube^2$, depicted as:
        \[\begin{tikzcd}
            00 \ar[r, equal]\ar[d, equal] & 01 \ar[d, equal] \\
            10 \ar[r, "l"]\ar[ru, Rightarrow, shorten <=2ex, shorten >=2ex, "\theta"] & 11
        \end{tikzcd}\quad\to\quad \begin{tikzcd}[row sep=1em]
            & 010 \arrow[dd, "t"] \arrow[rd, equal] & \\
           000 \arrow[dd, equal] \arrow[ru, equal] & & 011 \arrow[dd, equal] \\
            & 110 \arrow[rd, "r"]\arrow[ru, Rightarrow, shorten <=1ex, shorten >=1ex, "\beta"]\arrow[dd, Rightarrow, shorten <=2ex, shorten >=2ex, "\gamma"] & \\
           100 \arrow[rd, equal] \arrow[ru, "s"] \arrow[ruuu, Rightarrow, shorten <=6ex, shorten >=6ex, "\alpha"] & & 111 \\
            & 101 \arrow[ru, equal] & 
           \end{tikzcd}\quad\to\quad\begin{tikzcd}
            00 \ar[r, equal]\ar[d, "l'"] & 01 \ar[d, equal] \\
            10 \ar[r, equal]\ar[ru, Rightarrow, shorten <=2ex, shorten >=2ex, "\theta'"] & 11
        \end{tikzcd}.\]
        The assignments of relevant cells are $l\mapsto rs$, $\theta\mapsto \alpha\ast\beta$, $t, s\mapsto l'$, $r\mapsto \id_\ast$, $\alpha\mapsto \id_{l'}$, and $\beta, \gamma\mapsto \theta'$, so they compose to $\tau_I^{-1}: l\mapsto l', \theta\mapsto \theta'$.
    \end{proof}
    
    \subsection{Directed pushouts are absolute}
    As before, we let $J = \Lambda^2_0 = (1\leftarrow 0 \rightarrow 2)$ be the walking cospan category and let $W: (\Lambda^2_0)^\op\to \infty\Cat_\ast$ be the weight $S^0 \rightarrow \cube^1_+\leftarrow S^0$ for directed pushouts. Our goal is to show the absoluteness of directed pushouts
    $\ramalg = \colim^W: \Fun(J, \CatSp)\to \CatSp$.
    Let us start by guessing the dual weight by assuming it has a left adjoint as a right $\CatSp$-modules. The following lemma is useful:
    \begin{lemma}
        Let $J$ be a $(1, 1)$-category enriched in finite sets. Then there is an adjunction \[\begin{tikzcd}[column sep=large]
            \Fun(J, \CatSp) \arrow[r, shift left = 1ex, "\ev_i"]
            &  \CatSp \arrow[l, shift left = .5ex, "\yo(i)\otimes (\blank)"]
            \arrow[l, phantom, shift right = .2ex, "\scriptscriptstyle\boldsymbol{\bot}"]
        \end{tikzcd}\] in the category $\RMod_{\CatSp}(\PrL)$, or even in $\BMod{\CatSp}(\PrL)$. 
    \end{lemma}
    \begin{proof}
        In general, the evaluation functor $\ev_i: \Fun(J, \CMon)\to \CMon$ admits a right adjoint given by the right Kan extension $X\mapsto (j\mapsto [\Map_J(j, i), X])$. If $\Map_J(j, i)$ are finite sets, we have $[\Map_J(j, i), X]\simeq \Map_J(j, i)\otimes X$ by semiadditivity (with the opposite functoriality in $\Map_J(j, i)$ via transposing matrices), so the right adjoint is colimit-preserving. The adjunction of the lemma is obtained by tensoring $\CatSp$ from the right (note that $\CatSp$ is semiadditive, so it is an algebra over $\CMon$).
    \end{proof}
    Therefore, if there is a morphism $L: \CatSp\to \Fun(J, \CatSp)$ of right $\CatSp$-modules that is left adjoint to a functor $\colim^W: \Fun(J, \CatSp)\to\CatSp$, by composing with the adjunction in the above lemma, we see that $\ev_i\circ L$ must be given by tensoring with the left dual of $\colim^W \yo(i)$. 
    Specializing to $J = \Lambda^2_0$ and $W = S^0\xrightarrow{0}\cube^1_+\xleftarrow{1} S^0$, we see that $\ev_0\circ L = (0\ramalg_\FF 0)^L \simeq \FF[-1]$, $\ev_1\circ L = (\FF\ramalg_{\FF} 0)^L = \Sigma^{\infty-1} I^\op$, and $\ev_2 \circ L = (0\ramalg_{\FF} \FF)^L\simeq \Sigma^{\infty-1} I$ (see \cref{remark_duals_of_I}). 
    Therefore, in order to prove the absoluteness of directed pushouts, we must prove the following:
    \begin{theorem}\label{theorem_directed_pushouts_are_absolute}
        There is an adjunction \[\begin{tikzcd}[column sep=35mm]
            \CatSp \arrow[r, shift left = 1ex, "{\Sigma^{\infty-1}(I^\op\leftarrow S^0\rightarrow I)\otimes (\blank)}"]
            &  \Fun(\Lambda^2_0, \CatSp) \arrow[l, shift left = .5ex, "{\colim^W}"]
            \arrow[l, phantom, shift right = .2ex, "\scriptscriptstyle\boldsymbol{\bot}"]
        \end{tikzcd}.\]
        in $\RMod_{\CatSp}(\PrL)$. 
        In particular, the directed pushout $(X\leftarrow Y\rightarrow Z)\mapsto X\ramalg_Y Z$ is absolute. 
    \end{theorem}
    \begin{remark}
        One can also guess the right dual weight in a similar but easier manner. Namely, we already know that $X\mapsto [W(\blank), X]$ is the right adjoint to $\colim^W$, even though we do not know whether it lies in $\RMod_{\CatSp}(\PrL)$. Since the (co)presheaf categories are free modules on (co)representable functors, we can always cook up the best approximation of a functor by a right $\CatSp$-module morphism. In our case, it is simply $X\mapsto [W(\blank), \FF]\otimes X$, so we only need to know the right dual of each component of the weight. In our case, we will see that the right dual of $\Sigma^\infty_+\cube^1$ is $\Sigma^{\infty-1}(I\cup_{S^0} I^\op)$. 
    \end{remark}
    \begin{remark}
        It seems reasonable that the absoluteness follows formally from the above consideration (because, roughly speaking, the right adjoint of the prospective left adjoint is an absolute left Kan extension and should be computed pointwise), more generally for colimits over a finite posets with weights whose colimits of corepresentable presheaves are dualizable. We will not pursue this idea here and instead give a direct proof. 
    \end{remark}
    \begin{proof}
        We exhibit the adjunction by spelling out the unit and counit and checking the triangle identities. We will often omit $\Sigma^\infty$ in the following. Also beware of the frequent switch between $\ramalg$ and $\lamalg$ under $\Sigma$.
        \begin{enumerate}
            \item To define the unit $\eta: \id_\CatSp \to \colim^W(\Sigma^{-1}(I^\op\leftarrow S^0 \rightarrow I)\otimes(\blank))$, it suffices to define $\eta_{\FF}$; for general $X$, we must define $\eta_X = \eta_{\FF}\otimes X$. Note that the codomain is computed as
            \[\Sigma^{-1}I^\op\ramalg_{\Sigma^{-1}S^0} \Sigma^{-1}I \simeq \Sigma^{\infty-1}(I \ramalg_{S^0} I^\op),\]
            where the category $I\ramalg_{S^0} I^\op \simeq I\amalg_{S^0} \cube^1_+\amalg_{S^0} I^\op\in \Cat_\ast$ is the free category on the graph
            \[\begin{tikzcd}[column sep=2mm, row sep=3mm]
                0 \arrow[rr, bend left, "\cube^1"] && 1 \arrow[ldd, bend left, "{I^\op}"] \\ && \\
                & \ast \arrow[luu, bend left, "{I}"] & 
            \end{tikzcd}.\]
            We define $\eta_\FF: \FF\to \Sigma^{\infty-1}(I\ramalg_{S^0}I^\op)$ as the map classifying the loop $\ast\to 0\to 1\to \ast: \vS^1\to I\ramalg_{S^0}I^\op$. 
            \item Let $X\leftarrow Y\to Z\in \Fun(\Lambda^2_0, \CatSp)$. We define the counit map $\eps_{X\leftarrow Y\to Z}$ as the vertical composite of the following diagram:
            \[\begin{tikzcd}
                \Sigma^{-1} I^\op\otimes (X\ramalg_Y Z) \ar[d] & \Sigma^{-1}\FF \otimes (X\ramalg_Y Z) \ar[l]\ar[r]\ar[d] &
                \Sigma^{-1} I\otimes (X\ramalg_Y Z)\ar[d]\\
                \Sigma^{-1} I^\op\otimes (X\ramalg_X 0)\ar[d, "\simeq"] & \Sigma^{-1} (0\ramalg_Y 0) \ar[l]\ar[r] \ar[d, "\simeq"] &
                \Sigma^{-1} I\otimes (0\ramalg_Z Z)\ar[d, "\simeq"]\\
                \Sigma^{-1} (I^\op\owedge I^\op)\otimes X \ar[d ,"\eps_{I^\op}"] & Y \ar[l]\ar[r]\ar[d, equal] & \Sigma^{-1}(I\owedge I) \otimes Z \ar[d, "\eps_I"] \\
                X & Y\ar[l]\ar[r] & Z.
            \end{tikzcd}\]
            \item Now we must check the triangle identities. Since both our prospective left and right adjoints are morphisms of right $\CatSp$-modules, it suffices to check them on generators.
            \begin{enumerate}
                \item One of the triangle identities, evaluated at $\FF\in\CatSp$, says that after a suspension, the vertical compositions of the following diagram are the identities: 
                \[\adjustbox{center}{\begin{tikzcd}[column sep=6mm]
                    I^\op\ar[d]  &\FF\ar[l]\ar[r]\ar[d] &I\ar[d]\\
                    I^\op\otimes (\Sigma^{-1}I^\op\ramalg_{\Sigma^{-1}\FF}\Sigma^{-1}I)\ar[d]&(\Sigma^{-1}I^\op\ramalg_{\Sigma^{-1}\FF}\Sigma^{-1}I)\ar[l]\ar[r]\ar[d] &I\otimes (\Sigma^{-1}I^\op\ramalg_{\Sigma^{-1}\FF}\Sigma^{-1}I)\ar[d] \\
                    I^\op \otimes (\Sigma^{-1}I^\op\ramalg_{\Sigma^{-1}I^{\op}}0)\ar[d]& (0\ramalg_{\Sigma^{-1}\FF} 0)\ar[l]\ar[r]\ar[d] &I\otimes (0\ramalg_{\Sigma^{-1}I}\Sigma^{-1}I)\ar[d] \\
                    I^\op\otimes I^\op\otimes (\Sigma^{-1} I^\op)\ar[d] &\FF\ar[l]\ar[r]\ar[d] & I\otimes I\otimes(\Sigma^{-1}I)\ar[d] \\
                    I^\op&\FF\ar[l]\ar[r] &I
                \end{tikzcd}}\]
                After another suspension and canceling the desuspension using the half-central structure of $\vS^1$, the compositions can be computed in $\infty\Cat_\ast$:
                \[\begin{tikzcd}
                    \vS^1\owedge I^\op\ar[d]  &\vS^1\ar[l]\ar[r]\ar[d] &\vS^1\owedge I\ar[d]\\
                    I\owedge (I\ramalg_{S^0}I^\op)\ar[d] & (I \ramalg_{S^0}I^\op)\ar[l]\ar[r]\ar[d] &I^\op\owedge (I\ramalg_{S^0}I^\op)\ar[d] \\
                    I \owedge (0\ramalg_{I^{\op}}I^{\op})\ar[d]& (0\ramalg_{S^0} 0)\ar[l]\ar[r]\ar[d] &I^\op\owedge (I \ramalg_{I} 0)\ar[d] \\
                    I\owedge I\owedge I^\op\ar[d] &\vS^1 \ar[l]\ar[r]\ar[d] & I^\op\owedge I^\op\owedge I\ar[d] \\
                    \vS^1 \owedge I^\op&\vS^1\ar[l]\ar[r] &\vS^1\owedge I.
                \end{tikzcd}\]            
                The middle column unpacks to $\vS^1\to (I\ramalg_{S^0}I^\op)\simeq I\amalg_{S^0}\cube^1_+\amalg_{S^0} I^\op\to \ast\amalg_{S^0}\cube^1_+\amalg_{S^0}\ast\simeq \vS^1$, which is evidently the identity. The left and right columns are equivalent to the identity by the triangle identities for the duality of $I$ and $I^\op$. 
                \item The other triangle identity is that for any $X\leftarrow Y\to Z$, the following composition is the identity: 
                \[X\ramalg_Y Z\to \Sigma^{-1}(I \ramalg_{S^0} I^\op)\otimes (X\ramalg_Y Z) \to X\ramalg_Y Z, \]
                where the second map is the $W$-weighted colimit of the diagram in (2). Since $\Fun(\Lambda^2_0, \CatSp)$ is generated by $0\leftarrow 0\rightarrow \FF$, $\FF\leftarrow 0\rightarrow 0$, and $\FF\leftarrow\FF\rightarrow\FF$ as a right $\CatSp$-module, it suffices to check the identity for these three diagrams. The first case unpacks (after a suspension) to the statement that the composition
                \[\vS^1\to I\amalg_{S^0}\cube^1_+\amalg_{S^0}I^\op\to (I\owedge I)\amalg_{\ast} \ast\amalg_{\ast} \ast\simeq I\owedge I\to \vS^1,\]
                or equivalently $\vS^1\to I\amalg_{S^0}\cube^1_+\amalg_{S^0} I^\op\to I\amalg_{S^0} \ast \amalg_\ast \ast\simeq \vS^1$ is the identity, which is clear, and the second case is similar. The third case unpacks, after a suspension, to the statement that the following composition is the identity:
                \[\vS^1\owedge \cube^1_+\to (I\amalg_{S^0}\cube^1_+\amalg_{S^0}I^\op)\owedge \cube^1_+ \to\cube^1_+\owedge \vS^1\xrightarrow{\tau} \vS^1\owedge \cube^1_+. \]
                We can compute these maps explicitly by presenting them as (weak) quotients of grid-shaped (gaunt) categories (vertical and horizontal directions are the first and second tensor component, respectively):
                \[\left(\begin{tikzcd}                
                    \ast\ar[r, equal]\ar[d, "l_0"] & \ast \ar[d, "l_1"]\\
                    \ast \ar[r, equal]\ar[ru, Rightarrow, shorten <=2ex, shorten >=2ex, "\alpha"'] & \ast
                \end{tikzcd}\right)\quad \to \quad
                \left(\begin{tikzcd}
                    \ast\ar[r, equal]\ar[d, "a_0"] & \ast\ar[d, "a_1"]\\
                    00\ar[r, "d"']\ar[d, "b_0"]\ar[ru, Rightarrow, shorten <=2ex, shorten >=2ex, "\beta"'] & 01\ar[d, "b_1"]\\
                    10\ar[r, "e"']\ar[d, "c_0"]\ar[ru, Rightarrow, shorten <=2ex, shorten >=2ex, "\gamma"'] & 11\ar[d, "c_1"]\\
                    \ast\ar[r, equal]\ar[ru, Rightarrow, shorten <=2ex, shorten >=2ex, "\delta"'] & \ast
                \end{tikzcd}\right)\quad\to\quad 
                \left(\begin{tikzcd}
                    \ast \ar[r, "l_1'"]\ar[d, equal] & \ast\ar[d, equal]\\
                    \ast \ar[r, "l_0'"']\ar[ru, Rightarrow, shorten <=2ex, shorten >=2ex, "\alpha'"] & \ast
                \end{tikzcd}\right).\]
                The first map is $l_i\mapsto c_i\circ b_i\circ a_i$ ($i=0, 1$) and $\alpha\mapsto \delta\ast\gamma\ast\beta$. The second map is ($a_0, b_0, b_1, c_1\mapsto \id_\ast$), ($a_1, d\mapsto l_1'$), ($e, c_0\mapsto l_0'$), ($\beta\mapsto \id_{l_1'}$), ($\gamma\mapsto \alpha'$), ($\delta\mapsto \id_{l_0'}$), so these two maps and the half-central structure map compose to the identity. 
            \end{enumerate}
        \end{enumerate}
    \end{proof}
    Let us now spell out some special cases of the theorem:
    \begin{corollary}
        There are adjunctions
        \[\begin{tikzcd}[column sep=30mm]
            \CatSp \arrow[r, shift left = 1ex, "{\Sigma^{\infty-1}(S^0\to I)\otimes (\blank)}"]
            &  \Fun(\cube^1, \CatSp) \arrow[l, shift left = .5ex, "\rcof"]
            \arrow[l, phantom, shift right = .2ex, "\scriptscriptstyle\boldsymbol{\bot}"]
        \end{tikzcd}, and\]
        \[\begin{tikzcd}[column sep=40mm]
            \CatSp \arrow[r, shift left = 1ex, "{\Sigma^{\infty-1}(I^\op\to I\amalg_{S^0} I^\op)\otimes (\blank)}"]
            &  \Fun(\cube^1, \CatSp) \arrow[l, shift left = .5ex, "\rcyl"]
            \arrow[l, phantom, shift right = .2ex, "\scriptscriptstyle\boldsymbol{\bot}"]
        \end{tikzcd}.\]
        In particular, lax cofibers and lax cylinders are absolute colimits. 
    \end{corollary}
    \begin{proof}
        The inclusion $i: \cube^1 \simeq \{0\to 2\}\hookrightarrow \{1\leftarrow 0\to 2\} = \Lambda^2_0$ induces the following adjunction quadruple:
        \[i_!^L \dashv i_!\dashv i^\ast\dashv i_\ast :\begin{tikzcd}
            \Fun(\Lambda^2_0, \CatSp) \arrow[r, shift left = 1.6ex] \arrow[r]
            & \Fun(\cube^1, \CatSp). \arrow[l, shift left=.8ex] \arrow[l, shift right=.8ex]
        \end{tikzcd}\]
        The leftmost adjoint $i_!^L$ is the cobase-change functor $(X\leftarrow Y\rightarrow Z)\mapsto (X\to X\amalg_Y Z)$; the other three are restriction and left and right Kan extension. The left Kan extension takes $X\to Y$ to $X\xleftarrow{=} X\to Y$, and the right Kan extension takes $X\to Y$ to $0\leftarrow X\to Y$.
        Now, the first claim follows by composing the adjunction of the theorem with $i^\ast\dashv i_\ast$, whereas the second follows by composing it with $i_!^L\dashv i_!$. 
    \end{proof}
    
    \begin{corollary}
        The $1$-cube $\Sigma^\infty_+\cube^1$ is dualizable. The left and right duals are both $\Sigma^{\infty-1}(I\cup_{S^1} I^\op)$, i.e., the desuspension of the free spectrum on the pointed graph 
        \[\begin{tikzcd}
            \ast \ar[r, bend left] & \bullet \ar[l, bend left]
        \end{tikzcd}.\]
    \end{corollary}
    \begin{proof}
        Compose the adjunction of the theorem with the adjunction
        \[\begin{tikzcd}
            \Fun(\Lambda^2_0, \CatSp) \arrow[r, shift left = 1ex, "\colim"]
            & \CatSp. \arrow[l, shift left = .5ex, "\Delta"]
            \arrow[l, phantom, shift right = .2ex, "\scriptscriptstyle\boldsymbol{\bot}"]
        \end{tikzcd}\]
    \end{proof}
    \begin{corollary}\label{retracts_of_cubes_are_dualizable}
        Left- and right-dualizable objects are closed under shifts, tensor products, and retracts. 
        In particular, $\Sigma^{\infty-n} X$ is dualizable for any $X\in \widetilde{\Cube}$. This includes the cases $X\in \Cube, \Ori, \Theta$ by \cref{cubes_and_orientals_are_dense}. 
    \end{corollary}
    \begin{proof}
        By composition of adjunctions, if $X^L, Y^L$ (resp.\ $X^R, Y^R$) are the left (resp.\ right) duals of $X, Y$, then $Y^L\otimes X^L$ (resp.\ $Y^R\otimes X^R$) is the left (resp.\ right) dual of $X\otimes Y$.
        Note that $X$ is right- (resp.\ left-)dualizable if and only if $[X, \FF]\otimes Y\to [X, Y]$ (resp.\ $Y\otimes \llbracket X, \FF\rrbracket\to \llbracket X, Y\rrbracket$) induced by the counit $X\otimes [X, \FF]\to \FF$ (resp.\ $\llbracket X, \FF\rrbracket\otimes X\to \FF$), is an equivalence. This condition is stable under retracts. Explicitly, the right and left duals of a retract of $X$ are the corresponding retract of $X^R= [X, \FF]$ and $X^L = \llbracket X, \FF\rrbracket$. 
    \end{proof}
    A more complete characterization of dualizable categorical spectra will be explained in work in preparation. 

    \section{Extensions of categorical spectra}
    As an important corollary, we can define a lift of the Barratt--Puppe sequence in $\Sp$. 
    \begin{theorem}\label{extension_of_catsp}
        There is the following diagram, depending functorially on $f: \Omega X\to Y \in \Fun(\cube^1, \CatSp)$:
        \[\begin{tikzcd}
            \Omega X \ar[r, "f"]\ar[d] & Y \ar[r]\ar[d, "g"] & 0 \ar[d] \\
            0 \ar[r]\ar[ru, Rightarrow, shorten <=2.5ex, shorten >=2ex] & Z \ar[r, "h"]\ar[d]\ar[ru, phantom, "\rotatebox{40}{$\simeq$}"] & X\ar[d, "\Sigma f"] \\
            & 0 \ar[r]\ar[ru, Rightarrow, shorten <=2.5ex, shorten >=2ex] & \Sigma Y,
        \end{tikzcd}\]
        where the left and bottom squares exhibit $\rcof(f)\xrightarrow{\sim} Z\xrightarrow{\sim} \rfib(\Sigma f)$ and the top-right square is commutative and biCartesian. 
    \end{theorem}
    \begin{definition}
        In the situation of the theorem, we say that $Z$ is an \emph{extension} of $X$ by $Y$ classified by $f$ (or $\Sigma f$). As usual, we let $\Ext(X, Y)$ denote $\Map(X, \Sigma Y)\simeq \Map(\Omega X, Y)\in \CMon$, the commutative monoid of extensions of $X$ by $Y$. 
    \end{definition}
    \begin{remark}
        Given a map $f: \Omega X\to Y\in \Ext(X, Y)$, there is the transpose of the above diagram with $\lcof(f)\simeq Z\simeq \lfib(\Sigma f)$. We may say that $Z$ is the \emph{coextension} classified by $f$. There is no reason to choose one over the other; if $Y\to Z\to X$ is an extension of categorical spectra, $\Sigma Y\to \Sigma Z\to \Sigma X$ is a coextension. 
    \end{remark}
    \begin{proof}
        Let $g: Y\to Z$ be the lax cofiber of $f$. Note that the right adjoint $(\Sigma^{-1}\FF\to \Sigma^{\infty-1} I)\otimes (\blank)$ of $\rcof: \Fun(\cube^1, \CatSp)\to \CatSp$ is also the weight of the lax fiber of the shift: $\rfib(\Sigma f) \simeq \lim([I, \Sigma Y]\to [\FF, \Sigma Y]\leftarrow [\FF, X])$, so there is a canonical equivalence $Z\xrightarrow{\sim}\rfib(\Sigma f)$. We see that the spliced map $Z\xrightarrow{\sim}\rfib(\Sigma f)\to X$ is the same map as the map induced by functoriality of the construction on $f$, applied to $(\Omega X\xrightarrow{f} Y)\to (\Omega X\to 0)$, so by the pasting law (\cref{pasting_law_lax_pullbacks}), the sequence $Y\to Z\to \Sigma X$ is a bifiber sequence. 
    \end{proof}
    \begin{remark}
        The theorem is surprising given the failure of the pasting law for directed pullbacks; it suggests that the classical notion of fiber sequences splits into those of lax fiber, lax cofiber, and bifiber sequences, and that they appear $3$-periodically when we rotate a triangle.
    \end{remark}

    \chapter{Applications to TQFT}
    \label{chapter_categorical_spectra_with_adjoints}
    One of the central motivations in the study of $n$-categories is the theory of functorial quantum field theories, especially topological quantum field theories (TQFTs). The goal of this chapter is to give a few sample applications of the theory of categorical spectra in this direction. We refer the reader to \cite{lurieClassificationTopologicalField2009} for an introduction. Let us only sketch a definition of TQFTs here. Let $d\geq n\geq 0$ be integers. An \emph{$n$-category of ($d$-dimensional) cobordisms}, or an \emph{$(n-1)$-th extended cobordism category}, is roughly the univalent completion of an $n$-algebroid consisting of the following data:
    \begin{itemize}
        \item An object is a $(d-n)$-dimensional manifold (possibly with extra structure).
        \item A $1$-morphism from $M_0$ to $M_1$ is a $(d-n+1)$-dimensional manifold $W$ with boundary, equipped with the identification $\partial W = \overline{M_0} \sqcup M_1$, where $\overline{M_0}$ denotes the manifold $M_0$ with the ``reversed'' structure. We call $W$ the \emph{cobordism} from $M_0$ to $M_1$. 
        \item More generally, for $k\leq n$, a $k$-morphism is a cobordism between $(k-1)$-morphisms, i.e., a $(d-n+k)$-dimensional manifold with corners equipped with an identification of the boundary with $(d-n+k-1)$-dimensional manifolds corresponding to the source and target $(k-1)$-morphisms.
        \item An $(n+1)$-morphism is a diffeomorphism between $n$-morphisms. More generally, $k$-morphisms for $k>n$ are diffeomorphisms and isotopies, or equivalently trivial cobordisms.
        \item Composition is given by gluing cobordisms along the shared boundary. 
        \item The symmetric monoidal structure is given by disjoint union of manifolds.
    \end{itemize}
    We often denote this cobordism category by $\Bord_{d-n, \ldots, d}^\cS$, where $\cS$ indicates the relevant structure on the manifolds. We will focus on the \emph{topological} case, i.e., when $\cS$ is a tangential structure $X\in \sS_{/\rB\rO(d)}$, although it is possible to allow more sophisticated geometric structures by working over the site of manifolds.
    We note, however, that to match a structure on a manifold with that on its boundary, we need to identify the boundary with its collar or germ, or at least its first jet.
    A \emph{TQFT} is a symmetric monoidal functor $Z: \Bord_{d-n, \ldots, d}^\cS\to A$ into another symmetric monoidal $n$-category, typically of algebraic flavor, such as ``the $n$-category of $n$-vector spaces'' over $\CC$.

    An important feature of the cobordism category is the existence of duals and adjoints: every object $M$ admits a dual $\overline{M}$, every morphism $W: M_0\to M_1$ admits left and right adjoints $\overline{W}: M_1\to M_0$, and so on, up through $(n-1)$-morphisms (note that it is not reasonable to ask for adjointability of all $n$-morphisms in an $n$-category, as a top-level morphism is adjointable if and only if it is invertible). Consequently, a TQFT lands only in the sufficiently adjointable cells of the target.
 
    Let us say in general that a symmetric monoidal $\infty$-category is \emph{$n$-adjointful} if every $k$-morphism for $0\leq k<n$ is both left and right adjointable (dualizable when $k=0$). 
    The cobordism hypothesis states, in the greatest generality, that a cobordism $n$-category is a free $n$-adjointful symmetric monoidal $n$-category generated by certain data. This is, of course, equivalent to classifying TQFTs as certain collections of generating data.
    Only the case $d=n$ is systematically understood; this is when the objects are $0$-manifolds. Such cases are called \emph{fully extended}, or \emph{fully local}. We will focus on the fully local cases, but note that the other cobordism categories can be described as the $(d-n)$-th loop of a fully extended cobordism category (with a modified structure $\cS$).
    Full locality refers to the heuristic that we can cut global information on higher-dimensional manifolds, or a cobordism, into trivial pieces ($0$-morphisms). This suggests that the TQFT is determined by the value at the (codimension $d$ germ of) points.
    The framed version of the cobordism hypothesis makes this intuition precise and tells us that the cobordism category with $d$-framing is generated by a $d$-framed point. 
    When $d=1$, this translates to the fact that one can uniquely interpret a string diagram; $d> 1$ can be thought of as a higher-dimensional analog of this fact. 
    
    We start this section with the study of $n$-adjointful categorical spectra and, in particular, $n$-adjointful symmetric monoidal categories. It is a levelwise property $\CatSp^{n\mhy\adj}\subset \CatSp$ (in the sense of \cref{section_categorical_spectra_stable_properties}) with a localization $L^{n\mhy\adj}$, and it satisfies the following three properties:
    \begin{enumerate}
        \item $X$ is $n$-adjointful if and only if $\Sigma X$ is $(n+1)$-adjointful. 
        \item If $X$ is an $m$-categorical spectrum and $f: Y\to Z$ is $L^{n\mhy\adj}$-equivalence, then $X\otimes f$ is an $L^{(m+n)\mhy\adj}$-equivalence. 
        \item If $X, Y$ are $n$-adjointful and $X\rightarrowtail Z\twoheadrightarrow Y$ is a (co)extension, then $Z$ is also $n$-adjointful. 
    \end{enumerate}
    The first property is immediate, and the second and third reduces to a formula expressing $\cube^1\otimes (C_1\to \Adj)$ as a pushout of two copies of $C_1\to \Adj$ and a copy of its suspension, where $C_1\to \Adj$ is the inclusion of the right adjoint into a walking adjunction category. The reader may safely skip the proof if preferred. 

    Once this formal property is established, we will spend \cref{section_cobordism_hypothesis} translating the usual cobordism hypothesis into the language of categorical spectra. This translation is formal and superficial, but we point out that because $0$-adjoint categorical spectra form a monoidal subcategory, it is more convenient to shift the $n$-adjointful symmetric monoidal $n$-category $\eC$ to a $0$-adjointful $0$-categorical spectrum $\rB^{\infty-n}\eC$. Philosophically, this means that we use the codimensional indexing, placing the top level (i.e., the partition function) at the categorical level $0$. 

    A real application of our theory is in \cref{section_cobordism_hypothesis_with_singularities}. Here we apply our theory to the study of successive extensions of cobordism categories, which we view as cobordism categories with singularities (a.k.a.\ defects). The equivalence in \cref{extension_of_catsp} is precisely the categorical crux in generalizing the cobordism hypothesis to this setting. 

    We end with a short discussion of the cobordism hypothesis with stable tangential structures in \cref{section_stable_variant}, i.e., when points have infinite codimension. In this case, we will see that the stably framed bordism $\infty$-category is the tensor unit of the $\infty$-adjointful categorical spectra.

    \section{Categorical spectra with adjoints}\label{section_categorical_spectra_with_adjoints}
    Let us start by recalling the definition of adjunctions in higher categories. 
    Let $r: x\to y$ and $l: y\to x$ be morphisms in a $(2, 2)$-category $\eK$. 
    We say a $2$-cell $\eta: \id_y\to rl$ is the \emph{unit} of an adjunction if there is another $2$-cell $\eps: lr\to \id_x$ satisfying the properties $\id_r\simeq (r\eps)(\eta r)$ and $\id_l\simeq (\eps l)(l\eta)$. In this case, we say $l$ and $r$ are \emph{left} and \emph{right adjoint} of the adjunction, and $\eps$ is the \emph{counit} of the adjunction. Now let $X$ be a general $\infty$-category. A $1$-morphism $r: x\to y$ in $X$ admits a left adjoint if it admits a left adjoint in the homotopy $2$-category of $X$. If $k\geq 2$ and $x: C_k\to X$ is a $k$-morphism with $(k-2)$-source $s_{k-2}$ and $(k-2)$-target $t_{k-2}$, then $x$ is said to admit a left adjoint if it admits a left adjoint as a $1$-morphism in $\Hom_X(s_{k-2}, t_{k-2})$. In this section, we study categorical spectra whose cells of dimensions in a range admit adjoints.

    Let $\Adj$ denote the strict $2$-category of walking adjunction \cite{schanuelFreeAdjunction1986}; it is a theorem of Riehl--Verity that it is gaunt and corepresents a homotopy-coherent adjunction in $(\infty, 2)$-categories. We refer to \cite[\S 3.1]{riehlHomotopyCoherentAdjunctions2016}
    for a detailed description as a simplicial computad. Here we only name the atomic cells for reference: $\Adj$ has two objects $0, 1$ and two atomic $1$-cells $r: 0\to 1$, $l: 1\to 0$ (denoted $-, +, u, f$, respectively, in the reference). 
    There are two atomic $2$-cells $\eta: \id_{1}\to rl$ and $\eps: lr\to \id_0$ corepresenting the unit and counit, two atomic $2$-cells $\alpha: \id_r\xrightarrow{\sim} (r\eps)(\eta r)$ and $\beta:\id_l\xrightarrow{\sim} (\eps l)(l\eta)$ corepresenting the triangle identities, and the pattern continues, i.e., there are two atomic cells $\alpha^{(n)}$ and $\beta^{(n)}$ for each $n\geq 3$ corepresenting the ``higher triangle identities.'' The following is the model-independent translation of the main results of \cite[\S 4]{riehlHomotopyCoherentAdjunctions2016}:
    \begin{theorem}
        Consider the following subcategories generated by the atomic cells (with generation indicated by overlines):
        \[C_1 = \overline{\{r\}}\hookrightarrow\overline{\{r, l, \eps\}}\hookrightarrow \overline{\{r, l, \eps, \eta, \alpha\}} \hookrightarrow \overline{\{r, l, \eps, \eta, \alpha, \beta, \alpha^{(3)}\}}\hookrightarrow \Adj. \]
        Then the inclusions of the first three categories into $\Adj$ are epimorphisms in $\infty\Cat$ and the last one is an equivalence. 
    \end{theorem}
    We denote the maps $C_1\to \Adj$ classifying the morphisms $r$, $l$ by the same names. 
    The theorem says that $r, l: C_1\to \Adj$ are epi, i.e., a homotopy coherent adjunction extending a prospective right or left adjoint is unique if it exists.
    \begin{definition}
        Let $d\in \bar{\ZZ}\coloneqq \ZZ\cup \{\pm \infty\}$. 
        Let $S_d^\adj \coloneqq \{\Sigma_+^{\infty-i}\sigma^j l, \Sigma_+^{\infty-i}\sigma^j r\mid 0\leq j\leq d+i-2\}$. 
        A categorical spectrum $X$ is \emph{$d$-adjointful} if it is $S_d^\adj$-local, i.e., for any $f\in S_d^\adj$, the induced map $\Map(f, X)$ is an isomorphism. We let $\CatSp^{d\mhy\adj}\subset \CatSp$ denote the category of $d$-adjointful categorical spectra. 
        We also let $d\CatSp^\adj$ denote the intersection $d\CatSp\cap \CatSp^{d\mhy\adj}$. 
    \end{definition}
    
    In other words, a categorical spectrum $X=(X_n)$ is $d$-adjointful if and only if every $(j+1)$-morphism of $X_n$ has both left and right adjoints for $j=0, \ldots, d+n-2$. Also note that, as a symmetric monoidal category, $X_n= \Omega X_{n+1}$ automatically has duals for objects if $d+n-1\geq 0$.
    \begin{example}
        By definition, we have $0\CatSp^\adj \xrightarrow{\sim} \lim \CMon(n\Cat)^\dual$, where $\CMon(n\Cat)^\dual\subset\CMon(n\Cat)$ denotes the full subcategory of symmetric monoidal $n$-categories \emph{with duals} in the sense of \cite{lurieClassificationTopologicalField2009}. 
        Also, the functor $\rB^{\infty}$ restricts to the equivalence $\CMon(d\Cat)^\dual\xrightarrow{\sim} \CatSp^\cn\cap d\CatSp^{d\mhy\adj}$. 
        For any $n<d$, the intersection $n\CatSp\cap \CatSp^{d\mhy\adj}$ is $\Sp$. 
    \end{example}
    \begin{remark}
        Let $d\in \bar\ZZ$. The inclusion $\CatSp^{d\mhy\adj}\hookrightarrow \CatSp$ admits a left adjoint $L_d^{\adj}$, which freely adds left and right adjoints to stable $k$-cells with $k\leq d-1$. 
    \end{remark}
    The following lemma plays a fundamental role in the study of adjunctions in higher categories. The proof will be deferred to the end of the section.
    \begin{lemma}\label{pushout_formula_for_cylinder_on_Adj}
        There is a pushout diagram of $3$-categories\footnote{The original version claimed that this is true for non-univalent categories too. However, in the proof we use that $C_1\to \Adj$ is epi, which is only true in the univalent setting.}:
        \[\begin{tikzcd}[column sep=20mm]
            \sigma C_1 \sqcup \partial \cube^1\otimes C_1\ar[r, "\phi\sqcup (r_0\sqcup r_1)"]\ar[d, "\sigma r\sqcup \partial\cube^1\otimes r"'] & \cube^1\otimes C_1\ar[d, "\cube^1\otimes r"]\\
            \sigma \Adj \sqcup \partial \cube^1\otimes \Adj \ar[r] & \cube^1\otimes \Adj, 
        \end{tikzcd}\]
        where we label the arrows of $\cube^1\otimes C_1$ as
        \[\begin{tikzcd}
            a_0 \ar[r, "r_0"]\ar[d, "a"] &  b_0 \ar[d, "b"]\\
            a_1\ar[r, "r_1"]\ar[ru, Rightarrow, "\phi", shorten <=1ex, shorten >= 1ex] & b_1
        \end{tikzcd}\] 
        and the map $r: C_1\to \Adj$ picks out the universal right adjoint.
    \end{lemma}
    \begin{remark}
        Informally speaking, the lemma states that a morphism in the functor category with lax natural transformations has a left adjoint if and only if each of the component cells has a left adjoint. 
        Applying ${}^{\leq 2}(\blank)$, we recover \cite[Theorem 4.6]{haugsengLaxTransformationsAdjunctions2021}. The proof is also similar, usisng some mate calculus, but naturally our full $3$-categorical version is more complicated. 
    \end{remark}

    Recall that the class of maps inverted by a localization is characterized by being \emph{strongly saturated}: a class $S$ of morphisms in a presentable category $\eC$ is strongly saturated if it satisfies the following three conditions \cite[Definition 5.5.4.5]{LurieHTT}: 
    \begin{enumerate}
        \item $S$ satisfies the 2-out-of-3 property, i.e., if two of $f, g$, and $f\circ g$ belongs to $S$, then so is the third. 
        \item $S$ is closed under cobase change, i.e., if $f: X\to Y\in S$ and 
        \[\begin{tikzcd}
            X\ar[r, "f"]\ar[d] & Y\ar[d] \\
            Z\ar[r, "f'"] & Y\amalg_X Z
        \end{tikzcd}\]
        is a pushout diagram, then $f'$ belongs to $S$ as well. 
        \item $S$ is closed under colimits in $\Fun(\cube^1, \eC)$. 
    \end{enumerate}
    Notice the last condition implies $\id_{\emptyset}\in S$, so by the second condition any isomorphism belongs to $S$.
    For a set $S$, let $\overline{S}$ denote the smallest strongly saturated class containing $S$. These are precisely the maps that get inverted by the localization $\eC\to \eC[S^{-1}]$. 
    \begin{lemma}\label{suspension_and_strongly_saturated_class}
        Let $S$ be a class of morphisms in $\infty\Cat$ and let $\sigma: \infty\Cat\to \infty\Cat$ be the unpointed suspension endofunctor. Then we have the following: 
        \begin{enumerate}
            \item $\sigma\overline{S}\subset \overline{\sigma S}$,
            \item If $f: X\to Y\in S$, then $\sigma f \vee\cube^1: \sigma X\vee\cube^1\to \sigma Y\vee\cube^1 \in \overline{\sigma S}$ and similarly for $\cube^1\vee \sigma f$. 
            \item Let $f: X\to Y$ be a morphism in $\infty\Cat$ such that $\cube^1 \otimes f\in \overline{\{f, \sigma f\}}$. Then for any $n\geq 0$, the morphism $\cube^1\otimes \sigma^n f$ belongs to $\overline{\{\sigma^n f, \sigma^{n+1}f\}}$. 
        \end{enumerate}
    \end{lemma}
    \begin{proof}
        \begin{enumerate}
            \item We wish to show that $\overline{S}\subset T= \{f\mid \sigma f\in \overline{\sigma S}\}$. 
            We have $S\subset T$, so it suffices to show that $T$ is strongly saturated. The 2-out-of-3 property is clear. 
            Note that $\sigma: \infty\Cat\to \infty\Cat$ takes a colimit diagram to a colimit diagram under $\sigma\emptyset$, i.e., $\sigma(\colim_{\lambda\in \Lambda} X_\lambda) \simeq \colim_{\lambda \in \Lambda^{\triangleleft}} (\sigma X_\lambda)$, where if $\lambda=-\infty \in \Lambda^{\triangleleft}$ is the cone point, we set $X_{-\infty} = \emptyset$. 
            In particular, $\sigma$ preserves weakly contractible colimits, so if $f'$ is a cobase change of $f\in T$, then $\sigma f'$ is also a cobase change of $\sigma f\in \overline{\sigma S}$, so $f' \in T$. 
            Lastly, if $f' = \colim_{\lambda} f_\lambda$ in $\Fun(\cube^1, \infty\Cat)$ and $f_\lambda\in T$, the suspension $\sigma f'$ is computed by $\colim_{\lambda\in \Lambda^{\triangleleft}}(\sigma f_\lambda)$ with $f_{-\infty} \coloneqq \id_{\emptyset}$. As $\sigma f_\lambda, \sigma\id_{\emptyset}\in \overline{\sigma S}$, it follows that $\sigma f'\in \overline{\sigma S}$, i.e., $f' \in T$. 
            \item This is immediate from the pushout $f\vee \cube^1\coloneqq \sigma f\amalg_{\id_{\ast}} \id_{\cube^1}$ in $\overline{\sigma S}$.
            \item Recall the pushout formula 
            \[\cube^1\otimes \sigma^n f \simeq \sigma(\cube^1\otimes \sigma^{n-1} f)\sqcup_{\sigma^n f \sqcup \sigma^n f} ((\sigma^n f \vee \cube^1)\sqcup (\cube^1\vee\sigma^n f))\] of \cref{pushout_formula_for_gray_cylinder_of_suspensions}.
            By (2) we have $\sigma^n f \vee\cube^1$, $\cube^1\vee \sigma^n f\in \overline{\{\sigma^n f\}}$, so it suffices to check $\sigma(\cube^1\otimes \sigma^{n-1} f)\in \overline{\{\sigma^{n}f, \sigma^{n+1}f \}}$. This follows by induction on $n$ and (1) for $S = \{\sigma^{n-1}f, \sigma^n f\}$. 
        \end{enumerate}
    \end{proof}
    With \cref{pushout_formula_for_cylinder_on_Adj}, we obtain $\cube^1\otimes \sigma^n r\in \overline{\{\sigma^n r, \sigma^{n+1} r\}}$ and hence
    $\Sigma^{\infty-i}_+\cube^1\otimes \Sigma^{\infty-j}_+\sigma^n r\simeq \Sigma^{\infty-(i+j)}_+(\cube^1\otimes \sigma^n r)\in \overline{\{\Sigma_+^{\infty-(i+j)}(\sigma^n r, \sigma^{n+1} r)\}}$. 

    \begin{theorem}\label{theorem_tensor_product_localizes_to_adj}
        The graded monoidal structure $\{n\CatSp\subset \CatSp^\otimes\}_{n\in \bar{\ZZ}}$ is compatible with the localizations $L_n^\adj: n\CatSp\to n\CatSp^\adj$. More precisely, if $f$ is a $L_n^\adj$-equivalence and $X$ is a $m$-categorical spectrum, then $f\otimes X$ and $X\otimes f$ are $L_{m+n}^\adj$-equivalences.
        In particular, the tensor product of categorical spectra localizes to $\otimes^\adj: m\CatSp^\adj\otimes n\CatSp^\adj\to (m+n)\CatSp^\adj$. 
    \end{theorem}
    \begin{proof}
        The claim is trivial if one of $m, n$ is $-\infty$. 
        We must show that $f\in S_n^\adj$ and $X\in m\CatSp$ imply $f\otimes X, X\otimes f\in \overline{S_{m+n}^\adj}$. 
        For finite $m, n$, we have $\Sigma^{-n}S_n^\adj = S_0^\adj$ and $\Sigma^{-m}m\CatSp = 0\CatSp$, so after shifts we may assume $(m, n)=(0, 0)$. Other cases reduce to $(m, n)=(0, \infty), (\infty, 0), (\infty, \infty)$; the first two are special cases of the third. Now assume $(m, n)=(0, 0)$; the case $(m, n)=(\infty, \infty)$ is similar.
        Since $S_n^\adj$ and $m\CatSp$ are closed under duality involutions, using $(A\otimes B)\simeq (B^\op\otimes A^\op)^\op$, the theorem reduces to proving $X\otimes \Sigma^{\infty-i}\sigma^k r \in \overline{S_0^\adj}$. 
        Such $X$ are closed under tensor products and colimits. Note that $0\CatSp$ is generated under colimits by $\Sigma_+^{\infty-n}\cube^n$ because $\{\Omega^{\infty-n}: 0\CatSp\to n\Cat\xrightarrow{\Map(\cube^n, \blank)}\sS\}_{n\in\NN}$ is jointly conservative (note that $C_n$ is a retract of $\cube^n$). 
        Consequently, we need only check the case $X=\Sigma^{\infty-1}\cube^1\otimes \Sigma^{\infty-i}\sigma^k r\in \overline{S_0^\adj}$ for $0\leq k \leq i-2$, which follows from $\Sigma^{\infty-(i+1)}\cube^1\otimes \sigma^kr \in\overline{\{\Sigma_+^{\infty-(i+1)}\sigma^k r, \Sigma_+^{\infty-(i+1)}\sigma^{k+1} r\}}$. 
    \end{proof}
    \begin{remark}
        We expect that the localized monoidal product is more commutative than the original: if a category admits adjoints, then passage to adjoints allows us to identify the category with its various duality involutions. In particular, we likely have a natural isomorphism $X\otimes^L Y\simeq (X\otimes^L Y)^\op\simeq Y^\op\otimes^L X^\op\simeq Y\otimes^L X$. In fact, assuming a conjecture on factorization homology, one can show that the localized tensor product upgrades to an $\EE_\infty$-structure.
    \end{remark}
    \begin{corollary}\label{n_adjointable_catsp_are_closed_under_extensions}
        $n$-adjointful categorical spectra are closed under (co)extensions.
    \end{corollary}
    \begin{proof}
        Let $X, Y$ be $n$-adjointful categorical spectra, let $f: Y\to \Sigma X$ be a morphism, and let $Z= \rfib(f) = \lim(0\to \Sigma X\xleftarrow{\ev_0}[\cube^1, \Sigma X]\xrightarrow{\ev_1} \Sigma X\leftarrow Y)$, so that $X\to Z \to Y$ is an extension of categorical spectra. Since $\CatSp^{n\mhy\adj}\subset \CatSp$ is closed under limits, it suffices to show that $\Sigma X$ and $[\cube^1_+, \Sigma X]$ are also $n$-adjointful. $\Sigma X$ is $(n+1)$-adjointful, so, in particular, $n$-adjointful. Moreover, $[\cube^1_+, \Sigma X]$ is $S_n^\adj$-local because $\Sigma^\infty_+\cube^1\otimes \overline{S_n^\adj}\subset \overline{S_{n+1}^\adj}$ and $\Sigma X$ is $S_{n+1}^\adj$-local. 
    \end{proof}

    \begin{proof}[proof of \cref{pushout_formula_for_cylinder_on_Adj}]
        Let $P$ denote the pushout $(\cube^1\otimes C_1)\sqcup_{\sigma C_1\sqcup \partial\cube^1\otimes C_1}(\sigma\Adj\sqcup \partial\cube^1\otimes \Adj)$. 
        Since $r: C_1\to \Adj$ is an epimorphism, $\cube^1\otimes C_1\to \cube^1\otimes \Adj$ and $\cube^1\otimes C_1\to P$ are both epimorphisms. Since the epimorphisms from $\cube^1\otimes C_1$ form a poset, to show $P\simeq \cube^1\otimes \Adj$ in $\infty\Cat_{(\cube^1\otimes C_1)/}$, it suffices to give maps in both directions. %

        We first define $P\to \cube^1\otimes\Adj$, or equivalently the commutative square in the lemma. 
        The second component $\partial\cube^1\otimes \Adj\to \cube^1\otimes \Adj$ is induced from $\partial\cube^1\hookrightarrow\cube^1$. 
        To define the first component $\sigma \Adj\to \cube^1\otimes \Adj$, we must show that the $2$-cell $\phi: r_1\circ a\to b\circ r_0$ given by $\sigma C_1\to\cube^1\otimes C_1\to \cube^1\otimes \Adj$ admits a left adjoint. 
        \begin{itemize}
            \item Let $\psi$ denote the $2$-cell $C_2\hookrightarrow \cube^1\otimes C_1\xrightarrow{\id\otimes l} \cube^1\otimes \Adj$ and let $\phi^L$ be the $2$-cell of the mate square of $\id\otimes l$:
            \[\phi = \begin{tikzcd}
                a_0\ar[d, "a"']\ar[r, "="{name=L}]\ar[phantom, "=", from=L, to=2-2]\ar[rd, "r_0"'] & a_0\ar[d, "r_0"] \\
                a_1\ar[r, Rightarrow, shorten <=1.5ex, shorten >=1.5ex, "\phi"]\ar[d, "r_1"']\ar[rd, "r_1"] & b_0\ar[d, "b"] \\
                b_1\ar[r, "="'{name=K}]\ar[phantom, from=K, to=2-1, "="] & b_1.
            \end{tikzcd} \qquad \phi^L \coloneqq \begin{tikzcd}
                a_0\ar[d, "r_0"']\ar[r, "="{name=J}]\ar[Rightarrow, to=J, from=2-1, shorten <=2.5ex, shorten >=2ex, "\eps_0"] & a_0\ar[d, "a"] \\
                b_0\ar[r, Rightarrow, shorten <=1.5ex, shorten >=1.5ex, "\psi"]\ar[d, "b"']\ar[ru, "l_0"'] & a_1\ar[d, "r_1"] \\
                b_1\ar[r, "="'{name=I}]\ar[ru, "l_1"]\ar[Rightarrow, from=I, to=2-2, shorten <=2ex, shorten >=2.5ex, "\eta_1"'] & b_1.
            \end{tikzcd}\]
            \item We will exhibit $\phi^L\dashv \phi$ by defining $E: \phi^L\circ \phi\to \id_{r_1\circ a}$ and $H: \id_{b\circ r_0}\to \phi\circ\phi^L$ and checking the triangle identities. 
            Let $\id\otimes \eps, \id\otimes \eta: \cube^1\otimes C_2\to \cube^1\otimes \Adj$ also denote the corresponding $3$-cell. We define 
            \[\adjustbox{scale=0.9, center}{$E: $\begin{tikzcd}
                a_0\ar[d, "a"']\ar[r, "="{name=L}]\ar[rd, "r_0"'] & a_0\ar[r, "="{name=L}] & a_0\ar[d, "a"] \\
                a_1\ar[r, Rightarrow, shorten <=1.5ex, shorten >=1.5ex, "\phi"]\ar[d, "r_1"']\ar[rd, "r_1"] & b_0\ar[d, "b"]\ar[u, Rightarrow, "\eps_0"', shorten <=1ex, shorten >=1ex]\ar[r, Rightarrow, shorten <=1.5ex, shorten >=1.5ex, "\psi"]\ar[ru, "l_0"'] & a_1\ar[d, "r_1"] \\
                b_1\ar[r, "="'{name=K}]\ar[phantom, from=K, to=2-1, "="] & b_1\ar[r, "="'{name=I}]\ar[ru, "l_1"]\ar[Rightarrow, from=I, to=2-3, shorten <=2ex, shorten >=2.5ex, "\eta_1"'] & b_1.
            \end{tikzcd}$\quad \xRrightarrow{\eta_1\ast (r_1\ast(\id\otimes \eps))} \quad $
            \begin{tikzcd}
                a_0\ar[d, "a"']\ar[r, "="{name=L}] & a_0\ar[r, "="{name=L}] & a_0\ar[d, "a"] \\
                a_1\ar[rr, "="]\ar[d, "r_1"']\ar[rd, "r_1"]\ar[rru, phantom, "="] & {} & a_1\ar[d, "r_1"] \\
                b_1\ar[r, "="'{name=K}]\ar[phantom, from=K, to=2-1, "="] & b_1\ar[r, "="'{name=I}]\ar[ru, "l_1"]\ar[Rightarrow, from=I, to=2-3, shorten <=2ex, shorten >=2.5ex, "\eta_1"']\ar[u, Rightarrow, shorten <=1ex, shorten >=.5ex, "\eps_1"'] & b_1.
            \end{tikzcd}$\quad \xRrightarrow[\alpha^{-1}\ast a]{\sim}\quad \id_{r_1\circ a},$}
            \]
            \[\adjustbox{scale=0.9, center}{$H: \id_{b\circ r_0} \quad \xRrightarrow[b\ast\beta]{\sim}\quad$ 
            \begin{tikzcd}
                a_0\ar[d, "r_0"']\ar[r, "="{name=J}]\ar[Rightarrow, to=J, from=2-1, shorten <=2.5ex, shorten >=2ex, "\eps_0"] & a_0\ar[r, "="{name=L}]\ar[phantom, "=", from=L, to=2-3]\ar[rd, "r_0"'] & a_0\ar[d, "r_0"] \\
                b_0\ar[d, "b"']\ar[ru, "l_0"']\ar[rr, "="']\ar[rrd, phantom, "="] & {}\ar[u, Rightarrow, "\eta_0", shorten <=1ex, shorten >=1ex] & b_0\ar[d, "b"] \\
                b_1\ar[r, "="'{name=I}] & b_1\ar[r, "="'{name=K}] & b_1.
            \end{tikzcd} $\quad\xRrightarrow{((\id\otimes \eta)\ast r_0)\ast\eps_0}\quad$
            \begin{tikzcd}
                a_0\ar[d, "r_0"']\ar[r, "="{name=J}]\ar[Rightarrow, to=J, from=2-1, shorten <=2.5ex, shorten >=2ex, "\eps_0"] & a_0\ar[d, "a"]\ar[r, "="{name=L}]\ar[phantom, "=", from=L, to=2-3]\ar[rd, "r_0"'] & a_0\ar[d, "r_0"] \\
                b_0\ar[r, Rightarrow, shorten <=1.5ex, shorten >=1.5ex, "\psi"]\ar[d, "b"']\ar[ru, "l_0"'] & a_1\ar[r, Rightarrow, shorten <=1.5ex, shorten >=1.5ex, "\phi"]\ar[rd, "r_1"] & b_0\ar[d, "b"] \\
                b_1\ar[r, "="'{name=I}]\ar[ru, "l_1"] & b_1\ar[r, "="'{name=K}]\ar[u, Rightarrow, "\eta_1"', shorten <=0.5ex, shorten >=1ex] & b_1.
            \end{tikzcd}}\]
            The triangle identities are given similarly by the $4$-cells corresponding to $\id\otimes \alpha, \id\otimes \beta: \cube^1\otimes C_3\to \cube^1\otimes \Adj$ composed with $\alpha_0$, $\alpha_1$, $\beta_0$ and $\beta_1$. 
        \end{itemize}
        Therefore, there is a unique map $\sigma\Adj\to \cube^1\otimes \Adj$ extending $\phi$, which defines $P\to \cube^1\otimes \Adj$. 

        Next we define $\cube^1\otimes \Adj\to P$. We must show that the canonical map $\tilde{r}: a\to b$ in $[\cube^1, P]$ corresponding to the square $\id\otimes r: \cube^1\otimes C_1\to P$ admits a left adjoint. We define $\tilde{l}: b\to a$ as follows. Let $\phi: r_1\circ a\to b\circ r_0$ denote again the $2$-cell corresponding to $\tilde{r}$, i.e., the composition $\sigma C_1\xrightarrow{\sigma r} \sigma\Adj\to P$. Let $\phi^L: b\circ r_0\to r_1\circ a$ denote the composition $\sigma C_1\xrightarrow{\sigma l} \sigma\Adj\to P$, which is left adjoint to $\phi$, exhibited by the adjunction data $\sigma\Adj\to P$. Now we let $\tilde{l}$ be the mate square of $\phi^L$;the $2$-cell $l_1b\Rightarrow al_0$ will be called $\psi$ again:
        \[\tilde{l} \coloneqq\quad
        \begin{tikzcd}
            b_0 \ar[r, "l_0"]\ar[d, equal] & a_0 \ar[r, equal]\ar[d, "r_0"] & a_0\ar[r, equal]\ar[d, "a"] & a_0\ar[d, "a"] \\
            b_0 \ar[r, equal]\ar[d, "b"]\ar[ru, Rightarrow, shorten <=2ex, shorten >=2ex, "\eta_0"] & b_0 \ar[d, "b"]\ar[r, Rightarrow, shorten <=1ex, shorten >=1ex, "\phi^L"'] & a_1 \ar[r, equal]\ar[d, "r_1"] & a_1\ar[d, equal]\\
            b_1\ar[r, equal] & b_1\ar[r, equal] & b_1 \ar[r, "l_1"']\ar[ru, Rightarrow, shorten <=2ex, shorten >=2ex, "\eps_1"'] & a_1
        \end{tikzcd}\]
        We wish to show that $\tilde{l}$ is left adjoint to $\tilde{r}$. 
        The unit map $\tilde{\eta}: \id_b\to \tilde{r}\circ\tilde{l}$ and the counit map $\tilde{\eps}: \tilde{l}\circ\tilde{r}\to \id_a$ are defined as follows. We first define $\tilde{\eta}$. 
        We must define a $2$-cell $C_2\to \Fun^\lax(\cube^1, P)$, or equivalently a map $\cube^1\otimes C_2\to P$, extending the prescribed boundary $(\id_b, \tilde{r}\circ\tilde{l}): \cube^1\otimes \partial C_2\to P$. 
        Now recall that we have a pushout of $3$-algebroids $\cube^1\otimes C_2\simeq (\partial\cube^1\otimes C_2\sqcup_{\partial\cube^1\otimes \partial C_2} \cube^1\otimes \partial C_2)\sqcup_{\partial C_3}C_3$.
        We already know that $\tilde{\eta}|_{\partial \cube^1\otimes C_2}$ must be $\eta_0\sqcup \eta_1$, which agrees with the given map on $\cube^1\otimes \partial C_2$. It remains to define $C_3\to P$ whose boundary is prescribed as follows: 
        \[\tilde{\eta}: 
        \begin{tikzcd}
            & a_0\ar[rd, "r_0"] & \\
            b_0\ar[rr, "="]\ar[dd, "b"']\ar[ru, "l_0"] & {}\ar[u, Rightarrow, "\eta_0", shorten <=2ex, shorten >=1ex] & b_0\ar[dd, "b"] \\
            &&\\
            b_1\ar[rr, "="]\ar[rruu, phantom, "="] & & b_1
        \end{tikzcd} \quad \Rrightarrow\quad 
        \begin{tikzcd}
             & a_0\ar[d, "a"]\ar[rd, "r_0"] & \\
            b_0\ar[ru, "l_0"]\ar[d, equal]\ar[rd, phantom, "\xRightarrow{\psi}"] & a_1\ar[d, equal] & b_0\ar[dd, "b"]\\
            b_0\ar[d, "b"'] & a_1\ar[rd, "r_1"]\ar[ru, phantom, "\xRightarrow{\phi}"] & \\
            b_1\ar[rr, "="']\ar[ru, "l_1"'] & {}\ar[u, Rightarrow, "\eta_1", shorten <=1ex, shorten >=1ex] & b_1,
        \end{tikzcd}\]
        Massaging the diagram, the domain and codomain can be rewritten into the following grid shape, and now $\tilde{\eta}$ is clearly given by first inserting $\eta_\phi: \id_{br_0}\to \phi\phi^L$ and then $\alpha_1: \id_{r_1}\to (r_1\ast \eps_1)(\eta_1\ast r_1)$:
        {\small \[\begin{tikzcd}
            b_0 \ar[r, "l_0"] \ar[d, equal] & a_0 \ar[d, "r_0"'] \\
            b_0 \ar[r, equal]\ar[dd, "b"]\ar[ru, Rightarrow, shorten <=2ex, shorten >=2ex, "\eta_0"] & b_0 \ar[dd, "b"']\\
            {}& {} \\
            b_1\ar[r, equal] & b_1
        \end{tikzcd}
        \Rrightarrow
        \begin{tikzcd}
            b_0 \ar[r, "l_0"]\ar[d, equal] & a_0 \ar[r, equal]\ar[d, "r_0"] &  a_0\ar[d, "a"]\ar[r, equal] & a_0\ar[d, "r_0"']\\
            b_0 \ar[r, equal]\ar[dd, "b"]\ar[ru, Rightarrow, shorten <=2ex, shorten >=2ex, "\eta_0"] & b_0 \ar[dd, "b"]\ar[rd, phantom,  "\xRightarrow{\phi^L}"] & a_1\ar[dd, "r_1"] & b_0\ar[dd, "b"']\\
            && {}\ar[ru, phantom, "\xRightarrow{\phi}"] & \\
            b_1\ar[r, equal] & b_1\ar[r, equal] & b_1\ar[r, equal] & b_1
        \end{tikzcd} \Rrightarrow \begin{tikzcd}
            b_0 \ar[r, "l_0"]\ar[d, equal] & a_0 \ar[r, equal]\ar[d, "r_0"] & a_0\ar[r, equal]\ar[d, "a"] & a_0\ar[d, "a"]\ar[r, equal] & a_0\ar[d, "r_0"']\\
            b_0 \ar[r, equal]\ar[dd, "b"]\ar[ru, Rightarrow, shorten <=2ex, shorten >=2ex, "\eta_0"] & b_0 \ar[dd, "b"]\ar[rd, phantom,  "\xRightarrow{\phi^L}"] & a_1 \ar[r, equal]\ar[d, "r_1"] & a_1\ar[d, equal] & b_0\ar[dd, "b"']\\
            && b_1 \ar[ru, Rightarrow, shorten <=2ex, shorten >=2ex, "\eps_1"']\ar[d, equal]\ar[r, "l_1"] & a_1\ar[d, "r_1"]\ar[ru, phantom, "\xRightarrow{\phi}"] & \\
            b_1\ar[r, equal] & b_1\ar[r, equal] & b_1 \ar[r, equal]\ar[ru, Rightarrow, shorten <=2ex, shorten >=2ex, "\eta_1"] & b_1\ar[r, equal] & b_1
        \end{tikzcd}\]}
        The definition of the counit $\tilde{\eps}$ is similar: we need to find the $3$-cell $C_3\hookrightarrow \cube^1\otimes C_2\xrightarrow{\tilde{\eps}} P$ with prescribed boundaries, which is accomplished by $\alpha_0^{-1}$ and $\eps_{\phi}$:
        {\small \[\begin{tikzcd}
            a_0 \ar[r, equal] \ar[dd, "a"] & a_0 \ar[r, equal] \ar[d, "r_0"'] & a_0 \ar[d, equal] \ar[r, equal] & a_0 \ar[dd, "a"] \ar[r, equal] & a_0 \ar[dd, "a"']\\
            & b_0 \ar[r, "l_0"] \ar[d, equal] \ar[ru, "\eps_0", Rightarrow, shorten <=2ex, shorten >=2ex] & a_0 \ar[d, "r_0"']\ar[rd, phantom, "\xRightarrow{\phi^L}"] &&\\
            a_1 \ar[d, "r_1"]\ar[ru, phantom, "\xRightarrow{\phi}"] & b_0 \ar[d, "b"] \ar[r, equal] \ar[ru, "\eta_0", Rightarrow, shorten <=2ex, shorten >=2ex] & b_0 \ar[d, "b"'] & a_1 \ar[d, "r_1"'] \ar[r, equal] & a_1 \ar[d, equal] \\
            b_1 \ar[r, equal] & b_1 \ar[r, equal] & b_1 \ar[r,equal] & b_1 \ar[r, "l_1"]\ar[ru, Rightarrow, shorten <=2ex, shorten >=2ex, "\eps_1"]& a_1
        \end{tikzcd}\Rrightarrow\begin{tikzcd}
            a_0 \ar[r, equal] \ar[dd, "a"] & a_0 \ar[dd, "r_0"] \ar[r, equal] & a_0 \ar[dd, "a"] \ar[r, equal] & a_0 \ar[dd, "a"']\\
            &{} \ar[rd, phantom, "\xRightarrow{\phi^L}"] &&\\
            a_1 \ar[d, "r_1"]\ar[ru, phantom, "\xRightarrow{\phi}"] & b_0 \ar[d, "b"'] & a_1 \ar[d, "r_1"] \ar[r, equal] & a_1 \ar[d, equal] \\
            b_1 \ar[r, equal] & b_1 \ar[r,equal] & b_1 \ar[r, "l_1"]\ar[ru, Rightarrow, shorten <=2ex, shorten >=2ex, "\eps_1"]& a_1
        \end{tikzcd}\Rrightarrow\begin{tikzcd}
            a_0\ar[r, equal]\ar[dd, "a"] & a_0 \ar[dd, "a"']\\
            {} & {}\\
            a_1 \ar[r, equal]\ar[d, "r_1"] & a_1\ar[d, equal]\\
            b_1 \ar[r, "l_1"] & a_1
        \end{tikzcd}\]}

        To verify the triangle identity, we must provide $3$-cells $\tilde{\alpha}: \id_{\tilde{r}}\xrightarrow{\sim} (\tilde{r}\tilde{\eps})(\tilde{\eta}\tilde{r})$, $\tilde{\beta}: \id_{\tilde{l}}\to (\tilde{\eps}\tilde{l})(\tilde{l}\tilde{\eta})$ in $\Fun^\lax(\cube^1, P)$. The construction of $\tilde{\alpha}$ reduces to defining a $4$-cell $C_4\to \cube^1\otimes C_3\to P$ with domain
        \[C_3\to \cube^1\otimes C^-_2\cup_{\{1\}\otimes C^-_2} \{1\}\otimes C_3\xrightarrow{\id_{\tilde{r}}\cup \alpha_1} P\] and codomain 
        \[C_3\to \cube^1\otimes C^+_2\cup_{\{0\}\otimes C^+_2} \{0\}\otimes C_3\xrightarrow{(\tilde{r}\tilde{\eps})(\tilde{\eta}\tilde{r})\cup\alpha_0} P.\]
        This $4$-cell is described in the following diagram. We keep the above grid-shaped description of $\tilde{\eta}, \tilde{\eps}$ and omit the labels of the boundary or whiskered/padded cells. The left edge composes to the domain ($\alpha_1$ component) and the the top, right, and bottom edges compose to the codomain ($\alpha_0$, $\tilde{\eta}\tilde{r}$, $\tilde{r}\tilde{\eps}$ components, respectively). The top-left square commutes by $\alpha_\phi$, and the other squares commute by interchange law, so the whole square commutes.
        
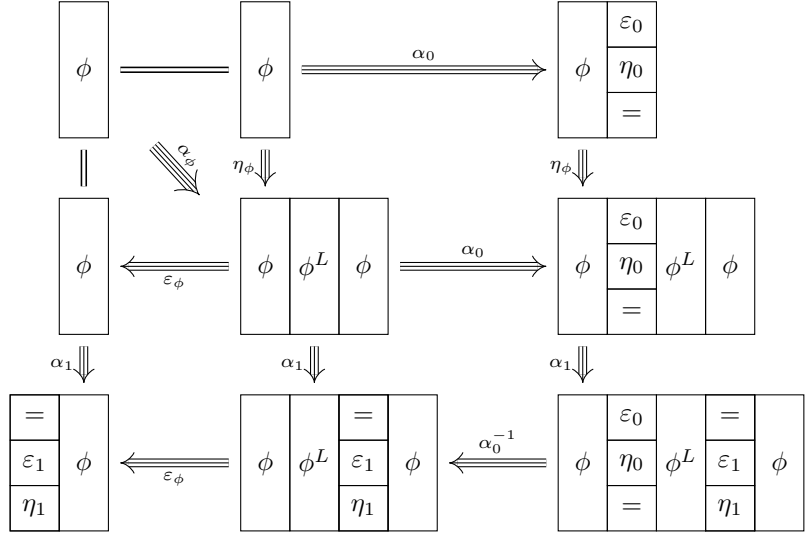
\[\begin{tikzpicture}[
        x=1cm,y=1cm,
        box/.style={draw, minimum width=.65cm, minimum height=1.8cm, inner sep=0pt},
        sbox/.style={draw, minimum width=.65cm, minimum height=.6cm, inner sep=0pt},
        darr/.style={->, double distance=1.4pt},
        label/.style={font=\scriptsize}
        ]
        \newcommand{\codalphazero}[2]{%
        \node[sbox] at (#1,#2+0.6) {$\varepsilon_0$};
        \node[sbox] at (#1,#2) {$\eta_0$};
        \node[sbox] at (#1,#2-0.6) {$=$};
        }
        \newcommand{\codalphaone}[2]{%
        \node[sbox] at (#1,#2+0.6) {$=$};
        \node[sbox] at (#1,#2) {$\varepsilon_1$};
        \node[sbox] at (#1,#2-0.6) {$\eta_1$};
        }

        \node[box] (A) at (0,0) {$\phi$};
        \node[box] (B) at (2.4,0) {$\phi$};
        \draw[line width=0.6pt, double, double distance=1.2pt, shorten <=1ex, shorten >=1ex] (A.east) -- (B.west);
        \node[box] (C) at (6.6,0) {$\phi$};
        \codalphazero{7.25}{0}
        \draw[triple, shorten <=1ex, shorten >=1ex] (B.east) -- node[above,label] {$\alpha_0$} (C.west);

        \node[box] (D) at (0,-2.6) {$\phi$};

        \node[box] (E1) at (2.4,-2.6) {$\phi$};
        \node[box] (E2) at (3.05,-2.6) {$\phi^{L}$};
        \node[box] (E3) at (3.7,-2.6) {$\phi$};

        \draw[triple, shorten <=1ex, shorten >=1ex] (B) -- node[left ,label] {$\eta_\phi$} (E1);
        \draw[triple, shorten <=1ex, shorten >=1ex] (E1.west) -- node[below,label] {$\varepsilon_\phi$} (D.east);
        \draw[line width=0.6pt, double, double distance=1.2pt, shorten <=1ex, shorten >=1ex] (A.south) -- (D.north);

        \node[box] (F1) at (6.6,-2.6) {$\phi$};
        \codalphazero{7.25}{-2.6}
        \node[box] (F2) at (7.9,-2.6) {$\phi^{L}$};
        \node[box] (F3) at (8.55,-2.6) {$\phi$};

        \draw[triple, shorten <=1ex, shorten >=1ex] (C) -- node[left,label] {$\eta_\phi$} (F1);
        \draw[triple, shorten <=1ex, shorten >=1ex] (E3.east) -- node[above,label] {$\alpha_0$} (F1.west);

        \node[box] (G1) at (-.65,-5.2) {};
        \codalphaone{-.65}{-5.2}
        \node[box] (G2) at (0,-5.2) {$\phi$};

        \draw[triple, shorten <=1ex, shorten >=1ex] (D) -- node[left,label] {$\alpha_1$} (G2);

        \node[box] (H1) at (2.4,-5.2) {$\phi$};
        \node[box] (H2) at (3.05,-5.2) {$\phi^{L}$};
        \node[box] (H3) at (3.7,-5.2) {};
        \codalphaone{3.7}{-5.2}
        \node[box] (H4) at (4.35,-5.2) {$\phi$};

        \draw[triple, shorten <=1ex, shorten >=1ex] (E2) -- node[left,label] {$\alpha_1$} (H2);
        \draw[triple, shorten <=1ex, shorten >=1ex] (H1.west) -- node[below,label] {$\varepsilon_\phi$} (G2.east);

        \node[box] (I1) at (6.6,-5.2) {$\phi$};
        \codalphazero{7.25}{-5.2}
        \node[box] (I2) at (7.9,-5.2) {$\phi^{L}$};
        \node[box] (I3) at (8.55,-5.2) {};
        \codalphaone{8.55}{-5.2}
        \node[box] (I4) at (9.2,-5.2) {$\phi$};

        \draw[triple, shorten <=1ex, shorten >=1ex] (F1) -- node[left,label] {$\alpha_1$} (I1);
        \draw[triple, shorten <=1ex, shorten >=1ex] (I1.west) -- node[above,label] {$\alpha_0^{-1}$} (H4.east);
        \draw[quadruple] (A) -- node[label, midway, above, sloped, auto] {$\alpha_\phi$} (E1);
        \end{tikzpicture}\]
        The construction of $\tilde{\beta}$ reduces to defining a $4$-cell $C_4\to \cube^1\otimes C_3\to P$ with the domain
        \[C_3\to \cube^1\otimes C^-_2\cup_{\{1\}\otimes C^-_2} \{1\}\otimes C_3\xrightarrow{\id_{\tilde{l}}\cup \beta_1} P\] and the codomain 
        \[C_3\to \cube^1\otimes C^+_2\cup_{\{0\}\otimes C^+_2} \{0\}\otimes C_3\xrightarrow{(\tilde{\eps}\tilde{l})(\tilde{l}\tilde{\eta})\cup\beta_0} P.\]
        This $4$-cell can be filled as in the following diagram. The homotopies between $3$-cells labeled $\alpha_i$ and $\beta_i$ are provided by the higher triangle identity. Namely, an L-shaped gluing of the unit and counit can be simplified by either $\alpha$ or $\beta$, and the next adjunction datum provides a homotopy between these two.
            
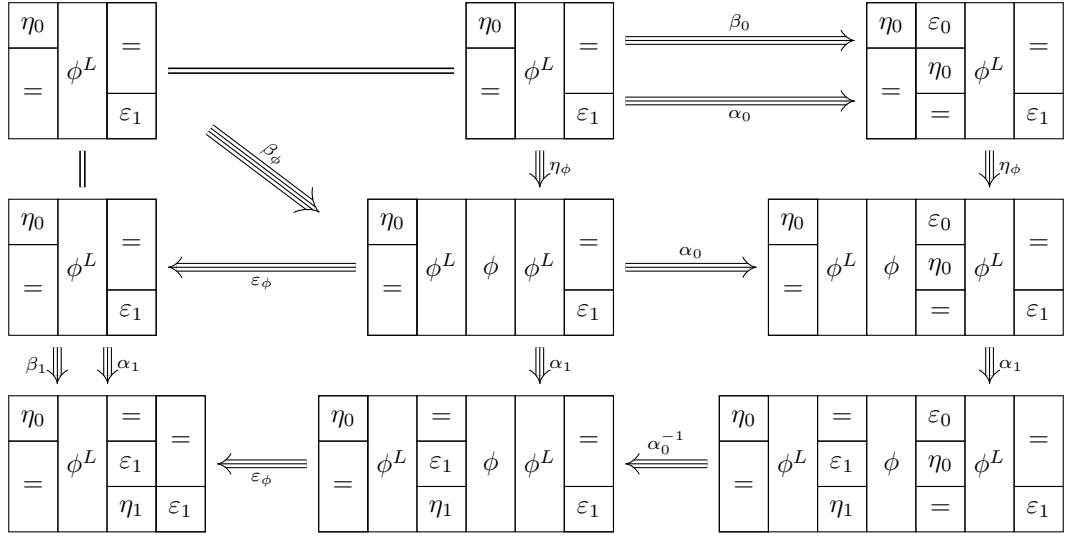
\[\begin{tikzpicture}[
            x=1cm,y=1cm,
            box/.style={draw, minimum width=.65cm, minimum height=1.8cm, inner sep=0pt},
            sbox/.style={draw, minimum width=.65cm, minimum height=.6cm, inner sep=0pt},
            dbox/.style={draw, minimum width=.65cm, minimum height=1.2cm, inner sep=0pt},
            label/.style={font=\scriptsize},
            short/.style={shorten <=1ex, shorten >=1ex}
            ]
            \newcommand{\codalphazero}[2]{%
            \node[sbox] at (#1,#2+0.6) {$\varepsilon_0$};
            \node[sbox] at (#1,#2) {$\eta_0$};
            \node[sbox] at (#1,#2-0.6) {$=$};
            }
            \newcommand{\codalphaone}[2]{%
            \node[sbox] at (#1,#2+0.6) {$=$};
            \node[sbox] at (#1,#2) {$\varepsilon_1$};
            \node[sbox] at (#1,#2-0.6) {$\eta_1$};
            }
            \newcommand{\leftunit}[2]{%
            \node[sbox] at (#1,#2+0.6) {$\eta_0$};
            \node[dbox] at (#1,#2-0.3) {$=$};
            }
            \newcommand{\rightcounit}[2]{%
            \node[dbox] at (#1,#2+0.3) {$=$};
            \node[sbox] at (#1,#2-0.6) {$\varepsilon_1$};
            }
            \leftunit{-1.95}{0}
            \node[box] (A2) at (-1.3,0) {$\phi^L$};
            \node[box] (A3) at (-0.65, 0) {};
            \rightcounit{-0.65}{0}

            \node[box] (B1) at (4.1, 0) {};
            \leftunit{4.1}{0}
            \node[box] (B2) at (4.75,0) {$\phi^L$};
            \node[box] (B3) at (5.4, 0) {};
            \rightcounit{5.4}{0}

            \draw[line width=0.6pt,double,double distance=1.2pt,short]
            (A3) -- (B1);

            \node[box] (C0) at (9.4, 0) {};
            \leftunit{9.4}{0}
            \codalphazero{10.05}{0}
            \node[box] (Cphi) at (10.7,0) {$\phi^L$};
            \rightcounit{11.35}{0}

            \draw[triple,short]
            (B3.east |- 0, 0.4) -- 
            node[above,label] {$\beta_0$}
            (C0.west |- 0, 0.4);

            \draw[triple,short]
            (B3.east |- 0, -0.4) --
            node[below,label] {$\alpha_0$}
            (C0.west |- 0, -0.4);

            \leftunit{-1.95}{-2.6}
            \node[box] (Dphi) at (-1.3,-2.6) {$\phi^L$};
            \node[box] (D3) at (-0.65, -2.6) {};
            \rightcounit{-0.65}{-2.6}

            \draw[line width=0.6pt,double,double distance=1.2pt,short]
            (A2.south) -- (Dphi.north);

            \node[box] (E0) at (2.8, -2.6) {};
            \leftunit{2.8}{-2.6}
            \node[box] (E1) at (3.45,-2.6) {$\phi^L$};
            \node[box] (E2) at (4.1,-2.6) {$\phi$};
            \node[box] (E3) at (4.75,-2.6) {$\phi^L$};
            \node[box] (E4) at (5.4, -2.6) {};
            \rightcounit{5.4}{-2.6}

            \draw[triple,short]
            (B2.south) --
            node[right,label] {$\eta_\phi$}
            (E3.north);

            \draw[triple,short]
            (E0.west) --
            node[below,label] {$\eps_\phi$}
            (D3.east);

            \draw[quadruple]
            (A3) --
            node[midway, above,sloped,label, auto] {$\beta_\phi$}
            (E0);

            \node[box] (F0) at (8.1, -2.6) {};
            \leftunit{8.1}{-2.6}
            \node[box] (F1) at (8.75,-2.6) {$\phi^L$};
            \node[box] (F2) at (9.4,-2.6) {$\phi$};
            \codalphazero{10.05}{-2.6}
            \node[box] (F3) at (10.7,-2.6) {$\phi^L$};
            \rightcounit{11.35}{-2.6}

            \draw[triple,short]
            (Cphi.south) --
            node[right,label] {$\eta_\phi$}
            (F3.north);

            \draw[triple,short]
            (E4.east) --
            node[above,label] {$\alpha_0$}
            (F0.west);

            \leftunit{-1.95}{-5.2}
            \node[box] (G1) at (-1.3,-5.2) {$\phi^L$};
            \codalphaone{-0.65}{-5.2}
            \node[box] (G2) at (0,-5.2) {};
            \rightcounit{0}{-5.2}

            \draw[triple,short]
            (Dphi.south west) --
            node[left,label] {$\beta_1$}
            (G1.north west);

            \draw[triple,short]
            (Dphi.south east) --
            node[right,label] {$\alpha_1$}
            (G1.north east);

            \node[box] (H0) at (2.15, -5.2) {};
            \leftunit{2.15}{-5.2}
            \node[box] (H1) at (2.8,-5.2) {$\phi^L$};
            \codalphaone{3.45}{-5.2}
            \node[box] (H2) at (4.1,-5.2) {$\phi$};
            \node[box] (H3) at (4.75,-5.2) {$\phi^L$};
            \node[box] (H4) at (5.4, -5.2) {};
            \rightcounit{5.4}{-5.2}

            \draw[triple,short]
            (E3.south) --
            node[right,label] {$\alpha_1$}
            (H3.north);

            \draw[triple,short]
            (H0) --
            node[below,label] {$\eps_\phi$}
            (G2);

            \node[box] (I0) at (7.45, -5.2) {};
            \leftunit{7.45}{-5.2}
            \node[box] (I1) at (8.1,-5.2) {$\phi^L$};
            \codalphaone{8.75}{-5.2}
            \node[box] (I2) at (9.4,-5.2) {$\phi$};
            \codalphazero{10.05}{-5.2}
            \node[box] (I3) at (10.7,-5.2) {$\phi^L$};
            \rightcounit{11.35}{-5.2}

            \draw[triple,short]
            (F3.south) --
            node[right,label] {$\alpha_1$}
            (I3.north);

            \draw[triple,short]
            (I0.west) --
            node[above,label] {$\alpha_0^{-1}$}
            (H4.east);

            \end{tikzpicture}\]
            The left edge composes is the domain of $\tilde{\beta}$. The top, right, bottom edge composes to the codomain of $\tilde{\beta}$. The top-left square commutes by triangle identity $\beta_\phi$ and the other squares commute by the interchange law. 
    \end{proof}
    
    \section{The cobordism hypothesis}\label{section_cobordism_hypothesis}
    The goal of this section is to review the basic form of the cobordism hypothesis and translate it into the context of categorical spectra.
    \begin{notation}
    Recall the straightening-unstraightening equivalence $\colim: \Fun(\rB \rO(n), \sS)\xrightarrow{\sim} \sS_{/\rB \rO(n)}$, which takes a groupoid $\tilde{X}$ with an $\rO(n)$-action to its quotient $X \simeq \tilde{X}/\rO(n) = \tilde{X}_{h\rO(n)}$. 
    The inverse is given by pulling back the universal $\rO(n)$-torsor $\ast\to \rB\rO(n)$ along $X\to \rB\rO(n)$. 
    The data is also equivalent to the pair of a CW complex representing the groupoid $X$ and the real vector bundle $\zeta\simeq \RR^n\times_{\rO(n)} \tilde{X}$ of rank $n$ with inner product, so we sometimes use the notation $\tilde{X}$ and $(X, \zeta)$ interchangeably or let $\zeta$ denote the classifying map $X\to \rB \rO(n)$ itself. 
    These equivalent data will be called \emph{tangential structures}. Despite the notation $\rO(n)$, unless we specifically refer to a vector bundle on a topological space, we will not need a topological group structure on $\rO(n)$; we only regard it as a group object in $\sS$. 
    \end{notation}
    Let $m\leq n$, let $M$ be a smooth $m$-dimensional manifold, and let $\tilde{X}$ be a tangential structure. Let $\tau_M: M\to \rB\rO(m)$ be the classifying map of the tangent bundle. A $\tilde{X}$-structure, or $(X, \zeta)$-structure, on $M$ is a commutative diagram in $\sS$ lifting the tangent classifier
    \[\begin{tikzcd}[column sep=25mm]
        & X\ar[d, "\zeta"]\\
        M\ar[r, "\tau_M\oplus \RR^{n-k}"']\ar[ru, dashed] & \rB\rO(n).
    \end{tikzcd}\]
    
    \begin{example}
        The most basic tangential structure is $X\to \rB\rO(n) = \id$, $\tilde{X} = \ast$. This is no additional structure, i.e., the structure of \emph{unoriented manifolds}. 
        Another fundamental case is $X=\ast\to \rB\rO(n)$, $\tilde{X} = \rO(n)$. In this case, a $\tilde{X}$-structure gives a \emph{framing}, i.e., an identification of the tangent bundle with the constant bundle of the reference point of $\rB\rO(n)$.
        We will use the superscript $\fr$ to indicate framing and write $\un$ (or omit the notation for tangential structure) for unoriented manifolds. 
    \end{example}
    Note that the groupoidification $\abs{C_k}$ of the cell $C_k$ admits a canonical stratified manifold structure as a quotient of the cube $[0, 1]^k$ by cylindrically collapsing appropriate faces.
    In particular, it comes equipped with a natural framing on every stratum in a compatible manner. Such manifold realization extends to $\Theta$ in a functorial way. 
    In the following, it is helpful to think of them as secretly being extended to the ``germ'' beyond the boundary, so that the objects restricted to the boundary come equipped with some ``glue.''
    Roughly speaking, the \emph{$n$-dimensional bordism category $\Bord^{\tilde{X}}_n$ with $\tilde{X}$-structure} is an $n$-category whose value on $\theta\in \Theta_n$ classifies finite submersive bundles of $\tilde{X}$-manifolds on $\abs{\theta}$ whose $\tilde{X}$-structure is compatible across stratifications. In particular, we have the following description of cells:
    \begin{enumerate}
        \setcounter{enumi}{-1}
        \item An object of $\Bord_n^{\tilde{X}}$ is a finite set of points $x = \{x_1, \ldots, x_n\}$ of $X$ equipped with identifications $(\zeta(x_i)\simeq \RR^n)\in\rO(n)$. 
        \item A $k$-cell for $1\leq k\leq n$ from $M_0$ to $M_1$ is a $k$-dimensional $\tilde{X}$-manifold $W$ together with identifications of the boundaries $(\abs{C_{k-1}}\xrightarrow{s}\abs{C_k})^\ast W\simeq M_0$, $(\abs{C_{k-1}}\xrightarrow{t}\abs{C_k})^\ast W \simeq M_1$, i.e., a cobordism from $M_0$ to $M_1$ with $\tilde{X}$-structures. 
        \item When $k> n$, one may think of a $k$-morphism as a bundle on $\abs{C_k}$ induced from one on $\abs{C_n}$ via the projection $C_n\to C_k$; these are the trivial cobordism given by cylinders (note, however, that the identifications with the boundary has the freedom of diffeomorphisms).
    \end{enumerate}
    Moreover, we equip $\Bord^{\tilde{X}}_n$ with a symmetric monoidal structure by disjoint unions of manifolds. 
    Giving a precise definition of the symmetric monoidal $(\infty, n)$-category of cobordisms is a nontrivial task. At least one has to encode the compatibility of tangential structures across different strata systematically, which may be done by considering collars or developing the theory of stratified manifolds. 
    We refer the reader to \cite{calaqueNote$inftyn$categoryCobordisms2019}\cite{ayalaCobordismHypothesis2017}\cite{gradyExtendedFieldTheories2023} for some constructions in the literature. 
    One should also note that the above naive definition ends up in a non-univalent category; the underlying groupoid consists of trivial cobordisms (i.e., cylinders), but there are nontrivial invertible cobordisms (these are roughly the h-cobordisms), so we usually apply the univalent completion (this inverts the difference between the smooth and $\rP\rL$ cobordism categories, for instance). 
    \begin{remark}
        One should define the cobordism categories in a way that clearly depends functorially on $X\in \sS_{/\rB\rO(n)}$. In particular, the framed bordism $n$-category $\Bord_n^\fr$ admits an action of $\Aut(\ast\to \rB\rO(n)) = \rO(n)$. It is a folklore result that $\Aut(\Bord_n^\fr) = \rP\rL(n)$, except for the unknown case $n=4$, but we will stick to the $\rO(n)$-action for simplicity. 
    \end{remark}
    Modulo the problem of definition, it is not difficult to show that the objects of cobordism categories are fully dualizable; for $k<n$, the dual of a $k$-morphism is the same manifold with the opposite $\tilde{X}$-structure.
    The cobordism hypothesis \cite{baezHigherDimensionalAlgebraII1996} states the universal property of the cobordism category as the free symmetric monoidal $n$-category with duals.
    \begin{hypothesis}\label{cobordism_hypothesis_plainver}
        The forgetful functor $\CMon(n\Cat)^{\dual}\hookrightarrow \CMon(n\Cat)\xrightarrow{(\blank)^{\leq 0}} \sS$ is corepresented by the \emph{framed bordism category} $\Bord_n^\fr$. In particular, the functor admits an action of $\rO(n)$, so the above functor canonically factors through the category of groupoids with $\rO(n)$-action:
        \[\begin{tikzcd}[column sep=large]
            n\SMC^\dual\ar[d, hook]\ar[r, dashed]  & \sS_{/\rB\rO(n)} \ar[r, "X\mapsto \tilde{X}", "\sim"'] & \Fun(\rB\rO(n), \sS)\ar[d, "\ev_\ast"]\\
            n\SMC \ar[rr, "(\blank)^{\leq 0}"] && \sS.
        \end{tikzcd}\]
        Moreover, the left adjoint of $\CMon(n\Cat)^{\dual}\to \Fun(\rB\rO(n), \sS)$ sends the object $\tilde{X}$ to the $n$-dimensional cobordism category $\Bord_{n}^{\tilde{X}}$, i.e., $\Bord_n^{\tilde{X}}\simeq \tilde{X}\otimes_{\rO(n)}\Bord_n^{\fr}$. 
    \end{hypothesis}
    \begin{remark}
        The second statement is equivalent to saying that $\tilde{X}\to \Bord^{\tilde{X}}_n$ is colimit-preserving, i.e., satisfies \emph{descent}. 
        This is expected from the \emph{locality} of fully extended cobordism categories; one can cut a manifold into small bordisms and assemble information on the original manifold from the restricted pieces. In the current situation, it essentially means that one can assemble a tangential structure on a manifold from local trivializations of the tangent bundle by asking for some properties/structures on the transition functions and cocycles. This perspective is central in \cite{gradyExtendedFieldTheories2023}. The unoriented case is, in some sense, universal: by unstraightening the framed bordism category as an $\rO(n)$-equivariant object $\Bord_n^\fr: \rB\rO(n)\to n\SMC$, one gets a coCartesian fibration $\Bord_n\coloneqq \Bord_n^\un = \int \Bord_n^\fr \to \rB\rO(n)$; intuitively, one obtains an unoriented bordism category from the framed bordism category by taking the orbits under change of framings.
        All the other cases are base changes of this fibration along the structure map $X\to \rB\rO(n)$. One may think of this construction as the \emph{categorical Madsen--Tillman spectrum}.
    \end{remark}
    \begin{remark}
        The cobordism hypothesis is widely believed, but there is no consensus on a rigorous proof yet, as far as the author knows (\cite{lurieClassificationTopologicalField2009} gives a sketch of a proof, and \cite{ayalaCobordismHypothesis2017} gives a proof assuming a fundamental conjecture on factorization homology of solidly framed manifolds. \cite{gradyGeometricCobordismHypothesis2022} claims a rigorous proof and generalization to the geometric context, parametrized by the site of manifolds, but peer review is still in progress). \textbf{From now on, we will work conditionally on this formulation of the cobordism hypothesis.}    
    \end{remark}
    Let us restate the cobordism hypothesis using categorical spectra. 
    As the functor $(\Omega^{\infty})^{\leq d}: \CatSp^{d\mhy\adj}\to \CMon(d\Cat)$ lands in $\CMon(d\Cat)^\dual$, the underlying groupoid functor $(\Omega^\infty)^{\leq 0}: \CatSp\to 0\CatSp^\cn \simeq \CMon(\sS)\to \sS$ lifts to $\Fun(\rB \rO(n), \sS)$. 
    \begin{corollary}
        The functor $\underline{\Omega}^{\infty-d}: \CatSp^{0\mhy\adj}\xrightarrow{\Omega^{\infty-d}} \CMon(\infty\Cat)\xrightarrow{(\blank)^{\leq 0}} \sS$ is represented by $\rB^{\infty-d}\Bord_d^\fr$ and factors through $\sS_{/\rB\rO(n)}$. 
        The left adjoint $\sS_{/\rB\rO(n)}\to \CatSp^{0\mhy\adj}$, formally given by $X\mapsto \tilde{X}\otimes_{\rO(n)} L_0^\adj\FF[-n]$, admits a description as the cobordism categorical spectrum $\rB^{\infty-d}\Bord_d^X$.  
    \end{corollary}
    The translation is superficial, but in a sense, this better reflects the intuition of the cobordism hypothesis. Via the cobordism hypothesis, the framed cobordism $n$-category is described as the free symmetric monoidal $n$-category with duals generated by a point. However, this point secretly looks like a germ of $\RR^n$, and there are $\rO(n)$-worth of such points. In the above, this generator is placed in dimension $-n$.
    For example, targets of TQFTs become dimension-independent in this way: for any $d$, one can write $\{\rB^{\infty-d}\Bord_d^{X}\to \underline{\CC}\}$ to mean $\Bord_d\to d\mathsf{Vect}_{\CC}$, where $\underline{\CC}$ is the categorical spectrum $\{n\Mod_\CC\}_n$ of \cite{stefanichHigherQuasicoherentSheaves2021}.
    Moreover, we can organize the framed cobordism categories of all dimensions into a single algebra in $0\CatSp^{0\mhy\adj}$:
    \begin{corollary}\label{multiplicative_cobordism_hypothesis}
        The categorical spectrum $\bigoplus_{n\geq 0}\rB^{\infty-n}\Bord^\fr_n$ admits the structure of a tensor algebra on $\rB^{\infty-1}\Bord_1^\fr$ in the monoidal category $\CatSp^{0\mhy\adj}$\footnote{Shai Keidar informed the author that this is already a tensor algebra in $\CatSp$. It follows from the improvement of the cobordism hypethesis that describes $\Fun^\lax(\Bord_n^\fr, \eC)$ \cite[Theorem 7.6]{johnson-freydOplaxNaturalTransformations2017}. This theorem is a corollary of \cref{pushout_formula_for_cylinder_on_Adj}.}. 
    \end{corollary}
    \begin{proof}
        The first claim is equivalent to the cobordism hypothesis by applying the monoidal localization $L^{0\mhy\adj}: \CatSp\to \CatSp^{0\mhy\adj}$ to the tensor algebra $\operatorname{Tens}(\FF[-1]) = \bigoplus \FF[-n]$. 
    \end{proof}
    \begin{remark}
        Intuitively, the $\EE_1$-rig structure on $\bigoplus_{n\geq 0}\rB^{\infty-n}\Bord_n^\fr$ is given by disjoint union and Cartesian product of manifolds. 
        As soon as such a Cartesian product functor $\Bord_m^\fr\otimes \Bord_n^\fr\to \Bord_{m+n}^\fr$ is shown to be well-defined, the universal property shows that it is the unique functor that sends the pair of reference codimension-$m$ and codimension-$n$ points to the reference codimension-$(m+n)$ point, but at the moment of writing, the author is not aware of a way that does not go back to a construction of the cobordism category. 
    \end{remark}
    \begin{remark}
        The cobordism hypothesis and its suggested proofs have purely categorical consequences that can be stated without mentioning manifolds. One is the $\rO(n)$-action on the underlying groupoid of a symmetric monoidal $n$-category with duals. This action remains elusive without the framed cobordism hypothesis, forcing us to state the cobordism hypothesis in two steps. Combining this with \cref{theorem_tensor_product_localizes_to_adj}, the groupoid of $k$-cells in a symmetric monoidal $n$-category with duals acquires an $\rO(n-k)$-action.
        This fact is mentioned in \cite[\S 4.3]{lurieClassificationTopologicalField2009} without proof and used to formulate the cobordism hypothesis with singularities. 

        Another categorical takeaway from Lurie's proof sketch is an explicit finite-step pushout description of $L^{(-1)\mhy\adj}\FF[-n]\to L^{0\mhy\adj}\FF[n]$. A priori, it is unclear whether the localization can be reached after finitely many pushouts.
        More precisely, in \cite[\S 3.4]{lurieClassificationTopologicalField2009}, Lurie describes the inclusion $\rB^{\infty-n}\Bord_{n-1}\to \rB^{\infty-n}\Bord_n$
        by introducing the \emph{index filtration}: 
        $\rB^{\infty-n}\Bord_{n-1} = \cF_{-1}\hookrightarrow \cF_{0}\hookrightarrow \cF_1 \hookrightarrow \cdots \hookrightarrow\cF_{n}=\rB^{\infty-n}\Bord_{n}$.
        Note that $\cF_{-1}\hookrightarrow\cF_{n}$ is the reflection under the inclusion $0\CatSp^{0\mhy\adj}\hookrightarrow 0\CatSp^{(-1)\mhy\adj}$.
        Roughly speaking, at least non-univalently, the category $\cF_k$ has the same $0, 1, \ldots, (n-1)$-cells as $\Bord_{n-1}$ and allows the $n$-cells (i.e., cobordisms) of $\Bord_{n}$ that can be built up from handles of index at most $k$; here we see the time coordinate of the cobordism as a (generalized) Morse function. Each step $\cF_{k-1}\hookrightarrow\cF_{k}$ is generated by the single $\rO(n-k)$-equivariant $n$-cell corresponding to the $k$-handle, subject to the relation of cancellation of a $(k-1)$-handle and a $k$-handle, as well as the ``triangle identity'' of cancellation. The case $k=0$ is particularly simple and stated as the lax cofiber sequence: 
        \[\begin{tikzcd}
        \Sigma^{\infty-1}_+ \rB\rO(n) \ar[r, "S^{n-1}"]\ar[d] & \rB^{\infty-n}\Bord_{n-1}\ar[d] \\
        0 \ar[r]\ar[ru, Rightarrow, shorten <=4ex, shorten >=2ex, "D^n"] & \cF_0
        \end{tikzcd}\]
        In other words, $\cF_0$ is an extension of $\Sigma_+^\infty \rB\rO(n)$ by $\rB^{\infty-n}\Bord_{n-1}$ classified by the standard $\rO(n)$-action on the $(n-1)$-sphere \[S^{n-1}\in \Map_{\sS}(\rB\rO(n), \Omega^{n-1}\Bord_{n-1})\simeq \Ext(\Sigma^\infty_+\rB\rO(n), \rB^{\infty-n}\Bord_{n-1}).\] 
        \Cref{n_adjointable_catsp_are_closed_under_extensions} ensures that this extension does not break the previous levels of adjointfulness. 
    \end{remark}

    \section{Cobordism hypothesis with singularities}
    \label{section_cobordism_hypothesis_with_singularities}
    As we saw in the last section, the cobordism hypothesis gives a geometric description of the categorical spectra with adjoints freely generated by a $\rO(n)$-equivariant groupoid of $(-n)$-cells.
    In this section, we study a generalization called the \emph{cobordism hypothesis with singularities}, sketched by Lurie in \cite[\S 4.3]{lurieClassificationTopologicalField2009}; 
    we describe the cobordism category with certain types of conical singularities (a.k.a.\ defects) allowed. It turns out that such cobordism categories/categorical spectra arise as \emph{cell complexes} in categorical spectra with adjoints, i.e., as iterated extensions of cobordism categorical spectra of different codimensions, and the singularity types precisely classify the extensions. 
    The goal of this section is to make Lurie's sketch of the argument precise, filling the unproven categorical claims using our previous results. We first make the following definition and later check that it qualifies for its name.

    \begin{definition}
        A \emph{cobordism categorical spectrum with singularities} is a categorical spectrum $B^0$ that fits into the following sequence, where  $d\geq 0$, $X^k\in\sS_{/\rB\rO(k)}$, and $B^k$ is an extension (resp.\ coextension) of $\rB^{\infty-k}\Bord_k^{\tilde{X}^k}$ by $B^{k+1}$  when $k$ is odd (resp.\ even):
        \[\begin{tikzcd}[column sep=small]
            0= B^{d+1}\ar[r, tail] & B^d\ar[r, tail]\ar[d, two heads] & B^{d-1}\ar[r, tail]\ar[d, two heads] & \cdots \ar[r, tail] & B^1\ar[r, tail]\ar[d, two heads] & B^0\ar[d, two heads] \\
            & \rB^{\infty-d} \Bord_d^{\tilde{X}^d} & \rB^{\infty-(d-1)}\Bord_{d-1}^{\tilde{X}^{d-1}} & &\rB^{\infty-1} \Bord_1^{\tilde{X}^1} & \rB^{\infty}\Bord_0^{\tilde{X}^0}
        \end{tikzcd}\]
        A \emph{singularity datum} for $B^0$ is the sequence $\vec{X} = (X^0, E^0, \ldots, X^{d-1}, E^{d-1}, X^d)$ where $B^k$ is classified by $E^k\in \Ext(\rB^{\infty-k}\Bord_k^{\tilde{X}^k}, B^{k+1})$. 
        In this case, we denote $\Bord^{\vec{X}}_d\coloneqq \Omega^{\infty-d} B^0$ and call it the \emph{cobordism category of $\vec{X}$-manifolds}.
    \end{definition}
    \begin{remark}
        The distinction between extensions and coextensions is not essential under the existence of adjoints. The above definition is designed to make the statement of \cref{cobordism_hypothesis_with_singularities} cleaner.
    \end{remark}
    \begin{remark}
        In the above definition, if $X^d=\emptyset$, we automatically have $B^d=0$ and $E^{d-1}=0$, so deleting $E^{d-1}, X^d$ from $\vec{X}$ does not change the resulting $B^0$.
        Also, since  $E^k$ only inductively depends on the entries on the right, it makes sense to consider a singularity datum $X^{\geq k}$ where $X^i$ and $E^i$ for $i< k$ are replaced by $\emptyset$ and $0$, respectively. The resulting sequence of extensions is $0\rightarrowtail B^d\rightarrowtail \cdots B^k\xrightarrow{=} \cdots\xrightarrow{=} B^k$, so $B^k$ is also a cobordism categorical spectrum with singularities. 
        Combining both considerations, when $X^i=\emptyset$ unless $k\leq i\leq l$, we may simply denote the singularity datum $\vec{X}\coloneqq \vec{X}^{[k, l]}\coloneqq (X^k, E^k, \ldots, E^{l-1}, X^l)$. Notice that the categorical spectrum $B^0$ does not depend on the choice of $d(\geq l)$, while $\Bord_d^{\vec{X}} =\Omega^{\infty-d}B^0$ does.
    \end{remark}
    \begin{remark}
        By \cref{n_adjointable_catsp_are_closed_under_extensions}, $B^k$  is $0$-adjointful, $0$-categorical, and $(-d)$-connective. In particular, $B^k = \rB^{\infty-d}\Bord_d^{\vec{X}^{\geq k}}$.
        By the cobordism hypothesis, one computes the Ext monoid as follows: 
        \begin{align*}
            \Ext(\rB^{\infty-k}\Bord_k^{\tilde{X}^k}, B^{k+1})&\simeq \Map(\tilde{X}^k \otimes_{\rO(k)} L^{0\mhy\adj}\FF[-k], \Sigma B^{k+1})\\
            &\simeq \Map_{\rO(k)}(\tilde{X}^k, \underline{\Omega}^{\infty-k-1}B^{k+1}).
        \end{align*}
        Here the $\rO(k+1)$-action on the codomain is restricted to $\rO(k)$ by the inclusion $\rO(k)\subset \rO(k+1)$. 
        By definition, this classifies an $\rO(k)$-equivariant local system of $(k+1)$-dimensional $\vec{X}^{\geq k+1}$-manifolds. By abuse of notation, we continue to denote the corresponding local system by $E^k$. Intuitively, the bundle $E^k\to \tilde{X}^k$ describes the universal \emph{link bundle} of the codimension $k$ stratum inside the whole manifold with singularity.
        Given a topological model of $\tilde{X^k}$, this should give %
        an honest bundle that continuously and equivariantly assigns an $X^{\geq k+1}$-manifold $E^k(\tilde{x})$ to a point $\tilde{x}\in\tilde{X}^k$. However, to achieve this, it seems necessary to construct the cobordism categories not merely as symmetric monoidal $d$-categories but as smooth stacks.
    \end{remark}
    \begin{theorem}[Cobordism hypothesis with singularities {\cite[Theorem 4.3.11]{lurieClassificationTopologicalField2009}}]\label{cobordism_hypothesis_with_singularities}
        Let $d> k\geq 0$, let $\vec{X} = (X^k, E^k, \ldots, X^d)$ be a singularity datum, and let $\vec{X'} = \vec{X}^{\geq k+1}$. 
        For any $0$-adjointful categorical spectrum $A = (A_n)$, there is a Cartesian square 
        \[\begin{tikzcd}
            \Map_{\CatSp}(\rB^{\infty-d}\Bord_d^{\vec{X}}, A)\ar[r]\ar[d] & \Map_{\CatSp}(\rB^{\infty-d}\Bord_d^{\vec{X'}}, A)\ar[d, "(E^k)^\ast"]\ar[r, phantom, "\ni"] & Z_0\ar[d, mapsto] \\
            \Map_{\rO(k)}(\tilde{X}^k, \Alg_{\EE_0}(
                A_{k+1})^{\leq 0}) \ar[r] & \Map_{\rO(k)}(\tilde{X}^k, A_{k+1}^{\leq 0})\ar[r, phantom, "\ni"] & \Omega^{\infty-k-1}Z_0\circ E^k.
        \end{tikzcd}\]
        
    \end{theorem}
    \begin{proof}
        By \cref{extension_of_catsp}, there is an (op)lax cofiber sequence $\rB^{\infty-k-1}\Bord_k^{\tilde{X}^k}\to B^{k+1} \to B^k$, or a pushout diagram
        \[\begin{tikzcd}
            \rB^{\infty-k-1}\Bord_k^{\tilde{X}^{k}} \ar[r]\ar[d] & B^{k+1}\ar[d] \\
            \Sigma^\infty D^{k+1} I \otimes\rB^{\infty-k-1}\Bord_k^{\tilde{X}^{k}} \ar[r] & B^k.
        \end{tikzcd}\]
        Mapping into $A$ and applying the cobordism hypothesis, one obtains
        \[\begin{tikzcd}
            \Map(B^k, A)\ar[r]\ar[d]\ar[rd, phantom, "\lrcorner", near start] & \Map(B^{k+1}, A)\ar[d] \\
            \Map_{k\SMC^\dual}(\Bord_k^{\tilde{X}^k}, \Omega^{\infty-k-1}[\Sigma^\infty I, A])\ar[r]\ar[d, "\simeq"] & \Map_{k\SMC^\dual}(\Bord_k^{\tilde{X}^k}, \Omega^{\infty-k-1}A)\ar[d, "\simeq"]\\
            \Map_{\rO(k)}(\tilde{X}^k, \underline{\Omega}^{\infty-k-1}[\Sigma^\infty I, A])\ar[r] & \Map_{\rO(k)}(\tilde{X}^k, \underline{\Omega}^{\infty-k-1}A)
        \end{tikzcd}\]
        Now observe that $\Omega^{\infty-k-1}[\Sigma^\infty D^{k+1} I, A]\simeq \Omega^\infty[\Sigma^\infty I, \Sigma^{k+1} A]$ so the underlying groupoid is \[\Map_\ast(I, \underline{\Omega}^{\infty-k-1}A)\simeq \Alg_{\EE_0}(A_{k+1})^{\leq 0}.\]
    \end{proof}

    \section{Stable tangential structures}\label{section_stable_variant}
    In this short section, we investigate $\CatSp^{\infty\mhy\adj}$.
    Having adjoints for all cells is a strong condition; for instance, such a category is coinductively equivalent to a groupoid.
    As a consequence, this category exhibits behavior closer to the category of spectra.

    A \emph{stable tangential structure} is an object of $\sS_{/\rB\rO}$.
    For $X\to \rB\rO\in \sS$, one can associate a sequence $\{X_n\}$ of finite-dimensional tangential structures by $X_n \coloneqq X \times_{\rB\rO}\rB\rO(n)$. Conversely, if a sequence $X_n\to \rB\rO(n)$ and homotopies $X_n\simeq X_{n+1}\times_{\rB\rO(n+1)}\rB\rO(n)$ are given, then one recovers $X$ as $X\simeq \colim_n X_n$.
    This data is also equivalent to a sequence of $\rO(n)$-equivariant groupoids $\tilde{X}^n$ with $\rO(n)$-equivariant equivalences $\tilde{X}^n\xrightarrow{\sim} \tilde{X}^{n+1}$. In the limit, the common underlying space $\tilde{X}$ acquires the $\rO = \rO(\infty)$-action, whose quotient is $X\to \rB\rO$. 
    The cobordism hypothesis implies that there is a sequence of cobordism categories that extends the assignment $\tilde{X}^n\xrightarrow{\sim}\tilde{X}^{n+1}$: 
    \[\Bord_0^{X^0}\to \Bord_1^{X^1} \to\Bord_2^{X^2}\to \cdots\]
    The colimit is the cobordism $\infty$-category $\Bord_\st^X$ of \emph{stably $X$-manifolds}.
    This is a natural example of an $(\infty, \infty)$-category with adjoints that is not truncated at any finite level. 
    We denote the stably framed case, when $X= \ast$, by $\Bord^{\fr}_\st$. 
    The following is the limit of the $d$-dimensional framed cobordism hypothesis as $d\to \infty$. 
    \begin{theorem}\label{stable_cobordism_hypothesis}
        Assume the cobordism hypothesis. Then $L^{\infty\mhy\adj} \FF$ is the infinite cobordism category $\rB^\infty\Bord^{\fr}_\st$ of stably framed categories.
        It is the tensor unit of the monoidal category $\CatSp^{\infty\mhy\adj}$, so for any $A\in \CatSp^{\infty\mhy\adj}$, we have 
        \[\ev_\ast: [\rB^\infty\Bord^{\fr}_\st, A]\xrightarrow{\sim} A. \]
    \end{theorem}
    \begin{remark}
        The last statement is an improved version of the stable cobordism hypothesis; we usually recover only the underlying groupoid, but using the lax internal hom we can now recover the whole category (or the categorical spectrum). An important consequence is that the group $\rO(\infty)$ (or $\mathrm{PL}(\infty)$, which is much closer to $\Aut(\SS)$) acts on the category $A$ itself.
        The cobordism hypothesis with tangential structures implies that, the stable cobordism $\infty$-category $\rB^\infty\Bord^X_{\st}$ is obtained by restricting the action of $\rO$ on $\rB^\infty\Bord^\fr_\st$ along $X\to \rB\rO$ and then taking the colimit. This should be thought of as \emph{categorical Thom spectra}.
        One can also stabilize cobordism categories with singularities in the previous section, giving a categorical lift of the Bass--Sullivan theory.
    \end{remark}

\appendix
\chapter{Steiner's theory for strict \texorpdfstring{$\infty$}{infinity}-categories}\label{appendix_Steiner_theory}
In this appendix, we give a summary of Steiner's theory. It provides an equivalence between a class of strict $\infty$-categories and a class of chain complexes with a kind of positivity structure. 
It is a powerful computational tool in the combinatorics of strict $\infty$-categories, especially when the dualities involved in making a construction functorial are confusing.

\textbf{Only in this appendix, a \emph{category} without specification will mean a $(1, 1)$-category, not an $(\infty, 1)$-category.} References include \cite{steinerOmegacategoriesChainComplexes2004}, \cite{CategoricalCharacterizationSteiner}, \cite{araJointTranchesPour2020}, \cite{ozornovaQuillenAdjunctionGlobular2023}.

\begin{remark}\label{globular definition of strict omega categories}
    A strict $\infty$-category $X$, as defined in \cref{def_strict_omega_categories}, admits the following more explicit description (cf.\ \cref{remark_strcat_as_set_valued_theta_presheaves}): 
    \begin{enumerate}
        \item for each integer $n\geq 0$, a set $X_n$, called the set of $n$-cells,
        \item for each $p > q \geq 0$, the structure of a category with objects $X_q$ and morphisms $X_p$, i.e., 
        \begin{enumerate}
            \item the ($q$-)source and ($q$-)target maps $s_q, t_q: X_p \to X_q$,
            \item the identity map $i_p: X_q\to X_p$,
            \item the composition map $\ast_q: {X_p}_{s_q}{\times_{X_q}}_{t_q} X_p \to X_p$ satisfying associativity and unitality, 
        \end{enumerate}
    \end{enumerate}
    which are compatible in the sense that for $p>q>r\geq 0$, the above data define a $2$-category structure on $(X_p, X_q, X_r)$, i.e., they satisfy the globularity conditions $s_rs_q = s_r = s_rt_q$, $t_rt_q = t_r = t_rs_q$ and the ``interchange law'' $(f\ast_q g)\ast_r (h\ast_q k) = (f\ast_r h)\ast_q (g\ast_r k)$ for $f, g, h, k\in C_p$. The data (1) and (2)(i)(ii) with the globularity condition are precisely the data of reflexive globular sets (i.e., set-valued presheaves on $\GG$), and (iii) is the structure required to extend it to presheaves on $\Theta$ satisfying the Segal conditions. 
\end{remark}       

\begin{remark}
    For any category $\eC$ with finite limits, the obvious modification of the above defines the notion of \emph{strict $\infty$-category objects in $\eC$}, which describes $\eC$-valued presheaves on $\Theta$ satisfying the Segal conditions. 
\end{remark}

\section{Steiner's adjunction}
Steiner's adjunction is between the category of augmented directed complexes and the category of strict $\omega$-categories. The main theorem of Steiner's theory states that it restricts to an adjoint equivalence on the full subcategory of \emph{strong Steiner objects}.
First, we compile the relevant definitions on the chain-complex side: 
\begin{definition}\label{definition_ADC_steiner_complex}
    \begin{enumerate}
        \item An \emph{augmented directed chain complex} (ADC for short) is a triple $(A, A^+, \epsilon) = ((A_\bullet, \partial_\bullet), A_\bullet^+, \epsilon)$, where $A\in \Ch_{\geq 0}(\ZZ)$ is a nonnegatively (homologically) graded chain complex, $\epsilon: A_0\to \ZZ$ is an augmentation, and $A_n^+\subset A_n$ is a sub-$\NN$-module (a.k.a.\ submonoid) for each $n$ (called the \emph{positivity submonoid}; we do \textbf{not} ask that $\partial_n(A_n^+)\subset A_{n-1}^+$). We often omit $\partial_\bullet$, $A^+$, and $\epsilon$ when it is not confusing or is clear from the context.
        \item A map $A\to B$ of ADCs is a chain map $f: A\to B$ that commutes with augmentations and satisfies $f(A^+)\subset B^+$. Let $\adCh$ denote the category of ADCs.
        \item A \emph{basis} of an ADC $(A_\bullet, A_\bullet^+)$ is a graded subset $\{B_q\subset A_q^+\}_{q\geq 0}$ that is both a $\ZZ$-basis of $A$ and an $\NN$-basis of $A^+$. This is unique if it exists\footnote{This is true in general for an $\NN$-module. In fact, any isomorphism $f: \NN^{\oplus I}\to \NN^{\oplus J}$ is induced by a bijection $I\to J$. To show this, note that both $f$ and $f^{-1}$ are injective, so they do not decrease the sum of the coefficients in the standard basis presentation.}, in which case we call it \emph{the} basis of $A$ and moreover make the following definitions: 
        \begin{enumerate}
            \item Any element $a = \sum_{b\in B_q} \lambda_b b\in A_q$ is uniquely a difference $a = a^+-a^-$ with $a^+, a^-\in A^+$. We write $\partial_q^\pm(a) \coloneqq (\partial_q(a))^\pm$. 
            Also, define $\operatorname{supp}(a)\subset B_q$ as the set of $b\in B_q$ with $\lambda_b\neq 0$. 
            \item The basis $\{B_q\}$ is \emph{unital} if for every $q \geq 0$ and $b\in B_q$, we have $\epsilon\circ\partial^+_0\circ\cdots\circ\partial^+_{q-1} (b) = 1 = \epsilon\circ\partial^-_0\circ\cdots\circ\partial^-_{q-1} (b)$.
            \item Consider the preorder on $\bigsqcup_{q\geq 0} B_q$ generated by the relation 
            \[\{(a, b)\mid b\in B_q, a\in \operatorname{supp}(\partial^-_{q-1}b)\}\cup \{(a, b)\mid a\in B_q, b\in \operatorname{supp}(\partial^+_{q-1}a)\}.\] 
            The basis is \emph{strongly loop-free} if this preorder is a partial order. 
        \end{enumerate}
        \item A \emph{strong Steiner complex} is an ADC with a strongly loop-free unital basis.
    \end{enumerate}
\end{definition}

\begin{remark}
    The category $\adCh$ is cocomplete, and colimits can be computed degreewise. More precisely, the forgetful functor $\adCh\to \grMod_{\ZZ}\times \grMod_{\NN}$ creates colimits. 
\end{remark}

\begin{example}
    For any CW complex $X$ with chosen orientations of cells, we regard the augmented cellular chain complex $C_\bullet(X)\to \ZZ$ as an ADC with basis consisting of the cells. 
    In particular, let $D^p\subset \RR^p\subset \RR^\infty$ be the unit $p$-disk with the CW structure $D^p = \Bigl(\bigsqcup_{q\leq p-1} (e_+^q\sqcup e_-^q)\Bigr)\sqcup e^p$, where $e_\pm^q = \{(x_0, \ldots, x_{q-1}, x_q, 0, \ldots, 0)\mid x_0^2+\cdots+x_q^2=1, \pm x_q>0 \}\subset \RR^p$. The cellular chain $\{C_\bullet(D^q)\}_{q\geq 0}$ extends to an $\omega$-category object in $\adCh^\op$; for $p>q\geq 0$,
    \begin{itemize}
        \item the co-source and co-target map $s^q, t^q: C_\bullet(D^q)\to C_\bullet(D^p)$ are induced by the inclusions $D^q\hookrightarrow D^p$ with the images $e_-^q$ and $e_+^q$,
        \item the co-identity map $i^q: C_\bullet(D^p)\to C_\bullet(D^q)$ is induced by the projection $D^p\to D^q$, 
        \item the co-composition map $\ast^p: C_\bullet(D^p)\to \colim (C_\bullet(D^p)\leftarrow C_\bullet(D^q)\rightarrow C_\bullet(D^p)) \cong C_\bullet(D^p\sqcup_{D^q}D^p)$ is induced by the $q$-fold unreduced suspension of the pinch map $D^{p-q}\to D^{p-q}\vee D^{p-q}$. These moreover satisfy the interchange law. 
    \end{itemize}
\end{example}
In other words, $\{C_\bullet(D^q)\}: \GG\to \adCh$ extends to a functor $\Theta\to \adCh$ satisfying the Segal conditions, so the restricted Yoneda embedding $\adCh\to \PSh_{\Set}(\Theta)$, which is right adjoint to the Yoneda extension $\PSh_{\Set}(\Theta)\to \adCh$, factors through the subcategory $\infty\strCat$. 
\begin{definition}
    Steiner's adjunction \[\begin{tikzcd}
        \infty\strCat \arrow[r, shift left = 1ex, "\lambda"]
        & \adCh \arrow[l, shift left = .5ex, "\nu"]
        \arrow[l, phantom, shift right = .2ex, "\scriptscriptstyle\boldsymbol{\bot}"]
    \end{tikzcd}\]
    is the restricted Yoneda extension adjunction of $\{C_\bullet(D^q)\}$.
\end{definition}
\begin{remark}
    One should think of $\lambda$ as the linearization functor: $q$-cells of a strict $\infty$-category $X$ generate $(\lambda X)_q$ and the relations are given by splitting a composition into a sum. The functor $\nu$ is a sort of cellular nerve: for $A\in \adCh$, a $q$-cell of $\nu A$ corresponds to a map $\lambda X_q\to A$. 
    For a detailed explicit description of these functors, see e.g., \cite[\S 2.3]{ozornovaQuillenAdjunctionGlobular2023}.
\end{remark}

We now explain the notion corresponding to strong Steiner complexes on the category side, following \cite{CategoricalCharacterizationSteiner}: 
\begin{definition}\label{definition_polygraph_strong_steiner}
    Let $X$ be a strict $\omega$-category. 
    \begin{enumerate}
        \item A set of cells $E=\bigsqcup_{n\geq 0} E_n$ of $X$, where $E_n \subset X_n$, is a \emph{polygraphic basis} if the following diagram is a pushout for any $n$:
        \[\begin{tikzcd}
            \bigsqcup_{E_q} \partial C_q \ar[d, hook] \ar[r] & X^{\leq q-1} \ar[d, hook]\\
            \bigsqcup_{E_q} C_q \ar[r] & X^{\leq q}
        \end{tikzcd}\]
        $C$ is called a (strict) \emph{polygraph} or a \emph{computad} if it admits a polygraphic basis. 
        \cite{makkaiWordProblem}\cite[Proposition 2.4]{CategoricalCharacterizationSteiner} shows that if a basis exists, it must be the set of nondegenerate indecomposables.
        \item Let $E$ be a polygraphic basis. For $c\in X_q$, define $\operatorname{supp}(c)\subset E_q$ to be the set of factors of $c$. 
        \item \label{preorder on the polygraphic basis} Consider the preorder on $E$ generated by the relation
        \[\bigcup_{p<q}\{(a, b)\in E_p\times E_q \mid a\in \operatorname{supp}(s_p b)\} \cup \bigcup_{p>q}\{(a, b)\in E_p\times E_q \mid b\in \operatorname{supp}(t_q a)\}.\]
        A polygraphic basis $E$ is \emph{strongly loop-free} if the preorder is a partial order.
        \item A strict $\omega$-category is \emph{strong Steiner} if it admits a strongly loop-free polygraphic basis. 
    \end{enumerate}
\end{definition}

\begin{remark}
    Any polygraph is gaunt. Note that the pushout in the definition of polygraph is taken strictly. These do not agree in general; for instance $C_2/\partial C_2$ is strictly $\rB^2\NN$ but weakly $\rB^2\Free_{\EE_2}(\ast)$. 
\end{remark}
\begin{theorem}[\cite{steinerOmegacategoriesChainComplexes2004}, \cite{CategoricalCharacterizationSteiner}]
    The adjunction $\lambda\dashv \nu$ restricts to an equivalence between the category of strong Steiner categories and the category of strong Steiner complexes. Moreover, $\lambda X$ is strong Steiner if and only if $X$ is, and similarly for $\nu$. 
\end{theorem}

\section{Operations on augmented directed complexes}\label{operations_on_ADCs}
Corresponding to the operations on strict $\infty$-categories, the category $\adCh$ has suspension, duality involutions, and a tensor product. A good reference is \cite[\S 1, 2]{ozornovaQuillenAdjunctionGlobular2023}.
\begin{definition}
    The \emph{suspension} functor $\sigma: \adCh\to \adCh$ sends an object
    \[A: \quad\cdots \to A_n\xrightarrow{\partial_n^A} A_{n-1}\to \cdots \to A_0\xrightarrow{\eps^A} \ZZ \] to 
    \[\begin{tikzcd}
        \quad \cdots \ar[r] & (\sigma A)_n\ar[d, equal]\ar[r, "\partial_n^{\sigma A}"] & (\sigma A)_{n-1}\ar[d, equal]\ar[r, "\partial_{n-1}^{\sigma A}"] & \cdots \ar[r]& (\sigma A)_1\ar[d, equal]\ar[r, "\partial_1^{\sigma A}"] & (\sigma A)_0\ar[d, equal]\ar[r, "\eps^{\sigma A}"] & \ZZ\ar[d, equal]\\
        \cdots \ar[r] & A_{n-1}\ar[r, "\partial_{n-1}^A"] & A_{n-2}\ar[r, "\partial^A_{n-2}"] & \cdots \ar[r]& A_0\ar[r, "{\binom{-\eps^A}{\eps^A}}"] & \bot\ZZ\oplus\top\ZZ \ar[r, "{(1, 1)}"] & \ZZ
    \end{tikzcd}\]
    with positivity submonoids $(\sigma A)^+_{n} = A^+_{n-1}$ for $n\geq 1$ and $(\sigma A)^+_0 = \bot\NN\oplus\top \NN$. On morphisms, $\sigma$ assigns $f: A\to B$ to its degree shift $\sigma f: \sigma A\to \sigma B$. 
    The functor $\sigma$ clearly lifts to a colimit-preserving, fully faithful functor $\sigma: \adCh\to \adCh_{\sigma 0/}$.
\end{definition}

\begin{definition}
    For $\tau\in (\ZZ/2)^{\ZZ_{\geq 1}}$, define the \emph{$\tau$-dual} functor $D_\tau: \adCh\to \adCh$ by
    \[(A_n, \partial_n: A_n\to A_{n-1}, \eps, A^+) \mapsto (A_n, (-1)^{\tau(n)}\partial_n, \eps, A^+)\] 
    on objects and as the identity on morphisms, viewed as morphisms of graded abelian groups. When $\tau$ is constantly $1$, we call the $\tau$-dual the \emph{total dual} and denote it by $(\blank)^\circ$. 
    Similarly, when $\tau(n)\equiv n \pmod{2}$ (resp.\ $\tau(n)\equiv n+1 \pmod{2}$) then we call the $\tau$-dual the \emph{odd dual} (resp.\ \emph{even dual}) and denote by $(\blank)^\op$ (resp.\ $(\blank)^\co$). 
\end{definition}
In the following, in order to match the \emph{lax} Gray tensor product of $\infty$-categories, we use the \emph{reverse} of the monoidal structure on $\adCh$ that is standard in the references \cite{steinerOmegacategoriesChainComplexes2004}\cite{ozornovaQuillenAdjunctionGlobular2023}. One can think of the reversed Koszul sign rule as the rule for a differential acting from the right.
\begin{definition}
    \begin{enumerate}
        \item The usual symmetric monoidal structure on $\Ch_{\geq 0}(\ZZ)$ with the Koszul sign rule is equivalent to its \emph{reverse}, i.e., the differential on a tensor product may as well be defined by 
        $\partial(x\otimes y) = (-1)^{\deg y}(\partial x)\otimes y + (x\otimes \partial y)$. We pick this \emph{reversed Koszul sign rule} convention for the tensor product. 
        \item We equip $\grMod_{\NN}$ with the Day convolution symmetric monoidal structure, i.e., $(A^+_\bullet \otimes_\NN B^+_\bullet)_n = \bigoplus_{i+j=n} A^+_i\otimes_{\NN} B^+_j$.
        \item The symmetric monoidal structure on $\Ch_{\geq 0}(\ZZ)$ induces a symmetric monoidal structure on the category $\Ch_{\geq 0}(\ZZ)_{/\ZZ}$ of \emph{augmented} complexes by $\epsilon_{A\otimes B}: (A\otimes B)_0\simeq A_0\otimes B_0\xrightarrow{\epsilon_A\otimes \epsilon_B}\ZZ\otimes \ZZ\simeq \ZZ$.
        \item Let $A = (A, A^+, \epsilon_A)$ and $B = (B, B^+, \epsilon_B)$ be ADCs. We define the tensor product $A\otimes B$ as $(A\otimes_\ZZ B, A^+\otimes_\NN B^+, \epsilon_{A\otimes B})$. This tensor product canonically extends to a monoidal structure.
    \end{enumerate}
\end{definition} 
\begin{remark}
    The positivity structure in $\adCh$ breaks the symmetry of the tensor product: if $A$ and $B$ are both free ADCs on a single basis element in degree $1$, say $a$ and $b$, then the symmetry morphism $a\otimes b\mapsto -b\otimes a$ does not preserve the positivity submonoid.
\end{remark}
Compatibility with the operations on $\infty\strCat$ through Steiner's adjunction is summarized as follows:
\begin{remark}
    \begin{enumerate}
        \item (\cite[Proposition A.3]{araJointTranchesPour2020}) The functors $\sigma$, $(\blank)^\circ$, and $\otimes$ preserves strong Steiner complexes. 
        \item (\cite[Proposition 2.12]{ozornovaQuillenAdjunctionGlobular2023}) There is a natural isomorphism $\nu\sigma \cong \sigma\nu: \adCh\to \infty\strCat$.
        \item (\cite[Proposition 2.19]{araJointTranchesPour2020})
        For any $\tau$, the $\tau$-dual functor naturally commutes with $\lambda$ and $\nu$.
        \item (\cite[Theorem A.15]{araJointTranchesPour2020}, \cite[Proposition 2.14]{ozornovaQuillenAdjunctionGlobular2023})
        There exists a unique biclosed monoidal structure on $\infty\strCat$ satisfying one of the following two equivalent conditions:
        \begin{itemize}
            \item $\nu|_{\adCh^{\Ste}}: \adCh^{\Ste}\to \infty\strCat$ promotes to a monoidal functor (so $X\otimes X'\coloneqq \nu((\lambda X)\otimes (\lambda X'))$ for strong Steiner categories $X, X'$), or
            \item $\lambda: \infty\Cat\to \adCh$ promotes to a monoidal functor, so $(\lambda X)\otimes (\lambda X') \simeq \lambda(X\otimes X')$ for strict $\infty$-categories $X, X'$.
        \end{itemize}
        The monoidal structure is called the \emph{lax Gray tensor product of strict $\infty$-categories} and denoted by $\otimes$ or $\otimes^\lax$.
    \end{enumerate}
\end{remark}

\section{The cubes and the orientals}
In this section, we define important families of strong Steiner categories and investigate their basic combinatorics.
Let $C_\bullet(\Delta^n)$ be the normalized chain complex of the $n$-simplex, i.e., $C_k(\Delta^n) = \bigoplus_{\alpha: [k]\rightarrowtail [n]} \ZZ\alpha$ with differential $\partial \alpha = \sum_{i=0}^k (-1)^i \alpha\circ \delta^i$, where $\delta^i: [k-1]\to [k]$ skips the value $i$.
We give it the structure of an ADC so that $\{\alpha: [k]\rightarrowtail[n]\}$ is a basis. 
In the following, we will notationally identify the nondegenerate simplex $\alpha: [k]\rightarrowtail [n]$ with the subset $\im\alpha \subset [n]$, and the $i$-th vertex is denoted by $\underline{i}$. We let $\cube^1\coloneqq C_1$ be the interval category and denote the degree-$1$ positive generator of $\lambda \cube^1$ by $\underline{?}$, so the complex is $\ZZ\underline{?}\to \ZZ\underline{0}\oplus\ZZ\underline{1}$ with $\partial(\underline{?}) = \underline{1}-\underline{0}$ and $\eps(\underline{0}) = \eps(\underline{1})=1$. 
\begin{definition}
    The \emph{$n$-oriental} is the strict $n$-category $\Ori^n \coloneqq \nu C_\bullet(\Delta^n)$.
    The \emph{$n$-cube} is the strict $n$-category $\cube^n\coloneqq (\cube^1)^{\otimes n} = \nu(C_\bullet(\Delta^1)^{\otimes n})$.
    These are strong Steiner categories (see below).
\end{definition}
We denote the basis elements of $\lambda \cube^n \coloneqq (\lambda \cube^1)^{\otimes n}$ by ``(partially undetermined) binary strings'' $\underline{a_n a_{n-1}\cdots a_1}\coloneqq \underline{a_n}\otimes \underline{a_{n-1}}\otimes\cdots \otimes \underline{a_1}$, where $a_i\in \{?, 0, 1\}$. 
The nontrivial part of checking the strong Steinerness of the ADCs above is to show that the basis is loop-free. For these categories, it turns out that the preorder on the basis is in fact a total order:
\begin{proposition}
    Let $n_1, \ldots, n_k\geq 0$ be integers. The augmented directed chain complex $C_\bullet(\Delta^{n_1})\otimes\cdots \otimes C_\bullet(\Delta^{n_k})$  has a basis, and the preorder of \cref{definition_ADC_steiner_complex} is a total order. 
\end{proposition}
\begin{proof}
    \cite[Examples 3.8 and 3.10]{steinerOmegacategoriesChainComplexes2004} explicate the preorder on the lax cone 
    construction and the tensor product. 
    The preorder on the tensor product is essentially the lexicographic order, twisted by the degree according to the Koszul sign rule. 
    The preorder on the lax (left) cone construction is similar; for a subset of $[n]$, it is a twisted lexicographic order on the indicator function, read as a binary string. Thus both are total orders. 
\end{proof}
To give an idea of ``the lexicographic order twisted by the degree'' in the proof, let us work out the case that we will use in \cref{trivial_automorphisms}.
\begin{lemma}\label{aut_of_cubes_are_trivial}
    $\cube^n$ is a strong Steiner category, and the preorder of \cref{definition_ADC_steiner_complex} on its polygraphic basis is a linear order. In particular, $\Aut(\cube^n)$ is trivial. 
\end{lemma}
\begin{proof}
    The second part follows from the first because any automorphism must preserve the order of the basis. 
    When $n=1$, the order on the basis is the total order $\underline{0}<\underline{?}<\underline{1}$.
    In general, let $\underline{a_n\cdots a_1}$, $\underline{b_n\cdots b_1}$ be two basis elements such that $a_i=b_i$ for $i<k$ and $a_k\neq b_k$. Unwinding the definition, we see that the order on the basis is the ``signed lexicographic order,'' i.e., $\underline{a_n\cdots a_1}<\underline{b_n\cdots b_1}$ if and only if either 
    \begin{itemize}
        \item $a_1, \ldots, a_{k-1}$ contains an even number of $?$ and $a_k < b_k$ as elements of $\{0< ?< 1\}$, or 
        \item $a_1, \ldots, a_{k-1}$ contains an odd number of $?$ and $a_k > b_k$. 
    \end{itemize}
    For example, if $a_1, \ldots, a_{k-1}$ contains even number of $?$ and $a_k=0$, $b_k=?$, then 
    \[\underline{a_n\cdots a_{k+1}0a_{k-1}\cdots a_1}\leq \underline{1\cdots 10a_{k-1}\cdots a_1}< \underline{1\cdots 1 ? a_{k-1}\cdots a_1}\leq \underline{b_n\cdots b_{k+1} ? a_{k-1}\cdots a_1}\]
    by the Koszul sign rule (observe that without the sign, $\underline{1\cdots 1}$ is maximal and $\underline{0\cdots 0}$ is minimal) and $\underline{1\cdots 10a_{k-1}\cdots a_1}\in \operatorname{supp}(\partial^-(\underline{1\cdots 1?a_{k-1}\cdots a_1}))$.
    In particular, two basis elements are always comparable by checking the rightmost different entries, so the preorder is linear. 
\end{proof}
\begin{remark}\label{cube_category_is_gaunt}
    The subcategory $\Gaunt\subset \infty\Cat$ is an exponential ideal for the Cartesian product, so in particular it is self-enriched by the functor category. 
    Using the suspension-hom adjunction, one sees that positive-dimensional cells of a $\Gaunt$-enriched $(\infty, \infty)$-category have a trivial $\infty$-groupoid of automorphisms.
    Therefore, a skeletal subcategory of the category $\Gaunt^{\mathrm{triv-aut}}\subset \Gaunt$ of gaunt $\infty$-categories with trivial automorphisms is gaunt. It follows that $\Cube, \Ori, \Theta$ are gaunt,both as underlying $1$-categories or as $(\infty, \infty)$-categories.
\end{remark}
We end this section by proving some retract relations for these fundamental gaunt $\infty$-categories. 
\begin{proposition}\label{lemma_cylinder_suspension_wedge_of_orientals}
    \begin{enumerate}
        \item The quotient by the full subcategory $\{1\}\otimes \Ori^n\subset \cube^1\otimes \Ori^n$ is isomorphic to $\Ori^{n+1}$, and the quotient map admits a section sending the vertex $\underline{n+1}$ to the vertex $\underline{1}\otimes \underline{n}$. 
        \item The further quotient by $\Ori^{\{0, \ldots n\}}\subset \Ori^{n+1}$ is isomorphic to the unreduced suspension $\sigma \Ori^n$, and the quotient map admits a section. 
        \item The concatenation $\Ori^n\vee \Ori^m$ is a retract of $\Ori^{n+m}$.
    \end{enumerate}
\end{proposition}
\begin{proof}
    \begin{enumerate}
        \item We define the map of Steiner complexes corresponding to $q: \cube^1\otimes \Ori^n\to \Ori^{n+1}$. As a map of graded abelian groups, let
        \[\underline{1}\otimes \alpha \mapsto \begin{cases}
            \underline{n+1} &\quad \text{if $\deg(\alpha) = 0$}\\
            0 & \quad \text{if $\deg(\alpha)>0$}
        \end{cases},\quad \underline{0}\otimes \alpha \mapsto \alpha,\quad \underline{?}\otimes \alpha\mapsto \alpha\sqcup \{n+1\}.  \]
        It is straightforward to check that the map defined is indeed a map of augmented directed chain complexes.
        Note that the following square commutes and is biCartesian in $\Ch(\ZZ)$, i.e., it is a short exact sequence in $\Ch(\ZZ)_{\ZZ//\ZZ}$ with the basepoint given by the sink vertex: 
        \[\begin{tikzcd}
            \lambda(\{1\}\otimes \Ori^n) \ar[d]\ar[r, hook] & \lambda(\cube^1\otimes \Ori^n)\ar[d, two heads]\\
            \ZZ\ar[r, "\underline{n+1}"] & \lambda(\Ori^{n+1})
        \end{tikzcd}\]
        Thus we have a short exact sequence of the reduced complexes
        \[\tilde{\lambda}(\{1\}\otimes \Ori^n)\to \tilde{\lambda}(\cube^1\otimes \Ori^n) \to \tilde{\lambda}(\Ori^{n+1}).\]
        Now the first map admits an obvious retraction, so the second map admits a section; this splitting is given by the disjoint union decomposition of the standard basis, so the section is again a map of (reduced) ADCs. Adding back the basepoint factor, we obtain the section $\lambda(\Ori^{n+1})\to \lambda(\cube^1\otimes \Ori^n)$. 
        \item The argument is similar to the first part; define the map $\lambda(\Ori^{n+1})\twoheadrightarrow \lambda(\sigma\Ori^n)$ of ADCs by
        \[\alpha\mapsto\begin{cases}
            \sigma(\alpha\setminus \{n+1\}) &\quad \text{if $\deg(\alpha)>0$, $n+1\in\alpha$} \\
            0 & \quad \text{if $\deg(\alpha)>0$, $n+1\not\in\alpha$}
        \end{cases}, \quad \underline{0}, \ldots, \underline{n}\mapsto \bot,\quad \underline{n+1}\mapsto \top. 
        \]
        Then the following is a short exact sequence in $\Ch(\ZZ)_{\ZZ//\ZZ}$ with the basepoint given by the source vertex: 
        \[\begin{tikzcd}
            \lambda(\Ori^{\{0, \ldots, n\}})\ar[r, hook]\ar[d] & \lambda(\Ori^{n+1})\ar[d, two heads]\\
            \ZZ\ar[r, "\bot"] & \lambda(\sigma\Ori^n)
        \end{tikzcd}\]
        The rest is the same as in (1). 
        \item The degree-$k$ part of the ADC $\lambda(\Ori^n\vee\Ori^m)$ is generated by injections $[k]\to [n+m]$ with the image contained either in $\{0,\ldots, n\}$ or $\{n, \ldots, n+m\}$, so there is an obvious inclusion $s: \lambda(\Ori^n\vee\Ori^m) \hookrightarrow \lambda(\Ori^{n+m})$. 
        The retraction $r$ is the one that exhibits $\Delta^n\vee\Delta^m$ as a simple deformation retract of $\Delta^{n+m}$, which we now explicitly describe. 
        For $\alpha\subset\{0, \ldots, n+m\}$, let $\alpha_L\coloneqq \alpha\cap \{0, \ldots, n-1\}$ and $\alpha_R\coloneqq \alpha\cap\{n+1, \ldots, n+m\}$. 
        Define a degree-$1$ map $h: \lambda(\Ori^{n+m})\to \lambda(\Ori^{n+m})$ of graded abelian groups by
        \[\alpha \mapsto \begin{cases}
            (-1)^{\#\alpha_L} (\alpha \cup \{n\}) & \text{if $\alpha_L, \alpha_R\neq \emptyset$, $n\not\in\alpha$,}
            \\
            0 & \text{otherwise}
        \end{cases}.\]
        Observe that $r = \id-\partial  h-h\partial: \lambda(\Ori^{n+m})\to \lambda(\Ori^{n+m})$ is a chain map. Moreover, we have
        \begin{itemize}
            \item if $\alpha_L$ or $\alpha_R$ is empty, then $h \alpha = h\partial \alpha = 0$, so $r\alpha = \alpha$,
            \item if $\alpha_L, \alpha_R \neq \emptyset$ and $n\not \in \alpha$, then $\partial h\alpha = \alpha +h\partial \alpha$, so $r\alpha = 0$,
            \item if $\alpha_L, \alpha_R\neq \emptyset$ and $n\in \alpha$, then $h\alpha = 0$ and $h \partial \alpha = h ((-1)^{\#\alpha_L}\alpha\circ \delta^{\# \alpha_L}) = \alpha$, so $r\alpha = 0$. 
        \end{itemize}
        Therefore, $r$ is a map of ADCs that lands in $\lambda(\Ori^n\vee \Ori^m)$ and satisfies $rs=\id$. 
    \end{enumerate}
\end{proof}
The following density results are proven by \cite{campionCubesAreDense2022} for the cubes and \cite{gepnerOrientedSimplicialSpacesInpreparation} for the orientals.
\begin{corollary}\label{cubes_and_orientals_are_dense}
    \begin{enumerate}
        \item The $n$-oriental $\Ori^n$ is a retract of the $n$-cube $\cube^n$.
        \item We have inclusion of the idempotent completions $\Theta\subset \widetilde{\Ori}\subset \widetilde{\Cube}$. In particular, the orientals and cubes are dense in $\infty\Algbrd$. 
    \end{enumerate}
\end{corollary}
\begin{proof}
    \begin{enumerate}
        \item This follows by inductively applying (1) of the proposition. Although it is notationally heavier, it is also not too hard to provide the section and retraction directly: the retraction collapses each face of the form $\underline{0\cdots 0 1 ?\cdots ?}$ (possibly with no $\underline{0}$). The section $\Ori^n\hookrightarrow\cube^n$ sends the simplex $\{i_0<\cdots<i_k\}: [k]\rightarrowtail [n]$ to the composition of the faces of $\cube^n$ of the form $\underline{a_n\cdots a_1}$ where:
        \begin{itemize}
            \item The possibly empty substrings $\underline{a_n\cdots a_{i_{k}+1}}$ and $\underline{a_{i_0}\cdots a_1}$ are $\underline{0\cdots 0}$ and $\underline{1\cdots 1}$, respectively.
            \item For each $0\leq j<k$, the substring $\underline{a_{i_{j+1}}\cdots a_{i_j+1}}$ is of the form $\underline{1\cdots 1?0\cdots 0}$, with exactly one $\underline{?}$ and possibly no $\underline{0}$ or $\underline{1}$. 
        \end{itemize}
        When $n=3$, the picture of ``coarsening'' the cube to the oriental is the following (observe that the directions of $2$- and $3$-cells also align):
        
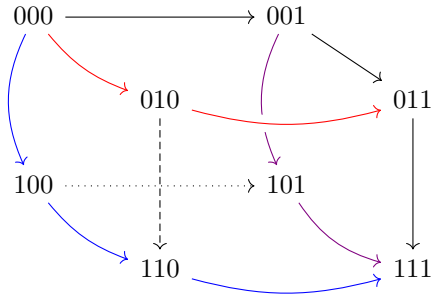
\[\begin{tikzcd}
            000 \arrow[rr] \arrow[rd, bend right=15, red] \arrow[dd, bend right=25, blue] & & 001 \arrow[rd] \arrow[dd, bend right=25, violet] & \\
            & 010 \arrow[rr, bend right=15, crossing over, red] \arrow[dd, dashed] & & 011 \arrow[dd] \\
            100 \arrow[rd, bend right=15, blue] \arrow[rr, dotted] & & 101 \arrow[rd, bend right=20, violet] &\\
            & 110 \arrow[rr, bend right=15, blue] & & 111
        \end{tikzcd}\]
        \item The second inclusion follows from (1). The first inclusion uses the strategy of \cite{campionCubesAreDense2022}; by our definition of $\Theta$ (\cref{def_cells_and_Theta}), it reduces to (2) and (3) of the proposition. 
    \end{enumerate}
\end{proof}

    \printbibliography
    \end{document}